\documentclass[11pt,reqno]{amsart}
\usepackage{a4wide}
\numberwithin{equation}{section}
\usepackage{mathrsfs}
\usepackage{amsfonts}
\usepackage{amsmath}
\usepackage{stmaryrd}
\usepackage{amssymb}
\usepackage{amsthm}
\usepackage{mathrsfs}
\usepackage{url}
\usepackage{amsfonts}
\usepackage{amscd}
\usepackage{indentfirst}
\usepackage{enumerate}
\usepackage{amsmath,amsfonts,amssymb,amsthm}
\usepackage{amsmath,amssymb,amsthm,amscd}
\usepackage{graphicx,mathrsfs}
\usepackage{appendix}

\usepackage[numbers,sort&compress]{natbib}
\usepackage{color}
 \usepackage[colorlinks, linkcolor=black, citecolor=blue]{hyperref}

\allowdisplaybreaks[4]

\newcommand\N{\mathbb N}

\newcommand\R{\mathbb R}

\newcommand\mbb\mathbb
\newcommand\mbf\mathbf
\newcommand\mcal\mathcal
\newcommand\mfrak\mathfrak
\newcommand\mrm\mathrm
\newcommand\msf\mathsf
\renewcommand\a\alpha
\renewcommand\b\beta
\newcommand\g\gamma
\newcommand\G\Gamma
\renewcommand\d\delta
\newcommand\D\Delta
\newcommand\e\varepsilon
\newcommand\z\zeta
\renewcommand\t\theta
\newcommand\Th\Theta
\newcommand\la\lambda
\newcommand\La\Lambda
\newcommand\s\sigma
\newcommand\si\varsigma
\newcommand\Si\Sigma
\newcommand\ups\upsilon
\newcommand\U\Upsilon
\newcommand\ph\varphi
\renewcommand\o\omega
\renewcommand\O\Omega
\newcommand\wt\widetilde
\newcommand\wh\widehat
\newcommand\ol\overline
\newcommand\ul\underline
\newcommand\mr\mathring
\newcommand\ub\underbrace
\newcommand\pa\partial
\newcommand\n\nabla
\newcommand\fa\forall
\newcommand\ex\exists
\newcommand\es\emptyset
\newcommand\wk\rightharpoonup
\newcommand\inc\hookrightarrow
\newcommand\linf\varliminf
\newcommand\lsup\varlimsup
\newcommand\os\overset
\newcommand\us\underset
\newcommand\sr\stackrel
\newcommand\Ot\Leftarrow
\newcommand\To\Rightarrow
\newcommand\map\mapsto
\newcommand\ot\leftarrow
\newcommand\lot\longleftarrow
\newcommand\lto\longrightarrow
\newcommand\tot\leftrightarrow
\newcommand\ltot\longleftrightarrow
\newcommand\sm\backslash
\renewcommand\Cup\bigcup
\renewcommand\Cap\bigcap
\newcommand\sub\subset
\newcommand\Sub\Subset
\newcommand\sne\subsetneq
\newcommand\bus\supset
\newcommand\Bus\Supset

\newcommand\eq\equiv
\newcommand\ox\otimes
\newcommand\Ox\bigotimes
\newcommand\pl\oplus
\newcommand\Pl\bigoplus
\newcommand\x\times
\renewcommand\c\circ
\newcommand\q\quad
\renewcommand\l\left
\renewcommand\r\right
\newcommand\fr\frac

\usepackage[numbers,sort&compress]{natbib}
\usepackage[dvipsnames]{xcolor}
\usepackage{color}

\definecolor{bondiblue}{rgb}{0.0, 0.58, 0.71}
\def\sideremark#1{\ifvmode\leavevmode\fi\vadjust{\vbox to0pt{\vss
			\hbox to 0pt{\hskip\hsize\hskip1em
				\vbox{\hsize2.1cm\tiny\raggedright\pretolerance10000
					\noindent #1\hfill}\hss}\vbox to15pt{\vfil}\vss}}}%

\newtheorem{Thm}{Theorem}[section]
\newtheorem{Lem}[Thm]{Lemma}

\newtheorem{Prop}[Thm]{Proposition}

\newtheorem{Rem}[Thm]{Remark}

\begin{document}

\title[Moser-Trudinger problem]
{Qualitative analysis for Moser-Trudinger nonlinearities with a low energy}
\author[P. Luo, K. Pan and S. Peng]{Peng Luo, Kefan Pan, Shuangjie Peng}

\address[Peng Luo]{School of Mathematics and Statistics $\&$ Hubei Key Laboratory of Mathematical Sciences, Central China
Normal University,
Wuhan, 430079, P. R. China}
 \email{pluo@ccnu.edu.cn}

\address[Shuangjie Peng]{School of Mathematics and Statistics $\&$ Hubei Key Laboratory of Mathematical Sciences, Central China
Normal University, Wuhan, 430079, P. R. China}\email{sjpeng@ccnu.edu.cn}

\address[Kefan Pan]{School of Mathematics and Statistics, Central China Normal University, Wuhan 430079, China}\email{kfpan@mails.ccnu.edu.cn}

\begin{abstract}
We are concerned with the Moser-Trudinger problem
\begin{equation*}
\begin{cases}
-\Delta u=\lambda ue^{u^2}~~&\mbox{in}~\Omega,\\[0.5mm]
 u>0 ~~ &{\text{in}~\Omega},\\[0.5mm]
u=0~~&\mbox{on}~\partial \Omega,
\end{cases}
\end{equation*}
where $\Omega\subset \R^2$ is a smooth bounded domain and $\lambda>0$ is sufficiently small.
Qualitative analysis for Moser-Trudinger nonlinearities has been studied in recent decades, however there is still a lot of clarity about this issue, even for a low energy.
The reason is that this problem is a critical exponent for dimension two  and will lose compactness.
Here by using a variety of local Pohozaev identities, we qualitatively analyze the positive solutions of Moser-Trudinger problem with a low energy, which contains the Morse index, non-degeneracy, asymptotic behavior, uniqueness and symmetry of solutions. Since the fundamental solution of $-\Delta$ in $\Omega \subset \mathbb{R}^2$ is in logarithmic form and the corresponding bubble is exponential growth, more precise asymptotic behavior of the solutions is needed, which is of independent interest.
Moreover, to obtain our results, some ODE's theory will be used to a prior estimate of the solutions and some elliptic theory in dimension two will play a crucial role.
\end{abstract}

\maketitle
{\small
\keywords {\noindent {\bf Keywords:} {\small Moser-Trudinger problem, Morse index
and non-degeneracy, asymptotic behavior, uniqueness and symmetry}
\smallskip
\newline
\subjclass{\noindent {\bf 2020 Mathematics Subject Classification:} 35A01 $\cdot$ 35B25 $\cdot$ 35J20 $\cdot$ 35J60}
}
\tableofcontents
\section{Introduction and main results}
\setcounter{equation}{0}
In this paper, we consider the following Moser-Trudinger problem
\begin{equation}\label{1.1}
\begin{cases}
-\Delta u=\lambda ue^{u^2}~~&\mbox{in}~\Omega,\\[0.5mm]
 u>0~~  &{\text{in}~\Omega},\\[0.5mm]
u=0~~&\mbox{on}~\partial \Omega,
\end{cases}
\end{equation}
where $\Omega\subset \R^2$ is a smooth bounded domain and $\lambda>0$ is sufficiently small.

\vskip 0.1cm

Problem \eqref{1.1} is related to the Euler-Lagrange equation for the functional
\begin{equation}\label{Ju-1.1}
J_{\lambda}(u)=\frac{1}{2}\int_{\Omega}|\nabla u|^2 \,dx -\frac{\lambda}{2} \int_{\Omega} e^{u^2} \,dx,~u \in H_0^1(\Omega),
\end{equation}
where $H_0^1(\Omega)$ is the usual Hilbert space with norm $\|u\|_{H^1_0(\Omega)}= (\int_{\Omega}|\nabla u|^2dx)^{\frac{1}{2}}$. The functional $J_{\lambda}(u)$ in \eqref{Ju-1.1} is associated to the critical Trudinger embedding (see \cite{Tru1967}): $e^{u^2}\in L^p(\Omega)$ for all $p \geq 1$ whenever
$u\in H_0^1(\Omega)$, which was sharpened by J. Moser \cite{Moser1971} in the following form
\begin{equation*}
\sup_{u\in H^1_0(\Omega), \|u\|^2_{H^1_0(\Omega)}\leq 4\pi}\int_{\Omega}\big(e^{u^2}-1\big) \,dx
\leq C|\Omega|,
\end{equation*}
and
\begin{equation}\label{Ju-1.1q}
\sup_{u\in H^1_0(\Omega), \|u\|^2_{H^1_0(\Omega)}\leq 4\pi+\delta}\int_{\Omega}\big(e^{u^2}-1\big) \,dx
= +\infty, ~~~~\mbox{for every}~~~~\delta>0.
\end{equation}
Since then, an amount of work has been devoted to the study of the functional
\begin{equation*}
E(u):=\int_{\Omega}e^{u^2}dx,~~~~~~~\,\,\, u\in H^1_0(\Omega),
\end{equation*}
and in particular of its critical points. Clearly, $u\equiv 0$ is the only global minimum point of $E$ and we can not find a global maximizer of $E$ due to \eqref{Ju-1.1q}. Hence it is interesting and hopeful to find a maximizer of $E|_{M_{\Lambda}}$, i.e. of $E$ constrained to the manifold
\begin{equation*}
E|_{M_{\Lambda}}:=\Big\{u\in H^1_0(\Omega):~~\|u\|^2_{H^1_0(\Omega)}=\Lambda\Big\}.
\end{equation*}
To find the critical point of above constrained problem, a lot of mathematicians have achieved a series of original results. An early result is \cite{CC1986}, in which Carleson and Chang proved that $E|_{M_{4\pi}}$ has a maximizer when $\Omega=B_1(0)$. From then on the critical points of $E|_{M_{\Lambda}}$ \big($\Lambda\in (0,4\pi)$, $\Lambda=4\pi$ and $\Lambda\in (4\pi,\infty)$\big) has been widely studied. It is impossible to mention all the contributions here just in these few lines, one can refer to \cite{DMR2012,F1992,LRS2009,MM2014,MM2017,Struwe2000} and the references therein for more details.

\vskip 0.1cm

We note that the critical point of $E|_{M_{\Lambda}}$ solves
\begin{equation}\label{1.1a}
\begin{cases}
-\Delta u=\lambda ue^{u^2}~~&\mbox{in}~\Omega,\\[0.5mm]
u=0~~&\mbox{on}~\partial \Omega,\\[0.5mm]
\|u\|^2_{H^1_0(\Omega)}=\Lambda.
\end{cases}
\end{equation}
From \cite{AD2004}, we know that as $\lambda\to 0$,
 $u_\lambda$ solves \eqref{1.1} with $J_\lambda(u_\lambda)\to  2\pi$  means that
 $u_\lambda$ solves \eqref{1.1a} with some $\Lambda=\Lambda_\lambda\to 4\pi$.
Also for high energy, the similar relationship still holds by using the results in \cite{DT2020}.
Hence considering the solutions of problem \eqref{1.1} will be useful for us to find the critical points of $E|_{M_{\Lambda}}$, and another natural  topic is to understand the qualitative properties of solution to \eqref{1.1} as $\lambda\to 0$. In this direction, a lot of work has already been done. An early and interesting work is \cite{Dru2006}, in which Adimurthi and Druet studied the blow-up behavior of sequences of solutions to problem \eqref{1.1}. To describe Adimurthi and Druet's results,  we introduce some notations and recall the known asymptotic characterization. We denote by $G(x,\cdot)$ the Green's function
 of $-\Delta$ in $\Omega$, i.e. the solution to
\begin{equation}\label{greensyst}
\begin{cases}
-\Delta G(x,\cdot)= \delta_x  &{\text{in}~\Omega}, \\[1mm]
G(x,\cdot)=0  &{\text{on}~\partial\Omega},
\end{cases}
\end{equation}
where $\delta_x$ is the Dirac function. We have the following well known decomposition formula of $G(x,y)$,
\begin{equation}\label{GreenS-H}
G(x,y)=S(x,y)-H(x,y)  ~\mbox{for}~(x,y)\in \Omega\times \Omega,
\end{equation}
where  $S(x,y):=-\frac{1}{2\pi}\log |x-y|$ and $H(x,y)$ is the regular part of $G(x,y)$. Furthermore, for any $x\in \Omega$, the Robin function is defined as
\begin{equation}\label{Robinf}
\mathcal{R}(x):=H(x,x).
\end{equation}
As in \cite{DIP2017}, the function
\begin{equation}\label{defU}
U(x)= -\log \Big(1+\frac{|x|^2}{4}\Big)
\end{equation}
is a solution of the Liouville equation
\begin{equation*}
\begin{cases}
-\Delta U=e^{2U} \,\,~\mbox{in}~\R^2,\\[1mm]
U \leq 0, \,\,~U(0)=0,\\[1mm]
\displaystyle\int_{\R^2}e^{2U}dx=4\pi.
\end{cases}
\end{equation*}
Now we state the conclusions about the asymptotic behavior of the solutions to problem \eqref{1.1}.

\vskip 0.2cm

\noindent\textbf{Theorem A(c.f. \cite{AD2004})} Let $u_\lambda$ be a family of solutions to  \eqref{1.1} satisfying $\displaystyle\limsup_{\lambda\to 0}J_{\lambda}(u_\lambda)\leq 2\pi$,  $x_\lambda$ be a point where $u_\lambda$ achieves its maximum. Then the following properties hold after passing to a subsequence, as $\lambda \to 0$:

\vskip 0.3cm

\noindent \textup{(1)}. $\|u_{\lambda}\|_2\to 0$, $\|\nabla u_\lambda\|^2_2\to 4\pi$.

\vskip 0.1cm

\noindent\textup{(2)}. $\gamma_\lambda:=u_\lambda(x_\lambda)\to \infty$.

\vskip 0.1cm

\noindent\textup{(3)}.  Set
\begin{equation}\label{defw}
w_{\lambda}(y):=\gamma_\lambda \Big(u_{\lambda}(x_{\lambda}+\theta_{\lambda}y)-
\gamma_\lambda \Big),~y\in \Omega_{\lambda}:=\frac{\Omega-x_{\lambda}}{\theta_\lambda},
\end{equation}
where $\theta_{\lambda}:=\big(\lambda \gamma^2_\lambda e^{\gamma_\lambda^2}\big)^{-\frac{1}{2}} \to 0$ as $\lambda \to 0$. One has
\begin{equation}\label{limw}
\lim_{\lambda \rightarrow 0} w_{\lambda}=U\,\,~\mbox{in}~C^2_{loc}(\R^2).
\end{equation}

\vskip 0.1cm

\noindent\textup{(4)}. The sequence of maxima $\{x_\lambda\}$ converges to $x_0\notin \partial \Omega$, which satisfies  $\nabla \mathcal{R}(x_0)=0$. Moreover,
\begin{equation}\label{lim-1}
\lim_{\lambda \rightarrow 0} \gamma_{\lambda} u_{\lambda}(x)=4\pi G(x,x_0)\,\,~\mbox{in} ~ C^2_{loc}\big(\Omega\backslash \{x_0\}\big).
\end{equation}
\vskip 0.1cm

\noindent\textup{(5)}. It holds
\begin{equation}\label{lim-2}
\lim_{\lambda \rightarrow 0} \lambda \gamma_{\lambda} \int_{\Omega} u_{\lambda}(x)e^{u^2_{\lambda}(x)}dx =4\pi~~~~\mbox{and}~~~~
\lim_{\lambda \rightarrow 0}  \lambda  \gamma^2_{\lambda} \int_{\Omega} e^{u^2_{\lambda}(x)}dx = 4\pi.
\end{equation}
\vskip 0.2cm

Here Theorem A describes the one-bubble case. Concerning multi-bubble solutions to problem \eqref{1.1}, Druet and Thizy \cite{DT2020} recently established the following asymptotic behavior.

\vskip 0.2cm

\noindent\textbf{Theorem B(c.f. \cite{Dru2006,DT2020})} Let $u_\lambda$ be a family of solutions to  \eqref{1.1} satisfying $\displaystyle\limsup_{\lambda\to 0}J_{\lambda}(u_\lambda)<+\infty$, then
$u_\lambda \rightharpoonup 0$ weakly in $H^1_0(\Omega)$. Moreover there exist $k\geq 1$ such that
\begin{equation*}
\int_{\Omega}|\nabla u_\lambda|^2 \,dx\to 4k\pi\,\,\mbox{as}\,\, \lambda\to 0
\end{equation*}
and $k$ sequences $\{x_{i,\lambda}\}$ of points in $\Omega$ such that

\vskip 0.3cm

\noindent \textup{(1)}. $x_{i,\lambda}$ is the local maximum point of $u_\lambda(x)$ satisfying
$x_{i,\lambda}\to x_{i}$ as $\lambda\to 0$ with $x_i\in\Omega$, $x_i\neq x_j$, $i,j=1,\cdots,k$ if $k>1$.

\vskip 0.1cm

\noindent\textup{(2)}. $u_\lambda\to 0$ in $C^2_{loc}( \overline{\Omega}\backslash \mathcal{S})$
where $\mathcal{S}=\{x_i\}^k_{i=1}$.
\vskip 0.1cm

\noindent\textup{(3)}.
For all $i=1,\cdots,k$, $\gamma_{i,\lambda}=u_\lambda(x_{i,\lambda})\to +\infty$ as $\lambda\to 0$ and
\begin{equation*}
\gamma_{i,\lambda}\Big(u_\lambda(x_{i,\lambda}+\theta_{i,\lambda}x)- \gamma_{i,\lambda}\Big) \to U(x),~~~~~\mbox{in}~~~~~C^2_{loc}(\R^2)~~~~\mbox{as}~~~~~\lambda\to 0,
\end{equation*}
where $\theta_{i,\lambda}^{-2}:=\lambda\gamma_{i,\lambda}^2e^{\gamma_{i,\lambda}^2}\to +\infty$
as $\lambda\to 0$.

\vskip 0.1cm

\noindent\textup{(4)}. For all $i=1,\cdots,k$, there exists $m_i>0$ such that
\begin{equation}\label{DT-lam_gam}
\sqrt{\lambda}\gamma_{i,\lambda}\to \frac{2}{m_i}~~~~\mbox{as}~~~~~\lambda\to 0.
\end{equation}

\vskip 0.1cm

\noindent\textup{(5)}. $\big(x_1,\cdots,x_k,m_1,\cdots,m_k\big)$ is the critical point of
\begin{equation}\label{9-13-1a}
\Phi\big(y_1,\cdots,y_k,\alpha_1,\cdots,\alpha_k\big)=
\sum^k_{i=1}\alpha^2_i\mathcal{R}(y_i) - \sum_{i\neq j} \alpha_i\alpha_j G(y_i,y_j) +\frac{1}{2}\sum^k_{i=1}\big(\alpha^2_i-\alpha^2_i\ln \alpha_i\big).
\end{equation}

\vskip 0.1cm

Here we point out that from Theorem B and the proofs in \cite{DT2020}, we know that
\begin{equation}\label{9-13-1}
\frac{u_\lambda(x)}{\sqrt{\lambda}}\to 2\pi\sum^k_{i=1}m_iG(x_i,x)
~~~~~~\mbox{in} ~~~~~~C^2_{loc}\Big(\overline{\Omega}\backslash \{x_i\}^k_{i=1}\Big)~~~~~~\mbox{as} ~~~~~~\lambda \to 0.
\end{equation}
In the direction of existence,  del Pino, Musso and Ruf \cite{DMR2010} have proved that
 problem \eqref{1.1} has  a multi-bubble solution satisfying  \eqref{9-13-1} by Lyapunov-Schmidt reduction,
which improved the precision of the result in \cite{Dru2006}. More precisely, for a given $k\geq 1$, they provided the conditions under which problem \eqref{1.1} admits a family of $k$-peak solutions $u_\lambda$ satisfying
\[
J_{\lambda}(u_\lambda) = 2k\pi + o(1),~~~\mbox{for any $\lambda>0$ small.}
\]
Function \eqref{9-13-1a} also appeared in \cite{DMR2010}, which is called Kirchhoff-Routh function and play a crucial role in many problems, such as plasma problem \cite{CF1980,CPY2010}, Brezis-Nirenberg problem \cite{BN1983,CLP2021,MP2002} and so on. As the study of properties of the critical points to Kirchhoff-Routh function is not the aim of this paper, we just refer to \cite{BP2015,BPW2010,CPY2010,GT2010} and the references therein.

\vskip 0.1cm

In this paper, we intend to continue the qualitative analysis of solutions to \eqref{1.1}, which mainly contains  the Morse index, non-degeneracy, asymptotic behavior, uniqueness and symmetry.

\vskip 0.1cm

Our first result concerns the Morse index of $u_\lambda$.
It is known that the Morse index gives a strong qualitative information on the solutions, such as
non-degeneracy, uniqueness, symmetries, singularities, nodal sets as well as classifying solutions.
See \cite{BL1992,DGIP2019,F2007,P2002,Y1998} and the references therein.
To state the results in this direction, we first introduce some notations and definitions. The Morse and augmented Morse index of a solution $u_\lambda$ to problem \eqref{1.1} can be defined as
\begin{equation*}
\begin{cases}
m(u_\lambda):=\sharp \big\{k\in \N: \mu_{\lambda,k}<1\big\},\\[2mm]
m_0(u_\lambda):=\sharp \big\{k\in \N: \mu_{\lambda,k}\leq 1 \big\},
\end{cases}
\end{equation*}
where $\mu_{\lambda,1}<\mu_{\lambda,2}\leq \mu_{\lambda,3}\leq \cdots$ is the sequence of eigenvalues for the linearized  problem
\begin{equation}\label{eigen}
\begin{cases}
-\Delta v=\mu_{\lambda,k} \lambda \big(1+2u^2_\lambda\big)e^{u^2_\lambda} v &~\mbox{in}~\Omega,\\
v=0 &~\mbox{on}~\partial \Omega.
\end{cases}
\end{equation}
We can see that when $\mu_{\lambda,k} =1$, the space generated by the solutions $v$ to \eqref{eigen} is the kernel of the linearized operator at $u_\lambda$. Hence $u_\lambda$ is degenerate if and only if $\mu_{\lambda,k}=1$ for some $k \in \N$. Therefore, it is obvious that $m(u_\lambda)= m_0(u_\lambda)$ for any non-degenerate solution $u_\lambda$.

\vskip 0.1cm

For any function $f\in C^2(\Omega)$, let $a \in \Omega$ be a critical point of $f$, we denote by $m\big(f(a)\big)$ and $m_0\big(f(a)\big)$ the Morse and augmented Morse index of $f$ at $a$, that is:
\begin{equation*}
\begin{cases}
m\big(f(a)\big):=\sharp \Big\{l\in \{1,2\}: \mu_l<0 \Big\},\\[2mm]
m_0\big(f(a)\big):=\sharp \Big\{l\in \{1,2\}: \mu_l \leq 0 \Big\},
\end{cases}
\end{equation*}
where $\mu_1\leq \mu_2$ are the eigenvalues of the hessian matrix $D^2 f(a)$.

\vskip 0.1cm

The computation  of Morse index of solutions is a very broad subject.
Since \eqref{1.1} is a two-dimensional problem, we just mentioned some results  on
the Morse index of blow-up solutions to some classical two-dimensional problems.

\vskip 0.2cm

\noindent \textup{(I)} Gelfand problem
\begin{equation}\label{9-13-2}
\begin{cases}
-\Delta u=\lambda e^{u}~~&\mbox{in}~\Omega,\\[0.5mm]
 u>0 ~~ &{\text{in}~\Omega},\\[0.5mm]
u=0~~&\mbox{on}~\partial \Omega,
\end{cases}
\end{equation}
where $\Omega\subset \R^2$ is a smooth bounded domain and $\lambda>0$ is sufficiently small.
In \cite{GGOS2014}, Gladiali, Grossi, Ohtsuka and Suzuki computed the Morse index of multi-peak solutions to \eqref{9-13-2} by using the asymptotic estimates of eigenvalues to the eigenvalue problem associated with Gelfand problem. Moreover, they concluded that there is a relationship between the Morse index of the solutions and the associated critical points of the Kirchhoff-Routh function. For further results on the blow-up solutions to Gelfand problem
\eqref{9-13-2}, one can refer to \cite{BJLY2019,GG2004,GOS2011} and the references therein.

\vskip 0.2cm

\noindent \textup{(II)} Lane-Emden problem
\begin{equation}\label{9-13-3}
\begin{cases}
-\Delta u=u^p~~&\mbox{in}~\Omega,\\[0.5mm]
 u>0 ~~ &{\text{in}~\Omega},\\[0.5mm]
u=0~~&\mbox{on}~\partial \Omega,
\end{cases}
\end{equation}
where $\Omega\subset \R^2$ is a smooth bounded domain and $p$ is sufficiently large.
De Marchis, Grossi, Ianni and Pacella \cite{DGIP2019} also calculated the Morse index of $1$-spike solutions to \eqref{9-13-3} by studying the asymptotic behavior of the solutions, eigenvalues and its corresponding eigenfunctions. Recently, Ianni, Luo and Yan \cite{ILY} generalized above results into the $k$-spike ($k\geq 2$) case. For further results on the blow-up solutions to Lane-Emden problem
\eqref{9-13-3}, one can refer to \cite{DGIP2019,DIP2015,RW1994}  and the references therein. For some results on the Morse index of blow-up solutions to related high-dimensional problems (for example: nearly critical problem, Brezis-Nirenberg problem), we can refer to \cite{BLR1995,GP2005,CKL2016}.

In this paper, we intend to calculate the Morse index of the low-energy sequences of solutions to Moser-Trudinger problem. Our first result is as follows:
\begin{Thm}\label{th_Morse}
Let $u_\lambda$ be a family of solutions to problem \eqref{1.1} satisfying $J_\lambda(u_\lambda)\leq 2\pi$, then there exists $\lambda_0>0$ such that
\begin{equation*}
1\leq 1+m\big(\mathcal{R}(x_0)\big)\leq m(u_\lambda)\leq m_0(u_\lambda)\leq 1+m_0\big(\mathcal{R}(x_0)\big)\leq 3, \,\,~\mbox{for any}~ \lambda\in (0,\lambda_0).
\end{equation*}
Furthermore, if the hessian matrix $D^2 \mathcal{R}(x_0)$ of the Robin function $\mathcal{R}(x)$ at the point $x_0$ is non-singular, then it holds
\begin{equation*}
m(u_\lambda)=1+m\big(\mathcal{R}(x_0)\big)\in [1,3], \,\,~\mbox{for any}~ \lambda\in (0,\lambda_0).
\end{equation*}
\end{Thm}

The proof of Theorem \ref{th_Morse} is based on analyzing the asymptotic behavior of eigenvalues and eigenfunctions of the linearized problem \eqref{eigen} at $u_\lambda$.
The precise asymptotic behavior, as $\lambda\to 0$, of the eigenvalues $\mu_{\lambda,i}$ and the eigenfunctions $v_{\lambda,i}$ to problem \eqref{eigen} for $i=2,3,4$ is described in the following theorem.

\begin{Thm}\label{th_eigenvalue}
Under the same assumptions of Theorem \ref{th_Morse}, one has, as $\lambda\to 0$,
\begin{equation}\label{iden-mu23}
\mu_{\lambda,l}=1 + 12\pi \theta^2_{\lambda}\Lambda_{l-1}+o\big(\theta^2_{\lambda}\big),
\end{equation}
where $l=2,3$, $\Lambda_1\leq \Lambda_2$ are the eigenvalues of the hessian matrix $D^2 \mathcal{R}(x_{0})$.
For
the fourth eigenfunction and eigenvalue, we have following estimate
\begin{equation}\label{9-13-20}
\begin{cases}
\displaystyle v_{\lambda,4}=-\frac{4\pi b}{\gamma^2_\lambda}G(x,x_\lambda)+o\Big(\frac{1}{\gamma^2_\lambda}\Big), ~~\,~~\mbox{in}~~~~C^1_{loc}\big( \overline{\Omega} \backslash \{x_{\lambda}\}\big)~~~~\,~~~ \mbox{for some}~~~~b\in \R \backslash \{0\},\\[4mm]
\displaystyle \mu_{\lambda,4}=1+\frac{3}{\gamma_\lambda^4}+ o\Big(\frac{1}{\gamma_\lambda^4}\Big).
\end{cases}
\end{equation}

\end{Thm}

Here we point out that Moser-Trudinger type equation contains exponential nonlinearity $u e^{u^2}$, which is critical and requires more accurate estimates. Specifically, for Gelfand problem, Lane-Emden problem and Moser-Trudinger problem, we can easily find that the first eigenvalue of the related linearized problem tends to some constant smaller than $1$. The second and third eigenvalues will tend to $1$, and whether these eigenvalues are greater or less than $1$ will depend on the eigenvalues of the hessian matrix of $\mathcal{R}(x)$ at the critical point. To compute the fourth eigenvalue (here we focus on  $1$-spike solutions), some calculations will be different.
For Gelfand problem, the fourth eigenfunction and eigenvalue possess following estimate
\begin{equation*}
\begin{cases}
\displaystyle v_{\lambda,4}=\frac{8\pi b}{\log \lambda}G(x,x_\lambda)+o\Big(\frac{1}{|\log \lambda|}\Big), ~~\,~~\mbox{in}~~~~C^1_{loc}\big(\overline{\Omega} \backslash \{x_{\lambda}\}\big)~~~~\,~~~ \mbox{for some}~~~~b\neq 0,\\[4mm]
\displaystyle \mu_{\lambda,4}=1-\frac{3}{2\log \lambda}+ o\Big(\frac{1}{|\log \lambda|}\Big).
\end{cases}
\end{equation*}
Also  for Lane-Emden problem,
 the fourth eigenfunction and eigenvalue possess following estimate
\begin{equation*}
\begin{cases}
\displaystyle v_{p,4}=-\frac{8\pi b}{p}G(x,x_p)+o\Big(\frac{1}{p}\Big), ~~\,~~\mbox{in}~~~~C^1_{loc} \big(\overline{\Omega} \backslash \{x_{p}\}\big)~~~~\,~~~ \mbox{for some}~~~~b\neq 0,\\[4mm]
\displaystyle \mu_{p,4}=1+\frac{6}{p}+ o\Big(\frac{1}{p}\Big).
\end{cases}
\end{equation*}
Thus we know that the main part of two terms $v_{\lambda,4}$ $(v_{p,4})$ and $\mu_{\lambda,4}-1$ $(\mu_{p,4}-1)$ in
Gelfand problem or Lane-Emden problem has the same order. However, this does not hold in Moser-Trudinger problem.
In fact, we know from \eqref{9-13-20} that the main part of $\mu_{\lambda,4}-1$ in Moser-Trudinger problem is a higher order term of $v_{\lambda,4}$. Hence to prove that the main part of $\mu_{\lambda,4}-1$ is positive, we need some more computations
and prior estimates to prove \eqref{3.5.2} below. Also we point out that the corresponding prior estimates are established by some ODE's theory.

\vskip 0.1cm

Using Theorem \ref{th_Morse}, we have the following result concerning the non-degeneracy.
\begin{Thm}\label{th_nondeg}
Let   $u_\lambda$ be a solution of \eqref{1.1} satisfying $J_\lambda(u_\lambda)\leq 2\pi$ and
 $\xi_\lambda\in H^1_0(\Omega)$ be a solution of $\mathcal{L}_\lambda\big(\xi_\lambda\big)=0$, where
$$\mathcal{L}_\lambda\big(\xi\big):= -\Delta \xi - \lambda \big(1+2u^2_\lambda\big)e^{u^2_\lambda}\xi$$
is the linearized operator of problem \eqref{1.1} at the solution $u_{\lambda}$.
Suppose that $x_{0}$ is a non-degenerate critical point of Robin function $\mathcal{R}(x)$, then there exists $\lambda_0>0$ such that
	\[\xi_\lambda \equiv 0, \,\,~\mbox{for any}~~ \lambda\in (0,\lambda_0).\]
\end{Thm}

\begin{Rem}
A usual method to  prove the non-degeneracy of the solutions is blow-up argument based on local Pohozaev identities, one can refer to \cite{G2005,GG2004,GILY2022,GLPY2021,GMPY2020,GMPY2022} and references therein for more details. But here we verify the non-degeneracy of the solutions to problem \eqref{1.1} by proving that each eigenvalue of the eigenvalue problem \eqref{eigen} is not equal to $1$. It is easy to see that this conclusion can be directly obtained after estimating the eigenvalues of problem \eqref{eigen} and this method greatly reduces the tedious calculations. Detailed proof is given in subsection \ref{s4.2}.
\end{Rem}

Another main result in this paper is the following local uniqueness result in any general smooth bounded domain $\Omega \subset \R^2$.
\begin{Thm}\label{th_uniq}
Let $u_\lambda^{(1)}$ and $u_\lambda^{(2)}$ be two solutions to \eqref{1.1} with
\begin{equation}\label{energy1}
J_{\lambda}(u_\lambda^{(l)})\leq 2\pi,~\,~\mbox{for}\,~l=1,2,
\end{equation}
which concentrate at the same  critical point $x_{0}\in\Omega$ of Robin function $\mathcal{R}(x)$.
If $x_{0}$ is non-degenerate, then  there exists  $\lambda_0>0$ such that
\[u_\lambda^{(1)}\equiv u_\lambda^{(2)}, \,\,~\mbox{for any}~~ \lambda\in (0,\lambda_0).\]
\end{Thm}

In recent years, the local uniqueness of concentrated solutions for elliptic problems has been extensively investigated. One can prove the local uniqueness of concentrated solutions either by counting the local degree of the corresponding finite dimensional problem as in \cite{CH2003,CNY1998,G1993} or by using local Pohozaev identities as in \cite{BJLY2019,CLL2015,DLY2015,GPY2017}. Local Pohozaev identities only involve first order derivatives, which simplifies the estimates and provides a good tool to deal with the local uniqueness problem for concentrated solutions. In the proof of Theorem \ref{th_uniq}, we mainly use the local Pohozaev identities. However, we point out that since the fundamental solution $S(x,y)$ of $-\Delta$ in $\Omega \subset \R^2$ is in logarithmic form, we can not directly use Proposition \ref{prop_eta_lambda} below or comparison principle to prove that the following conclusion is valid for any $R>0$,
\begin{equation}\label{etalambdas}
\eta_\lambda:=\frac{u_\lambda^{(1)}-u_\lambda^{(2)}}{\|u_\lambda^{(1)}-u_\lambda^{(2)}\|_{L^{\infty}(\Omega)}} = o(1)~\mbox{uniformly~in}~\Omega \backslash B_{R \theta_\lambda^{(1)}}(x_\lambda^{(1)}),
\end{equation}
where the definition of $x_\lambda^{(1)}$ and $\theta_\lambda^{(1)}$ can be seen in Section \ref{s6}. Hence, we adopt a new method to overcome this obstacle. More precisely, we divide $\Omega \backslash B_{R \theta_\lambda^{(1)}}
(x_\lambda^{(1)})$ into three parts, and then use ODE's theory and the properties of Green's function to prove that \eqref{etalambdas} holds for any $R>0$.

\vskip 0.1cm

In addition, we state another main difficulty to prove local uniqueness.
Compared with Gelfand problem and Lane-Emden problem, to prove $\widetilde{A}_\lambda=o\big(\lambda\big)$ as in Proposition \ref{prop-A} below, we need estimate two sides of Pohozaev identity \eqref{p1_ueta} up to the order $\displaystyle \frac{1}{\big(\gamma^{(1)}_\lambda\big)^5}$ (see \eqref{A_tilde3} below), which requires much more precise expansions on $\gamma_\lambda$ and $|x_\lambda-x_0|$. These will be given in next theorem.

\begin{Thm}\label{thm-lambdagamma}
Under the conditions in Theorem A above, we have that
for any fixed small $d>0$,
\begin{equation}\label{4.9.1}
u_\lambda(x)=  C_{\lambda}G(x_{\lambda},x) + \sum_{i=1}^2 \frac{2\pi c_i}{\gamma^3_\lambda} \theta_\lambda \partial_i G(x_\lambda,x) + o\Big(\frac{\theta_\lambda}{\gamma^3_\lambda}\Big) \,\,~\mbox{in}~C^1\big(\Omega \backslash   B_{2d}(x_{\lambda})\big),
\end{equation}
and
\begin{equation}\label{xlambda-x0}
x_\lambda-x_0 = -\frac{\theta_\lambda}{2 \gamma_\lambda^2} (c_1,c_2) +o\Big(\frac{\theta_\lambda}{\gamma^2_\lambda}\Big),
\end{equation}
for some $(c_1,c_2)\in \R^2$  and
\begin{equation}\label{def_C1}
C_{\lambda}  = \frac{4\pi}{\gamma_\lambda} + \frac{4\pi}{\gamma_\lambda^3} + o\Big(\frac{1}{\gamma^3_\lambda}\Big).
\end{equation}
Furthermore, it holds
\begin{equation}\label{lambda_gamma}
\sqrt{\lambda}\gamma_\lambda= e^{\frac{B_1}{2}} + \frac{B_2}{2} e^{-\frac{B_1}{2}} \lambda + \frac{4 B_3-3 B_2^2}{8} e^{-\frac{3 B_1}{2}} \lambda^2 +o\big(\lambda^2\big),
\end{equation}
where $\gamma_\lambda:=u_\lambda(x_\lambda)$, $B_1$, $B_2$ and $B_3$ are some constants given later.
\end{Thm}

\begin{Rem}
To obtain local uniqueness by using a variety of Pohozaev identities, the estimates of $u_\lambda$ and $|x_\lambda-x_0|$
obtained in Section \ref{s2} are not accurate enough. Besides, using \eqref{DT-lam_gam}, we know that there exists $m_1>0$ such that
$$ \sqrt{\lambda} \gamma_\lambda \to \frac{2}{m_1}~\mbox{as}~\lambda \to 0,$$
which indicates that
\begin{equation}\label{lim_theta}
\theta_\lambda =O\Big(e^{-\frac{1}{2}\gamma_\lambda^2}\Big).
\end{equation}
Note that \eqref{lim_theta} is essential and will be used throughout the paper, but it is not accurate enough to prove the local uniqueness result.
The estimates in Theorem \ref{thm-lambdagamma} are important and necessary, which are also of independent interest.
And to get estimate \eqref{lambda_gamma}, an amount of ODE's theorems will be used.
The concrete values of $B_1$, $B_2$ and $B_3$ are not important. Here we just use \eqref{lambda_gamma} to get
\begin{equation*}
 \gamma^{(1)}_\lambda- \gamma^{(2)}_\lambda= o\big(\lambda^\frac{3}{2}\big),
\end{equation*}
where $\gamma^{(1)}_\lambda$  and $\gamma^{(2)}_\lambda$ are the maximum values of two solutions $u_{\lambda}^{(1)}$ and $u_{\lambda}^{(2)}$.
\end{Rem}
In \cite{DGIP2019,GILY2022}, the authors proved that the Lane-Emden problem has a unique solution in any smooth bounded and convex domain $\Omega \subset \R^2$ for $p>1$ sufficiently large. Also in \cite{GG2004}, Gladiali and Grossi confirmed that Gelfand  problem has only one solution which blows up at the origin when $\Omega \subset \R^2$ is a bounded domain which is convex in the direction $x_1,x_2$ and symmetric with respect to the axis $x_i = 0$ for $i = 1,2$.
Therefore, it is natural to ask whether the similar uniqueness result can be obtained without assuming that $u_\lambda$ is a single-peak solution concentrated at a non-degenerate critical point of Robin function. Now we focus on the case where $\Omega$ is a convex domain and give the following results concerning uniqueness and symmetry.

\begin{Thm}\label{th_uniq2}
Let $\Omega$ be a smooth bounded and convex domain in $\R^2$. Then there exists $\lambda_0>0$ such that any solution $u_\lambda$ satisfying
\begin{equation}\label{9-13-11}
\limsup_{\lambda \to 0} J_\lambda(u_\lambda) < \infty
\end{equation}
is unique for any $\lambda \in (0,\lambda_0)$.
\vskip 0.1cm

Furthermore, if $\Omega$ is also symmetric with respect to $x_1,x_2$, then for any $\lambda \in (0,\lambda_0)$, it holds
\begin{equation*}
u_\lambda(x_1,x_2)=u_\lambda(x_1,-x_2)=u_\lambda(-x_1,x_2),
\end{equation*}
and
\begin{equation*}
u_\lambda(0)=\max_{x\in \Omega}u_\lambda(x).
\end{equation*}
\end{Thm}
\begin{Rem}
Condition \eqref{9-13-11} is necessary in Theorem \ref{th_uniq2}. For Lane-Emden problem, if $\Omega$ is convex, then
the energy is bounded by \cite{KS2018}.
Also from Malchiodi and Martinazzi's result in \cite{MM2014}, we know that
condition \eqref{9-13-11} is true when $\Omega=B_1(0)$. However, whether
\eqref{9-13-11} holds for Moser-Trudinger problem when $\Omega$ is convex is interesting but unknown.

\vskip 0.1cm

Here we also mention that
the nodal solutions for a large perturbation of the
Moser-Trudinger equation will be also interesting, one can refer to \cite{GMNP2021,GN2021,MT2022,N2021} and references therein.
\end{Rem}

Now we give another application of the local uniqueness  above to get an exact number of solutions to Moser-Trudinger problem  for some non-convex domains.

\begin{Thm}\label{th_uniq2d}
Let $\Omega$ be a smooth bounded and convex domain in $\R^2$, $\Omega_\e=\Omega\backslash B_\e(P)$ with $P\in \Omega$ and $\nabla \mathcal{R}_{\Omega}(P)\neq 0$, then there exist $\lambda_0,\e_0>0$ such that
for any $\lambda \in (0,\lambda_0)$ and $\e\in (0,\e_0)$,
the number of solutions $u_{\lambda,\e}$ to problem
\begin{equation*}
\begin{cases}
-\Delta u=\lambda ue^{u^2}~~&\mbox{in}~\Omega_\e,\\[0.5mm]
 u>0~~  &{\text{in}~\Omega_\e},\\[0.5mm]
u=0~~&\mbox{on}~\partial \Omega_\e,
\end{cases}
\end{equation*}
 satisfying
$\displaystyle\limsup_{\lambda \to 0, \e\to 0} J_\lambda(u_{\lambda,\e}) \leq 2\pi$ is exactly two.
\end{Thm}
\begin{Rem}
Theorem \ref{th_uniq2d} is  directly verified by Theorem \ref{th_uniq} and the number of critical points of $\mathcal{R}_{\Omega_\e}(x)$ obtained by Gladiali-Grossi-Luo-Yan \cite{GGLY2022}.
\end{Rem}

At the end of this section, we give some explanations why we just consider $1$-spike case. Or in other word, for multi-spike case, what is the main difficulty. From Theorem B above, we know
$\gamma_{i,\lambda}\approx \frac{2}{\sqrt{\lambda}m_i}$. Then to find $0<\frac{1}{C}\leq \frac{\theta_{i,\lambda}}{\theta_{j,\lambda}}\leq C$, by the definition of $\theta_{i,\lambda}$, we need
\begin{equation}\label{9-15-1}
0<\frac{1}{C}\leq e^{\gamma^2_{i,\lambda}-\gamma^2_{j,\lambda}} \leq C.
\end{equation}
To get \eqref{9-15-1}, a necessary condition is $m_i=m_j$ for any $1\leq i,j\leq k$.
This would be a luxury since $m_i$ is determined by the critical points of Kirchhoff-Routh function \eqref{9-13-1a}. Now we pay our attention to $k=2$ and give some domains, on which $m_1\neq m_2$. Our main idea is to deduce a contradiction by assuming $m_1=m_2$ and some computations on the critical points of Kirchhoff-Routh function \eqref{9-13-1a}.
To be specific, if $m_1=m_2$ and $\Omega=B_1(0) \backslash B_\e(0)$, we have, for $i=1,2$,
\begin{equation}\label{9-15-2}
\nabla \mathcal{R}_\Omega(x^{(i)}_{\e})-2 \sum_{j\neq i} \nabla_yG_\Omega(x^{(j)}_{\e},x^{(i)}_{\e})=0,
\end{equation}
and
\begin{equation}\label{9-15-3}
\mathcal{R}_\Omega(x^{(i)}_{\e})-\sum_{j\neq i} G_\Omega(x^{(j)}_{\e},x^{(i)}_{\e})=\frac{m_i\big(2\ln m_i-1\big)}{4\pi}.
\end{equation}
We denote $\Psi_\Omega(x,y)=
\mathcal{R}_\Omega(x)+\mathcal{R}_{\Omega}(y)-2 G_{\Omega}(x,y)$, then from
  \eqref{9-15-2}, we deduce that
$(x^{(1)}_{\e},x^{(2)}_{\e})$ must be the critical point of $\Psi_\Omega(x,y)$. Also
\eqref{9-15-3} gives us that
\begin{equation}\label{9-15-4}
\mathcal{R}_\Omega(x^{(1)}_{\e})=\mathcal{R}_\Omega(x^{(2)}_{\e}).
\end{equation}

\vskip 0.1cm

On the other hand, using Gladiali-Grossi-Luo-Yan's results on Kirchhoff-Routh function \cite{GGLY2023}, we know that only following two cases may occur:
\vskip 0.2cm

Case (1) $\big(x^{(1)}_{\e},x^{(2)}_{\e}\big)\to (0,Q)$ or $(Q,0)$ for some $Q\neq 0$.
\vskip 0.2cm

Case (2) $\big(x^{(1)}_{\e},x^{(2)}_{\e}\big)\to (0,0)$. In this case we know that $|x^{(1)}_{\e}|, |x^{(2)}_{\e}|\sim \e^{\frac{3}{4}}$.
\vskip 0.2cm

If Case (1) holds, by \eqref{9-15-4} and Gladiali-Grossi-Luo-Yan's results in \cite{GGLY2023}, then we have
$$\mathcal{R}_{B_1(0)}(0)= \lim_{\e\to0 }\mathcal{R}_\Omega(x^{(1)}_{\e})=\lim_{\e\to0 } \mathcal{R}_\Omega(x^{(2)}_{\e})=\mathcal{R}_{B_1(0)}(Q),$$
which is impossible.
\vskip 0.1cm

If Case (2) holds, then
\begin{equation*}
\underbrace{\mathcal{R}_\Omega(x^{(1)}_{\e})}_{=O(1)}- \underbrace{G_\Omega(x^{(1)}_{\e},x^{(2)}_{\e})}_{\to +\infty}=\underbrace{\frac{m_1\big(2\ln m_1-1\big)}{4\pi}}_{=O(1)},
\end{equation*}
which is also impossible.
\vskip 0.1cm

Hence from above discussions, we prove that $m_1\neq m_2$. Hence for two-spike solutions on $\Omega=B_1(0)\backslash B_\e(0)$, the corresponding two bubbles have different order, which gives us some crucial difficulty to analyze the qualitative properties  for Moser-Trudinger nonlinearities with multi-spike solutions. To our knowledge, these will need some more tools and be a very challenging subject.

\vskip 0.1cm

The paper is organized as follows. In Section \ref{s2}, we give some basic estimates on $u_\lambda$ and $w_\lambda$. Two essential quadratic forms $P(u,v)$ and $Q(u,v)$ are also given in this section. In Section \ref{s3}, we are committed to studying the more accurate expansion of $w_\lambda$, which will be used throughout the paper. In Section \ref{s4}, we study the properties concerning the eigenvalue $\mu_{\lambda,l}$ and its corresponding eigenfunction $v_{\lambda,l}$ to the linearized problem \eqref{eigen}. By analyzing the asymptotic behavior of eigenvalues and eigenfunctions, we complete the proof of Theorem \ref{th_Morse}. Furthermore, Theorem \ref{th_nondeg} can be derived directly from Theorem \ref{th_Morse}. In Section \ref{s5}, we prove Theorem \ref{thm-lambdagamma}, in other words, we give some more precise estimates of $u_\lambda$, $x_\lambda$ and $\gamma_\lambda$, which is crucial to prove the local uniqueness of the solution $u_\lambda$. The proofs of Theorem \ref{th_uniq} and Theorem \ref{th_uniq2} are contained in Section \ref{s6}. Finally, the proofs of Proposition \ref{prop_klambda} and Proposition \ref{prop_slambda} related to the properties of $k_\lambda$ and $s_\lambda$ are postponed in the Appendix.


\section{Some known facts and basic estimates} \label{s2}

\subsection{Some properties and computations on $U$}~
\vskip 0.2cm

Firstly, we review a crucial property of the kernel of the linearized operator at $U$ of the Liouville equation obtained in \cite{EG2004} and give a known important result.
\begin{Lem}\label{lem3.1}
Let $U$ be the function defined in \eqref{defU} and $v\in C^2(\R^2)$ be a solution of the following problem
\begin{equation*}
\begin{cases}
-\Delta v=2e^{2U}v\,\,~\mbox{in}~\R^2,\\[1mm]
\displaystyle\int_{\R^2}|\nabla v|^2dx<\infty.
\end{cases}
\end{equation*}
Then it holds
\begin{equation*}
v(x)= c_0\frac{4-|x|^2}{4+|x|^2}+\sum^2_{i=1}{c_i}\frac{x_i}{4+|x|^2}
\end{equation*}
with some $c_0,c_1,c_2\in \R$.
\end{Lem}

\begin{Lem}\label{add-lem1}\cite[Lemma 2.1]{EMP2006}
Let $f \in C^1([0,+\infty))$ such that $\displaystyle\int_0^{+\infty} t |\log t||f(t)| \,dt < +\infty$. There exists a $C^2$ radial solution $w(r)$ of equation
\begin{equation*}
\Delta w+\frac{8}{\big(1+|y|^2\big)^2}w=f(|y|)~\mbox{in}~\R^2,
\end{equation*}
such that as $r \to +\infty$,
\begin{equation*}
w(r)=\bigg( \int_0^{+\infty} t \frac{t^2-1}{t^2+1} f(t) \,dt\bigg) \log r + O\bigg( \int_r^{+\infty} s |\log s||f(s)| \,ds +\frac{|\log r|}{r^2} \bigg),
\end{equation*}
and
\begin{equation*}
\partial_r w(r)=\bigg( \int_0^{+\infty} t \frac{t^2-1}{t^2+1} f(t) \,dt\bigg)\frac{1}{r} + O\bigg(\frac{1}{r} \int_r^{+\infty} s |f(s)| \,ds +\frac{|\log r|}{r^3} \bigg).
\end{equation*}
\end{Lem}

\vskip 0.2cm

Next Lemma gives some calculations which will be used throughout the paper.

\begin{Lem}
We have
\begin{equation}\label{equa2}
\int_{\R^2} e^{2U(x)} \frac{4-|x|^2}{4+|x|^2} \,dx = 0,\quad \int_{\R^2} e^{2U(x)} \frac{|x|^2}{4+|x|^2} \,dx = 2\pi,
\end{equation}
and
\begin{equation}\label{equa1}
\int_{\R^2} U(x) e^{2U(x)} \frac{4-|x|^2}{4+|x|^2} \,dx = 2\pi,\quad \int_{\R^2} U^2(x) e^{2U(x)} \frac{4-|x|^2}{4+|x|^2} \,dx = -6\pi.
\end{equation}
\end{Lem}
\begin{proof}
Using integration by parts, we can easily calculate the above results. Here we omit it.

\end{proof}

\subsection{Two quadratic forms}~
\vskip 0.2cm

Recall that $x_\lambda$ is the local maximum point of the solution $u_\lambda$.
Let us define the following two quadratic forms
\begin{equation}\label{P}
\begin{split}
P(u,v):=&- 2d\int_{\partial B_d(x_{\lambda})}\big\langle \nabla u ,\nu\big\rangle \big\langle \nabla v,\nu\big\rangle \,d\sigma + d \int_{\partial B_d(x_{\lambda})} \big\langle \nabla u , \nabla v \big\rangle\,d\sigma,
\end{split}
\end{equation}
and
\begin{equation}\label{Q}
Q(u,v):=- \int_{\partial B_d(x_{\lambda})}\frac{\partial v}{\partial \nu}\frac{\partial u}{\partial x_i}\,d\sigma-
\int_{\partial B_d(x_{\lambda})}\frac{\partial u}{\partial \nu}\frac{\partial v}{\partial x_i}\,d\sigma
+ \int_{\partial B_d(x_{\lambda})}\big\langle \nabla u,\nabla v \big\rangle \nu_i\,d\sigma,
\end{equation}
where $u,v\in C^{2}(\overline{\Omega})$, $d>0$ is a small constant such that $B_{2d}(x_{\lambda})\subset\Omega$ and $ \nu= ( \nu_1, \nu_2)$ is the unit outward normal of $\partial B_d(x_{\lambda})$.

\vskip 0.2cm

Now we have the following properties about the above two quadratic forms.
\begin{Lem}\label{indep_d}
If $u$ and $v$ are harmonic in $ B_d(x_{\lambda})\backslash \{x_{\lambda}\}$, then $P(u,v)$ and $Q(u,v)$ are independent of $\theta \in (0,d]$.
\end{Lem}
\begin{proof}Let $\Omega'\subset \Omega$ be a smooth bounded domain, integrating by parts, we have
\begin{equation*}
\begin{split}
& -\int_{\Omega'}\big\langle \nabla v,x-x_\lambda\big\rangle \Delta u \,dx -\int_{\Omega'}\big\langle \nabla u,x-x_\lambda\big\rangle \Delta v \,dx \\
=& -\int_{\partial \Omega'} \frac{\partial u}{\partial \nu}\big\langle \nabla v,x-x_\lambda\big\rangle \,d\sigma
-\int_{\partial \Omega'} \frac{\partial v}{\partial \nu}\big\langle \nabla u,x-x_\lambda\big\rangle \,d\sigma
+\int_{\partial \Omega'} \big\langle \nabla u, \nabla v \big \rangle  \big\langle  \nu, x-x_\lambda \big\rangle \,d\sigma.
\end{split}
\end{equation*}
Since $u$ and $v$ are harmonic in $ B_d(x_\lambda)\backslash \{x_\lambda\}$, let $0<\theta_{1}<\theta_{2}<d$ and $\Omega'=B_{\theta_{2}}(x_\lambda)\backslash B_{\theta_{1}}(x_\lambda)$, the above identity can be written as
\begin{equation*}
-\int_{\partial \Omega'} \frac{\partial u}{\partial \nu}\big\langle \nabla v,x-x_\lambda\big\rangle \,d\sigma
-\int_{\partial \Omega'} \frac{\partial v}{\partial \nu}\big\langle \nabla u,x-x_\lambda\big\rangle \,d\sigma
+\int_{\partial \Omega'} \big\langle \nabla u,  \nabla  v \big \rangle  \big\langle  \nu, x-x_\lambda \big\rangle \,d\sigma =0.
\end{equation*}
Therefore, we derive that
\begin{align*}
&-2\theta_{2}\int_{\partial B_{\theta_{2}}(x_\lambda)} \frac{\partial u}{\partial \nu}\frac{\partial v}{\partial\nu} \,d\sigma +\theta_{2}\int_{\partial B_{\theta_{2}}(x_\lambda)} \big\langle \nabla u,  \nabla  v\big\rangle \,d\sigma \\
= &-2\theta_{1}\int_{\partial B_{\theta_{1}}(x_\lambda)} \frac{\partial u}{\partial \nu}\frac{\partial  v}{\partial\nu} \,d\sigma +\theta_{1}\int_{\partial B_{\theta_{1}}(x_\lambda)} \big\langle \nabla u,  \nabla  v\big\rangle \,d\sigma,
\end{align*}
which indicates that $P(u,v)$ is independent of $\theta\in (0,d]$.

\vskip 0.1cm

On the other hand, integrating by parts, we get
\begin{equation*}
\begin{split}
&-\int_{\Omega'} \Big(\Delta u \frac{\partial v}{\partial x_i}  +\Delta v \frac{\partial u}{\partial x_i}\Big) \,dx \\
=& -\int_{\partial \Omega'} \Big(\frac{\partial v}{\partial \nu} \frac{\partial u}{\partial x_i} +\frac{\partial u}
{\partial \nu} \frac{\partial v}{\partial x_i}\Big) \,d\sigma +\int_{\Omega'} \bigg(\Big\langle \nabla v, \nabla \frac{\partial u}{\partial x_i}\Big\rangle +\Big\langle \nabla u, \nabla \frac{\partial v}{\partial x_i}\Big\rangle\bigg) \,dx \\
=& -\int_{\partial \Omega'} \Big(\frac{\partial v}{\partial \nu} \frac{\partial u}{\partial x_i}
+\frac{\partial u}{\partial \nu} \frac{\partial v}{\partial x_i}\Big) \,d\sigma +\int_{\partial\Omega'}\big\langle \nabla u,\nabla v \big\rangle \nu_i \,d\sigma.
\end{split}
\end{equation*}
Since $u$ and $v$ are harmonic in $ B_d(x_\lambda)\backslash \{x_\lambda\}$, then arguing like before, we get that $Q(u,v)$ is  independent of $\theta\in (0,d]$.
\end{proof}

Using the above result, we derive the following key computations about the quadratic forms $P$ and $Q$ on Green's function.

\begin{Prop}\label{prop2-1}
It holds
\begin{equation}\label{pgg}
P \big(G(x_{\lambda},x), G(x_{\lambda},x)\big)=
-\frac{1}{2\pi},
\end{equation}
and
\begin{equation}\label{pggh}
P \big(G(x_{\lambda},x),\partial_hG(x_{\lambda},x)\big)=
-\frac{1}{2}\frac{\partial \mathcal{R}(x_{\lambda})}{\partial x_h}.
\end{equation}
Moreover
\begin{equation}\label{qgg}
Q \big(G(x_{\lambda},x),G(x_{\lambda},x)\big)=
-\frac{\partial \mathcal{R}(x_{\lambda})}{\partial{x_i}},
\end{equation}
and
\begin{equation}\label{qggh}
Q \big(G(x_{\lambda},x),\partial_h G(x_{\lambda},x)\big)=
- \frac{1}{2} \frac{\partial^2 \mathcal{R}(x_{\lambda})}{\partial{x_i} \partial{x_h} },
\end{equation}
where $G(x,y)$ and $\mathcal{R}(x)$ are the Green function and Robin function, respectively (see \eqref{greensyst}, \eqref{GreenS-H}, \eqref{Robinf}), $\partial_iG(y,x):=\frac{\partial G(y,x)}{\partial y_i}$ and
$D_{x_i}G(y,x):=\frac{\partial G(y,x)}{\partial x_i}$.
\end{Prop}
\begin{proof}
See \cite[Proposition 2.4]{GILY2022}. Note that there are some misprints about the coefficients in \cite[Proposition 2.4]{GILY2022} and we correct them here.
\end{proof}

Next, we have the following identities about the quadratic forms $P$ and $Q$ on the solution $u_\lambda$.
\begin{Lem}\label{lem2.3}
Let $u_\lambda\in C^2(\Omega)$ be a solution of \eqref{1.1} and $d>0$ is a fixed small constant such that $B_{2d}(x_{\lambda})\subset\Omega$. Then
\begin{equation}\label{puu}
P\big(u_\lambda,u_\lambda\big)=  d\lambda  \int_{\partial B_{d}(x_{\lambda})}e^{u^2_\lambda}  \,d\sigma-2\lambda \int_{B_{d}(x_{\lambda})} e^{u^2_\lambda} \,dx,
\end{equation}
and
\begin{equation}\label{quu}
Q\big(u_\lambda,u_\lambda\big)= \lambda \int_{\partial B_{d}(x_{\lambda})} e^{u^2_\lambda} \nu_i \,d\sigma.
\end{equation}
\end{Lem}
\begin{proof}
Multiplying $\big\langle x-x_{\lambda}, \nabla u_\lambda \big\rangle$ on both sides of \eqref{1.1} and integrating on $B_{d}(x_{\lambda})$, we have
\[\begin{split}
&\frac{1}{2}\int_{\partial B_{d}(x_{\lambda})} \big\langle x-x_{\lambda},\nu\big\rangle |\nabla u_\lambda|^2 \,d\sigma
-\int_{\partial B_{d}(x_{\lambda})}\frac{\partial u_\lambda}{\partial\nu} \big\langle x-x_{\lambda}, \nabla u_\lambda \big\rangle \,d\sigma \\
=&\frac{\lambda}{2}  \int_{\partial B_{d}(x_{\lambda})}e^{u^2_\lambda} \big\langle x-x_{\lambda},\nu\big\rangle \,d\sigma - \lambda \int_{B_{d}(x_{\lambda})} e^{u^2_\lambda} \,dx,
\end{split}
\]
which together with \eqref{P} implies \eqref{puu}.
Similarly, multiplying $\displaystyle \frac{\partial u_\lambda }{\partial x_i}$ on both sides of \eqref{1.1} and integrating on $B_{d}(x_{\lambda})$, we get
\begin{equation}\label{2.3-1}
\begin{split}
-\int_{\partial B_{d}(x_{\lambda})}\frac{\partial u_\lambda}{\partial \nu}\frac{\partial u_\lambda}{\partial x_i} \,d\sigma +\frac{1}{2}\int_{\partial B_{d}(x_{\lambda})}|\nabla u_\lambda|^2\nu_i \,d\sigma = \frac{\lambda}{2} \int_{\partial B_{d}(x_{\lambda})} e^{u^2_\lambda} \nu_i \,d\sigma.
\end{split}
\end{equation}
Then \eqref{quu} can be deduced from \eqref{2.3-1}.
\end{proof}

\vskip 0.2cm

\subsection{The asymptotic behavior of $w_\lambda$ and $u_\lambda$}~
\vskip 0.2cm

Let $w_\lambda$ be the function defined in \eqref{defw}. We will establish a prior estimate on $w_\lambda$.
\begin{Lem}\label{lem2.4}
For any $\varepsilon \in (0,2)$, there exist $R_\e>1$, $C_\e,~\widetilde{C}_\varepsilon>0$ and $\lambda_\e >0$ such that for any $\lambda\in (0,\lambda_\e)$
\begin{equation}\label{2.4.1}
w_{\lambda}(y)\leq \big(2-\e\big)\log \frac{1}{|y|}+C_\e ,
\,\, ~\mbox{for}~ y\in\Omega_\lambda \backslash B_{2R_\varepsilon}(0),
\end{equation}
and
\begin{equation}\label{2.4.2}
0\leq \Big(1+\frac{w_\lambda(y)}{\gamma^2_\lambda}\Big) e^{2w_\lambda(y)+\frac{w^2_\lambda(y)}{\gamma^2_\lambda}} \leq
\frac{\widetilde{C}_\varepsilon }{1+|y|^{4-2\varepsilon}}, \,\, ~\mbox{for}~ y\in\Omega_\lambda,
\end{equation}
where $\Omega_\lambda=\big\{ x: x_\lambda+\theta_\lambda x \in \Omega \big\}$.
\end{Lem}
\begin{proof}
Given $\e>0$, we can choose $R_\e>1$ such that
\begin{equation*}
\int_{B_{R_\e}(0)}e^{2U(z)}dz>(4-\e)\pi.
\end{equation*}
Then by Fatou's lemma and \eqref{limw}, we have
\begin{equation*}
\liminf_{\lambda \to 0}
\int_{B_{R_\e}(0)}\Big(1+\frac{w_\lambda(z)}{\gamma^2_\lambda}\Big)
e^{2w_\lambda(z)+\frac{w^2_\lambda(z)}{\gamma^2_\lambda}}dz\geq \int_{B_{R_\e}(0)}
e^{2U(z)}dz,
\end{equation*}
which indicates that there exists $\lambda_\varepsilon>0$ sufficiently small such that for any $\lambda\in(0,\lambda_\varepsilon)$, it holds
\begin{equation}\label{2.4-1}
\int_{B_{R_\e}(0)}\Big(1+\frac{w_\lambda(z)}{\gamma^2_\lambda}\Big)
e^{2w_\lambda(z)+\frac{w^2_\lambda(z)}{\gamma^2_\lambda}}dz > (4-\e)\pi.
\end{equation}
Also by the Green's representation theorem, we can write
\begin{equation*}
\begin{split}
u_\lambda\big(\theta_\lambda y+x_\lambda\big)=& \lambda \int_{\Omega} G\big(\theta_\lambda y+x_\lambda,x\big)
 u_{\lambda}(x)e^{u^2_{\lambda}(x)}\,dx\\=& \frac{1}{\gamma_\lambda}
 \int_{\Omega_\lambda}G\big(\theta_\lambda y+x_\lambda,\theta_\lambda z+x_\lambda\big) \Big(1+\frac{w_\lambda(z)}{\gamma^2_\lambda}\Big)
e^{2w_\lambda(z)+\frac{w^2_\lambda(z)}{\gamma^2_\lambda}}dz.
\end{split}
\end{equation*}
Hence, using \eqref{defw}, we find
\begin{align*}
w_{\lambda}(y) =& -\gamma^2_\lambda+\int_{\Omega_\lambda}G\big(\theta_\lambda y+x_\lambda,\theta_\lambda z+x_\lambda\big) \Big(1+\frac{w_\lambda(z)}{\gamma^2_\lambda}\Big) e^{2w_\lambda(z)+\frac{w^2_\lambda(z)}{\gamma^2_\lambda}} \,dz \\
=& -\gamma_\lambda u_\lambda(x_\lambda) +\int_{\Omega_\lambda}G\big(\theta_\lambda y+x_\lambda,\theta_\lambda z+x_\lambda\big) \Big(1+\frac{w_\lambda(z)}{\gamma^2_\lambda}\Big) e^{2w_\lambda(z)+\frac{w^2_\lambda(z)}{\gamma^2_\lambda}} \,dz \\
=& -\lambda \gamma_\lambda \int_{\Omega} G(x_\lambda,x) u_\lambda(x)e^{u_\lambda^2(x)} \,dx \\
& +\int_{\Omega_\lambda} G\big(\theta_\lambda y+x_\lambda,\theta_\lambda z+x_\lambda\big) \Big(1+\frac{w_\lambda(z)}{\gamma^2_\lambda}\Big) e^{2w_\lambda(z)+\frac{w^2_\lambda(z)}{\gamma^2_\lambda}} \,dz \\
=& -\int_{\Omega_\lambda} G\big(x_\lambda,\theta_\lambda z+x_\lambda\big) \Big(1+\frac{w_\lambda(z)}{\gamma^2_\lambda}\Big) e^{2w_\lambda(z)+\frac{w^2_\lambda(z)}{\gamma^2_\lambda}} \,dz \\
& +\int_{\Omega_\lambda} G\big(\theta_\lambda y+x_\lambda,\theta_\lambda z+x_\lambda\big) \Big(1+\frac{w_\lambda(z)}{\gamma^2_\lambda}\Big) e^{2w_\lambda(z)+\frac{w^2_\lambda(z)}{\gamma^2_\lambda}} \,dz.
\end{align*}
Now we can compute $w_\lambda(y)$ as follows:
\begin{align}\label{2.4-2}
w_{\lambda}(y)=&\frac{1}{2\pi} \int_{\Omega_\lambda}\log \frac{|z|}{|z-y|} \Big(1+\frac{w_\lambda(z)}{\gamma^2_\lambda}\Big) e^{2w_\lambda(z)+\frac{w^2_\lambda(z)}{\gamma^2_\lambda}} \,dz \notag\\
& -\int_{\Omega_\lambda}\Big(H\big(\theta_\lambda y+x_\lambda,\theta_\lambda z+x_\lambda\big)
-H\big(x_\lambda,\theta_\lambda z+x_\lambda\big) \Big) \Big(1+\frac{w_\lambda(z)}{\gamma^2_\lambda}\Big)
e^{2w_\lambda(z)+\frac{w^2_\lambda(z)}{\gamma^2_\lambda}} \,dz \notag\\=&
\frac{1}{2\pi}\int_{\Omega_\lambda}\log \frac{|z|}{|z-y|} \Big(1+\frac{w_\lambda(z)}{\gamma^2_\lambda}\Big)
e^{2w_\lambda(z)+\frac{w^2_\lambda(z)}{\gamma^2_\lambda}} \,dz +O(1) \notag\\=&
\frac{1}{2\pi}\int_{B_{R_\e}(0)}\log \frac{|z|}{|z-y|} \Big(1+\frac{w_\lambda(z)}{\gamma^2_\lambda}\Big)
e^{2w_\lambda(z)+\frac{w^2_\lambda(z)}{\gamma^2_\lambda}} \,dz \notag\\&
+ \frac{1}{2\pi}\int_{\Omega_\lambda\backslash B_{R_\e}(0) \bigcap \{|z|\leq 2|z-y|\} }\log \frac{|z|}{|z-y|} \Big(1+\frac{w_\lambda(z)}{\gamma^2_\lambda}\Big) e^{2w_\lambda(z)+\frac{w^2_\lambda(z)}{\gamma^2_\lambda}} \,dz
\notag\\&+ \frac{1}{2\pi} \int_{\Omega_\lambda\backslash B_{R_\e}(0) \bigcap \{|z|\geq 2|z-y|\} } \log  |z| \Big(1+\frac{w_\lambda(z)}{\gamma^2_\lambda}\Big) e^{2w_\lambda(z)+\frac{w^2_\lambda(z)}{\gamma^2_\lambda}} \,dz
\notag\\&+ \frac{1}{2\pi} \int_{\Omega_\lambda\backslash B_{R_\e}(0) \bigcap \{|z|\geq 2|z-y|\} } \log \frac{1}{|z-y|} \Big(1+\frac{w_\lambda(z)}{\gamma^2_\lambda}\Big) e^{2w_\lambda(z)+\frac{w^2_\lambda(z)}{\gamma^2_\lambda}} \,dz
+O(1) \notag\\[1mm]=:& I_1+I_2+I_3+I_4+ O(1).
\end{align}
The second equality is derived from the fact that $H$ is Lipschitz continuous and $|\theta_\lambda y| \leq C$, and the calculation that
\begin{equation}\label{2.4-3}
\int_{\Omega_\lambda} \Big(1+\frac{w_\lambda(z)}{\gamma^2_\lambda}\Big) e^{2w_\lambda(z)+\frac{w^2_\lambda(z)}
{\gamma^2_\lambda}} \,dz
= \lambda \gamma_\lambda \int_{\Omega} u_\lambda(z)e^{u_\lambda^2(z)} \,dz \overset{\eqref{lim-2}} = 4\pi+o_{\lambda}(1).
\end{equation}

Firstly, we estimate $I_1$. If $|z|\leq R_\e$, then for any $y\in \Omega_\lambda \backslash B_{2R_\e}(0)$, we have
$2|z|\leq |y|$ and then
\begin{equation*}
\frac{|z|}{|z-y|} \leq \frac{|z|}{|y|-|z|} \leq \frac{|z|}{|y|-\frac{|y|}{2}}\leq \frac{2R_\e}{|y|}.
\end{equation*}
Therefore, it holds
\begin{equation*}
I_1\leq
\frac{1}{2\pi}\log \frac{2R_\e}{|y|} \int_{B_{R_\e}(0)}\Big(1+\frac{w_\lambda(z)}{\gamma^2_\lambda}\Big)
e^{2w_\lambda(z)+\frac{w^2_\lambda(z)}{\gamma^2_\lambda}}dz,
\end{equation*}
which together with \eqref{2.4-1} indicates that
\begin{equation}\label{2.4-4}
I_1\leq \Big(2-\frac{\e}{2}\Big) \log \frac{1}{|y|} + C_1(\varepsilon),
\end{equation}
where $C_1(\varepsilon)>0$ is a constant dependent on $\varepsilon$.

\vskip 0.2cm

Next, we estimate the term $I_2$. If $|z|\leq 2|z-y|$, then $\log \frac{|z|}{|z-y|}\leq
\log 2$. Hence, from \eqref{2.4-1} and \eqref{2.4-3}, we calculate that
\begin{equation}\label{2.4-5}
I_2\leq \frac{\log 2}{2\pi} \int_{\Omega_\lambda\backslash B_{R_\e}(0)  } \Big(1+\frac{w_\lambda(z)}{\gamma^2_\lambda}\Big) e^{2w_\lambda(z)+\frac{w^2_\lambda(z)}{\gamma^2_\lambda}}dz \leq C_2(\varepsilon) ,
\end{equation}
where $C_2(\varepsilon)>0$ is a constant dependent on $\varepsilon$. Similarly, if $|z|\geq 2|z-y|$, we have  $|z|\leq 2|y|$ and
\begin{equation}\label{2.4-6}
I_3\leq \frac{\log |2y|}{2\pi} \int_{\Omega_\lambda\backslash B_{R_\e}(0) } \Big(1+\frac{w_\lambda(z)}{\gamma^2_\lambda}\Big)
e^{2w_\lambda(z)+\frac{w^2_\lambda(z)}{\gamma^2_\lambda}}dz \leq \frac{\varepsilon}{2} \log |y| + C_3(\varepsilon),
\end{equation}
where $C_3(\varepsilon)>0$ is a constant dependent on $\varepsilon$.

\vskip 0.1cm

Finally, we know that
\begin{equation*}
\begin{split}
& \Omega_\lambda\backslash B_{R_\e}(0) \cap \big\{ |z| \geq 2|z-y| \big\} \\
=&~ \Big( \Omega_\lambda\backslash B_{R_\e}(0) \cap \big\{2\leq 2|z-y|\leq |z|\big\} \Big) \bigcup
\Big( \Omega_\lambda\backslash B_{R_\e}(0) \cap \big\{ 2|z-y|\leq |z| \leq 2 \big\} \Big)  \\
&~ \bigcup \Big( \Omega_\lambda\backslash B_{R_\e}(0) \cap \big\{ 2|z-y| \leq 2 \leq |z|\big\} \Big),
\end{split}
\end{equation*}
then by H\"older's inequality, for any fixed $\alpha$ (closed to $1^+$), we compute $I_4$:
\begin{equation}\label{2.4-7}
\begin{split}
I_4 =&~
\frac{1}{2\pi} \int_{\Omega_\lambda\backslash B_{R_\e}(0)\bigcap \{2|z-y|\leq 2 \leq |z|\}} \log \frac{1}{|z-y|} \Big(1+\frac{w_\lambda(z)}{\gamma^2_\lambda}\Big) e^{2w_\lambda(z)+\frac{w^2_\lambda(z)}{\gamma^2_\lambda}}dz \\
&~ +\frac{1}{2\pi} \int_{\Omega_\lambda\backslash B_{R_\e}(0) \bigcap \{2|z-y|\leq |z| \leq 2\}} \log \frac{1}{|z-y|} \Big(1+\frac{w_\lambda(z)}{\gamma^2_\lambda}\Big) e^{2w_\lambda(z)+\frac{w^2_\lambda(z)}{\gamma^2_\lambda}} \,dz \\
&~ +\frac{1}{2\pi} \int_{\Omega_\lambda\backslash B_{R_\e}(0) \bigcap \{2\leq 2|z-y|\leq |z|\} } \log \frac{1}{|z-y|} \Big(1+\frac{w_\lambda(z)}{\gamma^2_\lambda}\Big) e^{2w_\lambda(z)+\frac{w^2_\lambda(z)}{\gamma^2_\lambda}} \,dz \\
\leq &~ \frac{1}{2\pi} \int_{\Omega_\lambda\backslash B_{R_\e}(0)\bigcap \{|z-y|\leq 1\}} \log \frac{1}{|z-y|} \Big(1+\frac{w_\lambda(z)}{\gamma^2_\lambda}\Big) e^{2w_\lambda(z)+\frac{w^2_\lambda(z)}{\gamma^2_\lambda}} \,dz \\
\leq &~ \frac{1}{2\pi} \bigg(\int_{|z-y|\leq 1} \Big| \log \frac{1}{|z-y|} \Big|^{\frac{\alpha}{\alpha-1}}  \,dz\bigg)^{\frac{\alpha-1}{\alpha}} \cdot \Bigg(\int_{\Omega_\lambda\backslash B_{R_\e}(0)} \bigg| \Big(1+\frac{w_\lambda(z)}{\gamma^2_\lambda}\Big) e^{2w_\lambda(z)+\frac{w^2_\lambda(z)}{\gamma^2_\lambda}} \bigg|^{\alpha} \,dz\Bigg)^{\frac{1}{\alpha}} \\ \leq &~ C_4(\varepsilon),
\end{split}
\end{equation}
here we use \eqref{2.4-1}, \eqref{2.4-3} and the fact that
\[
\log \frac{1}{|z-y|} \Big(1+\frac{w_\lambda(z)}{\gamma^2_\lambda}\Big) e^{2w_\lambda(z) +\frac{w^2_\lambda(z)}{\gamma^2_\lambda}} \leq0,~\mbox{in}~\big\{z\in \Omega_\lambda\backslash B_{R_\e}(0) : 2\leq 2|z-y|\leq |z|\big\}.
\]
Substituting \eqref{2.4-4}, \eqref{2.4-5}, \eqref{2.4-6} and \eqref{2.4-7} into \eqref{2.4-2}, we obtain that for any
$ \lambda \in (0,\lambda_\varepsilon)$ and $y \in \Omega_\lambda \backslash B_{2R_\varepsilon}(0)$, there exists $C_\varepsilon>0$ such that
\begin{eqnarray*}
w_\lambda(y) & \leq& \big(2-\frac{\varepsilon}{2} \big) \log \frac{1}{|y|} + \frac{\varepsilon}{2} \log |y| + C_1(\varepsilon)+C_2(\varepsilon)+C_3(\varepsilon)+C_4(\varepsilon)+O(1) \\
&\leq& (2-\varepsilon) \log \frac{1}{|y|} + C_\varepsilon.
\end{eqnarray*}

Next, by \textbf{Theorem A} $(2)$ and $(3)$, we find that $w_\lambda(y)\leq 0$ in $B_{2R_\varepsilon}(0) \subset \Omega_\lambda$, and $\displaystyle \frac{1}{\gamma_\lambda^2} = o(1)$ for $\lambda$ sufficiently small, then we derive that
\begin{equation}\label{2.4-8}
0\leq \Big(1+\frac{w_\lambda(y)}{\gamma^2_\lambda}\Big) e^{2w_\lambda(y)+\frac{w^2_\lambda(y)}{\gamma^2_\lambda}} \leq 1,
\end{equation}
for any $|y| \leq 2R_\varepsilon$ and $\lambda$ sufficiently small. On the other hand, for $\lambda$ sufficiently small, it follows from \eqref{2.4.1} that
\begin{equation}\label{2.4-9}
0\leq \Big(1+\frac{w_\lambda(y)}{\gamma^2_\lambda}\Big) e^{2w_\lambda(y)+\frac{w^2_\lambda(y)}{\gamma^2_\lambda}} \leq
C e^{2(2-\varepsilon)\log\frac{1}{|y|}+2C_\varepsilon } \leq \frac{\widehat{C}_\varepsilon }{|y|^{4-2\varepsilon} },
\end{equation}
for $y \in \Omega_\lambda \backslash B_{2R_\varepsilon}(0)$ and some $\widehat{C}_\varepsilon >0$. Therefore, \eqref{2.4.2} can be derived from \eqref{2.4-8} and \eqref{2.4-9}.
\end{proof}

In the sequel, we use the same $\delta$ to denote various small constants belonging to $(0,1)$.
Now we have a basic asymptotic behavior on the positive solution $u_\lambda$.
\begin{Lem}\label{prop3-1}
For any fixed small $d>0$, it holds
\begin{equation}\label{asym_u}
u_\lambda(x)=  C_{\lambda}G(x_{\lambda},x)+ o\Big( \frac{\theta_{\lambda}}{\gamma_\lambda}\Big)\,\,
~\mbox{in}~C^1\big(\Omega\backslash B_{2d}(x_{\lambda})\big),
\end{equation}
where
\begin{equation}\label{def_C}
C_{\lambda}:= \lambda \displaystyle\int_{B_d(x_{\lambda})} u_\lambda(y)e^{u_\lambda^2(y)}dy=\frac{4\pi}{\gamma_\lambda}\big(1+o(1)\big).
\end{equation}
\end{Lem}
\begin{proof}
By the Green's representation theorem, we have
\begin{equation}\label{2.5-1}
\begin{split}
u_\lambda(x)=& \lambda \int_{\Omega} G(y,x)
 u_{\lambda}(y)e^{u^2_{\lambda}(y)}\,dy\\=& \lambda
 \int_{B_d(x_{\lambda})} G(y,x)
 u_{\lambda}(y)e^{u^2_{\lambda}(y)}\,dy
+\lambda \int_{\Omega\backslash   B_{d}(x_{\lambda})} G(y,x)
 u_{\lambda}(y)e^{u^2_{\lambda}(y)}\,dy
 \\=&  C_{\lambda}G(x_{\lambda},x)+
\lambda  \int_{B_d(x_{\lambda})}   \big(G(y,x)-G(x_\lambda,x)\big)
 u_{\lambda}(y)e^{u^2_{\lambda}(y)}\,dy\\&+
\lambda  \int_{\Omega\backslash B_d(x_{\lambda})} G(y,x)
 u_{\lambda}(y)e^{u^2_{\lambda}(y)}\,dy.
\end{split}
\end{equation}
By Taylor's expansion, the dominated convergence theorem and \eqref{2.4.2}, there exists a small constant $\delta>0$ such that
\begin{equation}\label{2.5-2}
\begin{split}
& \lambda \int_{B_d(x_{\lambda})} \big(G(y,x)-G(x_\lambda,x)\big) u_{\lambda}(y)e^{u^2_{\lambda}(y)}\,dy \\
=& \lambda \int_{B_d(x_{\lambda})} \big\langle y-x_\lambda,\nabla G(x_\lambda,x)\big\rangle u_{\lambda}(y)e^{u^2_{\lambda}(y)}\,dy + O\Big(\lambda \int_{B_d(x_{\lambda})} \big| y-x_\lambda\big|^2
u_{\lambda}(y)e^{u^2_{\lambda}(y)}\,dy\Big)\\
=& \frac{\theta_{\lambda}}{\gamma_\lambda} \int_{B_{\frac{d}{\theta_\lambda}}(0)} \big\langle z,\nabla G(x_\lambda,x)\big\rangle \Big(1+\frac{w_\lambda(z)}{\gamma^2_\lambda}\Big)e^{2w_\lambda(z)+
\frac{w^2_\lambda(z)} {\gamma^2_\lambda}}\,dz +O\bigg(\frac{\theta^2_{\lambda}}{\gamma_\lambda}  \int_{B_{\frac{d}{\theta_\lambda}}(0)} \frac{|z|^2}{1+|z|^{4-\delta}} \,dz \bigg)\\
=& \frac{\theta_{\lambda}}{\gamma_\lambda} \Big( \int_{B_{\frac{d}{\theta_\lambda}}(0)} \big\langle z,\nabla G(x_\lambda,x)\big\rangle e^{2U(z)}\,dz+o(1)\Big)+O\Big(\frac{\theta_\lambda^{2-\delta}}{\gamma_\lambda} \Big)\\
=&o\Big(\frac{\theta_{\lambda}}{\gamma_\lambda}\Big).
\end{split}
\end{equation}
Next, by H\"older's inequality, for any fixed $\alpha>1$, 
we find
\begin{equation}\label{2.5-3}
\begin{split}
& \lambda \int_{\Omega\backslash   B_{d}(x_{\lambda})} G(y,x)
 u_{\lambda}(y)e^{u^2_{\lambda}(y)}\,dy\\= &O\left(
 \lambda  \bigg(\int_{\Omega\backslash   B_{d}(x_{\lambda})}\Big(
 u_{\lambda}(y)e^{u^2_{\lambda}(y)}\Big)^{\alpha} \,dy\bigg)^{\frac{1}{\alpha}}
 \cdot \bigg(\int_{\Omega\backslash   B_{d}(x_{\lambda})}\big(
G(y,x)\big)^{\frac{\alpha}{\alpha-1}} \,dy\bigg)^{1-\frac{1}{\alpha}}\right)\\=&
O\Bigg(\lambda \theta_\lambda^{\frac{2}{\alpha}}\gamma_\lambda e^{\gamma_\lambda^2}  \bigg(\int_{ \Omega_\lambda\backslash B_{\frac{d}{\theta_\lambda}}(0)}\frac{1}{\big(1+|z|^{4-\delta}\big)^\alpha} \,dz
\bigg)^{\frac{1}{\alpha}}\Bigg)\\
=&O\Big(\frac{\theta_\lambda^{2-\delta}}{\gamma_\lambda}  \Big).
\end{split}
\end{equation}
It follows from \eqref{2.5-1}, \eqref{2.5-2} and \eqref{2.5-3} that
\begin{equation}\label{2.5-4}
u_\lambda(x)=  C_{\lambda}G(x_{\lambda},x)+ o\Big( \frac{\theta_{\lambda}}{\gamma_\lambda}\Big)\,\,
~\mbox{in}~\Omega\backslash B_{2d}(x_{\lambda}).
\end{equation}
Similar to the above estimates, we also get
\begin{equation}\label{2.5-5}
\begin{split}
\frac{\partial u_\lambda}{\partial x_i}(x) &= \lambda \int_{\Omega} D_{x_i}G(y,x) u_{\lambda}(y)e^{u^2_{\lambda}(y)} \,dy  =
C_{\lambda} D_{x_i}G(x_{\lambda},x)+ o\Big( \frac{\theta_{\lambda}}{\gamma_\lambda}\Big) \,\,
~\mbox{in}~\Omega\backslash B_{2d}(x_{\lambda}).
\end{split}
\end{equation}
Hence \eqref{asym_u} follows from \eqref{2.5-4} and \eqref{2.5-5}.

\vskip 0.1cm

Finally, we compute $C_\lambda$. By the dominated convergence theorem, \eqref{def_C} can be obtained from
\begin{equation*}
\begin{split}
\lim_{\lambda\to 0}\big(\gamma_\lambda C_{\lambda}\big) =&  \lim_{\lambda\to 0} \int_{B_{\frac{d}{\theta_\lambda}}(0)}
\Big(1+\frac{w_\lambda(z)}{\gamma^2_\lambda}\Big) e^{2w_\lambda(z)+\frac{w^2_\lambda(z)} {\gamma^2_\lambda}}\,dz = 4\pi.
\end{split}
\end{equation*}
\end{proof}

Recall that the local maximum point $x_\lambda$ converges to $x_0 \notin \partial \Omega$, which is a critical point of Robin function $\mathcal{R}(x)$. We have the following estimate on $|x_\lambda-x_0|$.

\begin{Lem}\label{lem-h1}
It holds
\begin{equation}\label{asym_x}
\big|x_{\lambda}-x_{0}\big|=o \big(\theta_\lambda\big).
\end{equation}
\end{Lem}
\begin{proof}
Firstly, using \eqref{qgg} and \eqref{asym_u}, we have
\begin{equation}\label{2.6-1}
\begin{split}
\mbox{LHS of \eqref{quu}} &= C_\lambda^2 Q\big(G(x_\lambda,x),G(x_\lambda,x)\big) +o\Big(\frac{\theta_\lambda}{\gamma^2_\lambda}\Big)  = -C^2_\lambda \frac{\partial \mathcal{R}(x_{\lambda})}{\partial{x_i}} +o\Big(\frac{\theta_\lambda}{\gamma^2_\lambda}\Big).
\end{split}
\end{equation}
On the other hand, by \eqref{defw}, we have
\begin{equation}\label{2.6-2}
\begin{split}
\mbox{RHS of \eqref{quu}}= & O\bigg(\lambda \theta_\lambda \int_{\partial B_{\frac{d}{\theta_\lambda}}(0)} e^{\gamma^2_\lambda+2w_\lambda(x)+\frac{w^2_\lambda(x)}{\gamma^2_\lambda}} \,d\sigma\bigg) \\
=& O\bigg(\frac{1}{\theta_\lambda \gamma^2_\lambda}\int_{\partial B_{\frac{d}{\theta_\lambda}}(0)} \frac{1}{1+|x|^{4-\delta}} \, d\sigma \bigg) =O\Big( \frac{\theta_\lambda^{2-\delta}}{\gamma^2_\lambda}\Big).
\end{split}\end{equation}
Hence from \eqref{def_C}, \eqref{2.6-1} and \eqref{2.6-2}, we find
\begin{equation}\label{2.6-3}
 \frac{\partial \mathcal{R}(x_{\lambda})}{\partial{x_i}}
=o\big(\theta_\lambda\big).
\end{equation}
Then \eqref{asym_x} follows from \eqref{2.6-3} and the assumption  that $x_0$ is a non-degenerate critical point of $\mathcal{R}(x)$.
\end{proof}

\section{Some prior estimates}\label{s3}

\vskip 0.2cm

In order to study the Morse index and the local uniqueness of $u_\lambda$, we need a better expansion of $w_\lambda$.
Let us define
\begin{equation}\label{v-def}
v_{\lambda}(x):=\gamma^2_\lambda \big(w_{\lambda}(x)-U(x) \big),
\end{equation}
where $w_{\lambda}$ is the rescaled function in \eqref{defw} and $U$ is its limit function introduced in \eqref{defU}.
\begin{Lem}
For any small fixed $d_0>0$, it holds
\begin{equation}\label{v-equa}
\begin{split}
-\Delta v_{\lambda}=  2v_{\lambda}e^{2U}+v_{\lambda} \overline{g}_\lambda +g^*_\lambda~\,\quad \mbox{for}~~|x|\leq \frac{d_0}{\theta_\lambda},
\end{split}\end{equation}
where
\begin{equation}\label{v-equa1}
\overline{g}_\lambda(x) = \frac{ \gamma^2_\lambda \big(e^{ 2w_\lambda }-e^{2U}\big) - 2v_\lambda e^{2U} }{v_\lambda}
= O\Big(\frac{ |w_{\lambda}-U|}{ 1+|x|^{4-\delta} }\Big) \quad\mbox{for}~~|x|\leq \frac{d_0}{\theta_\lambda},
\end{equation}
\begin{equation}\label{v-equa2}
g^*_\lambda(x) = w_\lambda e^{ 2w_\lambda+ \frac{w^2_\lambda}{\gamma^2_\lambda}} + \gamma^2_\lambda \Big(e^{ 2w_\lambda+ \frac{w^2_\lambda}{\gamma^2_\lambda}} -e^{ 2w_\lambda }\Big)
=O\left(\frac{1}{ 1+|x|^{4-\delta} }\right)~\quad\mbox{for}~~|x|\leq \frac{d_0}{\theta_\lambda},
\end{equation}
and
\begin{equation}\label{v-equa3}
g^*_\lambda(x)\rightarrow \big( U^{2}(x)+ U(x)\big) e^{2U(x)}~\,\mbox{in}~\,C^2_{loc}(\R^2),~\,\,\mbox{as}~~\lambda
\rightarrow 0.
\end{equation}
\end{Lem}
\begin{proof}
Recall that $u_\lambda(x_\lambda+\theta_\lambda x)=\gamma_\lambda+\frac{w_\lambda(x)}{\gamma_\lambda}$, then
direct computation gives
\begin{align}\label{add_vequa}
-\Delta v_\lambda(x) =& \gamma_\lambda^2 \big(-\Delta w_\lambda(x)+\Delta U(x) \big) \notag\\
=& \gamma^2_\lambda\big(-\gamma_\lambda\theta^2_\lambda \Delta u_\lambda(x_\lambda+\theta_\lambda x)-e^{2U(x)}\big) \notag\\
=& \gamma^2_\lambda \Big( \lambda \gamma_\lambda\theta^2_\lambda  u_\lambda(x_\lambda+\theta_\lambda x)
e^{ u_\lambda^2(x_\lambda+\theta_\lambda x) }-e^{2U(x)}\Big) \notag\\
=& \gamma^2_\lambda \bigg( \Big(1+\frac{w_\lambda(x)}{\gamma^2_\lambda} \Big) e^{ 2w_\lambda(x)+ \frac{w^2_\lambda(x)}{\gamma^2_\lambda}}-e^{2U(x)}\bigg) \notag\\
=& \underbrace{w_\lambda(x) e^{ 2w_\lambda(x)+\frac{w^2_\lambda(x)}{\gamma^2_\lambda}} +\gamma^2_\lambda \Big(e^{ 2w_\lambda(x)+\frac{w^2_\lambda(x)}{\gamma^2_\lambda}} -e^{ 2w_\lambda(x)}\Big)}_{:=g^*_\lambda}
+\gamma^2_\lambda \big(e^{2w_\lambda(x)} -e^{2U(x)}\big) \notag\\
=&  w_\lambda(x) e^{ 2w_\lambda(x)+ \frac{w^2_\lambda(x)}{\gamma^2_\lambda}}+w_\lambda^2(x) e^{ 2w_\lambda(x)}
+O\Big( \frac{w_\lambda^4}{\gamma^2_\lambda} e^{ 2w_\lambda} \Big)+\gamma^2_\lambda \big(e^{ 2w_\lambda(x)} -e^{2U(x)}\big).
\end{align}
Hence we have \eqref{v-equa2} and \eqref{v-equa3} by Lemma \ref{lem2.4} and the dominated convergence theorem.
\vskip 0.1cm

Now we compute  the term $\gamma^2_\lambda \big(e^{2w_\lambda(x)}-e^{2U(x)}\big)$. By Taylor's expansion and \eqref{2.4.1}, it holds
\begin{equation*}
\begin{split}
\gamma^2_\lambda \big(e^{2w_\lambda(x)}-e^{2U(x)}\big) =& \gamma^2_\lambda \bigg(2e^{2U(x)}\big(w_\lambda(x)-U(x)\big)
+O\Big(e^{2(1-\xi)U(x)+2\xi w_\lambda(x)} \big|w_\lambda-U\big|^2 \Big) \bigg) \\
=& 2e^{2U(x)} v_\lambda(x) + O\Big(e^{2(1-\xi)U(x)+2\xi w_\lambda(x)} |v_\lambda| \big|w_\lambda-U\big| \Big) \\
=& v_\lambda(x)\bigg(2e^{2U(x)}+ \underbrace{ O\Big(\frac{|w_\lambda-U|}{1+|x|^{4-\delta}}}_{:=
\overline g_\lambda} \Big)\bigg)\,\,~~\mbox{in}~~B_{\frac{d_0}{\theta_\lambda}}(0),
\end{split}
\end{equation*}
where $\xi \in (0,1)$.
Hence we find \eqref{v-equa} and \eqref{v-equa1} from above computations.
\end{proof}

Using the above results, we can get the bound of $v_\lambda$.
\begin{Prop}\label{prop4.2}
Let $v_{\lambda}$ be the function defined in \eqref{v-def}. Then for any fixed $\tau\in (0,1)$, there exists $C>0$ such that
\begin{equation}\label{v-bound}
| v_{\lambda}(x)|\leq C\big(1+|x|\big)^\tau~\mbox{in}~B_{\frac{d_0}{\theta_\lambda}}(0).
\end{equation}
\end{Prop}

As a consequence of Proposition \ref{prop4.2}, we pass to the limit into equation \eqref{v-equa} and obtain the following result.

\begin{Prop}\label{prop4.3}
It holds
\begin{equation}\label{lim-v}
\lim_{\lambda \rightarrow 0} v_{\lambda}=v_0 ~\mbox{in}~C^2_{loc}(\R^2),
\end{equation}
where
$v_0$ solves the non-homogeneous linear equation
\begin{equation}\label{4.3.1}
-\Delta v-2e^{2U}v= e^{2U}\big(U^2+U\big)~\mbox{in}~\R^2,
\end{equation}
and for any $\tau\in (0,1)$, there exists $C>0$ such that
\begin{equation}\label{v0-bound}
|v_{0}(x)|\leq C(1+|x| )^\tau.
\end{equation}
\end{Prop}
\vskip 0.4cm
\begin{proof}[\underline{\textbf{Proof of Proposition \ref{prop4.2}}}]
To prove \eqref{v-bound}, we use similar ideas in  \cite[Estimate C]{CL2002}  and
\cite[Appendix B]{LY2018}. Let
\begin{equation*}
N_\lambda :=\max_{|x| \leq \frac{d_0}{\theta_\lambda}} \frac{| v_{\lambda}(x)|}{(1+|x|)^\tau},
\end{equation*}
then obviously
$$\eqref{v-bound} \Leftrightarrow N_\lambda \leq C.$$
We will prove that $N_\lambda  \leq C$ by contradiction. The proof is organized into two steps. Set
\begin{equation*}
N_\lambda^*:=\max_{|x| \leq \frac{d_0}{\theta_\lambda}}\max_{|x'|=|x|}\frac{|v_{\lambda }(x)-v_{\lambda }(x')|}{(1+|x|)^{\tau}}.
\end{equation*}
\emph{Step 1.
If $N_\lambda \to +\infty$, then it holds
\begin{equation}\label{4.2-step1}
N_\lambda^*=o(1)N_\lambda.
\end{equation}}
\proof[Proof of Step 1.]
Suppose that this is not true. Then there exists $c_0>0$ such that $N^*_\lambda\geq c_0N_\lambda$. Let $x'_\lambda$ and $x''_\lambda$ satisfy $|x'_\lambda|=|x''_\lambda| \leq \frac{d_0}{\theta_\lambda}$ and
\begin{equation*}
N_\lambda^*=  \frac{|v_{\lambda}(x'_\lambda)-v_{\lambda}(x''_\lambda)|}{(1+|x'_\lambda|)^{\tau}}.
\end{equation*}
Without loss of generality, we may assume that $x'_\lambda$ and
$x''_\lambda$ are symmetric with respect to the $x_1$ axis. Set
\begin{equation*}
\omega^*_\lambda(x):=v_{\lambda}(x)-v_{\lambda}(x^-),~\,~x^-:=(x_1,-x_2),\  \mbox{ for }x=(x_{1},x_{2}),  ~\,x_2>0.
\end{equation*}
Hence $x^{''}_{\lambda}={x^{'}_{\lambda}}^{-}$ and
$N_\lambda^*=\frac{ |\omega^*_\lambda(x'_{\lambda})| }{ (1+|x'_{\lambda}| )^{\tau} }$.
Let us define
\begin{equation}\label{omega-def}
\omega_\lambda(x):=\frac{\omega^*_\lambda(x)}{(1+x_2)^{\tau}}.
\end{equation}
By \eqref{add_vequa}, direct calculations show that $\omega_\lambda$ satisfies
\begin{equation}\label{omega-equa}
-\Delta \omega_\lambda-\frac{2\tau}{1+x_2}\frac{\partial \omega_\lambda}{\partial x_2}+
\frac{\tau(1-\tau)}{(1+x_2)^2}\omega_\lambda=\frac{g_\lambda(x)}{(1+x_2)^\tau}+
\frac{\widetilde{{g}}_\lambda(x)}{(1+x_2)^\tau},
\end{equation}
where $g_\lambda(x):=g^*_\lambda(x)-g^*_\lambda(x^-)$, $g^*_\lambda$ is defined in \eqref{v-equa2} and
\begin{equation}\label{tilde_g}
\begin{split}
\widetilde{g}_\lambda(x) :=&~ \gamma^2_\lambda \Big(e^{2w_{\lambda}(x)}-e^{2w_{\lambda}(x^-)}\Big)
= \gamma^2_\lambda e^{2(1-\xi)w_{\lambda}(x^-)+2\xi w_{\lambda}(x)} \big(2w_{\lambda}(x)-2w_{\lambda}(x^-)\big) \\
=&~ 2\gamma^2_\lambda e^{2(1-\xi)w_{\lambda}(x^-)+2\xi w_{\lambda}(x)} \Big(\big(w_{\lambda}(x)-U(x)\big)-\big(w_{\lambda}(x^-)-U(x)\big)\Big) \\
=&~ 2e^{2(1-\xi)w_{\lambda}(x^-)+2\xi w_{\lambda}(x)} \big(v_{\lambda}(x)-v_{\lambda}(x^-)\big) \\
=&~ 2e^{2(1-\xi)w_{\lambda}(x^-)+2\xi w_{\lambda}(x)} \omega_\lambda^{\ast}(x)\overset{Lemma~\ref{lem2.4}}=O\Big(\frac{|\omega_\lambda^{\ast}|}{ 1+|x|^{4-\delta}} \Big)~\,\mbox{for}~~|x| \leq \frac{d_0}{\theta_\lambda},
\end{split}
\end{equation}
where $\xi \in (0,1)$.
Also let $x^{**}_\lambda$ satisfies $|x^{**}_\lambda| \leq \frac{d_0}{\theta_\lambda}$, $x^{**}_{\lambda,2}\geq 0$ and
\begin{equation}\label{N_ast2-def}
\big|\omega_\lambda(x^{**}_\lambda)\big|= N^{**}_\lambda:=\max_{|x| \leq \frac{d_0}{\theta_\lambda},~x_2\geq 0} |\omega_\lambda(x)|.
\end{equation}
Then it follows
\begin{equation}\label{N_ast2-inequa}
N^{**}_\lambda \geq \frac{\big|\omega^*_\lambda(x'_\lambda)\big|}{\big(1+x'_{\lambda,2}\big)^{\tau}}
\geq \frac{\big|v_{\lambda}(x'_\lambda)-v_{\lambda}(x''_\lambda)\big|}
{\big(1+|x'_\lambda|\big)^{\tau}}=N^*_\lambda\to +\infty.
\end{equation}
We may assume that $\omega_\lambda(x^{**}_\lambda)>0$. We claim that
\begin{equation}\label{x_ast2-bound}
|x^{**}_\lambda|\leq C.
\end{equation}
The proof of \eqref{x_ast2-bound} is divided into two parts.

\vskip 0.2cm

\noindent
\emph{Part 1. We prove that $|x^{**}_\lambda|\leq \frac{d_0}{2\theta_{\lambda}}$.}

\vskip 0.2cm
Arguing by contradiction, we know that
\begin{align}\label{4.2-part1-1}
\omega^*_\lambda(x^{**}_\lambda) =& v_{\lambda}(x^{**}_\lambda)-v_{\lambda}(x^{**-}_\lambda)
=\gamma^2_\lambda \Big(\big(w_{\lambda}(x^{**}_\lambda)-w_{\lambda}(x^{**-}_\lambda)\big)
-\big(U(x^{**}_\lambda)-U(x^{**-}_\lambda)\big)\Big) \notag\\
=& \gamma^3_\lambda \Big(u_{\lambda}\big(\theta_{\lambda}x^{**}_\lambda+x_{\lambda}\big)
-u_{\lambda}\big(\theta_{\lambda}x^{**-}_\lambda+x_{\lambda}\big)\Big),
\end{align}
the last equality is derived from \eqref{defw} and the fact that $U$ is radial.
Moreover, since we suppose that $|x^{**}_\lambda| > \frac{d_0}{2\theta_{\lambda}}$, using \eqref{asym_u} with $d=\frac{d_0}{4}$, we get
\begin{equation}\label{4.2-part1-2}
\begin{split}
&~ u_{\lambda}(\theta_{\lambda}x^{**}_\lambda+x_{\lambda})-u_{\lambda}(\theta_{\lambda}x^{**-}_\lambda+x_{\lambda}) \\
=&~ C_{\lambda}
\Big(G\big(x_{\lambda},\theta_{\lambda}x^{**}_\lambda+x_{\lambda}\big)-
G\big(x_{\lambda},\theta_{\lambda}x^{**-}_\lambda+x_{\lambda}\big)  \Big)
+o\Big(\frac{\theta_{\lambda} }{\gamma_\lambda }\Big).
\end{split}
\end{equation}
Since $\frac{d_0}{2\theta_\lambda} < |x_\lambda^{**}| \leq \frac{d_0}{\theta_\lambda}$, we have $\theta_{\lambda}\big|x^{**}_\lambda- x^{**-}_\lambda\big| \leq 2d_0$. Hence, it is easy to find that
\begin{equation*}
\begin{split}
& G(x_{\lambda},\theta_{\lambda}x^{**}_\lambda+x_{\lambda})-G(x_{\lambda},
\theta_{\lambda}x^{**-}_\lambda+x_{\lambda}) \\=&
-H(x_{\lambda}, \theta_{\lambda}x^{**}_\lambda+x_{\lambda})+H(x_{\lambda},\theta_{\lambda}x^{**-}_\lambda+x_{\lambda}) \\=&
O\bigg(\theta_{\lambda}\big|x^{**}_\lambda- x^{**-}_\lambda\big|
\Big| \nabla H \big(x_{\lambda}, \xi\theta_{\lambda}x^{**-}_\lambda+(1-\xi)\theta_{\lambda}x^{**}_\lambda +x_{\lambda} \big) \Big|\bigg)~\,\mbox{with}~\,\xi\in (0,1),
\end{split}
\end{equation*}
and
$$\Big| \nabla H\big(x_{\lambda}, \xi\theta_{\lambda}x^{**-}_\lambda+(1-\xi)\theta_{\lambda}x^{**}_\lambda +x_{\lambda}\big) \Big| =O\big(1\big).$$
Then by above computations, it follows
\begin{equation}\label{4.2-part1-3}
\begin{split}
G(&x_{\lambda}, \theta_{\lambda}x^{**}_\lambda+x_{\lambda})-G(x_{\lambda},\theta_{\lambda}x^{**-}_\lambda+x_{\lambda}) =O\big(\theta_{\lambda} x^{**}_{\lambda,2}\big).
\end{split}
\end{equation}
So from \eqref{4.2-part1-2} and \eqref{4.2-part1-3}, we have
\begin{equation}\label{4.2-part1-4}
\begin{split}
 u_{\lambda} (\theta_{\lambda}x^{**}_\lambda+x_{\lambda})-u_{\lambda}(\theta_{\lambda}x^{**-}_\lambda+x_{\lambda})
=O\Big(\frac{\theta_{\lambda} x^{**}_{\lambda,2}}{\gamma_\lambda}\Big)
+o\Big( \frac{\theta_{\lambda}}{\gamma_\lambda} \Big).
\end{split}\end{equation}
Now from  \eqref{4.2-part1-1} and \eqref{4.2-part1-4}, we get
\begin{equation}\label{4.2-part1-5}
\begin{split}
\omega^*_\lambda(x^{**}_\lambda) = v_{\lambda}(x^{**}_\lambda)-v_{\lambda}(x^{**-}_\lambda)
= O\big(\gamma_\lambda^2\theta_{\lambda} x^{**}_{\lambda,2}\big) +o\big(\gamma_\lambda^2\theta_{\lambda} \big).
\end{split}\end{equation}
As a result,
\begin{equation*}
N^{**}_\lambda = \frac{\omega^*_\lambda(x_\lambda^{**})}{(1+x^{**}_{\lambda,2})^{\tau}}
\leq C \gamma_\lambda^2\theta_{\lambda}(x^{**}_{\lambda,2})^{1-\tau}+o\big(\gamma_\lambda^2\theta_{\lambda} \big) \leq  C \gamma_\lambda^2 \theta_{\lambda}^{\tau} +o\big(\gamma_\lambda^2\theta_{\lambda} \big)=o\big(1\big).
\end{equation*}
It is a contradiction with \eqref{N_ast2-inequa} and concludes the proof of \emph{Part 1.}

\vskip 0.2cm

\noindent
\emph{Part 2. We prove that $|x^{**}_\lambda|\leq C$.}

\vskip 0.2cm
Suppose this is not true. Since we have proved that $x^{**}_\lambda\in B_{\frac{d_0}{2\theta_{\lambda}}}(0)$, we can deduce that
\begin{equation}\label{4.2-part2-1}
\big|\nabla \omega_\lambda(x^{**}_\lambda)\big|=0~\mbox{and}~\Delta \omega_\lambda(x^{**}_\lambda)\leq 0.
\end{equation}
So from
\eqref{omega-equa}, \eqref{tilde_g} and \eqref{4.2-part2-1}, we find
\begin{equation}\label{4.2-part2-2}
\begin{split}
0\leq &-\Delta \omega_\lambda(x^{**}_\lambda)=-
\frac{\tau(1-\tau)}{(1+x^{**}_{\lambda,2})^2}\omega_\lambda(x^{**}_\lambda)
+\frac{g_\lambda(x^{**}_\lambda)}{(1+x^{**}_{\lambda,2})^\tau}+ \frac{\widetilde{{g}}_\lambda(x^{**}_\lambda)}{(1+x^{**}_{\lambda,2})^\tau}\\=&-
\left(\frac{\tau(1-\tau)}{(1+x^{**}_{\lambda,2})^2}+O\Big(\frac{1}{ 1+|x^{**}_{\lambda}|^{4-\delta}}\Big)
\right)\omega_\lambda(x^{**}_\lambda)
+\frac{g_\lambda(x^{**}_\lambda)}{(1+x^{**}_{\lambda,2})^\tau}.
\end{split}\end{equation}
Also by assumption $|x^{**}_\lambda|\to \infty$, we find that for some large $R>0$ and $\delta\in (0,1)$, there exists $\lambda_0$ such that $|x^{**}_\lambda|\geq R$ for any $\lambda\in (0,\lambda_0)$. And it also follows
\begin{equation}\label{4.2-part2-3}
O\Big(\frac{1}{ 1+|x^{**}_{\lambda}|^{4-\delta}} \Big) \leq  \frac{\tau(1-\tau)}{2(1+x^{**}_{\lambda,2})^2}\,\,\,\mbox{for}~~\lambda\in (0,\lambda_0).
\end{equation}
Hence using \eqref{4.2-part2-2}, \eqref{4.2-part2-3} and the fact that $|x^{**}_\lambda|\to \infty$, we deduce that
\begin{equation*}
\frac{\omega_\lambda(x^{**}_\lambda)}{(1+x^{**}_{\lambda,2})^2}\leq \frac{Cg_\lambda(x^{**}_\lambda)}{(1+x^{**}_{\lambda,2})^{\tau}}.
\end{equation*}
Then it follows
\begin{equation}\label{4.2-part2-4}
 \omega_\lambda(x^{**}_\lambda) \leq Cg_\lambda(x^{**}_\lambda)(1+x^{**}_{\lambda,2})^{2-\tau}.
\end{equation}
Combining \eqref{v-equa2}, \eqref{N_ast2-def} and \eqref{4.2-part2-4}, we obtain
\begin{equation*}
N^{**}_\lambda\leq
\frac{C }{(1+|x_\lambda^{**}|)^{2+\tau-\delta}}\to 0~~\mbox{as}~~\lambda\to 0,
\end{equation*}
which is a contradiction with \eqref{N_ast2-inequa}. Then the proof of \emph{Part 2} is completed and \eqref{x_ast2-bound} follows.\\

Now let
\begin{equation}\label{def_omegaast2}
\omega_\lambda^{**}(x)=\frac{\omega_\lambda(x)}{N_\lambda^{**}},
\end{equation}
where $\omega_\lambda$ is defined in \eqref{omega-def} and $N_\lambda^{**}$ is defined in \eqref{N_ast2-def}.
From \eqref{omega-equa}, it follows that $\omega_\lambda^{**}(x)$ solves
\begin{equation}\label{omega_ast2-equa}
-\Delta \omega^{**}_\lambda-\frac{2\tau}{1+x_2}\frac{\partial \omega^{**}_\lambda}{\partial x_2}+
\frac{\tau(1-\tau)}{(1+x_2)^2}\omega^{**}_\lambda=\frac{g_\lambda(x)}{N_\lambda^{**}(1+x_2)^\tau}+
\frac{\widetilde{{g}}_\lambda(x)}{N_\lambda^{**}(1+x_2)^\tau}.
\end{equation}
Moreover $\omega^{**}_{\lambda}(x^{**}_\lambda)=1$ and
\begin{equation}\label{bound_omegaast2}
|\omega^{**}_{\lambda}(x)|\leq 1~ \mbox{in}~ B^+_{\frac{d_0}{\theta\lambda}}:=\Big\{x \in B_{\frac{d_0}{\theta\lambda}}(0) : x_2 \geq 0 \Big\},
\end{equation}
hence $\omega^{**}_\lambda(x)\to \omega$ uniformly in any compact subset of $\R^2$ and $\omega\not\equiv0$ because $\omega_\lambda^{**}(x_{\lambda}^{**})=1$ and  $|x_{\lambda}^{**}|\le C$. Observe that from \eqref{v-equa2},
$$\frac{g_\lambda(x)}{N_\lambda^{**}(1+x_2)^{\tau}}\to 0~\mbox{uniformly in any compact set of}~ B^+_{\frac{d_0}{\theta\lambda}} ~\mbox{as}~\lambda \to 0.$$
For any  $\phi(x)\in C^{\infty}_0\big(B_{\frac{d_0}{\theta\lambda}}(0)\big)$, we also have
\begin{align*}
& \frac{1}{N_\lambda^{**}} \int_{B^+_{\frac{d_0}{\theta\lambda}}} \bigg( \frac{\widetilde{g}_\lambda(x)}{(1+x_2)^\tau} -2e^{2U(x)} \omega_\lambda(x)  \bigg)  \phi(x) \,dx \\
\overset{\eqref{tilde_g}}=& \frac{1}{N_\lambda^{**}} \int_{B^+_{\frac{d_0}{\theta\lambda}}} \left( \frac{\gamma^2_\lambda \big(e^{2w_{\lambda}(x)}-e^{2w_{\lambda}(x^-)}\big)} {(1+x_2)^{\tau}}- 2e^{2U(x)} \omega_\lambda(x)\right) \phi(x) \,dx \\
\overset{\eqref{omega-def}}=& \frac{1}{N_\lambda^{**}} \int_{B^+_{\frac{d_0}{\theta\lambda}}} \frac{1}{(1+x_2)^{\tau}} \left(\gamma^2_\lambda \big(e^{2w_{\lambda}(x)}-e^{2w_{\lambda}(x^-)}\big) - 2e^{2U(x)}\omega_\lambda^{\ast}(x)\right) \phi(x) \,dx \\
\overset{\eqref{tilde_g}}=& \frac{1}{N_\lambda^{**}} \int_{B^+_{\frac{d_0}{\theta\lambda}}} \frac{1}{(1+x_2)^{\tau}} \left( 2e^{2(1-\xi)w_{\lambda}(x^-)+2\xi w_{\lambda}(x)} \omega_\lambda^{\ast}(x) -2e^{2U(x)}\omega_\lambda^{\ast}(x)\right) \phi(x) \,dx \\
\overset{\eqref{omega-def}}=& \frac{4}{N_\lambda^{**}} \int_{B^+_{\frac{d_0}{\theta\lambda}}} e^{2(1-\zeta)U(x)+2\zeta [(1-\xi)w_{\lambda}(x^-)+\xi w_{\lambda}(x)]} \Big(\big(1-\xi \big)w_\lambda(x^-)+\xi w_\lambda(x)-U(x)\Big) \omega_\lambda(x) \phi(x) \,dx \\
\overset{\eqref{def_omegaast2}}=& 4\int_{B^+_{\frac{d_0}{\theta\lambda}}} e^{2(1-\zeta)U(x)+2\zeta [(1-\xi)w_{\lambda}(x^-)+\xi w_{\lambda}(x)]} \Big(\big(1-\xi \big)w_\lambda(x^-)+\xi w_\lambda(x)-U(x)\Big) \omega_\lambda^{\ast\ast}(x) \phi(x) \,dx \\
\overset{\eqref{bound_omegaast2}}=& O\bigg(\int_{B^+_{\frac{d_0}{\theta\lambda}}} e^{2(1-\zeta)U(x)+2\zeta [(1-\xi)w_{\lambda}(x^-)+\xi w_{\lambda}(x)]} \big|(1-\xi)w_\lambda(x^-)+\xi w_\lambda(x)-U(x)\big| \cdot |\phi(x)| \,dx \bigg),
\end{align*}
where $\xi,\zeta \in (0,1)$.
By Lemma \ref{lem2.4}, we know that
\[
e^{2(1-\zeta)U(x)+2\zeta [(1-\xi)w_{\lambda}(x^-)+\xi w_{\lambda}(x)]} = O\Big(\frac{1}{1+|x|^{4-\delta}}\Big),
\]
and then using the dominated convergence theorem and \eqref{limw}, we derive
\[
\frac{1}{N_\lambda^{**}} \int_{B^+_{\frac{d_0}{\theta\lambda}}} \bigg( \frac{\widetilde{g}_\lambda(x)}{(1+x_2)^\tau} -2e^{2U(x)} \omega_\lambda(x)  \bigg)  \phi(x) \,dx \rightarrow 0.
\]
Hence, passing to the limit $\lambda \to 0$ into \eqref{omega_ast2-equa}, we can deduce that $\omega$ solves
\begin{equation*}
\begin{cases}
\displaystyle -\Delta \omega -\frac{2\tau}{1+x_2}\frac{\partial \omega}{\partial x_2}
+\Big(\frac{\tau(1-\tau)}{(1+x_2)^2}-2e^{2U(x)} \Big)\omega =0,~x_2>0,\\[3mm]
\displaystyle \omega(x_1,0)=0.
\end{cases}
\end{equation*}
Set $\bar \omega=(1+x_2)^{\tau}\omega$, then $\bar \omega \not\equiv 0$ and we find that $\bar \omega$ satisfies
\begin{equation*}
\begin{cases}
-\Delta \bar \omega  -2e^{2U} \bar \omega =0,~x_2>0,\\[2mm]
\bar\omega(x_1,0)=0.
\end{cases}
\end{equation*}
Then using Lemma \ref{lem3.1}, we have
\begin{equation}\label{omega_bar}
\bar \omega(x)=c_0\frac{4-|x|^2}{4+|x|^2}+\sum^2_{i=1}{c_i}\frac{x_i}{4+|x|^2},~~
\mbox{for some constants $c_0$, $c_1$ and $c_2$}.
\end{equation}
On the other hand, by the definition of $v_{\lambda}(x)$ and $\omega^*_\lambda(x)$, we know
\begin{equation}\label{omega_0}
\begin{split}
\nabla \omega_\lambda^*(0)=&\nabla \big( v_{\lambda}(x)-v_{\lambda}(x^-)\big)\big|_{x=0}\\=&
\gamma^2_\lambda \big(\nabla w_{\lambda}(x)-\nabla U(x)\big)\big|_{x=0}-\gamma^2_\lambda\big(\nabla w_{\lambda}(x^-)-\nabla U(x^-)\big)\big|_{x=0} \\=&
\gamma^3_\lambda \theta_{\lambda} \nabla u_{\lambda}(x_{\lambda}+\theta_{\lambda}x)\big|_{x=0}-
 \gamma^3_\lambda \theta_{\lambda} \nabla u_{\lambda}(x_{\lambda}+\theta_{\lambda}x^-)\big|_{x=0}\\
 =&0.
\end{split}
\end{equation}
Here the last equality follows by the fact that $x_{\lambda}$ is a local maximum point of $u_{\lambda}(x)$.
Hence
\begin{equation*}
\nabla \big((1+x_2)^{\tau} \omega_\lambda^{**}\big)  \big|_{x=0}=
\frac{\nabla \omega_\lambda^*(0)}{N_\lambda^{**}}= 0,
\end{equation*}
which means $\nabla \bar\omega(0)=0$. Then from \eqref{omega_bar}, we find $c_1=c_2=0$.
Moreover from $\bar\omega(x_1,0)=0$, we also have $c_0=0$. So $\bar \omega=0$, which is a contradiction.
Hence \eqref{4.2-step1} follows, which concludes the proof of \emph{Step 1}.
\vskip 0.2cm
\noindent\emph{Step 2. We prove that $N_{\lambda}\leq C$ arguing by contradiction.}
\proof[Proof of Step 2]
Suppose by contradiction that $N_\lambda\to +\infty$, as $\lambda\to 0$ and set
\begin{equation*}
\psi_{\lambda}(r)=\frac{1}{2\pi}\int^{2\pi}_0v_{\lambda}(r,\theta)d \theta,~\,\,r=|x|.
\end{equation*}
By \emph{Step 1}, it follows that
\begin{equation*}
\begin{split}
 \frac{|\psi_{\lambda}(r)|}{(1+r)^\tau}=&
 \frac{1}{2\pi(1+r)^\tau} {\Big |}\int^{2\pi}_0v_{\lambda}(r,\theta)d \theta
 {\Big |}\\\leq  &
 \max_{\theta\in [0,2\pi]} \frac{|v_{\lambda}(r,\theta)|}{ (1+r)^\tau}
 +
 \frac{1}{2\pi(1+r)^\tau}  \left| \int^{2\pi}_0\Big (v_{\lambda}(r,\theta)- \max_{\theta\in [0,2\pi]}  \big | v_{\lambda}(r,\theta) \big | \Big)  d \theta  \right|.
\end{split}
\end{equation*}
Then we have
\begin{equation*}
\begin{split}
  \max_{r\leq \frac{d_0}{\theta_{\lambda}}}\frac{|\psi_{\lambda}(r)|}{(1+r)^\tau}\leq &
  \max_{r\leq \frac{d_0}{\theta_{\lambda}}} \max_{\theta\in [0,2\pi]} \frac{|v_{\lambda}(r,\theta)|}{ (1+r)^\tau}
 +C \max_{r\leq \frac{d_0}{\theta_{\lambda}}} \max_{\theta_1,\theta_2\in [0,2\pi]}
 \frac{\big|v_{\lambda}(r,\theta_1)-v_{\lambda}(r,\theta_2)\big|}{ (1+r)^\tau}
 \\=&
  \max_{|x|\leq \frac{d_0}{\theta_{\lambda}}} \frac{|v_{\lambda}(x)|}{ (1+|x|)^\tau} +
 C \max_{|x|\leq \frac{d_0}{\theta_{\lambda}}} \max_{|x'|=|x|}\frac{|v_{\lambda}(x)-v_{\lambda}(x')|}{ (1+|x|)^\tau}
\\=&N_\lambda+C N^*_\lambda \overset{\eqref{4.2-step1}}=N_\lambda\big(1+o(1)\big).
\end{split}
\end{equation*}
Similarly, it follows
\begin{equation*}
\begin{split}
  \max_{r\leq \frac{d_0}{\theta_{\lambda}}}\frac{|\psi_{\lambda}(r)|}{(1+r)^\tau}\geq &
  \max_{r\leq \frac{d_0}{\theta_{\lambda}}} \max_{\theta\in [0,2\pi]} \frac{|v_{\lambda}(r,\theta)|}{ (1+r)^\tau}
 -C \max_{r\leq \frac{d_0}{\theta_{\lambda}}} \max_{\theta_1,\theta_2\in [0,2\pi]}
 \frac{\big|v_{\lambda}(r,\theta_1)-v_{\lambda}(r,\theta_2)\big|}{ (1+r)^\tau}
 \\=&
 \max_{|x|\leq \frac{d_0}{\theta_{\lambda}}} \frac{|v_{\lambda}(x)|}{ (1+|x| )^\tau} -
 C \max_{|x|\leq \frac{d_0}{\theta_{\lambda}}} \max_{|x'|=|x|}\frac{|v_{\lambda}(x)-v_{\lambda}(x')|}{ (1+|x|)^\tau}
\\=&N_\lambda-C N^*_\lambda =N_\lambda\big(1+o(1)\big).
\end{split}
\end{equation*}
Hence we can find that
\begin{equation}\label{psi-lamda}
\max_{r\leq \frac{d_0}{\theta_{\lambda}}}\frac{|\psi_{\lambda}(r)|}{(1+r)^\tau}=N_\lambda\big(1+o(1)\big).
\end{equation}

Assume that $\frac{|\psi_{\lambda}(r)|}{(1+r)^\tau}$ attains its maximum at $r_\lambda$,
then we claim
\begin{equation}\label{r-bound}
r_\lambda\leq C.
\end{equation}
In fact, let $\phi(x)=
\frac{4-|x|^2}{4+|x|^2}$ and  we recall that
$-\Delta\phi(x)=2e^{2U(x)}\phi(x)$.
Now multiplying \eqref{v-equa} by $\phi$ and integrating by parts, we have
\begin{equation}\label{v-phi-1}
\begin{split}
\int_{|x|=r} -\Big(\frac{\partial v_{\lambda}}{\partial \nu}\phi-  \frac{\partial \phi}{\partial \nu}v_{\lambda}\Big)\,d\sigma
 =&
\int_{|x|\leq r}  v_{\lambda}\overline{g}_\lambda\phi\, dx +
\int_{|x|\leq r} g^*_\lambda\phi\, dx\\
=&  O\Big(N_\lambda \int_{|x|\leq r}\overline{g}_\lambda\phi(1+|x|)^{\tau}\,dx+
\int_{|x|\leq r} g^*_\lambda\phi\,dx\Big)\\
=&  o(1)N_\lambda+O(1),
\end{split}
\end{equation}
the last identity is due to \eqref{limw}, \eqref{v-equa1} and \eqref{v-equa2}. Also for $r \geq 3$, we know
\begin{equation}\label{v-phi-2}
\begin{split}
\int_{|x|=r}\Big(\frac{\partial v_{\lambda}}{\partial \nu}\phi-  \frac{\partial \phi}{\partial \nu}v_{\lambda}\Big)\,d\sigma
= & \frac{4-r^2}{4+r^2}\int_{|x|=r} \frac{\partial v_{\lambda}}{\partial \nu}\, d\sigma
+ \frac{16r}{(4+r^2)^2} \int_{|x|=r}  v_{\lambda}\, d\sigma\\
=& \frac{4-r^2}{4+r^2}\int_{0}^{2\pi} r \frac{\partial v_\lambda(r,\theta)}{\partial r} \,d\theta + \frac{16r}{(4+r^2)^2} \int_{0}^{2\pi} r v_\lambda(r,\theta) \,d\theta \\
=& 2\pi r \frac{4-r^2}{4+r^2} \psi_{\lambda}'(r) + \frac{32\pi r^2}{(4+r^2)^2} \psi_{\lambda}(r).
\end{split}
\end{equation}
Hence, it follows from \eqref{v-phi-1} and \eqref{v-phi-2} that
\begin{equation}\label{add_psi}
2\pi r\frac{r^2-4}{4+r^2} \psi_{\lambda}'(r) = o(1)N_\lambda +O(1) +\frac{32\pi r^2}{(4+r^2)^2} \psi_{\lambda}(r),
~\mbox{for}~ r\geq 3,
\end{equation}
and \eqref{psi-lamda} implies
\begin{equation}\label{add_psi2}
\frac{32\pi r^2}{(4+r^2)^2} \psi_{\lambda}(r) \leq 32\pi \frac{ r^2(1+r)^{\tau}}{(4+r^2)^2} N_\lambda.
\end{equation}
Note that for $r\geq 3$ and $\tau\in(0,1)$, it holds
\begin{equation*}
\frac{ r^2(1+r)^{\tau}}{(4+r^2)^2} \leq \frac{C_\tau}{r}~~\mbox{and}~~\frac{4+r^2}{r^2-4}\leq \frac{13}{5},
\end{equation*}
where $C_\tau:=\max\limits_{r>3}\frac{ r^3(1+r)^{\tau}}{(4+r^2)^2}$. Therefore, we deduce from \eqref{add_psi} and \eqref{add_psi2} that
\begin{equation*}
|\psi'_{\lambda}(r)|=  \frac{o\big(N_\lambda\big)}{r} +  \frac{O(1)}{r} +\frac{O\big( N_\lambda \big)}{r^2}.
\end{equation*}
Furthermore, we find
\begin{equation}\label{r-bound-1}
\begin{split}
&~ (1+r_\lambda)^{\tau} N_\lambda\big(1+o(1)\big)  = |\psi_{\lambda}(r_\lambda)|\leq \int^{r_\lambda}_3|\psi'_{\lambda}(r)| \,dr +|\psi_{\lambda}(3)| \\
\leq &~ o\big(N_\lambda \big)\big(\ln r_\lambda-\ln3\big) + O(1)\big(\ln r_\lambda-\ln3\big) + O\big( N_\lambda \big)\Big(\frac{1}{3}-\frac{1}{r_\lambda}\Big) + |\psi_{\lambda}(3)| ,
\end{split}
\end{equation}
and
\begin{equation}\label{r-bound-2}
\big|\psi_{\lambda}(3)\big|\leq 4^{\tau}\max_{r\leq \frac{d_0}{\theta_{\lambda}}}\frac{|\psi_{\lambda}(r)|}{(1+r)^\tau}=O\big( N_\lambda \big).
\end{equation}
Hence  from \eqref{r-bound-1} and \eqref{r-bound-2}, we have
$(1+r_\lambda)^{\tau}N_\lambda= O\big(N_\lambda \big)$, which implies \eqref{r-bound}.
\vskip 0.2cm

Now integrating \eqref{v-equa}, we get
\begin{equation}\label{equazionePsi}
-\Delta \psi_{\lambda}=2\psi_{\lambda}e^{2U}+\overline{h}_\lambda+\widehat{h}_\lambda,\qquad~\mbox{ for}~|x|\leq \frac{d_0}{\theta_{\lambda}},
\end{equation}
where 
$$\overline{h}_\lambda(r):=\frac{1}{2\pi}\displaystyle\int^{2\pi}_0 v_{\lambda}(r,\theta) \overline{g}_\lambda(r,\theta)d \theta~~~~\,\,~\mbox{and}~~~~\,\,~\widehat{h}_{\lambda}(r):=\frac{1}{2\pi}\displaystyle\int^{2\pi}_0 g^*_\lambda(r,\theta)d \theta,$$
here $\overline{g}_\lambda$  and $g^*_\lambda$ are defined as \eqref{v-equa1} and \eqref{v-equa2} respectively.

\vskip 0.1cm

Next we define $\psi^*_{\lambda}(x)=\frac{\psi_{\lambda}(|x|)}{\psi_{\lambda}(r_\lambda)}$. We calculate that
\begin{equation*}
|\psi^*_{\lambda}(x)|\leq \frac{C\big|\psi_{\lambda}(|x|)\big|}{N_\lambda}\leq C\big(1+|x|\big)^{\tau}
\end{equation*}
and
\begin{equation*}
\frac{\widehat{h}_\lambda(r)}{\psi_{\lambda}(r_\lambda)}\leq \frac{C \widehat{h}_\lambda(r)}{N_\lambda} \overset{\eqref{v-equa2}} \leq \frac{C}{N_\lambda}\to 0.
\end{equation*}
Moreover, from the dominated convergence theorem, we have
\begin{equation*}
\int_{B_{\frac{d_0}{\theta_{\lambda}}}(0)}
\frac{\overline{h}_\lambda(|x|)}{\psi_{\lambda}(r_\lambda) } \phi(x)dx \to 0\,\,~\mbox{for any radial function}~\phi(x)\in C^{\infty}_0\big(B_{\frac{d_0}{\theta_{\lambda}}}(0)\big).
\end{equation*}
Hence from the above computations and the dominated convergence theorem, passing to the limit in equation \eqref{equazionePsi} divided by $\psi_\lambda(r_\lambda)$, we can deduce that $\psi^*_{\lambda}\to \psi(|x|)$ in  $C^2_{loc}(\R^2)$ and $\psi$ satisfies
\begin{equation*}
-\psi''-\frac{1}{r}\psi'=2e^{2U}\psi.
\end{equation*}
Therefore, it follows from Lemma \ref{lem3.1} that
\begin{equation}\label{psi-res}
\psi(|x|)=c_0\frac{4-|x|^2}{4+|x|^2},~\mbox{with some constant}~c_0.
\end{equation}
Since $\psi^*_{\lambda}{ (r_{\lambda})}=1$ and $r_\lambda\leq C$, we find $\psi\not\equiv 0$.

\vskip 0.1cm

On the other hand, we know that $\psi_{\lambda}(0){  =v_{\lambda}(0)=0}$ by definition. Then, we have $\psi(0)=0$, which together with \eqref{psi-res} implies $\psi\equiv 0$. This is a contradiction.

As a result, $N_\lambda \leq C$ and we complete the proof of Proposition \ref{prop4.2}.
\end{proof}

Next  we introduce
\begin{equation}\label{k-def}
k_{\lambda}(x):=\gamma^2_\lambda\big(v_{\lambda}(x)-v_0(x)\big),
\end{equation}
where $v_{\lambda}$ is defined in \eqref{v-def} and $v_{0}$ is its limit function as $\lambda\rightarrow 0$ (see Proposition \ref{prop4.3}).
\begin{Lem}
For any small fixed $d_0>0$, it holds
\begin{equation}\label{k-equa}
\begin{split}
-\Delta k_{\lambda}=  2k_{\lambda}e^{2U}+  h^*_\lambda +h^{**}_\lambda ~\,\,\mbox{for}~~|x|\leq \frac{d_0}{\theta_\lambda},
\end{split}
\end{equation}
where
\begin{equation*}
h_\lambda^*(x) \to e^{2U } \Big(\frac{1}{2}U^4+2v_0^2+ U^3+  v_0\big(2U^2+4U+1\big) \Big)~\,\,\mbox{in}~C_{loc}^2(\R^2),
~\,\,\mbox{as}~ \lambda \to 0,
\end{equation*}
\begin{equation}\label{h-ast}
h_\lambda^*(x) = O\Big( \frac{1}{1+|x|^{4-\delta-2\tau}} \Big)~\,\,\mbox{for}~~|x|\leq \frac{d_0}{\theta_\lambda},
\end{equation}
and
\begin{equation}\label{h-ast2}
h_\lambda^{**}(x) = O\Big(\frac{1}{\gamma^2_\lambda \big(1+|x|^{4-\delta-3\tau}\big)}\Big)~\,\,\mbox{for}~~|x|\leq \frac{d_0}{\theta_\lambda}.
\end{equation}
\end{Lem}
\begin{proof}
Firstly, using \eqref{add_vequa} and \eqref{4.3.1}, it holds
\begin{equation}\label{4.4-1}
\begin{split}
-\Delta k_\lambda(x)=&-\gamma^2_\lambda \Delta \big(v_\lambda(x)-v_0(x)\big)\\=&
\Big( \frac{w_\lambda}{\gamma^2_\lambda} e^{ 2w_\lambda+ \frac{w^2_\lambda}{\gamma^2_\lambda}}+
e^{ 2w_\lambda+ \frac{w^2_\lambda}{\gamma^2_\lambda}}
-e^{2U}\Big) \gamma^4_\lambda-\Big(2v_0 e^{2U}+e^{2U}\big(U^2+U\big)\Big) \gamma^2_\lambda \\
=& \Big( w_\lambda  e^{ 2w_\lambda+ \frac{w^2_\lambda}{\gamma^2_\lambda}}
-2v_0 e^{2U}-e^{2U}\big(U^2+U\big)\Big)\gamma^2_\lambda
+\underbrace{\Big(  e^{ 2w_\lambda+ \frac{w^2_\lambda}{\gamma^2_\lambda}}
-e^{2U}\Big) \gamma^4_\lambda }_{I_{\lambda,1}}.
\end{split}
\end{equation}
Next, by Taylor's expansion, we can compute
\begin{equation*}
\begin{split}
 I_{\lambda,1}=w^2_\lambda e^{2w_\lambda}\gamma^2_\lambda+\frac{1}{2}w^4_\lambda e^{2w_\lambda}+
 2\gamma^2_\lambda v_\lambda e^{2U} +2v^2_\lambda e^{2U}+O\Big(\frac{w^6_\lambda}{\gamma^2_\lambda} e^{2w_\lambda}\Big)+O\Big( \frac{v^3_\lambda}{\gamma^2_\lambda} e^{2U} \Big)\,\,\,\mbox{ in}~~\Big\{|x|\leq \frac{d_0}{\theta_\lambda}\Big\}.
\end{split}
\end{equation*}
Hence for $|x|\leq \frac{d_0}{\theta_\lambda}$, it holds
\begin{equation}\label{4.4-2}
\begin{split}
-\Delta k_\lambda(x)= &
\underbrace{\Big( w_\lambda  e^{ 2w_\lambda+ \frac{w^2_\lambda}{\gamma^2_\lambda}} -2v_0 e^{2U} -e^{2U}\big(U^2+U\big)
+ w^2_\lambda e^{2w_\lambda} + 2 v_\lambda e^{2U} \Big) \gamma^2_\lambda }_{I_{\lambda,2}} \\
& +\frac{1}{2}w^4_\lambda e^{2w_\lambda} +2v^2_\lambda e^{2U}
+O\Big(\frac{w^6_\lambda}{\gamma^2_\lambda} e^{2w_\lambda}\Big)+O\Big( \frac{v^3_\lambda}{\gamma^2_\lambda} e^{2U} \Big).
\end{split}
\end{equation}
Let $f(x)=(x+x^2)e^{2x}$, for any fixed $y$, it holds
\begin{equation*}
f(x)-f(y)=(x-y)(1+4y+2y^2)e^{2y}+O\Big((x-y)^2(2y^2+6y+3)e^{2y}\Big), ~~\mbox{as}~~x\to y.
\end{equation*}
Then for $|x|\leq \frac{d_0}{\theta_\lambda}$, we compute
\begin{equation*}
\begin{split}
I_{\lambda,2} &=
2k_\lambda e^{2U}+w_\lambda e^{2w_\lambda}\Big( e^{\frac{w^2_\lambda}{\gamma^2_\lambda}}-1\Big)\gamma^2_\lambda
+ \Big( e^{2w_\lambda}\big(w_\lambda^2+w_\lambda\big)- e^{2U}\big(U^2+U\big) \Big)\gamma^2_\lambda \\
&=2k_\lambda e^{2U}+ w^3_\lambda e^{2w_\lambda} +v_\lambda e^{2U}\big(2U^2+4U+1\big)
+ O\Big(  \frac{w^5_\lambda}{\gamma^2_\lambda} e^{2w_\lambda} \Big)  +O\Big( \frac{v_\lambda^2}{\gamma^2_\lambda}e^{2U}\big(U^2+1\big)\Big).
\end{split}
\end{equation*}
Therefore, \eqref{4.4-2} can be rewritten as
\begin{equation*}
\begin{split}
-\Delta k_\lambda(x) =& 2k_\lambda e^{2U}+ \underbrace{ w^3_\lambda e^{2w_\lambda} +v_\lambda e^{2U}\big(2U^2+4U+1\big) +\frac{1}{2}w^4_\lambda e^{2w_\lambda} +2v^2_\lambda e^{2U}}_{:=h_\lambda^*} \\
& +\underbrace{ O\Big(\frac{w^5_\lambda}{\gamma^2_\lambda} e^{2w_\lambda} \Big) +O\Big( \frac{v_\lambda^2}{\gamma^2_\lambda} e^{2U}\big(U^2+1\big)\Big) +O\Big(\frac{w^6_\lambda}{\gamma^2_\lambda} e^{2w_\lambda}\Big)+O\Big( \frac{v^3_\lambda}{\gamma^2_\lambda} e^{2U} \Big) }_{:=h_\lambda^{**}},
\end{split}
\end{equation*}
where
\begin{equation*}
h_\lambda^* \to e^{2U } \Big(\frac{1}{2}U^4+2v_0^2+ U^3+  v_0\big(2U^2+4U+1\big)\Big)~\,\,\mbox{in}~C_{loc}^2(\R^2), ~\,\,\mbox{as}~ \lambda \to 0.
\end{equation*}
Also \eqref{h-ast} and \eqref{h-ast2} follow from Lemma \ref{lem2.4} and Proposition \ref{prop4.3}.

\end{proof}

Now we can also prove the bound of $k_{\lambda}$.
\begin{Prop}\label{prop_klambda}
Let $k_{\lambda}$ be as in \eqref{k-def}. Then for any small fixed $d_0,\tau_1>0$, there exists $C>0$ such that
\begin{equation}\label{k-inequa}
|k_{\lambda}(x)|\leq C(1+|x| )^{\tau_1}~\mbox{ in}~ B_{\frac{d_0}{\theta_{\lambda}}}(0).
\end{equation}
Moreover, we deduce that
\begin{equation}\label{k-lim}
\lim_{\lambda\to 0}k_{\lambda}=k_0 \,\,\mbox{in}~C^{2}_{loc}(\mathbb R^{2}),
\end{equation}
where $k_0$ solves the non-homogeneous linear equation
\begin{equation}\label{k0-equa}
-\Delta v-2e^{2U}v=  e^{2U } \Big(\frac{1}{2}U^4+2v_0^2+ U^3+  v_0\big(2U^2+4U+1\big) \Big) ~~\mbox{in}~~\R^2.
\end{equation}
Furthermore,  there exists $C>0$ such that
\begin{equation*}
|k_{0}(x)|\leq C(1+|x| )^{\tau_1}.
\end{equation*}
\end{Prop}
\begin{proof}
The proof of this proposition will be given in Appendix \ref{s-k}.
\end{proof}

Using above analysis, we can write
\begin{equation*}
w_\lambda(x)=U(x)+\frac{v_0(x)}
{\gamma^2_\lambda}+\frac{k_{\lambda}(x)}{\gamma^4_{\lambda}}~\,\,\mbox{in}~~ \Omega_{\lambda}:=\frac{\Omega-x_{\lambda}}{\theta_\lambda},
\end{equation*}
with some $C, \tau>0$ such that
\begin{equation*}
|v_0(x)|+|k_\lambda(x)| \leq C(1+|x| )^{\tau}.
\end{equation*}
But sometimes, this is not enough. Particularly, to get the local uniqueness of concentrated solutions, we need further precise estimates.

Let us introduce
\begin{equation}\label{s-def}
s_{\lambda}(x):=\gamma^2_\lambda\big(k_{\lambda}(x)-k_0(x)\big).
\end{equation}
\begin{Lem}For any small fixed $d_0>0$, it holds
\begin{equation}
\label{s-equa}
\begin{split}
-\Delta s_{\lambda}=2s_{\lambda}e^{2U}+  q^*_\lambda~\,\,\,\mbox{ in}~~\Big\{x: |x|\leq \frac{d_0}{\theta_\lambda}\Big\},
\end{split}
\end{equation}
with $\displaystyle q^*_\lambda(x)=O\Big(\frac{1}{ 1+|x|^{4-\delta-3\tau}}\Big)$.
\end{Lem}

\begin{proof}
Firstly, using \eqref{4.4-1} and \eqref{k0-equa}, we can compute that
\begin{equation}\label{4.6-1}
\begin{split}
-\Delta s_\lambda(x) =& -\gamma_\lambda^2 \Delta k_\lambda +\gamma_\lambda^2 \Delta k_0 \\
=& \gamma_\lambda^6 \bigg( \Big(1+\frac{w_\lambda}{\gamma_\lambda^2}\Big) e^{2w_\lambda+\frac{w_\lambda^2}{\gamma_\lambda^2}}-e^{2U} \bigg) -\gamma_\lambda^4 \Big(2e^{2U}v_0+e^{2U}(U^2+U)\Big) \\ & -\gamma_\lambda^2 \bigg(2e^{2U}k_0+e^{2U}\Big(\frac{1}{2}U^4+2v_0^2+U^3+v_0\big(2U^2+4U+1\big)\Big)\bigg) \\
=& \gamma_\lambda^6 \Big(e^{2w_\lambda+\frac{w_\lambda^2}{\gamma_\lambda^2}}-e^{2U}\Big) +\gamma_\lambda^4 \Big(w_\lambda e^{2w_\lambda+\frac{w_\lambda^2}{\gamma_\lambda^2}} -2e^{2U}v_0 -e^{2U}(U^2+U)\Big) \\
& -\gamma_\lambda^2 \bigg(2e^{2U}k_0+e^{2U}\Big(\frac{1}{2}U^4+2v_0^2+U^3+v_0\big(2U^2+4U+1\big)\Big)\bigg).
\end{split}
\end{equation}
By Taylor's expansion, we have
\begin{equation}\label{4.6-2}
\begin{split}
& \gamma_\lambda^6 \Big(e^{2w_\lambda+\frac{w_\lambda^2}{\gamma_\lambda^2}}-e^{2U}\Big) \\
=& \gamma^4_\lambda w^2_\lambda e^{2w_\lambda} +\frac{1}{2} \gamma^2_\lambda w^4_\lambda e^{2w_\lambda}+ \frac{1}{6} w^6_\lambda e^{2w_\lambda} +O\Big(\frac{w^8_\lambda}{\gamma^2_\lambda} e^{2w_\lambda}\Big)
+2\gamma^4_\lambda v_\lambda e^{2U} +2\gamma^2_\lambda v^2_\lambda e^{2U} \\
&+ \frac{4}{3} e^{2U} v_\lambda^3 +O\Big( \frac{v^4_\lambda}{\gamma^2_\lambda} e^{2U} \Big) \\
=& \gamma^4_\lambda \Big( w^2_\lambda e^{2w_\lambda} +2e^{2U}v_\lambda \Big) + \gamma^2_\lambda \Big(\frac{1}{2} w^4_\lambda e^{2w_\lambda} +2 v^2_\lambda e^{2U}\Big) + \Big(\frac{1}{6} w^6_\lambda e^{2w_\lambda} +\frac{4}{3} e^{2U} v_\lambda^3\Big) \\
& +O\Big(\frac{w^8_\lambda}{\gamma^2_\lambda}e^{2w_\lambda}\Big)+O\Big(\frac{v^4_\lambda}{\gamma^2_\lambda}e^{2U} \Big).
\end{split}
\end{equation}
Substituting \eqref{4.6-2} into \eqref{4.6-1}, by \eqref{k-def} and \eqref{s-def}, we obtain
\begin{equation*}
\begin{split}
-\Delta s_\lambda(x) =& 2e^{2U} s_\lambda +\gamma_\lambda^4 \Big(w_\lambda e^{2w_\lambda+\frac{w_\lambda^2}{\gamma_\lambda^2}} + w^2_\lambda e^{2w_\lambda} -e^{2U}(U^2+U) \Big) \\
& +\gamma_\lambda^2 \left(\frac{1}{2} w^4_\lambda e^{2w_\lambda} +2 v^2_\lambda e^{2U} -e^{2U}
\Big(\frac{1}{2}U^4+2v_0^2+U^3+v_0\big(2U^2+4U+1\big)\Big)\right) \\
& +\Big(\frac{1}{6} w^6_\lambda e^{2w_\lambda} +\frac{4}{3}e^{2U}v_\lambda^3\Big)
+O\Big(\frac{w^8_\lambda}{\gamma^2_\lambda} e^{2w_\lambda}\Big)+O\Big(\frac{v^4_\lambda}{\gamma^2_\lambda}e^{2U} \Big).
\end{split}
\end{equation*}
Similar to the proof of \eqref{4.4-2}, we can also calculate that
\begin{equation*}
\begin{split}
-\Delta s_\lambda =& 2e^{2U} s_\lambda +e^{2U}\Big(\big(U^4+4U^3+3U^2\big)v_\lambda +\big(4v_0+2U^2+4U+1\big)k_\lambda +\frac{4}{3} v_\lambda^3 \\
& +v_\lambda^2 \big(2U^2+6U+3\big) \Big) +e^{2w_\lambda} \Big(\frac{1}{2} w_\lambda^5 +\frac{1}{6} w^6_\lambda \Big)
+O\bigg(\frac{1}{\gamma_\lambda^2 \big(1+|x|^{4-\delta}\big)}\bigg) \\
=& 2e^{2U} s_\lambda + q_\lambda^\ast (x),
\end{split}
\end{equation*}
with $\displaystyle q^*_\lambda(x)=O\Big(\frac{1}{ 1+|x|^{4-\delta-3\tau}}\Big)$.
\end{proof}

Now we establish the following bound on $s_{\lambda}$:
\begin{Prop}\label{prop_slambda}
Let $s_{\lambda}$ be as in \eqref{s-def}. Then for any small fixed $d_0,\tau_2>0$, there exists $C>0$ such that
\begin{equation}\label{s-inequa}
|s_{\lambda}(x)|\leq C(1+|x| )^{\tau_2}~\mbox{ in}~ B_{\frac{d_0}{\theta_{\lambda}}}(0).
\end{equation}
Moreover, it holds
\[\lim_{\lambda\to 0}s_{\lambda}=s_0\,\,\mbox{ in }C^{2}_{loc}(\mathbb R^{2}),\]
where $s_0$ solves the non-homogeneous linear equation
\begin{equation*}
\begin{split}
-\Delta v-2e^{2U}v =& e^{2U } \Big( \frac{1}{2}U^5+\frac{1}{6}U^6+\big(U^4+4U^3+3U^2\big)v_0
\\ &+\big(2U^2+6U+3\big)v^2_0 +\frac{4}{3}v_0^3 + \big(4v_0+2U^2+4U+1\big)k_0\Big) ~~\mbox{in}~~\R^2.
\end{split}
\end{equation*}
Furthermore,  there exists $C>0$ such that
\begin{equation*}
|s_{0}(x)|\leq C\big(1+|x|\big)^{\tau_2}.
\end{equation*}
\end{Prop}
\begin{proof}
The proof of this proposition will be given in Appendix \ref{s-s}.
\end{proof}

\begin{Prop}\label{prop_wlambda}
We can write
\begin{equation*}
w_\lambda(x)=U(x)+\frac{v_0(x)}
{\gamma^2_\lambda}+\frac{k_{0}(x)}{\gamma^4_{\lambda}}
+\frac{s_{\lambda}(x)}{\gamma^6_{\lambda}},~~\mbox{in}~~ \Omega_{\lambda}:=\frac{\Omega-x_{\lambda}}{\theta_\lambda},
\end{equation*}
with some $C, \tau_0>0$ such that
\begin{equation*}
|v_0(x)|+|k_0(x)|+|s_\lambda(x)|\leq C\big(1+|x|\big)^{\tau_0}.
\end{equation*}
\end{Prop}
\begin{proof}
These can be deduced by Propositions \ref{prop4.3}, \ref{prop_klambda} and \ref{prop_slambda}.
\end{proof}

\section{The Morse index and non-degeneracy of the positive solutions }\label{s4}

\subsection{Basic estimates on the eigenvalues and eigenfunctions}~

\vskip 0.2cm

Using the properties of $v_\lambda$, we can obtain a more accurate estimate of $C_\lambda$ defined in \eqref{def_C}, which is essential to compute Morse index of $u_\lambda$.

\begin{proof}[\underline{\textbf{Proof of \eqref{def_C1} in Theorem \ref{thm-lambdagamma}}}]
Firstly, we have
\begin{align}\label{add-C1}
C_{\lambda} =& \lambda \displaystyle\int_{B_d(x_{\lambda})} u_\lambda(x)e^{u_\lambda^2(x)} \,dx \notag\\
=& \frac{1}{\gamma_\lambda} \int_{B_{\frac{d}{\theta_\lambda}(0)}} \Big( \frac{w_\lambda(x)}{\gamma_\lambda^2}+1 \Big)
e^{2w_\lambda(x)+\frac{w_\lambda^2(x)}{\gamma_\lambda^2}} \,dx \notag\\
=& \frac{1}{\gamma_\lambda} \int_{B_{\frac{d}{\theta_\lambda}(0)}} e^{2U(x)}
\bigg( 1+\frac{2v_\lambda(x)+w_\lambda^2(x)}{\gamma_\lambda^2}
+O\Big(\frac{\big(2v_\lambda+w^2_\lambda\big)^2}{\gamma^4_\lambda}\Big) \bigg) \,dx \notag\\
&+ \frac{1}{\gamma_\lambda^3} \int_{B_{\frac{d}{\theta_\lambda}(0)}} w_\lambda(x) e^{2w_\lambda(x)+\frac{w_\lambda^2(x)}{\gamma_\lambda^2}} \,dx \notag\\
=& \frac{1}{\gamma_\lambda} \int_{\R^2} e^{2U(x)} \,dx +o\Big(\frac{\theta_\lambda}{\gamma_\lambda}\Big) +\frac{1}{\gamma_\lambda^3} \int_{\R^2} e^{2U(x)} \big(2v_0(x)+U^2(x)+U(x)\big) \,dx + o\Big(\frac{1}{\gamma^3_\lambda}\Big) \notag\\
=& \frac{4\pi}{\gamma_\lambda} +\frac{1}{\gamma_\lambda^3} \underbrace{ \int_{\R^2} -\Delta v_0(x) \,dx }_{:=c_0} + o\Big(\frac{1}{\gamma^3_\lambda}\Big)
= \frac{4\pi}{\gamma_\lambda} + \frac{c_0}{\gamma_\lambda^3} + o\Big(\frac{1}{\gamma^3_\lambda}\Big),
\end{align}
the penultimate equality is derived from \eqref{4.3.1}.
We can calculate that
\begin{equation}\label{c0-1}
c_0 = -\int_{\R^2} \Delta v_0(x) \,dx = -\lim_{R \to \infty} \int_{\partial B_R(0)} \frac{\partial v_0}{\partial \nu} \,d\sigma = -\lim_{R \to \infty} \int_{\partial B_R(0)} \nabla v_0 \cdot \frac{x}{|x|} \,d\sigma.
\end{equation}
Moreover, we know that
\begin{equation*}
\begin{split}
v_0(x)=w(x)+ e_0\frac{4-|x|^2}{4+|x|^2}+\sum^2_{i=1}{e_i}\frac{x_i}{4+|x|^2},
\end{split}\end{equation*}
where $w(x)$ is the radial solution of $-\Delta u-2e^{2U}u=e^{2U}\big(U^2+U\big)$.
Direct calculation gives that
\begin{equation*}
\begin{split}
\int_{\partial B_R(0)} e_0 \nabla \Big(\frac{4-|x|^2}{4+|x|^2}\Big) \cdot \frac{x}{|x|} \,d\sigma
&= e_0 \int_{\partial B_R(0)} \frac{-16 |x|}{\big(4+|x|^2\big)^2} \frac{x}{|x|} \cdot \frac{x}{|x|} \,d\sigma \\
&= e_0 \int_{\partial B_R(0)} \frac{-16 |x|}{\big(4+|x|^2\big)^2} \,d\sigma \to 0,~\mbox{as}~R \to \infty,
\end{split}
\end{equation*}
and
\begin{equation*}
\int_{\partial B_R(0)} \nabla \Big(\sum^2_{i=1}{e_i}\frac{x_i}{4+|x|^2}\Big) \cdot \frac{x}{|x|} \,d\sigma = 0.
\end{equation*}
Then \eqref{c0-1} can be rewritten as
\begin{equation}\label{c0-2}
c_0 =  -\lim_{R \to \infty} \int_{\partial B_R(0)} \frac{\partial w}{\partial \nu} \,d\sigma.
\end{equation}

Define $\widetilde{w}(x):=w(2x)$. Since $w(x)$ is a radial function satisfying
\[
-\Delta w - \frac{2}{\big(1+\frac{|x|^2}{4}\big)^2} w = e^{2U}\big(U^2+U\big),
\]
we know that $\widetilde{w}(x)=\widetilde{w}\big(|x|\big)$ solves
\begin{equation*}
\Delta \widetilde{w}(x) + \frac{8}{\big(1+|x|^2\big)^2} \widetilde{w}(x) = f\big(|x|\big),
\end{equation*}
where
\begin{eqnarray*} f\big(|x|\big) = -4 e^{2U(2x)}\big(U^2(2x)+U(2x)\big)= -\frac{4}{(1+|x|^2)^2}
\Big( -\log\big(1+|x|^2\big) + \log^2\big(1+|x|^2\big) \Big).
\end{eqnarray*}
By Lemma \ref{add-lem1}, we deduce that as $r \to +\infty$,
\begin{equation}\label{parw-1}
\begin{split}
\partial_r \widetilde{w}(r)
=& \bigg( \int_{0}^{\infty} t \frac{t^2-1}{t^2+1} \cdot \frac{1}{(1+t^2)^2} \Big( 4\log\big(1+t^2\big) -4\log^2\big(1+t^2\big) \Big) \,dt \bigg) \frac{1}{r} \\
& +O\bigg( \frac{1}{r} \int_{r}^{\infty} s \frac{4}{(1+s^2)^2} \cdot \Big| -\log\big(1+s^2\big) + \log^2\big(1+s^2\big) \Big| \,ds + \frac{|\log r|}{r^3} \bigg).
\end{split}
\end{equation}
Using integration by parts, we compute that
\begin{equation}\label{parw-2}
\begin{split}
& \bigg( \int_{0}^{\infty} t \frac{t^2-1}{(1+t^2)^3} \Big( 4\log\big(1+t^2\big) -4\log^2\big(1+t^2\big) \Big) \,dt \bigg) \frac{1}{r} \\
=& \frac{4}{r} \int_{0}^{\infty} \left( \frac{t}{(1+t^2)^2} - \frac{2t}{(1+t^2)^3} \right) \log\big(1+t^2\big) \,dt
- \frac{4}{r} \int_{0}^{\infty} \left( \frac{t}{(1+t^2)^2} - \frac{2t}{(1+t^2)^3} \right) \log^2\big(1+t^2\big) \,dt \\
=& \frac{4}{r} \cdot \frac{1}{4} - \frac{4}{r} \cdot \frac{3}{4} = -\frac{2}{r},
\end{split}
\end{equation}
and
\begin{equation}\label{parw-3}
\begin{split}
&~ \frac{1}{r} \int_{r}^{\infty} s \frac{4}{(1+s^2)^2} \cdot \Big| -\log\big(1+s^2\big) + \log^2\big(1+s^2\big) \Big| \,ds + \frac{|\log r|}{r^3} \\
\leq &~ \frac{4}{r} \int_{r}^{\infty} \frac{s}{(1+s^2)^2} \log\big(1+s^2\big) \,ds + \frac{4}{r} \int_{r}^{\infty} \frac{s}{(1+s^2)^2} \log^2\big(1+s^2\big) \,ds + \frac{|\log r|}{r^3} \\
=&~ \frac{6\log(1+r^2)}{r(1+r^2)} + \frac{6}{r(1+r^2)} + \frac{2\log^2(1+r^2)}{r(1+r^2)} + \frac{|\log r|}{r^3}.
\end{split}
\end{equation}
Substituting \eqref{parw-2} and \eqref{parw-3} into \eqref{parw-1}, we derive
\begin{equation}\label{parw-4}
\partial_r \widetilde{w}(r) = -\frac{2}{r} + O\left( \frac{ \log^2 r }{r^3} \right)~\mbox{as}~r \to +\infty,
\end{equation}
which together with \eqref{c0-2} gives
\begin{equation}\label{c0-3}
\begin{split}
c_0 &= -\lim_{R \to \infty} \int_{\partial B_R(0)} \frac{\partial w}{\partial \nu} \,d\sigma = -\lim_{R \to \infty} \int_{\partial B_R(0)} \nabla w(x) \cdot \frac{x}{|x|} \,d\sigma \quad (\mbox{set}~x=2y) \\
&= -\lim_{R \to \infty} \int_{\partial B_{\frac{R}{2}}(0)} \nabla \widetilde{w}(y) \cdot \frac{y}{|y|} \,d\sigma'
= -\lim_{R \to \infty} \int_{\partial B_{\frac{R}{2}}(0)} \partial_r \widetilde{w}\big(|y|\big) \frac{y}{|y|} \cdot \frac{y}{|y|} \,d\sigma' \\
&= -\lim_{R \to \infty} \int_{\partial B_{\frac{R}{2}}(0)} \partial_r \widetilde{w}\big(|y|\big) \,d\sigma'
= -\lim_{R \to \infty} \int_{\partial B_{\frac{R}{2}}(0)} -\frac{2}{|y|} \,d\sigma' = 4\pi.
\end{split}
\end{equation}
Combining  with \eqref{add-C1} and \eqref{c0-3}, we find that
\begin{equation*}
C_\lambda = \frac{4\pi}{\gamma_\lambda} + \frac{4\pi}{\gamma_\lambda^3} + o\Big(\frac{1}{\gamma^3_\lambda}\Big).
\end{equation*}
\end{proof}

Now we recall that $\mu_{\lambda,l}$ and $v_{\lambda,l}$ are the eigenvalues and the associated eigenfunctions of the linearized problem
\begin{equation}\label{eigen-1}
\begin{cases}
-\Delta  {v}_{\lambda,l}(x) =\mu_{\lambda,l}\lambda \big(1+2u_\lambda^2\big)e^{u_\lambda^2}{v}_{\lambda,l}
(x)&~\mbox{in}~~\Omega,\\[1mm]
{v}_{\lambda,l}(x)=0&~\mbox{on}~~\partial \Omega,\\[1mm]
\| {v}_{\lambda,l}\|_{L^{\infty}(\Omega)}=1,
\end{cases}
\end{equation}
for $l\in \N$. We have the following identities.
\begin{Lem}
It holds
\begin{equation}\label{puv}
\begin{split}
P\big(u_\lambda,v_{\lambda,l}\big) =&
\lambda \int_{\partial B_d(x_\lambda)} u_\lambda e^{u_\lambda^2}  v_{\lambda,l} \langle x-x_\lambda,\nu\rangle \,d\sigma -2\lambda \int_{ B_d(x_\lambda)} u_\lambda e^{u_\lambda^2}  v_{\lambda,l} \,dx \\&
+\big(\mu_{\lambda,l}-1\big) {\lambda}  \int_{ B_d(x_\lambda)}\big(1+2u^2_\lambda\big)e^{u_\lambda^2} \langle x-x_\lambda,\nabla u_\lambda\rangle  v_{\lambda,l} \,dx,
\end{split}
\end{equation}
and
\begin{equation}\label{quv}
\begin{split}
Q\big(u_\lambda,v_{\lambda,l}\big)
 =&\lambda\int_{\partial B_d(x_\lambda)}u_\lambda e^{u_\lambda^2}  v_{\lambda,l} \nu_i \,d\sigma
+\big(\mu_{\lambda,l}-1\big){\lambda}  \int_{B_d(x_\lambda)}\big(1+2u^2_\lambda\big)e^{u_\lambda^2} v_{\lambda,l} \frac{\partial u_\lambda}{\partial x_i} \,dx.
\end{split}
\end{equation}
\end{Lem}
\begin{proof}

Multiplying \eqref{1.1} by  $\langle x-x_\lambda, \nabla v_{\lambda,l}\rangle$   and integrating on $B_d(x_\lambda)$, we have
\begin{equation}\label{3.2-1}
-\int_{B_d(x_\lambda)}\Delta u_\lambda \big \langle x-x_\lambda, \nabla v_{\lambda,l}\big \rangle \,dx = \lambda \int_{B_d(x_\lambda)} u_\lambda e^{u_\lambda^2} \big \langle x-x_\lambda, \nabla v_{\lambda,l}\big \rangle \,dx.
\end{equation}
By integrating by parts, it follows
\begin{equation}\label{3.2-2}
\begin{split}
\mbox{LHS of}~ \eqref{3.2-1}= &-
\int_{\partial B_d(x_\lambda)} \frac{\partial u_\lambda}{\partial \nu} \big\langle x-x_\lambda,\nabla v_{\lambda,l} \big\rangle \,d\sigma +\int_{B_d(x_\lambda)}\nabla u_\lambda\cdot\nabla v_{\lambda,l} \,dx
\\& +\int_{B_d(x_\lambda)}\Big\langle \nabla u_\lambda, \big\langle \nabla^2 v_{\lambda,l},x-x_\lambda \big\rangle \Big\rangle \,dx,
\end{split}\end{equation}
and
\begin{equation}\label{3.2-3}
\begin{split}
\mbox{RHS of}~ \eqref{3.2-1} = &
\lambda \int_{\partial B_d(x_\lambda)} u_\lambda e^{u_\lambda^2} v_{\lambda,l} \big \langle x-x_\lambda,\nu \big\rangle \,d\sigma -2\lambda \int_{B_d(x_\lambda)}u_\lambda e^{u_\lambda^2}  v_{\lambda,l} \,dx \\
& -\lambda \int_{B_d(x_\lambda)}\big(1+2u^2_\lambda\big) e^{u_\lambda^2} \big \langle x-x_\lambda,\nabla u_\lambda\big\rangle  v_{\lambda,l} \,dx.
\end{split}\end{equation}
Then from \eqref{3.2-2} and \eqref{3.2-3}, we get

\begin{equation}\label{3.2-4}
\begin{split}
&-\int_{\partial B_d(x_\lambda)} \frac{\partial u_\lambda}{\partial \nu}\big \langle x-x_\lambda,\nabla v_{\lambda,l}\big\rangle \,d\sigma +\int_{B_d(x_\lambda)}\nabla u_\lambda\cdot\nabla v_{\lambda,l} \,dx
+\int_{B_d(x_\lambda)}\Big\langle \nabla u_\lambda, \big\langle \nabla^2 v_{\lambda,l},x-x_\lambda \big\rangle\Big\rangle \,dx \\
=& \lambda \int_{\partial B_d(x_\lambda)} u_\lambda e^{u_\lambda^2} v_{\lambda,l} \big \langle x-x_\lambda,\nu\big \rangle \,d\sigma -2\lambda \int_{B_d(x_\lambda)}u_\lambda e^{u_\lambda^2} v_{\lambda,l} \,dx
\\&-\lambda \int_{B_d(x_\lambda)}\big(1+2u^2_\lambda\big)e^{u_\lambda^2} \big \langle x-x_\lambda,\nabla u_\lambda\big \rangle  v_{\lambda,l} \,dx.
\end{split}
\end{equation}
Also multiplying $-\Delta v_{\lambda,l}=\mu_{\lambda,l}\lambda  \big(1+2u^2_\lambda\big) e^{u_\lambda^2} v_{\lambda,l}$ by  $\langle x-x_\lambda, \nabla u_\lambda\rangle$ and integrating on $B_d(x_\lambda)$, we have
\begin{equation}\label{3.2-5}
\begin{split}
&\mu_{\lambda,l}\lambda \int_{B_d(x_\lambda)} \big(1+2u^2_\lambda\big) e^{u_\lambda^2}  \big\langle x-x_\lambda,\nabla u_\lambda\big\rangle  v_{\lambda,l} \,dx \\
=& -\int_{\partial B_d(x_\lambda)} \frac{\partial v_{\lambda,l}}{\partial \nu}\big \langle x-x_\lambda,\nabla u_\lambda\big \rangle \,d\sigma +\int_{B_d(x_\lambda)}\nabla u_\lambda\cdot\nabla  v_{\lambda,l} \,dx \\
& +\int_{B_d(x_\lambda)}\Big\langle \nabla v_{\lambda,l}, \big\langle \nabla^2 u_\lambda,x-x_\lambda \big\rangle\Big\rangle \,dx.
\end{split}
\end{equation}
By divergence theorem, we know
\begin{equation}\label{3.2-6}
\begin{split}
&\int_{\partial B_d(x_\lambda)}  \big \langle \nabla u_\lambda, \nabla v_{\lambda,l}\big \rangle \big\langle x-x_\lambda,\nu\big\rangle \,d\sigma \\
=& 2\int_{B_d(x_\lambda)}\nabla u_\lambda\cdot\nabla  v_{\lambda,l} \,dx +\int_{B_d(x_\lambda)}\Big\langle \nabla u_\lambda, \big\langle \nabla^2 v_{\lambda,l},x-x_\lambda \big\rangle\Big\rangle \,dx \\
& +\int_{B_d(x_\lambda)} \Big\langle \nabla v_{\lambda,l}, \big\langle \nabla^2 u_\lambda,x-x_\lambda \big\rangle\Big\rangle \,dx.
\end{split}
\end{equation}
Hence \eqref{puv} can be obtained from \eqref{3.2-4}, \eqref{3.2-5} and \eqref{3.2-6}.

\vskip 0.1cm

Similarly, \eqref{quv} can also be obtained by multiplying \eqref{1.1} and $-\Delta v_{\lambda,l} =\mu_{\lambda,l}
\lambda  \big(1+2u^2_\lambda\big) e^{u_\lambda^2} v_{\lambda,l}$ by $\displaystyle \frac{\partial v_{\lambda,l}}{\partial x_i}$ and $\displaystyle \frac{\partial u_{\lambda}}{\partial x_i}$ respectively, and then integrating on $B_d(x_\lambda)$.
\end{proof}

Taking
\begin{equation*}
\widetilde{v}_{\lambda,l}(x):= v_{\lambda,l}\big(x_{\lambda}+\theta_{\lambda}x\big),~~x\in \Omega_{\lambda}:=\big\{x: x_{\lambda}+\theta_{\lambda}x\in \Omega\big\},
\end{equation*}
it follows from \eqref{eigen-1} that
\begin{equation}\label{3.4-2}
\begin{cases}
-\Delta \widetilde{v}_{\lambda,l}(x) =\mu_{\lambda,l}\widetilde{V}_\lambda(x)
\widetilde{v}_{\lambda,l}(x)&~\mbox{in}~~\Omega_\lambda,\\[1mm]
\widetilde{v}_{\lambda,l}(x)=0&~\mbox{on}~~\partial \Omega_\lambda,\\[1mm]
\|\widetilde{v}_{\lambda,l}\|_{L^{\infty}(\Omega_\lambda)}=1,
\end{cases}
\end{equation}
where
\begin{equation}\label{3.4-3}
\widetilde{V}_\lambda(x)=
\bigg(\frac{1}{\gamma^2_\lambda}+2\Big(1+\frac{w_\lambda}{\gamma^2_\lambda}\Big)^2\bigg)e^{2w_\lambda(x)+
\frac{w^2_\lambda}{\gamma^2_\lambda}}.
\end{equation}

\vskip 0.1cm

By standard elliptic regularity,  we have following result.
\begin{Lem}\label{lem3.3}
For $l\in \N$, suppose
$\mu_{l}:=\displaystyle\lim_{\lambda\to 0}\mu_{\lambda,l}$,
by taking a subsequence, then there exists $V_{l}(x)\in  C^{2}_{loc}\big(\R^2\big)$
satisfying
\begin{equation*}
\widetilde{v}_{\lambda,l}\to V_{l}~~\mbox{in}~~C^{2}_{loc}\big(\R^2\big),
\end{equation*}
 and
 \begin{equation}\label{3.3.1}
\begin{cases}
-\Delta V_{l}=2\mu_{l} e^{2U} V_{l}\,\,~\mbox{in}~\R^2,\\[1mm]
\|V_{l}\|_{L^{\infty}(\R^2)}\leq 1.
\end{cases}
\end{equation}
\end{Lem}

\begin{proof}
Observe that $\nabla \widetilde{v}_{\lambda,l}$ is uniformly bounded in $L^2(\R^2)$.
Indeed, by Lemma \ref{lem2.4}, we have
\begin{equation*}
\begin{split}
\int_{\R^2}\big|\nabla \widetilde{v}_{\lambda,l}(y)\big|^2 \,dy
&= \int_{\Omega_\lambda}\big|\nabla\widetilde{v}_{\lambda,l} (y)\big|^2 \,dy\\
&=\mu_{\lambda,l} \int_{\Omega_\lambda}\bigg(\frac{1}{\gamma^2_\lambda}+2\Big(1+\frac{w_\lambda(y)}{\gamma^2_\lambda}\Big)^2 \bigg)e^{2w_\lambda(y)+\frac{w^2_\lambda(y)}{\gamma^2_\lambda}} \widetilde{v}^2_{\lambda,l}(y) \,dy \\
& \leq \mu_{\lambda,l} \int_{\Omega_\lambda} \bigg(\frac{1}{\gamma^2_\lambda}+2\Big(1+\frac{w_\lambda}
{\gamma^2_\lambda}\Big)^2\bigg)e^{2w_\lambda+\frac{w^2_\lambda}{\gamma^2_\lambda}}  \,dy \leq C\int_{\Omega_\lambda} \frac{1}{1+|y|^{4-\delta}} \,dy \leq C.
\end{split}
\end{equation*}
Hence by the standard elliptic regularity theory (see \cite{GT1983}), it holds
$$\widetilde{v}_{\lambda,l} \to V_l~~\mbox{in}~~C^{2}_{loc}(\R^2)~~\mbox{as}~~\lambda\to 0,$$
with $V_l$ satisfying \eqref{3.3.1}.

\end{proof}

To estimate the Morse index of $u_\lambda$, we need to study the case $\mu_{l}=0$ and $\mu_{l}=1$. Without loss of generality, we further assume that the eigenfunctions are orthogonal in Dirichlet norm, more precisely
\begin{equation}\label{ortho}
\int_{\Omega}\nabla v_{\lambda,l}\cdot \nabla v_{\lambda,l'}dx =0,~\forall~ l\neq l'.
\end{equation}

\begin{Prop}\label{prop3.4}
Let $l\in \N$, for any $\mu_{l}\in [0,\infty)$, it holds
\begin{equation}\label{3.4.1}
V_{l}\not\equiv 0.
\end{equation}
Furthermore, the following conclusions hold:

\vskip 0.2cm

\noindent \textup{(1)}~ If $\mu_{l}=0$, then $V_{l} \equiv c_{l}$, where $c_{l}$ is a nonzero constant.

\vskip 0.2cm

\noindent \textup{(2)}~ If $\mu_{l}=1$, then there exists $ (a_{l,1},a_{l,2},b_{l}) \in \R^3\backslash \{0\}$ such that
\begin{equation}\label{3.4.2}
\widetilde{v}_{\lambda,l}(x)
=\sum^2_{q=1}\frac{a_{l,q}x_q}{4+|x|^2}+b_{l}\frac{4-|x|^2}{4+|x|^2}+o(1)\,~\mbox{in}~C^1_{loc}(\R^2),
~~\mbox{as}~\lambda\to 0.
\end{equation}
Moreover, we denote $\textbf{a}_l:=\big(a_{l,1},a_{l,2}\big)\in \R^{2}$ and $b_l\in \R$. If $\mu_{l}=\mu_{l'}=1$ for $l\neq l'$, then we have
\begin{equation}\label{3.4.3}
\textbf{a}_l\cdot \textbf{a}_{l'} + 16 b_lb_{l'}=0.
\end{equation}
\end{Prop}
\begin{proof}
Firstly, since $\|v_{\lambda,l}\|_{L^{\infty}(\Omega)}=1$, using \eqref{lim-1}, it holds
\begin{equation}\label{9-30}
\begin{split}
\mu_{\lambda,l}&\lambda \big(1+2u_\lambda^2\big)e^{u_\lambda^2} v_{\lambda,l}
= O\Big(\mu_l \lambda e^{\frac{C}{\gamma^2_\lambda}}\Big)
 \to 0~~\mbox{ locally uniformly in}~~\Omega\backslash \{x_0\}, ~~\mbox{as}~~ \lambda\to 0.
\end{split}
\end{equation}
We also have
\[
\begin{split}
\int_{\Omega} \big|\nabla v_{\lambda,l}\big|^2 \,dx &= \mu_{\lambda,l} \int_{\Omega_\lambda} \bigg(\frac{1}{\gamma_\lambda^2}+2\Big(1+\frac{w_\lambda(x)}{\gamma_\lambda^2}\Big)^2\bigg) e^{2w_\lambda(x)+\frac{w^2_\lambda(x)}{\gamma^2_\lambda}} \widetilde{v}_{\lambda,l}^2(x) \,dx \\
& \leq C \int_{\Omega_\lambda} \frac{1}{1+|x|^{4-\delta}} \,dx \leq C,
\end{split}
\]
which implies that $\{v_{\lambda,l}\}$ is bounded in $H_0^1(\Omega)$. By the standard elliptic regularity theory, we find that
\[
\|v_{\lambda,l}\|_{C^{2,\alpha}_{loc}(\Omega)} \leq C,~\mbox{for~some}~\alpha \in (0,1),
\]
where $C$ is independent of $\lambda$. Then we can assume that there exists a bounded function $v_l$ such that
\[
v_{\lambda,l} \to v_l~\mbox{in}~C^2_{loc}(\Omega),~\mbox{as}~\lambda \to 0.
\]
Furthermore, \eqref{eigen-1} and \eqref{9-30} imply that $v_l$ satisfies
\[
-\Delta v_l = 0~\mbox{in}~\Omega \backslash \{x_0\}, \quad v_l \in H_0^1(\Omega).
\]
Hence, we deduce that $v_l \equiv 0$ in $\overline{\Omega} \backslash \{x_0\}$, and since $v_{\lambda,l}=0$ on $\partial \Omega$, it follows
\begin{equation}\label{3.4-1}
\begin{split}
{v}_{\lambda,l}(x)\to 0~~\mbox{locally uniformly in}~~\overline{\Omega}\backslash \{x_0\}, ~~\mbox{as}~~ \lambda\to 0.
\end{split}
\end{equation}
Recall that $\widetilde{v}_{\lambda,l}$ satisfies \eqref{3.4-2} and \eqref{3.4-3}. Let us denote by $\widetilde{t}_\lambda\in \Omega_\lambda$ a point such that
\begin{equation}\label{3.4-4}
\begin{split}
\widetilde{v}_{\lambda,l}(\widetilde{t}_\lambda)=\max_{x\in\Omega_\lambda} \widetilde{v}_{\lambda,l}(x)=1.
\end{split}
\end{equation}
Similar to the choice of $d$ in Lemma \ref{lem2.3}, we have $B_{\frac{d}{\theta_\lambda}}(0)\subset \Omega_\lambda$. Moreover, since $v_{\lambda,l} \big(x_\lambda+\theta_\lambda \widetilde{t}_\lambda \big) =\widetilde{v}_{\lambda,l}(\widetilde{t}_\lambda)=1$ and $v_{\lambda,l}$ satisfies \eqref{3.4-1}, the fact that $x_\lambda \to x_0$ as $\lambda \to 0$ indicates that $x_\lambda+\theta_\lambda \widetilde{t}_\lambda \in B_d(x_\lambda)$ for $\lambda>0$ small. That is,
\begin{equation}\label{3.4-5}
|\widetilde{t}_\lambda|<\frac{d}{\theta_\lambda},~~\mbox{for}~~\lambda~~\mbox{small}.
\end{equation}
Assume by contradiction that $V_{l}\equiv 0$, namely that
\begin{equation}\label{3.4-6}
\begin{split}
 \widetilde{v}_{\lambda,l}(x) \to 0~~\mbox{locally uniformly in}~~\R^2, ~~\mbox{as}~~ \lambda\to 0.
\end{split}
\end{equation}
Then $\widetilde{t}_\lambda \to +\infty$ and so in particular
\begin{equation}\label{3.4-7}
|\widetilde{t}_\lambda| > 1,~\mbox{for}~\lambda~\mbox{small}.
\end{equation}

Let $\widehat{v}_{\lambda,l}(x)$ be the Kelvin transform of $\widetilde{v}_{\lambda,l}(x)$, namely
\begin{equation*}
\begin{split}
\widehat{v}_{\lambda,l}(x):=\widetilde{v}_{\lambda,l}\Big(\frac{x}{|x|^2}\Big).
\end{split}
\end{equation*}
Since $B_{\frac{d}{\theta_\lambda}}(0) \subset \Omega_\lambda$ and $\widetilde{v}_{\lambda,l}(x)$ is defined in $\Omega_\lambda$, we observe that $\widehat{v}_{\lambda,l}(x)$ is well defined in $\R^2\backslash
B_{\frac{\theta_\lambda}{d}}(0)$, and by \eqref{3.4-2}, it satisfies
\begin{equation}\label{3.4-8}
-\Delta \widehat{v}_{\lambda,l}(x) =\frac{\mu_{\lambda,l}}{|x|^4}\widetilde{V}_\lambda\Big(\frac{x}{|x|^2}\Big)
\widehat{v}_{\lambda,l}(x),~\mbox{in}~~\R^2\backslash B_{\frac{\theta_\lambda}{d}}(0).
\end{equation}
Moreover, by \eqref{3.4-6}, we have
\begin{equation}\label{3.4-9}
 \widehat{v}_{\lambda,l}(x)\to 0 ~\mbox{as}~\lambda \to 0,~\mbox{pointwise for any}~x\neq 0.
\end{equation}
Now we define
\begin{equation*}
\check{V}_{\lambda,l}(x):=
\begin{cases}
\displaystyle \frac{\mu_{\lambda,l}}{|x|^4}\widetilde{V}_\lambda\Big(\frac{x}{|x|^2}\Big)
\widehat{v}_{\lambda,l}(x),&~\mbox{in}~~B_{1}(0)\backslash \overline{B_{\frac{\theta_\lambda}{d}}(0)},\\[3mm]
\displaystyle 0,&~~\mbox{in}~~B_{\frac{\theta_\lambda}{d}}(0).
\end{cases}
\end{equation*}
By \eqref{limw}, \eqref{3.4-3} and \eqref{3.4-9}, we have that $\check{V}_{\lambda,l}(x)\to 0$ pointwise in $B_1(0)$ as $\lambda\to 0$. Furthermore, \eqref{2.4.2} implies that
\begin{equation*}
\check{V}^2_{\lambda,l}(x) \leq \widehat{C}_{\varepsilon} \Big(\frac{1}{|x|^4} \frac{1}{1+\frac{1}{|x|^{4-2\varepsilon}}}
\Big)^2
= \frac{\widehat{C}_{\varepsilon}}{|x|^{4\varepsilon}  \big(1+ |x|^{4-2\varepsilon} \big)^2}
\in L^1\big(B_1(0)\big),~\mbox{by choosing}~\varepsilon<\frac{1}{4}.
\end{equation*}
Thus by the dominated convergence theorem, we find
\begin{equation}\label{3.4-10}
\lim_{\lambda\to 0} \| \check{V}_{\lambda,l}\|_{L^2\big(B_1(0)\big)} =0.
\end{equation}
Let $\overline{V}_{\lambda,l}(x)\in H^1_0\big(B_1(0)\big)$ be a function such that
\begin{equation}\label{3.4-11}
 \begin{cases}
-\Delta \overline{V}_{\lambda,l}  =\check{V}_{\lambda,l},&~\mbox{in}~~B_{1}(0),\\[1mm]
  \overline{V}_{\lambda,l}=0,&~~\mbox{on}~~  \partial B_{1}(0).
 \end{cases}
\end{equation}
Then by \eqref{3.4-10} and the elliptic regularity, we have that
\begin{equation}\label{3.4-12}
\begin{split}
\overline{V}_{\lambda,l}(x) \to 0~~\mbox{uniformly in}~~B_1(0) ~~\mbox{as}~~ \lambda\to 0.
\end{split}
\end{equation}
Now we consider the difference $\overline{V}_{\lambda,l}-\widehat{v}_{\lambda,l}$. \eqref{3.4-8} and \eqref{3.4-11} imply that $\overline{V}_{\lambda,l}-\widehat{v}_{\lambda,l}$ is harmonic in $B_{1}(0)\backslash \overline{B_{\frac{\theta_\lambda}{d}}(0)}$, then by \eqref{3.4-6}, \eqref{3.4-11}, \eqref{3.4-12} and the maximum principle for harmonic functions, we derive
\begin{equation}\label{3.4-13}
\begin{split}
\|\overline{V}_{\lambda,l}-\widehat{v}_{\lambda,l}\|_{L^{\infty} \big(B_{1}(0)\backslash\overline
{B_{\frac{\theta_\lambda}{d}}(0)}\big)}
&\leq  \|\overline{V}_{\lambda,l}-\widehat{v}_{\lambda,l}\|_{L^{\infty}\big (\partial B_{1}(0)\big)}+
\|\overline{V}_{\lambda,l}-\widehat{v}_{\lambda,l}\|_{L^{\infty}\big(\partial {B_{\frac{\theta_\lambda}{d}}(0)}\big)}\\
&\leq \|\widehat{v}_{\lambda,l}\|_{L^{\infty}\big(\partial B_{1}(0)\big)}+ \|\overline{V}_{\lambda,l}\|_{L^{\infty}\big(\partial {B_{\frac{\theta_\lambda}{d}}(0)}\big)}+
\|\widehat{v}_{\lambda,l}\|_{L^{\infty}\big(\partial {B_{\frac{\theta_\lambda}{d}}(0)}\big)}\\
&= \|\widetilde{v}_{\lambda,l}\|_{L^{\infty}\big(\partial B_{1}(0)\big )} +\|\widetilde{v}_{\lambda,l}\|_{L^{\infty}\big( \partial {B_{\frac{d}{\theta_\lambda}}(0)} \big)} +o\big(1\big)\\
&= \|{v}_{\lambda,l}\|_{L^{\infty}\big(\partial B_{d}(x_\lambda)\big)} +o\big(1\big) \\
&\leq  \|{v}_{\lambda,l}\|_{L^{\infty}\big(\Omega \backslash B_{\frac{d}{2}}(x_0)\big)} +o\big(1\big) \overset{\eqref{3.4-1}}=o\big(1\big),
\end{split}
\end{equation}
the last inequality follows by the fact that $x_\lambda \to x_0$, as $\lambda \to 0$. Combining \eqref{3.4-12} with \eqref{3.4-13}, we obtain that
\begin{equation}\label{3.4-14}
\begin{split}
\| \widehat{v}_{\lambda,l}\|_{L^{\infty}\big(B_{1}(0)\backslash \overline{B_{\frac{\theta_\lambda}{d}}(0)}\big)}\leq
\|\overline{V}_{\lambda,l} \|_{L^{\infty}\big(B_{1}(0)\big)}+
\|\overline{V}_{\lambda,l}-\widehat{v}_{\lambda,l}\|_{L^{\infty}\big(B_{1}(0)\backslash \overline{B_{\frac{\theta_\lambda}{d}}(0)}\big)} =o\big(1\big).
\end{split}
\end{equation}

On the other hand, it follows from \eqref{3.4-5} and \eqref{3.4-7} that
\begin{equation*}
\frac{\widetilde{t}_\lambda}{|\widetilde{t}_\lambda|^2}\in B_{1}(0)\backslash \overline{B_{\frac{\theta_\lambda}{d}}(0)},~\mbox{for}~\lambda ~\mbox{small}.
\end{equation*}
Also by the definition of $\widehat{v}_{\lambda,l}$ and \eqref{3.4-4}, we get
\begin{equation*}
\begin{split}
\widehat{v}_{\lambda,l}\Big(\frac{\widetilde{t}_\lambda}{|\widetilde{t}_\lambda|^2}\Big) =\widetilde{v}_{\lambda,l}\big(\widetilde{t}_\lambda\big)=1,
\end{split}
\end{equation*}
which contradicts \eqref{3.4-14}. Hence the proof of \eqref{3.4.1} is completed.

\vskip 0.1cm

If $\mu_l=0$, then from \eqref{3.3.1}, we have
\begin{equation*}
\begin{cases}
-\Delta V_{l}=0\,\,~\mbox{in}~\R^2,\\[1mm]
\|V_{l}\|_{L^{\infty}(\R^2)}\leq 1.
\end{cases}
\end{equation*}
Hence we find that $V_{l}$ is a constant, which together with \eqref{3.4.1} implies that $V_{l}$ is a nonzero constant.

\vskip 0.1cm

If $\mu_{l}=1$,  from \eqref{3.3.1}, we have

\begin{equation*}
\begin{cases}
-\Delta V_{l}=2 e^{2U} V_{l}\,\,~\mbox{in}~\R^2,\\[1mm]
\|V_{l}\|_{L^{\infty}(\R^2)}\leq 1.
\end{cases}
\end{equation*}
It follows from Lemmas \ref{lem3.1} and \ref{lem3.3} that there exists $(a_{l,1},a_{l,2},b_{l})\in \R^3$ such that
\begin{equation}\label{3.4-15}
\widetilde{v}_{\lambda,l}(x)
=\sum^2_{q=1}\frac{a_{l,q}x_q}{4+|x|^2}+b_{l}\frac{4-|x|^2}{4+|x|^2}+o\big(1\big)\,~\mbox{in}~C^1_{loc}(\R^2),
~~\mbox{as}~\lambda\to 0.
\end{equation}
Combining \eqref{3.4-15} with \eqref{3.4.1}, we deduce that $(a_{l,1},a_{l,2},b_{l})\neq (0,0,0)$.

\vskip 0.2cm

Furthermore, if $\mu_{l}=\mu_{l'}=1$ for $l\neq l'$, then from \eqref{ortho}, we have
\begin{align}\label{3.4-16}
0 =& \lambda \int_{\Omega} \big(1+2u_\lambda^2(x)\big)e^{u_\lambda^2(x)}v_{\lambda,l}(x)v_{\lambda,l'}(x) \,dx \notag\\
=& \int_{\Omega_\lambda} \bigg(\frac{1}{\gamma^2_\lambda}+2\Big(1+\frac{w_\lambda(z)}{\gamma_\lambda^2}\Big)^2\bigg)
e^{2w_\lambda(z)+\frac{w_\lambda^2(z)}{\gamma^2_\lambda}} \widetilde{v}_{\lambda,l}(z) \widetilde{v}_{\lambda,l'}(z) \,dz \notag\\
=& \int_{B_{\frac{d}{\theta_\lambda}}(0)} \bigg(\frac{1}{\gamma^2_\lambda}+2\Big(1+\frac{w_\lambda(z)}
{\gamma_\lambda^2}\Big)^2\bigg) e^{2w_\lambda(z)+\frac{w_\lambda^2(z)}{\gamma^2_\lambda}} \widetilde{v}_{\lambda,l}(z) \widetilde{v}_{\lambda,l'}(z) \,dz \notag\\
&+O\bigg(\int_{\Omega_\lambda \backslash B_{\frac{d}{\theta_\lambda}}(0)} \frac{1}{1+|z|^{4-\delta}} \,dz \bigg) \notag\\
=& \int_{B_{\frac{d}{\theta_\lambda}}(0)} \bigg(\frac{1}{\gamma^2_\lambda}+2\Big(1+\frac{w_\lambda(z)}
{\gamma_\lambda^2}\Big)^2\bigg) e^{2w_\lambda(z)+\frac{w_\lambda^2(z)}{\gamma^2_\lambda}} \widetilde{v}_{\lambda,l}(z) \widetilde{v}_{\lambda,l'}(z) \,dz + O\big( \theta_\lambda^{2-\delta} \big).
\end{align}
Using the dominated convergence theorem and passing to the limit in \eqref{3.4-16}, we obtain
\begin{equation*}
\begin{split}
0 &=\int_{\R^2} 2e^{2U(z)} \bigg(\sum^2_{q=1}\frac{a_{l,q}a_{l',q} z_q^2}{\big(4+|z|^2\big)^2}+
b_{l}b_{l'}\frac{\big(4-|z|^2\big)^2}{\big(4+|z|^2\big)^2}\bigg) \,dz \\
&= \sum\limits_{q=1}^2 a_{l,q} a_{l',q} \int_{\R^2} e^{2U(z)} \frac{|z|^2}{(4+|z|^2)^2} \,dz + b_{l}b_{l'} \int_{\R^2} 2e^{2U(z)} \frac{\big(4-|z|^2\big)^2}{\big(4+|z|^2\big)^2} \,dz \\
&= \frac{\pi}{6} \textbf{a}_l \cdot \textbf{a}_{l'} + \frac{8\pi}{3} b_l b_{l'},
\end{split}
\end{equation*}
which means that $\textbf{a}_l \cdot \textbf{a}_{l'} + 16 b_l b_{l'}=0$.
\end{proof}

\begin{Lem}\label{lem3.5}
Let $l\in \mathbb{N}$, if $\mu_{l}=1$, then
\begin{equation}\label{3.5.1}
v_{\lambda,l}(x)= -\frac{4\pi b_{l}}{\gamma_\lambda^2} G(x,x_{\lambda})+o\Big(\frac{1}{\gamma_\lambda^2}\Big) ~\,~\mbox{in}~C^1\big( \Omega \backslash B_{2d}(x_{\lambda})\big),
\end{equation}
where $b_{l}$ is the constant in \eqref{3.4.2}.
Moreover, if $b_{l}\neq 0$, then
\begin{equation}\label{3.5.2}
\mu_{\lambda,l}=1+ \frac{3}{\gamma_\lambda^4} +o\Big(\frac{1}{\gamma_\lambda^4}\Big)~\,\mbox{as}~~\lambda \to 0.
\end{equation}
\end{Lem}

\begin{proof}
Firstly, for $x\in \Omega\backslash   B_{2d}(x_{\lambda})$, we have
\begin{equation}\label{3.5-1}
\begin{split}
v_{\lambda,l}(x)=& \lambda \mu_{\lambda,l} \int_{\Omega}G(y,x)
 e^{u^2_\lambda(y)}\big(1+2u^2_\lambda (y)\big) v_{\lambda,l}(y) \,dy\\=& \lambda  \mu_{\lambda,l}
  \underbrace{\int_{B_d(x_{\lambda})}G(y,x)
 e^{u^2_\lambda (y)}\big(1+2u^2_\lambda (y)\big)v_{\lambda,l}(y) \,dy}_{:=K_{\lambda,l}}\\& + \lambda \mu_{\lambda,l} \int_{ \Omega\backslash  B_{d}(x_{\lambda})}G(y,x)
 e^{u^2_\lambda (y)}\big(1+2u^2_\lambda (y)\big) v_{\lambda,l}(y) \,dy.
\end{split}
\end{equation}
Similar to \eqref{2.5-3}, for any fixed $\alpha>1$, we have
\begin{equation}\label{3.5-2}
\begin{split}
&~ \lambda \mu_{\lambda,l} \int_{ \Omega\backslash  B_{d}(x_{\lambda})} G(y,x) e^{u^2_\lambda (y)}\big(1+2u^2_\lambda (y)\big) v_{\lambda,l}(y) \,dy \\
\leq &~ \lambda \mu_{\lambda,l} \bigg(\int_{ \Omega\backslash B_{d}(x_{\lambda})} \Big|e^{u^2_\lambda (y)} \big(1+2u^2_\lambda(y)\big) v_{\lambda,l}(y)\Big|^\alpha \,dy\bigg)^{\frac{1}{\alpha}}  \bigg(
\int_{ \Omega\backslash B_{d}(x_{\lambda})} \big(G(y,x)\big)^{\frac{\alpha}{\alpha-1}} \,dy\bigg)^{1-\frac{1}{\alpha}} \\
=&~ O\Bigg( \lambda \gamma_\lambda^2 e^{\gamma_\lambda^2} \theta_\lambda^{\frac{2}{\alpha}} \bigg(
\int_{\Omega_\lambda \backslash B_{\frac{d}{\theta_\lambda}(0)}} \frac{1}{\big(1+|z|^{4-\delta}\big)^\alpha} \,dz \bigg)^{\frac{1}{\alpha}}\Bigg)= O\big( \theta_\lambda^{2-\delta} \big),~~\mbox{for any small fixed}~~\delta>0.
\end{split}
\end{equation}

Moreover, it holds
\begin{align}\label{K}
K_{\lambda,l}=&  G(x_{\lambda},x)\underbrace{\int_{B_d(x_{\lambda})} e^{u_\lambda^2(y)}\big(1+2u^2_\lambda (y)\big) v_{\lambda,l}(y) \,dy}_{:=A_{\lambda,l}} \notag\\
& +\underbrace{\int_{B_d(x_{\lambda})}\big(G(y,x)-G(x_{\lambda},x)\big)
e^{u_\lambda^2(y)}\big(1+2u^2_\lambda (y)\big) v_{\lambda,l}(y) \,dy}_{:=B_{\lambda,l}},
\end{align}
and by Taylor's expansion, it follows
\begin{align*}
B_{\lambda,l}=&\frac{\theta_\lambda}{\lambda  }
\int_{B_{\frac{d}{\theta_\lambda}}(0)}
\big\langle\nabla G(x_{\lambda},x),z\big\rangle\bigg(\frac{1}{\gamma_\lambda^2}+2\Big(1+\frac{w_{\lambda}(z)}{\gamma_\lambda^2}\Big)^2\bigg)
e^{2w_{\lambda}(z)+\frac{w^2_{\lambda}(z)}{\gamma_\lambda^2} }\widetilde{v}_{\lambda,l}(z) \,dz\\&
+O\left(\frac{\theta^2_{\lambda} }{\lambda }
\int_{B_{\frac{d}{\theta_{\lambda}}}(0)}
\big| z \big|^2 \cdot e^{2w_{\lambda}(z)+\frac{w^2_{\lambda}(z)}{\gamma_\lambda^2} } \,dz\right) =
O\Big(\frac{\theta_\lambda}{\lambda }\Big),
\end{align*}
which together with \eqref{3.5-1}, \eqref{3.5-2} and \eqref{K} imply that
\begin{equation}\label{3.5-3}
v_{\lambda,l}(x)= \lambda \mu_{\lambda,l} A_{\lambda,l} G\big(x_{\lambda},x\big)+ O(\theta_\lambda)\,\, ~\mbox{in}~ \Omega\backslash  B_{2d}(x_{\lambda}).
\end{equation}
Similar to the estimates of \eqref{3.5-3}, we find
\begin{equation*}
\begin{split}
\frac{\partial v_{\lambda,l}(x)}{\partial x_i}=& \lambda \mu_{\lambda,l} \int_{\Omega}D_{x_i} G(y,x) e^{u_\lambda^2(y)}
\big(1+2u^2_\lambda (y)\big)v_{\lambda,l}(y) \,dy\\
=&  \lambda  \mu_{\lambda,l} \int_{B_d(x_{\lambda})} D_{x_i} G(y,x) e^{u_\lambda^2(y)}\big(1+2u^2_\lambda (y)\big) v_{\lambda,l}(y) \,dy\\
& + \lambda \mu_{\lambda,l} \int_{ \Omega \backslash B_{d}(x_{\lambda})}D_{x_i} G(y,x)e^{u_\lambda^2(y)}
\big(1+2u^2_\lambda (y)\big) v_{\lambda,l}(y) \,dy \\
=& \lambda \mu_{\lambda,l} A_{\lambda,l} D_{x_i}G\big(x_{\lambda},x\big)+ O(\theta_\lambda)~\,~\mbox{in}~\Omega \backslash B_{2d}\big(x_{\lambda}\big).
\end{split}
\end{equation*}
Thus we have
\begin{equation}\label{3.5-6}
v_{\lambda,l}(x)= \lambda \mu_{\lambda,l} A_{\lambda,l} G\big(x_{\lambda},x\big)+ O(\theta_\lambda)\,\, ~\mbox{uniformly in}~ C^1 \big(\Omega\backslash  B_{2d}(x_{\lambda})\big).
\end{equation}

We recall that
\begin{equation*}
-\Delta  {u}_{\lambda} = \lambda  u_\lambda e^{u_\lambda^2}~\mbox{in}~~\Omega ~\,\,\mbox{and}~\,\,
-\Delta  {v}_{\lambda,l} =\mu_{\lambda,l}\lambda \big(1+2u_\lambda^2\big)e^{u_\lambda^2}
 {v}_{\lambda,l}~\mbox{in}~~\Omega.
\end{equation*}
Then multiplying the above two identities by $v_{\lambda,l}$ and $u_\lambda$ respectively, integrating by parts and subtracting, we get
\begin{equation}\label{3.5-4}
\int_{\Omega}\lambda u_\lambda e^{u_\lambda^2}v_{\lambda,l} \Big(\mu_{\lambda,l}\big(1+2u_\lambda^2\big)-1\Big) \,dx=0.
\end{equation}
Since $u_\lambda(x_\lambda)=u_\lambda(y)-\frac{1}{\gamma_\lambda} w_\lambda\big(\frac{y-x_\lambda}{\theta_\lambda}\big)$ in $\Omega$, by \eqref{3.5-4}, we derive
\begin{align}\label{3.5-5}
A_{\lambda,l}=&  \frac{1}{\gamma_\lambda} \int_{B_d(x_{\lambda})} e^{u^2_\lambda(y)}\big(1+2u^2_\lambda (y)\big) v_{\lambda,l}(y)u_\lambda(x_\lambda) \,dy \notag\\
=& \frac{1}{\gamma_\lambda} \int_{B_d(x_{\lambda})} e^{u^2_\lambda(y)}\big(1+2u^2_\lambda (y)\big) v_{\lambda,l}(y)
\Big(u_\lambda(y)-\frac{1}{\gamma_\lambda} w_\lambda \big(\frac{y-x_\lambda}{\theta_\lambda}\big)\Big) \,dy \notag\\
=&  \frac{1}{\gamma_\lambda \mu_{\lambda,l}} \int_{\Omega}  u_\lambda(y) e^{u_\lambda^2(y)} v_{\lambda,l}(y) \,dy
- \frac{1}{\gamma_\lambda} \int_{\Omega \backslash B_d(x_{\lambda})} e^{u^2_\lambda(y)}\big(1+2u^2_\lambda (y)\big) v_{\lambda,l}(y) u_\lambda(y) \,dy \notag\\
& -\frac{1}{\gamma^2_\lambda} \int_{B_d(x_{\lambda})} e^{u^2_\lambda(y)}\big(1+2u^2_\lambda (y)\big) v_{\lambda,l}(y) w_\lambda \big(\frac{y-x_\lambda}{\theta_\lambda}\big) \,dy \notag\\
=& \frac{1}{\lambda \gamma_\lambda^2 \mu_{\lambda,l}} \int_{\Omega_\lambda} \Big(1 +\frac{w_\lambda(z)}{\gamma^2_\lambda} \Big) e^{2w_\lambda(z) +\frac{w^2_\lambda(z)}{\gamma^2_\lambda} }  \widetilde{v}_{\lambda,l}(z) \,dz
+O\Big(\frac{\theta_\lambda^{2-\delta}}{\lambda}\Big) \notag\\
&-\frac{1}{\lambda \gamma_\lambda^2} \int_{B_{\frac{d}{\theta_\lambda}}(0)} \bigg(\frac{1}{\gamma_\lambda^2}+2\Big(1+\frac{w_{\lambda}(z)}
{\gamma_\lambda^2}\Big)^2\bigg) e^{2w_{\lambda}(z)+\frac{w^2_{\lambda}(z)}{\gamma_\lambda^2} }\widetilde{v}_{\lambda,l}(z) w_\lambda \big(z\big) \,dz \notag\\
=& \frac{1}{\lambda \gamma_\lambda^2 \mu_{\lambda,l}} \left(\int_{\R^2} e^{2U(z)}\Big(\sum^2_{q=1}\frac{a_{l,q}z_q}
{4+|z|^2} + b_{l}\frac{4-|z|^2}{4+|z|^2}\Big) \,dz+o(1)\right) +O\Big(\frac{\theta_\lambda^{2-\delta}}{\lambda}\Big) \notag\\
& -\frac{1}{\lambda \gamma_\lambda^2} \left(\int_{\R^2} 2 e^{2U(z)}\Big(\sum^2_{q=1}\frac{a_{l,q}z_q}{4+|z|^2}
+b_{l}\frac{4-|z|^2}{4+|z|^2}\Big) U\big(z\big) \,dz+o(1)\right) \notag\\
=& -\frac{ 4\pi b_{l}}{\lambda \gamma_\lambda^2} +o\Big(\frac{1}{\lambda \gamma_\lambda^2}\Big).
\end{align}
It follows from \eqref{3.5-6} and \eqref{3.5-5} that
\begin{equation*}
v_{\lambda,l}(x) = -\frac{4\pi b_{l}}{\gamma_\lambda^2} G(x,x_{\lambda})
+o\Big(\frac{1}{\gamma_\lambda^2}\Big)~\,~\mbox{in}~C^1 \big(\Omega\backslash  B_{2d}(x_{\lambda})\big).
\end{equation*}

Next, we  prove \eqref{3.5.2}.
Firstly, it holds
\begin{equation}\label{add1-1}
\begin{split}
A_0 &= \lambda \int_{B_d(x_\lambda)} u_\lambda e^{u_\lambda^2} v_{\lambda,l} \,dx \\
&= \lambda \int_{B_d(x_\lambda)} u_\lambda e^{u_\lambda^2} v_{\lambda,l} \frac{u_\lambda(x)-\big(u_\lambda(x)-\gamma_\lambda\big)}{\gamma_\lambda} \,dx \\
&= -\frac{\lambda}{\gamma_\lambda} \int_{B_d(x_\lambda)} u_\lambda e^{u_\lambda^2} v_{\lambda,l} \big(u_\lambda(x)-\gamma_\lambda\big) \,dx + \frac{\lambda}{\gamma_\lambda} \int_{B_d(x_\lambda)} u_\lambda^2 e^{u_\lambda^2} v_{\lambda,l} \,dx \\
&:= I_1+I_2,
\end{split}
\end{equation}
and
\begin{equation*}
I_2 = \frac{\lambda}{2\gamma_\lambda} \int_{B_d(x_\lambda)} \big(1+2u_\lambda^2\big) e^{u_\lambda^2} v_{\lambda,l} \,dx - \frac{\lambda}{2\gamma_\lambda} \int_{B_d(x_\lambda)} e^{u_\lambda^2} v_{\lambda,l} \,dx := I_{21}+I_{22}.
\end{equation*}
Direct calculations show that
\begin{equation}\label{add1-2}
\begin{split}
I_1 &= -\frac{1}{\gamma_\lambda^3} \int_{B_{\frac{d}{\theta_\lambda}}(0)} \Big(\frac{w_\lambda(x)}{\gamma_\lambda^2}+1\Big) e^{2w_\lambda(x)+\frac{w_\lambda^2(x)}{\gamma_\lambda^2}}
\widetilde{v}_{\lambda,l}(x) w_\lambda(x) \,dx \\
&= -\frac{1}{\gamma_\lambda^3} \bigg( b_l \int_{\R^2} U(x) e^{2U(x)} \frac{4-|x|^2}{4+|x|^2} \,dx + o(1)\bigg)
= -\frac{2\pi b_l}{\gamma_\lambda^3} + o\Big(\frac{1}{\gamma^3_\lambda}\Big),
\end{split}
\end{equation}
\begin{equation}\label{add1-3}
\begin{split}
I_{21} &= -\frac{1}{2\gamma_\lambda \mu_{\lambda,l}} \int_{B_d(x_\lambda)} \Delta v_{\lambda,l} \,dx
= -\frac{1}{2\gamma_\lambda \mu_{\lambda,l}} \int_{\partial B_d(x_\lambda)} \frac{\partial v_{\lambda,l}}{\partial \nu} \,d\sigma \\
&= -\frac{1}{2\gamma_\lambda \mu_{\lambda,l}} \int_{\partial B_d(x_\lambda)} -\frac{4\pi b_l}{\gamma^2_\lambda}  \frac{\partial G(x,x_{\lambda})}{\partial \nu} \,d\sigma + o\Big(\frac{1}{\gamma^3_\lambda}\Big) \\
&= \frac{2\pi b_l}{\gamma^3_\lambda \mu_{\lambda,l}} \int_{\partial B_d(x_\lambda)} \frac{\partial G(x,x_{\lambda})}{\partial \nu} \,d\sigma + o\Big(\frac{1}{\gamma^3_\lambda}\Big)
= -\frac{2\pi b_l}{\gamma^3_\lambda } + o\Big(\frac{1}{\gamma^3_\lambda}\Big),
\end{split}
\end{equation}
and
\begin{equation}\label{add1-4}
\begin{split}
I_{22} &= -\frac{1}{2\gamma_\lambda^3} \int_{B_{\frac{d}{\theta_\lambda}}(0)} e^{2w_\lambda(x)+\frac{w_\lambda^2(x)}{\gamma_\lambda^2}} \widetilde{v}_{\lambda,l}(x) \,dx \\
&= -\frac{1}{2\gamma_\lambda^3} \bigg( b_l \int_{\R^2} e^{2U(x)} \frac{4-|x|^2}{4+|x|^2} \,dx +o(1) \bigg) = o\Big(\frac{1}{\gamma^3_\lambda}\Big).
\end{split}
\end{equation}
Substituting \eqref{add1-2}, \eqref{add1-3} and \eqref{add1-4} into \eqref{add1-1}, we derive that
\begin{equation}\label{add1-A0}
A_0 = \lambda \int_{B_d(x_\lambda)} u_\lambda e^{u_\lambda^2} v_{\lambda,l} \,dx = -\frac{4\pi b_l}{\gamma^3_\lambda } + o\Big(\frac{1}{\gamma^3_\lambda}\Big).
\end{equation}
Also by scaling, we can compute
\begin{equation}\label{3.5-8}
\begin{split}
&\int_{B_{d}(x_{\lambda})} \big(1+  2u_\lambda^2\big) e^{u^2_\lambda} \langle x-x_\lambda,\nabla u_\lambda\rangle  v_{\lambda,l} \,dx \\
=& \theta_\lambda^2 \int_{B_{\frac{d}{\theta_\lambda}(0)}} \bigg(1+2\Big(\gamma_\lambda+
\frac{w_\lambda(z)}{\gamma_\lambda}\Big)^2\bigg) e^{\big(\gamma_\lambda+\frac{w_\lambda(z)}{\gamma_\lambda}\big)^2}
\big\langle z,\nabla \widetilde{u}_{\lambda}(z) \big\rangle  \widetilde{v}_{\lambda,l}(z) \,dz \\
=& \frac{2}{\lambda \gamma_\lambda} \left(\int_{\R^2} e^{2U(z)}\Big(\sum^2_{q=1}\frac{a_{l,q}z_q}{4+|z|^2}
+b_{l}\frac{4-|z|^2}{4+|z|^2}\Big)\big\langle z,\nabla U(z) \big\rangle \,dz+o(1)\right) \\
=& \frac{8\pi }{3 \lambda \gamma_\lambda } \big( b_{l} +o(1)\big),
\end{split}
\end{equation}
where $\widetilde{u}_{\lambda}(z):=u_{\lambda}\big(x_{\lambda}+\theta_\lambda z\big),~~z\in \Omega_{\lambda}:=\big\{z: x_{\lambda}+\theta_\lambda z\in \Omega\big\}$.
And since $\|v_{\lambda,l}\|_{L^{\infty}(\Omega)}=1$, we derive from Lemma \ref{lem2.4} that
\begin{equation}\label{3.5-10}
\lambda \int_{\partial B_d(x_\lambda)} u_\lambda e^{u_\lambda^2}  v_{\lambda,l} \langle x-x_\lambda,\nu \rangle \,d\sigma
= o\Big(\frac{\theta_\lambda}{\gamma_\lambda}\Big).
\end{equation}
Therefore, it follows from \eqref{add1-A0}, \eqref{3.5-8} and \eqref{3.5-10} that
\begin{equation}\label{3.5-7}
\mbox{RHS of}~ (\ref{puv}) = \frac{8\pi b_l}{\gamma_\lambda^3} + \frac{8\pi (\mu_{\lambda,l}-1)}{3\gamma_\lambda} \big( b_{l} +o(1)\big) + o\Big( \frac{1}{\gamma_\lambda^3} \Big).
\end{equation}
On the other hand, from \eqref{pgg}, \eqref{asym_u} and \eqref{3.5.1}, we have
\begin{equation}\label{L3.3-1}
\begin{split}
\mbox{LHS of}~ (\ref{puv})
=& -\frac{4\pi b_l}{\gamma_\lambda^2} C_\lambda P\big(G(x_{\lambda},x),G(x_{\lambda},x)\big)+
o\Big(\frac{1}{\gamma^3_\lambda}\Big)=\frac{8\pi b_{l}}{\gamma_\lambda^3} + o\Big(\frac{1}{\gamma_\lambda^3}\Big),
\end{split}
\end{equation}
which together with \eqref{3.5-7} implies
\begin{equation*}
\frac{8\pi b_{l}}{\gamma_\lambda^3} + o\Big(\frac{1}{\gamma_\lambda^3}\Big) =
\frac{8\pi b_l}{\gamma_\lambda^3} + \frac{8\pi (\mu_{\lambda,l}-1)}{3\gamma_\lambda} \big( b_{l} +o(1)\big) + o\Big( \frac{1}{\gamma_\lambda^3} \Big).
\end{equation*}
Hence, we know that
\begin{equation}\label{3.5-11}
\mu_{\lambda,l} = 1+o\Big(\frac{1}{\gamma_\lambda^2}\Big).
\end{equation}

However,  it is not enough to compute the Morse index of the solution $u_\lambda$ only by \eqref{3.5-11}. We need more precise estimates of the above calculations.

\vskip 0.2cm

Using \eqref{equa1}, we calculate that
\begin{align}\label{add2-1}
I_1 &= -\frac{1}{\gamma^3_\lambda} \int_{B_{\frac{d}{\theta_\lambda}}(0)} \Big(\frac{w_\lambda(x)}{\gamma_\lambda}+\gamma_\lambda\Big) e^{2w_\lambda(x)+\frac{w_\lambda^2(x)}{\gamma_\lambda^2}} \widetilde{v}_{\lambda,l}(x) \frac{w_\lambda(x)}{\gamma_\lambda} \,dx \notag\\
&= -\frac{1}{\gamma^5_\lambda} \int_{B_{\frac{d}{\theta_\lambda}}(0)} w_\lambda^2 e^{2w_\lambda+\frac{w_\lambda^2}{\gamma_\lambda^2}} \widetilde{v}_{\lambda,l} \,dx \underbrace{ -\lambda \int_{B_d(x_\lambda)} e^{u_\lambda^2} v_{\lambda,l} \big(u_\lambda(x)-\gamma_\lambda\big) \,dx }_{:=I_3} \notag\\
&= -\frac{1}{\gamma^5_\lambda} \bigg( b_l \int_{\R^2} U^2(x) e^{2U(x)} \frac{4-|x|^2}{4+|x|^2} \,dx + o(1)\bigg) + I_3 \notag\\
&= \frac{6\pi b_l}{\gamma_\lambda^5} +I_3 +o\Big(\frac{1}{\gamma^5_\lambda}\Big),
\end{align}
and it follows from \eqref{equa1}, \eqref{add1-A0} and the definition of $A_{\lambda,l}$ that
\begin{align}\label{add2-2}
I_2 &= \frac{\lambda}{2\gamma_\lambda} \int_{B_d(x_\lambda)} \big(1+2u_\lambda^2\big) e^{u_\lambda^2} v_{\lambda,l} \,dx - \frac{\lambda}{2\gamma_\lambda} \int_{B_d(x_\lambda)} e^{u_\lambda^2} v_{\lambda,l} \,dx \notag\\
&= \frac{\lambda}{2\gamma_\lambda} A_{\lambda,l} - \frac{\lambda}{2\gamma^2_\lambda} \int_{B_d(x_\lambda)} u_\lambda e^{u_\lambda^2} v_{\lambda,l} \,dx + \frac{\lambda}{2\gamma^2_\lambda} \int_{B_d(x_\lambda)} \big(u_\lambda(x)-\gamma_\lambda\big) e^{u_\lambda^2} v_{\lambda,l} \,dx \notag\\
&= \frac{\lambda}{2\gamma_\lambda} A_{\lambda,l} - \frac{1}{2\gamma_\lambda^2} \bigg( -\frac{4\pi b_l}{\gamma_\lambda^3} +o\Big( \frac{1}{\gamma_\lambda^3} \Big) \bigg) + \frac{1}{2\gamma^5_\lambda} \int_{B_{\frac{d}{\theta_\lambda}}(0)} w_\lambda(x) e^{2w_\lambda(x)+\frac{w_\lambda^2(x)}{\gamma_\lambda^2}} \widetilde{v}_{\lambda,l}(x) \,dx \notag\\
&= \frac{\lambda}{2\gamma_\lambda} A_{\lambda,l} + \frac{2\pi b_l}{\gamma_\lambda^5} + \frac{b_l}{2\gamma^5_\lambda} \int_{\R^2} U(x) e^{2U(x)} \frac{4-|x|^2}{4+|x|^2} \,dx + o\Big( \frac{1}{\gamma_\lambda^5} \Big) \notag\\
&= \frac{\lambda}{2\gamma_\lambda} A_{\lambda,l} + \frac{3\pi b_l}{\gamma_\lambda^5} + o\Big( \frac{1}{\gamma_\lambda^5} \Big).
\end{align}
Then we can deduce from \eqref{add1-1}, \eqref{add2-1} and \eqref{add2-2} that
\begin{equation}\label{add2-A0-1}
A_0 = \frac{\lambda}{2\gamma_\lambda} A_{\lambda,l} + I_3 + \frac{9\pi b_l}{\gamma_\lambda^5} +
o\Big( \frac{1}{\gamma_\lambda^5} \Big).
\end{equation}
On the other hand, it follows from \eqref{asym_u} and \eqref{3.5-3} that
\begin{equation}\label{add2-3}
\begin{split}
 &\int_{B_d(x_\lambda)} \lambda u_\lambda e^{u_\lambda^2}v_{\lambda,l}\big(\mu_{\lambda,l}(1+2u_\lambda^2)-1\big) \,dx \\
=& \int_{\partial B_d(x_\lambda)} \Big( -\frac{ \partial v_{\lambda,l} }{\partial \nu} u_\lambda + \frac{\partial u_\lambda}
{\partial \nu} v_{\lambda,l} \Big) \,d\sigma \\
=& \int_{\partial B_d(x_\lambda)} \left[ - \Big( \lambda \mu_{\lambda,l} A_{\lambda,l} \frac{\partial G(x,x_\lambda)}{\partial \nu} +O(\theta_\lambda) \Big) \left(C_\lambda G(x,x_\lambda) +o\Big(\frac{\theta_\lambda}{\gamma_\lambda}\Big)\right) \right. \\
& \left.+\Big( \lambda \mu_{\lambda,l} A_{\lambda,l} G(x,x_\lambda) +O(\theta_\lambda) \Big) \left(C_\lambda \frac{\partial G(x,x_\lambda)}{\partial \nu} +o\Big(\frac{\theta_\lambda}{\gamma_\lambda}\Big)\right)\right]\,d\sigma  = o\big(\theta_\lambda\big).
\end{split}
\end{equation}
Using \eqref{add2-3}, we can rewrite \eqref{3.5-5} as
\begin{align}\label{add2-4}
\lambda A_{\lambda,l} =& \lambda \int_{B_d(x_\lambda)} \big(1+2u_\lambda^2\big) e^{u_\lambda^2} v_{\lambda,l} \,dx \notag\\
=& \frac{\lambda}{\gamma_\lambda} \int_{B_d(x_\lambda)} e^{u_\lambda^2} u_\lambda v_{\lambda,l} \big(1+2u_\lambda^2\big) \,dx - \frac{\lambda}{\gamma_\lambda} \int_{B_d(x_\lambda)} \big(1+2u_\lambda^2\big) e^{u_\lambda^2} v_{\lambda,l} \big(u_\lambda(x)-\gamma_\lambda\big) \,dx \notag\\
=& \frac{\lambda}{\gamma_\lambda \mu_{\lambda,l}} \int_{B_d(x_\lambda)} e^{u_\lambda^2} u_\lambda v_{\lambda,l} \,dx
- \frac{1}{\gamma_\lambda^4} \int_{B_{\frac{d}{\theta_\lambda}(0)}} e^{2w_\lambda(x)+\frac{w_\lambda^2(x)}{\gamma_\lambda^2}} \widetilde{v}_{\lambda,l}(x) w_\lambda(x) \,dx \notag\\
& -\underbrace{\frac{2\lambda}{\gamma_\lambda} \int_{B_d(x_\lambda)} u_\lambda^2 e^{u_\lambda^2} v_{\lambda,l} \big(u_\lambda(x)-\gamma_\lambda\big) \,dx }_{:=I_4} +o\Big(\frac{\theta_\lambda}{\gamma_\lambda}\Big) \notag\\
=& \frac{A_0}{\gamma_\lambda \mu_{\lambda,l}} -\frac{2\pi b_l}{\gamma_\lambda^4} -I_4 +o\Big(\frac{1}{\gamma^4_\lambda}\Big)
= \frac{A_0}{\gamma_\lambda \mu_{\lambda,l}} +\frac{22\pi b_l}{\gamma_\lambda^4} +2\gamma_\lambda I_3 +o\Big(\frac{1}{\gamma^4_\lambda}\Big),
\end{align}
since direct calculation gives
\begin{align*}
I_4 =& \frac{2}{\gamma_\lambda^4} \int_{B_{\frac{d}{\theta_\lambda}(0)}} \Big( \frac{w_\lambda^2(x)}{\gamma_\lambda^2} + 2w_\lambda(x) + \gamma_\lambda^2 \Big) e^{2w_\lambda(x)+\frac{w_\lambda^2(x)}{\gamma_\lambda^2}} \widetilde{v}_{\lambda,l}(x) w_\lambda(x) \,dx \\
=& 2\lambda \gamma_\lambda \int_{B_d(x_\lambda)} e^{u_\lambda^2(x)} v_{\lambda,l}(x) \big(u_\lambda(x)-\gamma_\lambda\big) \,dx \\
& + \frac{4}{\gamma_\lambda^4} \int_{B_{\frac{d}{\theta_\lambda}(0)}} w_\lambda^2(x) e^{2w_\lambda(x)+\frac{w_\lambda^2(x)}{\gamma_\lambda^2}} \widetilde{v}_{\lambda,l}(x) \,dx +o\Big(\frac{1}{\gamma^4_\lambda}\Big) \\
=& -2\gamma_\lambda I_3 - \frac{24\pi b_l}{\gamma_\lambda^4} +o\Big(\frac{1}{\gamma^4_\lambda}\Big).
\end{align*}
Then we deduce from \eqref{add1-A0} and \eqref{add2-4} that
\begin{equation}\label{add2-5}
\begin{split}
I_3 &= \frac{\lambda A_{\lambda,l}}{2 \gamma_\lambda} -\frac{A_0}{2\gamma_\lambda^2 \mu_{\lambda,l}} -\frac{11\pi b_l}{\gamma_\lambda^5} +o\Big(\frac{1}{\gamma^5_\lambda}\Big) \\
&= \frac{\lambda A_{\lambda,l}}{2 \gamma_\lambda} -\frac{1}{2\gamma_\lambda^2 \mu_{\lambda,l}}
\bigg( -\frac{4\pi b_l}{\gamma^3_\lambda } + o\Big(\frac{1}{\gamma^3_\lambda}\Big) \bigg)
-\frac{11\pi b_l}{\gamma_\lambda^5} +o\Big(\frac{1}{\gamma^5_\lambda}\Big) \\
&= \frac{\lambda A_{\lambda,l}}{2 \gamma_\lambda} - \frac{9\pi b_l}{\gamma_\lambda^5} +o\Big(\frac{1}{\gamma^5_\lambda}\Big).
\end{split}
\end{equation}
And then substituting \eqref{add2-5} into \eqref{add2-A0-1}, we derive that
\begin{equation}\label{add2-A0-2}
\begin{split}
A_0 &= \frac{\lambda A_{\lambda,l}}{2\gamma_\lambda} + \frac{\lambda A_{\lambda,l}}{2 \gamma_\lambda} - \frac{9\pi b_l}
{\gamma_\lambda^5} + \frac{9\pi b_l}{\gamma_\lambda^5} + o\Big( \frac{1}{\gamma_\lambda^5} \Big)
=  \frac{\lambda}{\gamma_\lambda} A_{\lambda,l} + o\Big( \frac{1}{\gamma_\lambda^5} \Big).
\end{split}
\end{equation}

Furthermore, similar to \eqref{L3.3-1}, it follows from \eqref{def_C1}, \eqref{asym_u}, \eqref{3.5-3}, \eqref{3.5-5} and \eqref{3.5-11} that
\begin{equation}\label{L3.3-2}
\begin{split}
\mbox{LHS of}~ (\ref{puv})
=& -\frac{1}{2\pi} \lambda \mu_{\lambda,l} A_{\lambda,l} C_\lambda +o(\theta_\lambda) \\
=& -\frac{1}{2\pi} \lambda \mu_{\lambda,l} A_{\lambda,l} \bigg( \frac{4\pi}{\gamma_\lambda} + \frac{4\pi}{\gamma_\lambda^3} + o\Big( \frac{1}{\gamma_\lambda^3} \Big)\bigg) + o(\theta_\lambda) \\
=& -\frac{2\lambda}{\gamma_\lambda} A_{\lambda,l} - \frac{2 \lambda}{\gamma^3_\lambda} A_{\lambda,l} +
o\Big( \frac{1}{\gamma_\lambda^5} \Big).
\end{split}
\end{equation}
We can also deduce from  \eqref{3.5-8} and \eqref{3.5-10} that
\begin{equation}\label{R3.3-2}
\mbox{RHS of}~ (\ref{puv})
= -2 A_0 + \frac{8\pi(\mu_{\lambda,l}-1)}{3\gamma_\lambda}\big(b_l+o(1)\big) + o\Big(\frac{\theta_\lambda}{\gamma_\lambda}\Big).
\end{equation}
Therefore, it follows from \eqref{3.5-5}, \eqref{add2-A0-2}, \eqref{L3.3-2} and \eqref{R3.3-2} that
\begin{equation*}
\begin{split}
\frac{8\pi(\mu_{\lambda,l}-1)}{3} \big(b_l+o(1)\big)
&= 2 A_0\gamma_\lambda -2\lambda A_{\lambda,l} - \frac{2 \lambda}{\gamma^2_\lambda} A_{\lambda,l} + o\Big( \frac{1}{\gamma_\lambda^4} \Big)  = -\frac{2 \lambda}{\gamma^2_\lambda} A_{\lambda,l} + o\Big( \frac{1}{\gamma_\lambda^4} \Big) \\
&= -\frac{2}{\gamma^2_\lambda}\bigg(\frac{-4\pi b_l}{\gamma_\lambda^2} + o\Big( \frac{1}{\gamma_\lambda^2} \Big)\bigg)
+ o\Big( \frac{1}{\gamma_\lambda^4} \Big)  = \frac{8\pi b_l}{\gamma_\lambda^4} + o\Big( \frac{1}{\gamma_\lambda^4} \Big),
\end{split}
\end{equation*}
which implies \eqref{3.5.2}.
\end{proof}

\vskip 0.2cm
\subsection{The computation of Morse index to $u_\lambda$}\label{s4.2} ~
\vskip 0.2cm

Now we compute the value of $\mu_{\lambda,l}$ for $l=1,2,3,4$ term by term. Taking $r>0$ small fixed such that
$B_{3r}(x_\lambda)\subset B_{4r}(x_{0}) \subset \Omega$.
We set $\phi\in C^{\infty}_0\big(B_{3r}(x_\lambda)\big)$ with $\phi\equiv 1$ in $B_{2r}(x_\lambda)$ and $0\leq \phi\leq 1$ in $B_{3r}(x_\lambda)$. Taking
\begin{equation*}
\hat{u}_{\lambda}=\phi u_\lambda, ~~\psi_{\lambda,q}=\phi \frac{\partial u_\lambda}{\partial x_q}~(q=1,2).
\end{equation*}
Then we have following identities.
\begin{Lem} It holds
\begin{equation}\label{iden-1}
\begin{split}
 \int_{\Omega}\big|\nabla (\phi  u_\lambda )\big|^2 \,dx
=  \int_{\Omega}|\nabla \phi|^2u_\lambda^2 \,dx+\lambda  \int_{\Omega} \phi^2 u^2_\lambda e^{u^2_\lambda} \,dx.
\end{split}
\end{equation}
\end{Lem}
\begin{proof}
Multiplying $-\Delta u_\lambda=\lambda u_\lambda e^{u^2_\lambda}$ in $\Omega$ by $\phi^2u_\lambda$ and integrating by parts, we have \eqref{iden-1}.
\end{proof}
Firstly, we compute the eigenvalue $\mu_{\lambda,1}$, which can be stated as follows.
\begin{Prop}\label{prop_mu1}
It holds
\begin{equation*}
0<\mu_{\lambda,1}\leq \frac{1}{2\gamma^2_\lambda}\big(1+o(1)\big),~\,~\,~\mbox{for}~\lambda~\mbox{small}.
\end{equation*}
\end{Prop}

\begin{proof}
Let $V=span \{\phi u_{\lambda}\}$. By the Courant-Fischer-Weyl min-max principle, we have
\begin{equation}\label{3.7-1}
\mu_{\lambda,1}=\min_{W\subset H^1_0(\Omega) \atop \dim W=1}\max_{f\in W\backslash\{0\}}\frac{\displaystyle\int_{\Omega}|\nabla f|^2dx}{\lambda \displaystyle\int_{\Omega}f^2\big(1+2u_\lambda^2\big) e^{u_\lambda^{2}}dx}
\leq \max_{f\in V\backslash\{0\}}\frac{\displaystyle\int_{\Omega}|\nabla f|^2dx}{\lambda \displaystyle\int_{\Omega}f^2\big(1+2u_\lambda^2\big) e^{u_\lambda^{2}}dx}.
\end{equation}
For any $f=\alpha \phi u_{\lambda} \in V$ with $\alpha \in \R\backslash \{0\}$, we have
\begin{equation}\label{3.7-2}
\begin{split}
\frac{\displaystyle\int_{\Omega}|\nabla f|^2dx}{\lambda \displaystyle\int_{\Omega}f^2\big(1+2u_\lambda^2\big) e^{u_\lambda^{2}}dx}
= &\frac{ \displaystyle\int_{\Omega}\big|\nabla (\phi u_{\lambda})\big|^2dx}{
\lambda \displaystyle\int_{\Omega} \phi^2 u_\lambda^2\big(1+2u_\lambda^2\big) e^{u_\lambda^{2}}dx}
 \\  = &
 \frac{ \displaystyle  \int_{\Omega}\phi^2 u^2_\lambda e^{u^2_\lambda} dx}{
  \displaystyle\int_{\Omega} \phi^2 u_\lambda^2\big(1+2u_\lambda^2\big) e^{u_\lambda^{2}}dx}
+\frac{ \displaystyle\int_{\Omega}|\nabla \phi|^2u_\lambda^2dx}{
\lambda \displaystyle\int_{\Omega} \phi^2 u_\lambda^2\big(1+2u_\lambda^2\big) e^{u_\lambda^{2}}dx}.
\end{split}
\end{equation}

Now we estimate each term in \eqref{3.7-2}. By \eqref{defw} and \eqref{limw}, we have
\begin{equation}\label{3.7-3}
\begin{split}
\int_{\Omega}\phi^2 u^2_\lambda e^{u^2_\lambda} \,dx = & \int_{B_{r}(x_{\lambda})} u^2_\lambda e^{u^2_\lambda} \,dx +\int_{\Omega \backslash B_{r}(x_{\lambda})} \phi^2 u^2_\lambda e^{u^2_\lambda} \,dx \\
=& \frac{1}{\lambda}\int_{B_{\frac{r}{\theta_\lambda}(0)}} \Big(1+\frac{w_\lambda(z)}{\gamma^2_\lambda} \Big)^2 e^{2 w_\lambda(z)+\frac{w^2_\lambda(z)}{\gamma^2_\lambda} } \,dz \\
& +O\bigg(\frac{1}{\lambda}\int_{\Omega_\lambda \backslash B_{\frac{r}{\theta_\lambda}(0)}} \Big(1+\frac{w_\lambda(z)}{\gamma^2_\lambda} \Big)^2 e^{2 w_\lambda(z)+\frac{w^2_\lambda(z)}{\gamma^2_\lambda} } \,dz \bigg) \\
=& \frac{1}{\lambda}\big(4\pi+o(1)\big).
\end{split}
\end{equation}
Similarly,  it holds
\begin{equation}\label{3.7-4}
\begin{split}
&\int_{\Omega}\phi^2 u_\lambda^2 \big(1+ 2u_\lambda^2\big) e^{u_\lambda^{2}} \,dx \\
=& \int_{B_{r}(x_{\lambda})} u^2_\lambda \big(1+2u_\lambda^2\big) e^{u^2_\lambda}dx +\int_{\Omega \backslash B_{r}(x_{\lambda})} \phi^2 u^2_\lambda \big(1+2u_\lambda^2\big) e^{u^2_\lambda} \,dx \\
=& \frac{\gamma^2_\lambda}{\lambda} \int_{B_{\frac{r}{\theta_\lambda}(0)}} \Big(1+\frac{w_\lambda(z)}{\gamma_\lambda^2}\Big)^2 \bigg(\frac{1}{\gamma_\lambda^2} +2\Big(1+\frac{w_\lambda(z)}{\gamma_\lambda^2}\Big)^2\bigg) e^{2w_\lambda(z)+\frac{w^2_\lambda(z)}
{\gamma^2_\lambda}} \,dz \\
&+ O\Bigg(\frac{\gamma^2_\lambda}{\lambda}\int_{\Omega_\lambda \backslash B_{\frac{r}{\theta_\lambda}(0)}} \Big(1+\frac{w_\lambda(z)}{\gamma_\lambda^2}\Big)^2 \bigg(\frac{1}{\gamma_\lambda^2} +2\Big(1+\frac{w_\lambda(z)}{\gamma_\lambda^2}\Big)^2\bigg) e^{2 w_\lambda(z)+\frac{w^2_\lambda(z)}{\gamma^2_\lambda}} \,dz \Bigg) \\
=& \frac{\gamma^2_\lambda}{\lambda}\big(8\pi+o(1)\big).
\end{split}
\end{equation}
Obviously, it holds
\begin{equation}\label{3.7-5}
\int_{\Omega}|\nabla \phi|^2u_\lambda^2 \,dx \leq C\int_{B(x_{\lambda},3r)\backslash B(x_{\lambda},2r)} u^2_\lambda \,dx \overset{\eqref{lim-1}}=O\Big(\frac{1}{\gamma_\lambda^2}\Big).
\end{equation}
Substituting \eqref{3.7-2}--\eqref{3.7-5} into \eqref{3.7-1}, we get
\begin{equation*}
\mu_{\lambda,1}\leq \frac{1}{2\gamma^2_\lambda}\big(1+o(1)\big).
\end{equation*}
\end{proof}

Before analyzing the eigenvalues $\mu_{\lambda,l}$ with $l=2,3$, we give some key computations. Let
\begin{equation}\label{def_zlambda}
z_{\lambda}:=\beta_{\lambda,1}\frac{\partial u_{\lambda}}{\partial x_1}+\beta_{\lambda,2}\frac{\partial u_{\lambda}}{\partial x_2}~\,~\mbox{for}~~\beta_{\lambda,1},~\beta_{\lambda,2}\in \R,
\end{equation}
then we have following identities.
\begin{Lem}
It holds
\begin{equation}\label{iden-2}
\int_{\Omega}  \big|\nabla (z_{\lambda} \phi) \big|^2 \,dx =
\int_{\Omega}  z_{\lambda}^2 \big| \nabla \phi \big|^2 \,dx +
\lambda \int_{\Omega} \big(1+2u^2_\lambda\big) e^{u^2_\lambda} \phi^2 z_{\lambda}^2 \,dx,
\end{equation}
\begin{equation}\label{iden-3}
\lambda \int_{\Omega} \big(1+2u^2_\lambda\big) e^{u^2_\lambda} u_\lambda z_\lambda \phi^2 \,dx = \int_{\Omega}\phi^2 \big(\nabla u_\lambda \cdot \nabla z_{\lambda}\big) \,dx + 2 \int_{\Omega} u_\lambda \phi \big(\nabla \phi \cdot \nabla z_{\lambda}\big) \,dx,
\end{equation}
and
\begin{equation}\label{iden-4}
\lambda \int_{\Omega} e^{u^2_\lambda} u_\lambda z_\lambda \phi^2 \,dx=
\int_{\Omega} \phi^2 \big(\nabla u_\lambda \cdot \nabla z_{\lambda}\big) \,dx+ 2\int_{\Omega} z_{\lambda} \phi \big(\nabla \phi \cdot \nabla u_\lambda \big) \,dx.
\end{equation}
\end{Lem}
\begin{proof}
Recall that
\begin{equation}\label{3.8-1}
-\Delta z_{\lambda}=\lambda\big(1+2u_\lambda^2\big)e^{u_\lambda^2}z_{\lambda}~\,~\mbox{in}~~\Omega,
\end{equation}
multiplying it by $\phi^2z_{\lambda}$ and integrating by parts, we have \eqref{iden-2}.
Multiplying \eqref{3.8-1} by $\phi^2 u_\lambda$ and integrating by parts, we have \eqref{iden-3}.
Similarly, multiplying $-\Delta u_\lambda=\lambda u_\lambda e^{u_\lambda^2}$ by $\phi^2 z_{\lambda}$ and integrating by parts in $\Omega$, we derive \eqref{iden-4}.
\end{proof}

\begin{Lem}\label{lem3.9}
For $l=2,3$, it holds
\begin{equation}\label{eigenv-23}
\mu_{\lambda,l}\leq 1+O\big( \theta_\lambda^2 \big),~\,~\,~\mbox{for}~\lambda~\mbox{small}.
\end{equation}
\end{Lem}
\begin{proof}
To prove \eqref{eigenv-23}, we just need to prove
\begin{equation}\label{eigenv-3}
\mu_{\lambda,3}\leq 1+O\big( \theta_\lambda^2 \big).
\end{equation}
By the variational characterization of the eigenvalues of the linearized problem \eqref{eigen-1}, we have
\begin{equation*}
\mu_{\lambda,3}=\min_{W\subset H^1_0(\Omega) \atop  \dim W=3}\max_{v\in W\backslash\{0\}}\frac{\displaystyle\int_{\Omega}|\nabla v|^2dx}{\lambda \displaystyle\int_{\Omega}v^2\big(1+2u_\lambda^2\big) e^{u_\lambda^{2}}dx}.
\end{equation*}
Taking
\begin{equation*}
W=span\Big\{\hat{u}_{\lambda},\psi_{\lambda,1},\psi_{\lambda,2}\Big\}.
\end{equation*}
It is easy to see that the functions $\hat{u}_{\lambda},\psi_{\lambda,1},\psi_{\lambda,2}$ are linearly independent for $\lambda$ sufficiently small, one can refer to \cite[Lemma 3.1]{GP2005} for more details. Then $\dim W=3$ and each function $v_{\lambda}\in W$ can be written as
\begin{equation*}
\begin{split}
v_\lambda=& \alpha_{\lambda} \hat{u}_{\lambda}+ \beta_{\lambda,1}\psi_{\lambda,1}+ \beta_{\lambda,2}\psi_{\lambda,2}\\
=&
\alpha_{\lambda} u_{\lambda}\phi+ \beta_{\lambda,1} \frac{\partial u_{\lambda}}{\partial x_1}\phi + \beta_{\lambda,2}\frac{\partial u_{\lambda}}{\partial x_2}\phi \overset{\eqref{def_zlambda}}=\alpha_{\lambda} u_{\lambda}\phi+ z_\lambda \phi,
\end{split}
\end{equation*}
with $(\alpha_\lambda,\beta_{\lambda,1},\beta_{\lambda,2})\in \R^3$.
Hence, it holds
\begin{equation}\label{3.9-1}
\mu_{\lambda,3}\leq \max_{\mathfrak{L}} \frac{\displaystyle\int_{\Omega}\big|\nabla (\alpha_{\lambda} u_{\lambda}\phi+ z_\lambda \phi)\big|^2 \,dx}{\lambda \displaystyle\int_{\Omega}\big(\alpha_{\lambda} u_{\lambda}\phi+ z_\lambda \phi\big)^2\big(1+2u_\lambda^2\big) e^{u_\lambda^{2}} \,dx},
\end{equation}
where $$\mathfrak{L}:=\Big\{\big(\alpha_{\lambda},\beta_{\lambda,1},\beta_{\lambda,2}\big)\in \R^{3},~~~ \alpha^2_{\lambda}+\beta^2_{\lambda,1}+\beta^2_{\lambda,2} =1\Big\}.$$

We can deduce from \eqref{iden-3} and \eqref{iden-4} that
\begin{equation}\label{iden-34}
\begin{split}
\int_{\Omega} \nabla ( u_{\lambda}\phi ) \cdot \nabla (  z_\lambda \phi) \,dx= \lambda \int_{\Omega} \big(1+u_\lambda^2\big) e^{u_\lambda^{2}}\phi^2 u_\lambda z_{\lambda} \,dx + \int_{\Omega} z_{\lambda} u_\lambda \big|\nabla \phi  \big|^2 \,dx.
\end{split}
\end{equation}
Combining \eqref{iden-1}, \eqref{iden-2} with \eqref{iden-34}, we find
\begin{equation}\label{3.9-2}
\begin{split}
&\int_{\Omega}\big|\nabla  (\alpha_{\lambda} u_{\lambda}\phi + z_\lambda \phi)\big|^2 \,dx \\
=&\alpha_{\lambda}^2\int_{\Omega}\big|\nabla ( u_{\lambda}\phi )\big|^2 \,dx+2\alpha_{\lambda} \int_{\Omega}\nabla
( u_{\lambda}\phi ) \cdot \nabla (  z_\lambda \phi) \,dx+\int_{\Omega}\big|\nabla (z_\lambda \phi) \big|^2 \,dx\\
=& \alpha^2_{\lambda } \int_{\Omega}\big|\nabla  \phi \big|^2u_\lambda^2 \,dx+ \lambda\alpha^2_{\lambda} \int_{\Omega}\phi^2 u_\lambda^2e^{u_\lambda^2} \,dx\\
&+ 2 \lambda\alpha_{\lambda} \int_{\Omega} \big(1+u_\lambda^2\big)e^{u_\lambda^2} \phi^2 z_\lambda u_\lambda \,dx + 2  \alpha_{\lambda}\int_{\Omega} z_{\lambda} u_\lambda \big|\nabla \phi \big|^2  \,dx\\
&+ \int_{\Omega} \big|\nabla \phi \big|^2 z_{\lambda}^2 \,dx + \lambda\int_{\Omega} \big(1+2u_\lambda^2\big)   e^{u_\lambda^2} \phi^2 z_{\lambda}^2 \,dx  \\
=& \underbrace{\lambda\alpha^2_{\lambda} \int_{\Omega}\phi^2 u_\lambda^2e^{u_\lambda^2} \,dx}_{:=J_{\lambda,1}} +\underbrace {2 \lambda\alpha_{\lambda} \int_{\Omega} \big(1+u_\lambda^2\big) e^{u_\lambda^2} \phi^2 z_\lambda u_\lambda \,dx}_{:=J_{\lambda,2}}\\
&+ \underbrace{\lambda\int_{\Omega} \big(1+2u_\lambda^2\big) e^{u_\lambda^2} \phi^2 z_{\lambda}^2 \,dx}_{:=J_{\lambda,3}} + \underbrace{ \int_{\Omega}\big(\alpha_{\lambda} u_{\lambda} + z_\lambda \big)^2\big|\nabla \phi \big|^2 \,dx}_{:=
J_{\lambda,4}}.
\end{split}
\end{equation}
Furthermore, we compute that
\begin{equation}\label{3.9-3}
\begin{split}
& \lambda \displaystyle\int_{\Omega}\big(\alpha_{\lambda} u_{\lambda}\phi+ z_\lambda \phi\big)^2\big(1+2u_\lambda^2\big) e^{u_\lambda^{2}} \,dx\\
=& \lambda \displaystyle\int_{\Omega}\Big(\alpha^2_{\lambda} u^2_{\lambda}+ z^2_\lambda +2 \alpha_{\lambda} u_{\lambda}z_\lambda \Big) \big(1+2u_\lambda^2\big) e^{u_\lambda^{2}}\phi^2 \,dx\\
=& \underbrace{\lambda \alpha^2_{\lambda}  \displaystyle\int_{\Omega} u^2_{\lambda} \big(1+2u_\lambda^2\big) e^{u_\lambda^{2}}\phi^2 \,dx}_{:=J_{\lambda,5}} +\underbrace{2 \alpha_{\lambda} \lambda  \displaystyle\int_{\Omega} u_{\lambda}z_\lambda \big(1+2u_\lambda^2\big) e^{u_\lambda^{2}}\phi^2 \,dx }_{:=J_{\lambda,6}}
+J_{\lambda,3}.
\end{split}
\end{equation}
Hence substitute \eqref{3.9-2} and \eqref{3.9-3} into \eqref{3.9-1}, we can find
\begin{equation}\label{3.9-4}
\mu_{\lambda,3}\leq \max_{\mathfrak{L}} \frac{J_{\lambda,1}+J_{\lambda,2}+J_{\lambda,3}+J_{\lambda,4}}{J_{\lambda,3}+J_{\lambda,5}+J_{\lambda,6}}=
1+ \max_{\mathfrak{L}} \frac{J_{\lambda,1}+J_{\lambda,2}+J_{\lambda,4}-J_{\lambda,5}-J_{\lambda,6}}{J_{\lambda,3}+J_{\lambda,5}+J_{\lambda,6}}.
\end{equation}
Next we compute the terms $J_{\lambda,i}$ for $i=1,2,\cdots,6$.

\vskip 0.1cm

Firstly, by \eqref{3.7-3} and \eqref{3.7-4}, we have
\begin{equation*}
\begin{split}
J_{\lambda,1}= \lambda\alpha^2_{\lambda} \int_{\Omega}\phi^2 u_\lambda^2e^{u_\lambda^2}dx
=\alpha^2_{\lambda}\big(4\pi+o(1)\big),
\end{split}
\end{equation*}
and
\begin{equation*}
\begin{split}
J_{\lambda,5}= \lambda \alpha^2_{\lambda}  \displaystyle\int_{\Omega} u^2_{\lambda} \big(1+2u_\lambda^2\big) e^{u_\lambda^{2}}\phi^2 \,dx =\gamma^2_\lambda \alpha^2_{\lambda}\big(8\pi+o(1)\big).
\end{split}
\end{equation*}
Moreover, using integration by parts, we get
\begin{equation*}
\begin{split}
J_{\lambda,2}= &\lambda \alpha_\lambda \sum^2_{i=1}\beta_{\lambda,i}
\int_{\Omega} \frac{\partial \big(u^2_\lambda e^{u^2_\lambda}\big)}{\partial x_i} \phi^2  \,dx\\
=& O\Big(\lambda\int_{\Omega} u^2_\lambda e^{u^2_\lambda}  \big|\nabla \phi\big| \,dx\Big)
=O\Big( \lambda \int_{B_{3r}(x_{\lambda})\backslash B_{2r}(x_{\lambda})} u^2_\lambda e^{u^2_\lambda}  \,dx \Big)
=O\big( \theta_\lambda^{2-\delta} \big),
\end{split}
\end{equation*}
and
\begin{equation*}
\begin{split}
J_{\lambda,6}=& 2J_{\lambda,2}-2\alpha_\lambda \lambda \displaystyle\int_{\Omega} u_{\lambda}z_\lambda e^{u_\lambda^{2}}\phi^2 \,dx \\
=& 2J_{\lambda,2}-\lambda \alpha_\lambda \sum^2_{i=1}\beta_{\lambda,i} \int_{\Omega} \frac{\partial \big(  e^{u^2_\lambda}\big)}{\partial x_i} \phi^2 \,dx = O\big( \theta_\lambda^{2-\delta}\big).
\end{split}
\end{equation*}
Finally, direct computation gives
\begin{equation*}
\begin{split}
J_{\lambda,3}= & \lambda\int_{\Omega} \big(1+2u_\lambda^2\big) e^{u_\lambda^2} \phi^2  \Big(\beta_{\lambda,1}
\frac{\partial u_{\lambda}}{\partial x_1}+\beta_{\lambda,2}\frac{\partial u_{\lambda}}{\partial x_2}\Big)^2 \,dx\\
=& \frac{1}{\gamma^2_{\lambda} \theta_\lambda^2} \bigg( \sum^2_{i=1}\beta^2_{\lambda,i} \int_{\R^2} 2 e^{2U} \Big| \frac{\partial U}{\partial x_i} \Big|^2dx+o(1)\bigg)
= \frac{2\pi}{3\gamma^2_{\lambda}\theta_\lambda^2} \Big( \sum^2_{i=1}\beta^2_{\lambda,i} +o(1)\Big),
\end{split}
\end{equation*}
and
\begin{equation*}
\begin{split}
0\leq J_{\lambda,4}=\int_{B_{3r}(x_{\lambda})\backslash B_{2r}(x_{\lambda})} \big(\alpha_{\lambda} u_\lambda+ z_{\lambda}\big)^2 \big|\nabla \phi  \big|^2 \,dx \overset{\eqref{lim-1}}\leq \frac{C}{\gamma_\lambda^2}.
\end{split}
\end{equation*}

Next, we divide into two cases:

\vskip 0.2cm

\noindent\textup{(1)}
If $\alpha^2_{\lambda} \gamma^4_\lambda \rightarrow +\infty$ as $\lambda\to 0$, then
we find
\begin{equation*}
 \gamma^2_\lambda  \big(J_{\lambda,1}+J_{\lambda,2}+J_{\lambda,4}-J_{\lambda,5}-J_{\lambda,6}\big)\to -\infty,~~\mbox{as}~~\lambda\to 0,
\end{equation*}
which gives us that $J_{\lambda,1}+J_{\lambda,2}+J_{\lambda,4}-J_{\lambda,5}-J_{\lambda,6}<0$ for $\lambda$ small. Also \eqref{3.9-3} implies that $J_{\lambda,3}+J_{\lambda,5}+J_{\lambda,6} > 0$. Hence we have
\begin{equation}\label{3.9-5}
\begin{split}  \frac{J_{\lambda,1}+J_{\lambda,2}+J_{\lambda,4}-J_{\lambda,5}-J_{\lambda,6}}{J_{\lambda,3}+
J_{\lambda,5}+J_{\lambda,6}} < 0.
\end{split}\end{equation}

\vskip 0.1cm

\noindent\textup{(2)}
If $\alpha^2_{\lambda} \gamma^4_\lambda\leq C$ as $\lambda \to 0$, then we have
\begin{equation*}
\beta_{\lambda,1}^2+
\beta_{\lambda,2}^2 \to 1.
\end{equation*}
Hence it holds
\begin{equation*}
J_{\lambda,3}+J_{\lambda,5}+J_{\lambda,6} \geq \frac{\pi}{3 \gamma^2_{\lambda} \theta_\lambda^2},
\end{equation*}
and
\begin{equation*}
 \big|J_{\lambda,1}+J_{\lambda,2}+J_{\lambda,4}-J_{\lambda,5}-J_{\lambda,6}\big|\leq \frac{C}{\gamma_\lambda^2 }.
\end{equation*}
So we have
\begin{equation*}
\bigg| \frac{J_{\lambda,1}+J_{\lambda,2}+J_{\lambda,4}-J_{\lambda,5}-J_{\lambda,6}}{J_{\lambda,3}+J_{\lambda,5}
+J_{\lambda,6}} \bigg| \leq C_0 \theta_\lambda^2,~~~\,\mbox{for some}~~C_0>0,
\end{equation*}
which means that
\begin{equation}\label{3.9-6}
\frac{J_{\lambda,1}+J_{\lambda,2}+J_{\lambda,4}-J_{\lambda,5}-J_{\lambda,6}}{J_{\lambda,3}+J_{\lambda,5}+J_{\lambda,6}} =O\big(\theta_\lambda^2\big).
\end{equation}
Finally, we can deduce \eqref{eigenv-3} by \eqref{3.9-4}, \eqref{3.9-5} and \eqref{3.9-6}.
\end{proof}

\begin{Lem}\label{lem_mu23}
For $l=2,3$, it holds
\begin{equation}\label{3.10.1}
\mu_{\lambda,l}=1+o(1)~\mbox{as}~\lambda \to 0.
\end{equation}
\end{Lem}
\begin{proof}
The key point is to compute $\mu_{\lambda,2}$. By Lemma \ref{lem3.9}, we know that $\displaystyle\lim_{\lambda \to 0}\mu_{\lambda,2} = \mu_{2} \in [0,1]$. Assume by contradiction that $\mu_2 < 1$. It follows from Lemma \ref{lem3.3} that
$\widetilde{v}_{\lambda,2} \to V_{2}~~\mbox{in}~~C^1_{loc}\big(\R^2\big)$ as $\lambda \to 0$ and $V_{2}(x)$ satisfies
\begin{equation*}
\begin{cases}
-\Delta V_{2}=2\mu_{2} e^{2U} V_{2}\,\,~\mbox{in}~\R^2,\\[1mm]
\|V_{2}\|_{L^{\infty}(\R^2)}\leq 1.
\end{cases}
\end{equation*}
Similar to the proof of Proposition 3.1 in \cite{DGIP2019}, we know that $\mu_2=0$. And then $V_2\equiv c_2$, where $c_2$ is a nonzero constant in Proposition \ref{prop3.4}.

\vskip 0.1cm

Note that $v_{\lambda,1}$ and $v_{\lambda,2}$ are orthogonal in $H_0^1(\Omega)$, then it holds
\begin{equation}\label{3.10-1}
\begin{split}
\int_{\Omega}\nabla v_{\lambda,1}\cdot \nabla v_{\lambda,2} \,dx =0.
\end{split}
\end{equation}
On the other hand,
\begin{equation*}
\int_{\Omega}\nabla v_{\lambda,1}\cdot \nabla v_{\lambda,2} \,dx = \mu_{\lambda,1} \lambda \int_{\Omega} \big(1+2u_\lambda^2\big) e^{u_\lambda^2} v_{\lambda,1} v_{\lambda,2} \,dx,
\end{equation*}
which together with \eqref{3.10-1} yields that
\begin{equation}\label{3.10-2}
\begin{split}
0=& \lambda \int_{\Omega} \big(1+2u_\lambda^2\big)e^{u_\lambda^2} v_{\lambda,1} v_{\lambda,2} \,dx \\
=& \lambda \int_{B_d(x_\lambda)} \big(1+2u_\lambda^2\big) e^{u_\lambda^2} v_{\lambda,1} v_{\lambda,2} \,dx
+ \lambda \int_{\Omega\backslash B_d(x_\lambda)} \big(1+2u_\lambda^2\big) e^{u_\lambda^2} v_{\lambda,1} v_{\lambda,2} \,dx \\
=& \int_{B_{\frac{d}{\theta_\lambda}}(0)} \bigg(\frac{1}{\gamma_\lambda^2}+2\Big(1+\frac{w_\lambda(z)}{\gamma^2_\lambda}
\Big)^2\bigg) e^{2w_\lambda(z)+\frac{w_\lambda^2(z)}{\gamma^2_\lambda}} \widetilde{v}_{\lambda,1}(z) \widetilde{v}_{\lambda,2}(z) \,dz + O\big( \theta^{2-\delta}_\lambda \big) \\
=& \Big(\int_{\R^2} 2e^{2U(z)} c_1 c_2 \,dz + o(1) \Big) + O\big( \theta^{2-\delta}_\lambda \big)\\
=& c_1 c_2 \big(8\pi+o(1)\big) + O\big( \theta^{2-\delta}_\lambda \big),
\end{split}
\end{equation}
where $d>0$ is a small constant such that $B_{2d}(x_\lambda) \subset \Omega$ and $c_1$ is a nonzero constant in Proposition \ref{prop3.4}.
Hence, we can deduce from \eqref{3.10-2} that
\begin{equation*}
c_1 c_2 \big(8\pi+o(1)\big) = O\big( \theta^{2-\delta}_\lambda \big),
\end{equation*}
that is $c_1 c_2=0$, which is a contradiction. So we have $\mu_{2}=1$, and \eqref{3.10.1} follows from \eqref{eigenv-23}.
\end{proof}

\begin{Lem}
For $l=2,3$, it holds
\begin{equation}\label{eigenfr-23}
\widetilde{v}_{\lambda,l}(y)
=\sum^2_{q=1}\frac{a_{l,q}y_q}{4+|y|^2}+o\big(1\big)\,~\mbox{in}~C^1_{loc}(\R^2),
~~\mbox{as}~\lambda\to 0,
\end{equation}
for some vectors $\textbf{a}_{l}= \big(a_{l,1},a_{l,2}\big) \in \R^2 \backslash \{0\}$. Furthermore,
\begin{equation}\label{ortho23}
\textbf{a}_2\cdot \textbf{a}_{3}=0.
\end{equation}
\end{Lem}
\begin{proof}
Using \eqref{3.10.1} and Proposition \ref{prop3.4}, we know that for any $l=2,3$, there exists $(a_{l,1},a_{l,2},b_{l}) \in \R^3 \backslash \{0\}$ such that
\begin{equation*}
\widetilde{v}_{\lambda,l}(y)
=\sum^2_{q=1}\frac{a_{l,q}y_q}{4+|y|^2}+b_{l}\frac{4-|y|^2}{4+|y|^2}+o\big(1\big)\,~\mbox{in}~C^1_{loc}(\R^2),
~~\mbox{as}~\lambda\to 0.
\end{equation*}
Assume by contradiction that $b_{l}\neq 0$, then by \eqref{3.5.2} in Lemma \ref{lem3.5}, we obtain that for $\lambda$ sufficiently small,
\begin{equation*}
\mu_{\lambda,l}\geq 1+\frac{3}{2\gamma_\lambda^4},
\end{equation*}
which contradicts \eqref{eigenv-23}. Hence $b_{l}=0$ and \eqref{eigenfr-23} holds. Moreover, \eqref{ortho23} can be deduced by \eqref{3.4.3}.

\end{proof}

Using the above estimates, we can prove \eqref{iden-mu23} in Theorem \ref{th_eigenvalue}.

\begin{proof}[\underline{\textbf{Proof of \eqref{iden-mu23} in Theorem \ref{th_eigenvalue}}}]
It follows from \eqref{3.5-1}, \eqref{3.5-2} and \eqref{K} that
\begin{equation*}
v_{\lambda,l}(x)= \lambda \mu_{\lambda,l} A_{\lambda,l} G\big(x_{\lambda},x\big)+\lambda \mu_{\lambda,l} B_{\lambda,l}+
O\big(\theta^{2-\delta}_\lambda\big).
\end{equation*}
By Taylor's expansion, it follows
\begin{equation*}
\begin{split}
B_{\lambda,l}=& \frac{\theta_\lambda}{\lambda} \int_{B_{\frac{d}{\theta_\lambda}}(0)}
\big\langle\nabla G(x_{\lambda},x),z \big\rangle \bigg(\frac{1}{\gamma_\lambda^2}+2\Big(1+\frac{w_{\lambda}(z)}
{\gamma_\lambda^2}\Big)^2\bigg) e^{2w_{\lambda}(z)+\frac{w^2_{\lambda}(z)}{\gamma_\lambda^2} }
\widetilde{v}_{\lambda,l}(z) \,dz \\
&+ O\bigg(\frac{\theta^2_{\lambda}}{\lambda} \int_{B_{\frac{d}{\theta_{\lambda}}}(0)}
|z|^2 \cdot e^{2w_{\lambda}(z)+\frac{w^2_{\lambda}(z)}{\gamma_\lambda^2} } \,dz \bigg)\\
=& \frac{\theta_\lambda}{\lambda} \bigg(\int_{\R^2} 2e^{2U(z)} \sum^2_{q=1}a_{l,q} \cdot \partial_q G(x_{\lambda},x) \frac{z_q^2}{4+|z|^2} \,dz + o(1)\bigg) + O\Big(\frac{\theta_\lambda^{2-\delta}}{\lambda}\Big) \\
=& \frac{\theta_\lambda}{\lambda} \sum^2_{q=1} a_{l,q} \cdot \partial_q G(x_{\lambda},x) \int_{\R^2} e^{2U(z)} \frac{|z|^2}{4+|z|^2} \,dz + o\Big(\frac{\theta_\lambda}{\lambda}\Big) \\
=&\frac{2\pi\theta_\lambda}{\lambda } \sum^2_{q=1}a_{l,q} \cdot \partial_q G(x_{\lambda},x) + o\Big(\frac{\theta_\lambda}{\lambda }\Big).
\end{split}
\end{equation*}
Hence it holds
\begin{equation}\label{v_lambdal}
v_{\lambda,l}(x)= \lambda \mu_{\lambda,l} A_{\lambda,l} G\big(x_{\lambda},x\big)+
2\pi\theta_\lambda \sum^2_{q=1}a_{l,q} \cdot \partial_q G(x_{\lambda},x) + o(\theta_\lambda)~\,~\mbox{in}~C^1\big(\Omega \backslash B_{2d}(x_{\lambda})\big).
\end{equation}
Similar to \eqref{3.5-5}, we can calculate that
\[
A_{\lambda,l} = o\Big(\frac{1}{\lambda \gamma_\lambda^2}\Big).
\]

Next, by \eqref{qgg}, \eqref{qggh}, \eqref{v_lambdal} and Lemmas \ref{prop3-1}, \ref{lem_mu23}, we can compute that
\begin{equation}\label{3.12-1}
\begin{split}
\mbox{LHS of}~~(\ref{quv})=&\lambda \mu_{\lambda,l} A_{\lambda,l} C_{\lambda} Q\Big(G(x_{\lambda},x),G(x_{\lambda},x)\Big) \\
&+ 2\pi \theta_{\lambda} C_{\lambda} \sum^2_{q=1}a_{l,q} Q\Big(G(x_{\lambda},x),\partial_q G(x_{\lambda},x)\Big)
+o\Big(\frac{\theta_\lambda}{\gamma_\lambda}+\lambda |A_{\lambda,l}| \frac{\theta_\lambda}{\gamma_\lambda}\Big)\\
=& -\frac{4\pi \lambda }{\gamma_\lambda} A_{\lambda,l}\frac{\partial \mathcal{R}(x_\lambda)}{\partial x_i} \big(1+o(1)\big)+2\pi \theta_\lambda\frac{4\pi}{\gamma_\lambda} \sum^2_{q=1} a_{l,q} \Big(-\frac{1}{2} \frac{\partial^2 \mathcal{R}(x_\lambda)}{\partial x_i \partial x_q} \Big) + o\Big(\frac{\theta_\lambda}{\gamma_\lambda}\Big) \\
=& -\frac{4\pi^2 \theta_\lambda}{\gamma_\lambda} \sum^2_{q=1} a_{l,q} \frac{\partial^2 \mathcal{R}(x_0)}{\partial x_i \partial x_q} + o\Big(\frac{\theta_\lambda}{\gamma_\lambda}\Big),
\end{split}
\end{equation}
the last equality is derived from \eqref{asym_x} and \eqref{2.6-3}. On the other hand, by the definition of $w_\lambda$, we know
\begin{align*}
& {\lambda}  \int_{B_{d}(x_{\lambda})} \big(1+2u^2_\lambda\big)e^{u_\lambda^2} v_{\lambda,l} \frac{\partial u_\lambda}{\partial x_i} \,dx\\
=& \int_{B_{\frac{d}{\theta_\lambda}(0)}} \bigg(\frac{1}{\gamma^2_\lambda}+ 2\Big(1+ \frac{w_{\lambda}(z)}{\gamma^2_\lambda}\Big)^2 \bigg) e^{2w_{\lambda}(z)+\frac{w^2_{\lambda}(z)}{\gamma_\lambda^2}}  \widetilde{v}_{\lambda,l}(z) \frac{\partial w_\lambda(z)}{\partial z_i} \frac{1}{\theta_\lambda \gamma_\lambda} \,dz \\
=& \frac{2}{\gamma_\lambda \theta_\lambda} \left( \sum^2_{q=1}
\int_{\R^2} e^{2U(z)}\frac{a_{l,q}z_q}{4+|z|^2}\frac{\partial U(z)}{\partial z_i} \,dz+o(1)\right)
= -\frac{\pi}{3 \gamma_\lambda \theta_\lambda} a_{l,i} \big(1+o(1)\big).
\end{align*}
Direct calculation gives that
\begin{equation*}
\lambda\int_{\partial B_d(x_\lambda)}u_\lambda e^{u_\lambda^2}  v_{\lambda,l} \nu_i \,d\sigma
=o\Big(\frac{\theta_\lambda}{\gamma_\lambda}\Big).
\end{equation*}
Then we have
\begin{equation}\label{3.12-2}
\mbox{RHS of}~~(\ref{quv})= -(\mu_{\lambda,l}-1)\frac{\pi}{3 \gamma_\lambda \theta_\lambda} a_{l,i} \big(1+o(1)\big) + o\Big(\frac{\theta_\lambda}{\gamma_\lambda}\Big).
\end{equation}

Thus, from \eqref{3.12-1} and \eqref{3.12-2}, we find
\begin{equation}\label{3.12-3}
\begin{split}
-\frac{4\pi^2 \theta_\lambda}{\gamma_\lambda} \sum^2_{q=1} a_{l,q} \frac{\partial^2 \mathcal{R}(x_0)}{\partial x_i \partial x_q} + o\Big(\frac{\theta_\lambda}{\gamma_\lambda}\Big) =
-(\mu_{\lambda,l}-1)\frac{\pi}{3 \gamma_\lambda \theta_\lambda} a_{l,i} \big(1+o(1)\big),
~~\mbox{for}~~i=1,2.
\end{split}
\end{equation}
Since $a_{l,i}\neq 0$ for some $i$, pass to the limit as $\lambda \to 0$, then \eqref{3.12-3} implies
\begin{equation*}
\begin{split}
\mu_{\lambda,l} = 1+12\pi \theta_\lambda^2 \bigg(\sum^2_{q=1} a_{l,q} \frac{\partial^2 \mathcal{R}(x_{0})}{\partial x_i \partial x_q}\bigg) \frac{1}{a_{l,i}} + o\big(\theta_\lambda^2\big)
=1+ 12\pi \theta_\lambda^2 \eta_l + o\big(\theta_\lambda^2\big),
\end{split}
\end{equation*}
where
\begin{equation*}
\eta_l=\frac{1}{a_{l,i}}\sum^2_{q=1} a_{l,q} \frac{\partial^2 \mathcal{R}(x_{0})}{\partial x_i \partial x_q},
\end{equation*}
for any $i$ such that $a_{l,i} \neq 0$.

\vskip 0.2cm

Therefore, we have
\begin{equation*}
\sum^2_{q=1} a_{l,q} \frac{\partial^2 \mathcal{R}(x_{0})}{\partial x_i \partial x_q} = \eta_l a_{l,i},
\end{equation*}
which holds both if $a_{l,i} \neq 0$ or if $a_{l,i} = 0$ thanks to \eqref{3.12-3}. It means that $\eta_l$ is an eigenvalue of the hessian matrix $D^2 \mathcal{R}(x_0)$, with $\mathbf{a}_l = (a_{l,1}, a_{l,2})$ as corresponding eigenvector. By \eqref{ortho23}, we know that $\eta_l = \Lambda_j$, for some $j = 1,2$. Since $\mu_{\lambda,2} \leq \mu_{\lambda,3}$, we derive that $\eta_l = \Lambda_{l-1}$ and \eqref{iden-mu23} holds.

\end{proof}

\begin{Prop}\label{add-prop_mu4}
It holds
\begin{equation}\label{lim-mu4}
\mu_{\lambda,4} = 1+o(1)~\mbox{as}~\lambda \to 0.
\end{equation}
\end{Prop}

\begin{proof}
Recall that the variational characterization of the eigenvalues
\begin{equation}\label{chara-mu4}
\mu_{\lambda,4} = \inf_{v\in H_0^1(\Omega), v \not\equiv 0 \atop v\perp \{v_{\lambda,1},v_{\lambda,2},v_{\lambda,3}\}}
\frac{\displaystyle\int_\Omega \big|\nabla v\big|^2 \,dx}{\lambda \displaystyle\int_\Omega e^{u_\lambda^2} \big(1+2u_\lambda^2\big) v^2 \,dx}.
\end{equation}
Let $\psi_{\lambda,4}:=(x-x_\lambda) \cdot \nabla u_\lambda$,
which satisfies
\begin{equation}\label{equa-psi4}
-\Delta \psi_{\lambda,4} = 2\lambda u_\lambda e^{u_\lambda^2}+\lambda e^{u_\lambda^2} \big(1+2u_\lambda^2\big) \psi_{\lambda,4}.
\end{equation}
And then define the function
$v:=\widehat{\phi}_\lambda \psi_{\lambda,4} + \displaystyle\sum_{l=1}^3 \alpha_{\lambda,l} v_{\lambda,l}$,
with
\begin{equation}\label{def_phi}
\widehat{\phi}_\lambda(x):=
\begin{cases}
1,~~&\mbox{if}~|x-x_\lambda|\leq  \theta_\lambda,\\[3mm]
\displaystyle \frac{1}{\log \frac{\theta_\lambda}{d}} \log \frac{|x-x_\lambda|}{d},~~&\mbox{if}~\theta_\lambda <|x-x_\lambda| \leq d,\\[3mm]
0,~~&\mbox{if}~|x-x_\lambda|>d,
\end{cases}
\end{equation}
where $d>0$ is a constant such that $B_{2d}(x_{\lambda})\subset\Omega$ and
\begin{equation*}
\alpha_{\lambda,l}:=-\frac{\displaystyle\int_\Omega \lambda e^{u_\lambda^2}\big(1+2u_\lambda^2\big)\widehat{\phi}_\lambda \psi_{\lambda,4} v_{\lambda,l} \,dx}{\displaystyle\int_\Omega \lambda e^{u_\lambda^2}\big(1+2u_\lambda^2\big) v_{\lambda,l}^2 \,dx} = -\frac{N_{\lambda,l}}{D_{\lambda,l}},\quad l=1,2,3.
\end{equation*}
By the definition of $\alpha_{\lambda,l}$, we calculate that
\begin{equation*}
\begin{split}
\int_\Omega \nabla v \cdot \nabla v_{\lambda,l} \,dx
&= \int_\Omega \nabla\big(\widehat{\phi}_\lambda \psi_{\lambda,4}\big) \cdot \nabla v_{\lambda,l} \,dx + \alpha_{\lambda,l} \int_\Omega \big|\nabla v_{\lambda,l}\big|^2 \,dx \\
&= \int_\Omega \mu_{\lambda,l} \lambda e^{u_\lambda^2}\big(1+2u_\lambda^2\big) v_{\lambda,l} \widehat{\phi}_\lambda \psi_{\lambda,4} \,dx -\frac{N_{\lambda,l}}{D_{\lambda,l}} \int_\Omega \mu_{\lambda,l} \lambda e^{u_\lambda^2}\big(1+2u_\lambda^2\big) v_{\lambda,l}^2 \,dx =0,
\end{split}
\end{equation*}
for $l=1,2,3$, which implies that $v\perp \{v_{\lambda,1},v_{\lambda,2},v_{\lambda,3}\}$ in $H_0^1(\Omega)$.

\vskip 0.1cm

Now, we estimate $N_{\lambda,l}$, $D_{\lambda,l}$ and $\alpha_{\lambda,l}$.
Firstly, by Lemma \ref{lem2.4}, Proposition \ref{prop3.4} $(1)$ and Proposition \ref{prop_mu1}, we have
\begin{align}\label{D1}
D_{\lambda,1} &= \lambda \int_\Omega e^{u_\lambda^2} \big(1+2u_\lambda^2\big) v_{\lambda,1}^2 \,dx \notag\\
&= \lambda \int_{B_d(x_\lambda)} e^{u_\lambda^2}\big(1+2u_\lambda^2\big) v_{\lambda,1}^2 \,dx + \lambda \int_{\Omega\backslash B_d(x_\lambda)} e^{u_\lambda^2}\big(1+2u_\lambda^2\big) v_{\lambda,1}^2 \,dx \notag\\
&= \int_{B_{\frac{d}{\theta_\lambda}}(0)} e^{2w_\lambda(z)+\frac{w^2_\lambda(z)}{\gamma^2_\lambda}} \bigg(\frac{1}{\gamma_\lambda^2}+2\Big(\frac{w_\lambda(z)}{\gamma_\lambda^2}+1\Big)^2\bigg) \widetilde{v}_{\lambda,1}^2(z) \,dz +O\big(\theta_\lambda^{2-\delta}\big) \notag\\
&= 2 c_1^2 \int_{\R^2} e^{2U(z)} \,dz +o\big(1\big)\notag\\
& = 8\pi c_1^2 +o\big(1\big):=d_1+o\big(1\big),
\end{align}
where $c_1$ is the nonzero constant in Proposition \ref{prop3.4}.
Similarly, for $l=2,3$, we calculate that
\begin{align}\label{D23}
D_{\lambda,l} &= \lambda \int_{B_d(x_\lambda)} e^{u_\lambda^2}\big(1+2u_\lambda^2\big) v_{\lambda,l}^2 \,dx + \lambda \int_{\Omega\backslash B_d(x_\lambda)} e^{u_\lambda^2}\big(1+2u_\lambda^2\big) v_{\lambda,l}^2 \,dx \notag\\
&= \int_{B_{\frac{d}{\theta_\lambda}}(0)} e^{2w_\lambda(z)+\frac{w^2_\lambda(z)}{\gamma^2_\lambda}} \bigg(\frac{1}{\gamma_\lambda^2}+2\Big(\frac{w_\lambda(z)}{\gamma_\lambda^2}+1\Big)^2\bigg) \widetilde{v}_{\lambda,l}^2(z) \,dz +O\big(\theta_\lambda^{2-\delta}\big) \notag\\
&= 2\int_{\R^2} e^{2U(z)} \Big(\sum_{q=1}^2 \frac{a_{l,q}z_q}{4+|z|^2}\Big)^2 \,dz +o\big(1\big)\notag\\
&= \frac{\pi}{6}\sum_{q=1}^2 a_{l,q}^2 +o\big(1\big)
:=d_l+o\big(1\big),
\end{align}
where the third equality is derived from \eqref{eigenfr-23}. Hence, for any $l=1,2,3$, we deduce from \eqref{D1} and \eqref{D23} that there exists $d_l>0$ such that
\begin{equation}\label{D123}
D_{\lambda,l}=d_l+o\big(1\big),~\mbox{as}~\lambda \to 0.
\end{equation}
Using the definition of $\widehat{\phi}_{\lambda}$, we know that
\begin{align*}
N_{\lambda,l} =& \int_\Omega \lambda e^{u_\lambda^2}\big(1+2u_\lambda^2\big)\widehat{\phi}_\lambda \psi_{\lambda,4} v_{\lambda,l} \,dx \\
=& \int_{B_{\theta_\lambda}(x_\lambda)} \lambda e^{u_\lambda^2}\big(1+2u_\lambda^2\big) \psi_{\lambda,4} v_{\lambda,l} \,dx \\
& +\frac{\lambda}{\log\frac{\theta_\lambda}{d}} \int_{\{\theta_\lambda <|x-x_\lambda|\leq d\}}  e^{u_\lambda^2}\big(1+2u_\lambda^2\big) \psi_{\lambda,4} v_{\lambda,l} \log\frac{|x-x_\lambda|}{d} \,dx \\
=& \frac{1}{\gamma_\lambda} \int_{B_1(0)} e^{2w_\lambda(z)+\frac{w^2_\lambda(z)}{\gamma^2_\lambda}} \bigg(\frac{1}{\gamma_\lambda^2}+2\Big(\frac{w_\lambda(z)}{\gamma_\lambda^2}+1\Big)^2\bigg) \big(z \cdot \nabla w_\lambda(z)\big) \widetilde{v}_{\lambda,l}(z) \,dz \\
& +\frac{1}{\gamma_\lambda} \int_{\big\{1 <|z|\leq \frac{d}{\theta_\lambda}\big\}} e^{2w_\lambda(z)+\frac{w^2_\lambda(z)}{\gamma^2_\lambda}} \bigg(\frac{1}{\gamma_\lambda^2}+2\Big(\frac{w_\lambda(z)}{\gamma_\lambda^2}+1\Big)^2\bigg) \big(z \cdot \nabla w_\lambda(z)\big) \widetilde{v}_{\lambda,l}(z) \,dz \\
& +\frac{1}{\gamma_\lambda \log\frac{\theta_\lambda}{d}} \int_{\big\{1 <|z|\leq \frac{d}{\theta_\lambda}\big\}} e^{2w_\lambda(z)+\frac{w^2_\lambda(z)}{\gamma^2_\lambda}} \bigg(\frac{1}{\gamma_\lambda^2}+2\Big(\frac{w_\lambda(z)}{\gamma_\lambda^2}+1\Big)^2\bigg) \big(z \cdot \nabla w_\lambda(z)\big) \widetilde{v}_{\lambda,l}(z) \log|z| \,dz.
\end{align*}
Then we have
\begin{equation}\label{N1}
\begin{split}
N_{\lambda,1} =& \frac{2c_1}{\gamma_\lambda} \int_{B_1(0)} e^{2U(z)} \big(z \cdot \nabla U(z)\big) \,dz
+\frac{2c_1}{\gamma_\lambda} \int_{\R^2\backslash B_1(0)} e^{2U(z)} \big(z \cdot \nabla U(z)\big) \,dz \\
& +\frac{2c_1}{\gamma_\lambda \log \frac{\theta_\lambda}{d}} \int_{\R^2\backslash B_1(0)} e^{2U(z)} \big(z \cdot \nabla U(z)\big) \log|z| \,dz +o\Big(\frac{1}{\gamma_\lambda}\Big) \\
=& \frac{2c_1}{\gamma_\lambda} \int_{\R^2} e^{2U(z)} \big(z \cdot \nabla U(z)\big) \,dz +o\Big(\frac{1}{\gamma_\lambda}\Big)\\
=& -\frac{8 \pi c_1}{\gamma_\lambda} +o\Big(\frac{1}{\gamma_\lambda}\Big),
\end{split}
\end{equation}
here we use the fact that
\[
\frac{1}{\log \frac{\theta_\lambda}{d}} = \frac{1}{-\frac{1}{2}\log \big(\lambda \gamma_\lambda^2\big) -\frac{1}{2}\gamma_\lambda^2 -\log d} = o\Big(\frac{1}{\gamma_\lambda}\Big).
\]
We deduce from \eqref{eigenfr-23} that
\begin{equation}\label{N23}
\begin{split}
N_{\lambda,l} =& \frac{2}{\gamma_\lambda} \int_{\R^2} e^{2U(z)} \big(z \cdot \nabla U(z)\big) \sum_{q=1}^2 \frac{a_{l,q}z_q}{4+|z|^2} \,dz \\
& +\frac{2}{\gamma_\lambda \log \frac{\theta_\lambda}{d}} \int_{\R^2\backslash B_1(0)} e^{2U(z)} \big(z \cdot \nabla U(z)\big) \log|z| \sum_{q=1}^2 \frac{a_{l,q}z_q}{4+|z|^2} \,dz +o\Big(\frac{1}{\gamma_\lambda}\Big) \\
=& o\Big(\frac{1}{\gamma_\lambda}\Big), \quad \mbox{for}~l=2,3.
\end{split}
\end{equation}
Then, it follows from \eqref{N1} and \eqref{N23} that
\begin{equation}\label{N123}
N_{\lambda,l} =
\begin{cases}
\displaystyle -\frac{8 \pi c_1}{\gamma_\lambda} +o\Big(\frac{1}{\gamma_\lambda}\Big),~~&\mbox{for}~l=1,\\[5mm]
\displaystyle o\Big(\frac{1}{\gamma_\lambda}\Big),~~&\mbox{for}~l=2,3.
\end{cases}
\end{equation}
Therefore, by \eqref{D123} and \eqref{N123}, we obtain
\begin{equation}\label{alpha}
\alpha_{\lambda,l} = -\frac{N_{\lambda,l}}{D_{\lambda,l}}=
\begin{cases}
\displaystyle \frac{8 \pi c_1}{d_1 \gamma_\lambda} +o\Big(\frac{1}{\gamma_\lambda}\Big),~~&\mbox{for}~l=1,\\[5mm]
\displaystyle o\Big(\frac{1}{\gamma_\lambda}\Big),~~&\mbox{for}~l=2,3.
\end{cases}
\end{equation}
Since $\{v_{\lambda,1},v_{\lambda,2},v_{\lambda,3}\}$ are orthogonal in $H_0^1(\Omega)$ and $v_{\lambda,l}$ satisfies \eqref{eigen-1}, using Proposition \ref{prop_mu1}, Lemma \ref{lem_mu23}, \eqref{D123}, \eqref{N123} and \eqref{alpha}, we have
\begin{align}\label{numerator1}
\int_\Omega \big|\nabla v\big|^2 \,dx
=& \int_\Omega \Big|\nabla \big(\widehat{\phi}_\lambda \psi_{\lambda,4}\big)\Big|^2 \,dx +\sum_{l=1}^3 \alpha_{\lambda,l}^2 \int_\Omega \big|\nabla v_{\lambda,l}\big|^2 \,dx + 2\sum_{l=1}^3 \alpha_{\lambda,l} \int_\Omega \nabla \big(\widehat{\phi}_\lambda \psi_{\lambda,4}\big) \cdot \nabla v_{\lambda,l} \,dx \notag\\
=& \int_\Omega \Big|\nabla \big(\widehat{\phi}_\lambda \psi_{\lambda,4}\big)\Big|^2 \,dx + \sum_{l=1}^3 \mu_{\lambda,l} \alpha_{\lambda,l}^2 D_{\lambda,l} +2\sum_{l=1}^3 \mu_{\lambda,l} \alpha_{\lambda,l} N_{\lambda,l} \notag\\
=& \int_\Omega \Big|\nabla \big(\widehat{\phi}_\lambda \psi_{\lambda,4}\big)\Big|^2 \,dx +o\Big(\frac{1}{\gamma_\lambda^2}\Big).
\end{align}
Then by \eqref{equa-psi4} and \eqref{def_phi}, we calculate that
\begin{align}\label{numerator2}
& \int_\Omega \Big|\nabla \big(\widehat{\phi}_\lambda \psi_{\lambda,4}\big)\Big|^2 \,dx \notag\\
=& \int_\Omega \big|\nabla \widehat{\phi}_\lambda \big|^2 \psi_{\lambda,4}^2 \,dx -\int_\Omega \Delta \psi_{\lambda,4} \widehat{\phi}_\lambda^2 \psi_{\lambda,4} \,dx \notag\\
=& \int_\Omega \big|\nabla \widehat{\phi}_\lambda \big|^2 \psi_{\lambda,4}^2 \,dx + \int_\Omega 2\lambda e^{u_\lambda^2} u_\lambda \widehat{\phi}_\lambda^2 \psi_{\lambda,4} \,dx + \int_\Omega \lambda e^{u_\lambda^2} \big(1+2u_\lambda^2\big) \widehat{\phi}_\lambda^2 \psi_{\lambda,4}^2 \,dx \notag\\
=& \int_\Omega \big|\nabla \widehat{\phi}_\lambda \big|^2 \psi_{\lambda,4}^2 \,dx +\int_\Omega \lambda e^{u_\lambda^2} \big(1+2u_\lambda^2\big) \widehat{\phi}_\lambda^2 \psi_{\lambda,4}^2 \,dx \notag\\
& +\frac{2}{\gamma_\lambda^2} \int_{B_{\frac{d}{\theta_\lambda}}(0)} e^{2w_\lambda(z)+\frac{w^2_\lambda(z)}{\gamma^2_\lambda}} \Big(\frac{w_\lambda(z)}{\gamma_\lambda^2}+1\Big) \widehat{\phi}_\lambda^2(\theta_\lambda z+x_\lambda) \cdot \big(z \cdot \nabla w_\lambda(z)\big) \,dz \notag\\
=& \int_\Omega \big|\nabla \widehat{\phi}_\lambda \big|^2 \psi_{\lambda,4}^2 \,dx +\int_\Omega \lambda e^{u_\lambda^2} \big(1+2u_\lambda^2\big) \widehat{\phi}_\lambda^2 \psi_{\lambda,4}^2 \,dx + \frac{2}{\gamma_\lambda^2} \bigg(\int_{\R^2} e^{2U(z)} \big(z \cdot \nabla U(z)\big) \,dz +o(1)\bigg) \notag\\
=& \int_\Omega \big|\nabla \widehat{\phi}_\lambda \big|^2 \psi_{\lambda,4}^2 \,dx +\int_\Omega \lambda e^{u_\lambda^2} \big(1+2u_\lambda^2\big) \widehat{\phi}_\lambda^2 \psi_{\lambda,4}^2 \,dx -\frac{8\pi}{\gamma_\lambda^2} +o\Big(\frac{1}{\gamma_\lambda^2}\Big).
\end{align}
Hence, substituting \eqref{numerator2} into \eqref{numerator1}, we derive
\begin{equation}\label{numerator}
\int_\Omega \big|\nabla v\big|^2 \,dx = \int_\Omega \big|\nabla \widehat{\phi}_\lambda \big|^2 \psi_{\lambda,4}^2 \,dx +\lambda \int_\Omega e^{u_\lambda^2} \big(1+2u_\lambda^2\big) \widehat{\phi}_\lambda^2 \psi_{\lambda,4}^2 \,dx -\frac{8\pi}{\gamma_\lambda^2} +o\Big(\frac{1}{\gamma_\lambda^2}\Big).
\end{equation}
Similarly, by the definition of $\alpha_{\lambda,l}$, we obtain from \eqref{D123} and \eqref{N123} that
\begin{equation}\label{denominator}
\begin{split}
\lambda \int_\Omega e^{u_\lambda^2} \big(1+2u_\lambda^2\big) v^2 \,dx
&= \lambda \int_\Omega e^{u_\lambda^2} \big(1+2u_\lambda^2\big) \widehat{\phi}_\lambda^2 \psi_{\lambda,4}^2 \,dx +\sum_{l=1}^3 \alpha_{\lambda,l}^2 D_{\lambda,l} +2\sum_{l=1}^3 \alpha_{\lambda,l} N_{\lambda,l} \\
&= \lambda \int_\Omega e^{u_\lambda^2} \big(1+2u_\lambda^2\big) \widehat{\phi}_\lambda^2 \psi_{\lambda,4}^2 \,dx -\sum_{l=1}^3 \frac{N_{\lambda,l}^2}{D_{\lambda,l}} \\
&= \lambda \int_\Omega e^{u_\lambda^2} \big(1+2u_\lambda^2\big) \widehat{\phi}_\lambda^2 \psi_{\lambda,4}^2 \,dx
-\frac{64 \pi^2 c_1^2}{d_1 \gamma_\lambda^2} +o\Big(\frac{1}{\gamma_\lambda^2}\Big) \\
&= \lambda \int_\Omega e^{u_\lambda^2} \big(1+2u_\lambda^2\big) \widehat{\phi}_\lambda^2 \psi_{\lambda,4}^2 \,dx
-\frac{8\pi}{\gamma_\lambda^2} +o\Big(\frac{1}{\gamma_\lambda^2}\Big),
\end{split}
\end{equation}
since $d_1=8\pi c_1^2$.
Inserting \eqref{numerator} and \eqref{denominator} into \eqref{chara-mu4}, we get
\begin{equation}\label{num-deno}
\mu_{\lambda,4} \leq 1+ \frac{\displaystyle\int_\Omega \big|\nabla \widehat{\phi}_\lambda \big|^2 \psi_{\lambda,4}^2 \,dx +o\Big(\frac{1}{\gamma_\lambda^2}\Big)}{\displaystyle\lambda \int_\Omega e^{u_\lambda^2} \big(1+2u_\lambda^2\big) \widehat{\phi}_\lambda^2 \psi_{\lambda,4}^2 \,dx -\frac{8\pi}{\gamma_\lambda^2} +o\Big(\frac{1}{\gamma_\lambda^2}\Big)}.
\end{equation}
Next, we estimate the two integrals.
\begin{align}\label{num-deno1}
& \lambda \int_\Omega e^{u_\lambda^2} \big(1+2u_\lambda^2\big) \widehat{\phi}_\lambda^2 \psi_{\lambda,4}^2 \,dx \notag\\
=& \lambda \int_{B_d(x_\lambda)} e^{u_\lambda^2} \big(1+2u_\lambda^2\big) \widehat{\phi}_\lambda^2 \psi_{\lambda,4}^2 \,dx \notag\\
=& \frac{1}{\gamma_\lambda^2} \int_{B_{\frac{d}{\theta_\lambda}}(0)} e^{2w_\lambda(z)+\frac{w^2_\lambda(z)}{\gamma^2_\lambda}} \bigg(\frac{1}{\gamma_\lambda^2}+2\Big(\frac{w_\lambda(z)}{\gamma_\lambda^2}+1\Big)^2\bigg) \widehat{\phi}_\lambda^2(\theta_\lambda z+x_\lambda) \cdot \big(z \cdot \nabla w_\lambda(z)\big)^2 \,dz \notag\\
=& \frac{1}{\gamma_\lambda^2} \bigg(\int_{\R^2} 2e^{2U(z)} \big(z \cdot \nabla U(z)\big)^2 \,dz +o(1) \bigg)
= \frac{32 \pi}{3 \gamma_\lambda^2} +o\Big(\frac{1}{\gamma_\lambda^2}\Big),
\end{align}
and
\begin{align}\label{num-deno2}
\int_\Omega \big|\nabla \widehat{\phi}_\lambda \big|^2 \psi_{\lambda,4}^2 \,dx
&= \int_{\{\theta_\lambda<|x-x_\lambda|\leq d\}} \big|\nabla \widehat{\phi}_\lambda \big|^2 \big( (x-x_\lambda) \cdot \nabla u_\lambda \big)^2 \,dx \notag\\
&= \frac{1}{|\log \frac{\theta_\lambda}{d}|^2} \int_{\{\theta_\lambda<|x-x_\lambda|\leq d\}} \frac{1}{|x-x_\lambda|^2} \big( (x-x_\lambda) \cdot \nabla u_\lambda \big)^2 \,dx \notag\\
&= \frac{1}{\gamma_\lambda^2 |\log \frac{\theta_\lambda}{d}|^2} \int_{\big\{1<|z|\leq \frac{d}{\theta_\lambda}\big\}} \frac{1}{\theta_\lambda^2 |z|^2} \big(z \cdot \nabla w_\lambda(z) \big)^2 \cdot \theta_\lambda^2 \,dz \notag\\
&= \frac{1}{\gamma_\lambda^2 |\log \frac{\theta_\lambda}{d}|^2} \int_{\big\{1<|z|\leq \frac{d}{\theta_\lambda}\big\}} \frac{\big(z \cdot \nabla U(z) \big)^2}{|z|^2} \,dz +o\Big(\frac{1}{\gamma_\lambda^2 |\log \frac{\theta_\lambda}{d}|^2}\Big) \notag\\
&= \frac{1}{\gamma_\lambda^2 |\log \frac{\theta_\lambda}{d}|^2} \int_{\big\{1<|z|\leq \frac{d}{\theta_\lambda}\big\}} \frac{4|z|^2}{(4+|z|^2)^2} \,dz +o\Big(\frac{1}{\gamma_\lambda^2 |\log \frac{\theta_\lambda}{d}|^2}\Big) \notag\\
&\leq \frac{C}{\gamma_\lambda^2 |\log \frac{\theta_\lambda}{d}|} +o\Big(\frac{1}{\gamma_\lambda^2 |\log \frac{\theta_\lambda}{d}|^2}\Big) = o\Big(\frac{1}{\gamma_\lambda^2}\Big).
\end{align}
Substituting \eqref{num-deno1} and \eqref{num-deno2} into \eqref{num-deno}, we derive
\begin{equation*}
\mu_{\lambda,4} \leq 1 +\frac{\displaystyle o\Big(\frac{1}{\gamma_\lambda^2}\Big)}{\displaystyle \frac{8\pi}{3\gamma_\lambda^2} +o\Big(\frac{1}{\gamma_\lambda^2}\Big)} =1+o(1).
\end{equation*}
Since $\mu_{\lambda,1} < \mu_{\lambda,2} \leq \mu_{\lambda,3} \leq \cdots$ and $\mu_{\lambda,3}=1+o(1)$, we complete the proof.

\end{proof}

Next, we complete the proof of \eqref{9-13-20} in Theorem \ref{th_eigenvalue} by Lemma \ref{lem3.5} and Proposition \ref{add-prop_mu4}.

\begin{proof}[\underline{\textbf{Proof of \eqref{9-13-20} in Theorem \ref{th_eigenvalue}}}]
From $\eqref{lim-mu4}$, we find $\mu_{4}=1$, and by \eqref{3.5.1} in Lemma \ref{lem3.5}, we find that
\begin{equation*}
v_{\lambda,4}= -\frac{4\pi b_4}{\gamma^2_\lambda}G(x,x_\lambda)+o\Big(\frac{1}{\gamma^2_\lambda}\Big), ~~\,~~\mbox{in}~~~~C^1_{loc}\big( \overline{\Omega} \backslash \{x_{\lambda}\}\big).
\end{equation*}
Assume $b_{4}=0$, then by \eqref{3.4.3}, we have
\begin{equation}\label{3.13-2}
\textbf{a}_{4}\neq \textbf{0}~~\mbox{and}~~\textbf{a}_{4}\cdot \textbf{a}_{l}=0~~\mbox{for}~~l=2,3.
\end{equation}
On the other hand, from \eqref{ortho23}, we have
\begin{equation*}
\mbox{span}\big\{\textbf{a}_{i},~~i=2,3\big\}=\R^{2},
\end{equation*}
which is a contradiction with \eqref{3.13-2}. Hence $b_4\neq 0$, which together with Lemma \ref{lem3.5} implies that
\begin{equation*}
\mu_{\lambda,4} = 1+\frac{3}{\gamma_\lambda^4}+o\Big(\frac{1}{\gamma_\lambda^4}\Big).
\end{equation*}
\end{proof}

Using the above asymptotic analysis of eigenvalues and eigenfunctions to problem \eqref{eigen-1}, we give the proofs of Theorem \ref{th_Morse}  and Theorem \ref{th_nondeg}.

\begin{proof}[\textbf{Proof of Theorem \ref{th_Morse}}]
It follows from Proposition \ref{prop_mu1} and \eqref{9-13-20} that
\begin{equation*}
\mu_{\lambda,1}<1~\mbox{and}~\mu_{\lambda,4}>1.
\end{equation*}
Also, for $l=2,3$, we have
\begin{equation*}
\begin{split}
\sharp \Big\{l,~\mu_{\lambda,l}<1,~2\leq l\leq 3\Big\} \geq
\sharp \Big\{ \mbox{the negative eigenvalues of $\big( D^2 \mathcal{R}(x_{0}) \big)$}\Big\}= m\big(\mathcal{R}(x_{0})\big),
\end{split}
\end{equation*}
and
\begin{equation*}
\begin{split}
\sharp \Big\{l,~\mu_{\lambda,l} \leq 1,~2\leq l\leq 3\Big\} \leq
\sharp \Big\{ \mbox{the non-positive  eigenvalues of $\big( D^2 \mathcal{R}(x_{0}) \big)$}\Big\}
= m_0\big(\mathcal{R}(x_{0})\big).
\end{split}
\end{equation*}
Hence we have
\begin{equation*}
1\leq 1+m\big(\mathcal{R}(x_{0})\big)\leq  m(u_\lambda)\leq m_0(u_\lambda)\leq  1+m_0\big(\mathcal{R}(x_{0})\big)\leq 3.
\end{equation*}
Furthermore, if $D^2\big(\mathcal{R}(x_{0})\big)$ is nondegenerate, then
$m\big(\mathcal{R}(x_{0})\big)=m_0\big(\mathcal{R}(x_{0})\big)$
and
\begin{equation*}
m(u_\lambda)=1+m\big(\mathcal{R}(x_{0})\big)\in [1,3].
\end{equation*}
\end{proof}

\begin{proof}[\textbf{Proof of Theorem \ref{th_nondeg}}]
If $x_{0}$ is a non-degenerate critical point of Robin function $\mathcal{R}(x)$, from Theorem \ref{th_eigenvalue} and Proposition \ref{prop_mu1}, we have
\begin{equation*}
\mu_{\lambda,k}\neq 1,~\mbox{for any}~k\in \N,
\end{equation*}
which means  that $\mathcal{L}_\lambda(u)=0$ has no nonzero solutions. Hence
from $\mathcal{L}_\lambda\big(\xi_\lambda\big)=0$, we can find $\xi_\lambda=0$ for any $\lambda>0$ sufficiently
small.
\end{proof}

\section{More precise estimates on $u_\lambda$, $x_\lambda$ and $\gamma_\lambda$}\label{s5}

\vskip 0.2cm

Using Proposition \ref{prop_wlambda} and \eqref{def_C1}, we can obtain better estimates on $u_\lambda$, $x_\lambda$ and $\gamma_\lambda$, which is crucial to prove the local uniqueness of the concentrated solutions. In this section, we will prove \eqref{4.9.1}, \eqref{xlambda-x0} and \eqref{lambda_gamma} in Theorem \ref{thm-lambdagamma}.

\begin{proof}[\underline{\textbf{Proof of \eqref{4.9.1} in Theorem \ref{thm-lambdagamma}}}]
Direct calculation gives
\begin{equation}\label{4.9-1}
\begin{split}
&\lambda \int_{B_d(x_{\lambda})}   \big(G(y,x)-G(x_\lambda,x)\big)u_{\lambda}(y)e^{u^2_{\lambda}(y)}\,dy \\
=& \frac{\theta_{\lambda}}{\gamma_\lambda} \int_{B_{\frac{d}{\theta_\lambda}}(0)} \big\langle z,\nabla G(x_\lambda,x)\big\rangle \Big(1+\frac{w_\lambda(z)}{\gamma^2_\lambda}\Big)e^{2w_\lambda(z)+\frac{w^2_\lambda(z)}
{\gamma^2_\lambda}}\,dz +O\Big(\frac{\theta_\lambda^{2-\delta}}{\gamma_\lambda}  \Big).
\end{split}
\end{equation}
Recall that
\[
\lim_{\lambda \to 0} v_\lambda(x) = v_0(x) = w(x)+ c_0\frac{4-|x|^2}{4+|x|^2} + \sum^2_{i=1} c_i \frac{x_i}{4+|x|^2}~\mbox{in}~C^2_{loc}(\R^2),
\]
where $w(x)$ is the radial solution of $-\Delta u-2e^{2U}u = e^{2U}\big(U^2+U\big)$.
Using \eqref{2.4.1} and \eqref{v-def}, we find
\begin{align*}
&\int_{B_{\frac{d}{\theta_\lambda}}(0)} \big\langle z,\nabla G(x_\lambda,x)\big\rangle
\Big(1+\frac{w_\lambda(z)}{\gamma^2_\lambda}\Big)e^{2w_\lambda(z)+\frac{w^2_\lambda(z)}{\gamma^2_\lambda}}\,dz \\=&
\int_{B_{\frac{d}{\theta_\lambda}}(0)}  \big\langle z,\nabla G(x_\lambda,x)\big\rangle e^{2U(z)}
\Big(1+\frac{U(z)}{\gamma^2_\lambda}+\frac{v_\lambda(z)}{\gamma^4_\lambda}\Big)
\bigg( 1+\frac{2v_\lambda(z)}{\gamma^2_\lambda}  +\frac{w^2_\lambda(z)}{\gamma^2_\lambda} +O\Big(\frac{\big(2v_\lambda+w^2_\lambda\big)^2}{\gamma^4_\lambda}\Big) \bigg) \,dz \\
=& \int_{B_{\frac{d}{\theta_\lambda}}(0)}  \big\langle z,\nabla G(x_\lambda,x)\big\rangle e^{2U(z)}
\Big(1+\frac{U(z)+2v_\lambda(z) +w^2_\lambda(z)}{\gamma^2_\lambda} \Big)
\,dz+ O\Big( \frac{1}{\gamma^4_\lambda}\Big) \\
=& \sum_{i=1}^2 \partial_i G(x_\lambda,x) \int_{B_{\frac{d}{\theta_\lambda}}(0)} z_i e^{2U(z)} \frac{2v_\lambda(z)}{\gamma^2_\lambda} \,dz + o\Big( \frac{1}{\gamma^2_\lambda}\Big) \\
=& \sum_{i=1}^2 \frac{2 c_i}{\gamma^2_\lambda} \partial_i G(x_\lambda,x) \int_{B_{\frac{d}{\theta_\lambda}}(0)} z_i e^{2U(z)} \frac{z_i}{4+|z|^2} \,dz + o\Big( \frac{1}{\gamma^2_\lambda}\Big) \\
=& \sum_{i=1}^2 \frac{c_i}{\gamma^2_\lambda} \partial_i G(x_\lambda,x) \int_{\R^2} e^{2U(z)} \frac{|z|^2}{4+|z|^2} \,dz + o\Big( \frac{1}{\gamma^2_\lambda}\Big)
= \sum_{i=1}^2 \frac{2\pi c_i}{\gamma^2_\lambda} \partial_i G(x_\lambda,x) + o\Big( \frac{1}{\gamma^2_\lambda}\Big),
\end{align*}
which together with \eqref{4.9-1} implies
\begin{equation}\label{4.9-2}
\lambda \int_{B_d(x_{\lambda})}   \big(G(y,x)-G(x_\lambda,x)\big) u_{\lambda}(y)e^{u^2_{\lambda}(y)} \,dy
= \sum_{i=1}^2 \frac{2\pi c_i}{\gamma^3_\lambda} \theta_\lambda \partial_i G(x_\lambda,x)
+ o\Big(\frac{\theta_\lambda}{\gamma^3_\lambda}\Big).
\end{equation}
Substituting \eqref{2.5-3} and \eqref{4.9-2} into \eqref{2.5-1}, we have
\begin{equation}\label{4.9-3}
u_\lambda(x)=  C_{\lambda}G(x_{\lambda},x) + \sum_{i=1}^2 \frac{2\pi c_i}{\gamma^3_\lambda} \theta_\lambda \partial_i G(x_\lambda,x) + o\Big(\frac{\theta_\lambda}{\gamma^3_\lambda}\Big) \,\,~\mbox{in}~ \Omega\backslash B_{2d}(x_{\lambda}).
\end{equation}

Similar to \eqref{4.9-3}, we have
\begin{align*}
\frac{\partial u_\lambda(x)}{\partial x_i} &= \lambda \int_{\Omega} D_{x_i} G(y,x) u_{\lambda}(y)e^{u^2_{\lambda}(y)} \,dy \\
&= \lambda \int_{\Omega\backslash B_d(x_\lambda)} D_{x_i} G(y,x) u_{\lambda}(y)e^{u^2_{\lambda}(y)} \,dy + \lambda \int_{B_d(x_\lambda)} D_{x_i} G(y,x) u_{\lambda}(y)e^{u^2_{\lambda}(y)} \,dy\\
&= C_{\lambda} D_{x_i} G(x_{\lambda},x) + \sum_{j=1}^2 \frac{2\pi c_j}{\gamma^3_\lambda} \theta_\lambda \partial_j \big( D_{x_i} G(x_\lambda,x)\big) + o\Big(\frac{\theta_\lambda}{\gamma^3_\lambda}\Big).
\end{align*}
Therefore  \eqref{4.9.1} holds.
\end{proof}

By \eqref{4.9.1}, we obtain a more precise estimate on $\big|x_\lambda-x_0\big|$ than Lemma \ref{lem-h1}, which can be stated as follows.

\begin{proof}[\underline{\textbf{Proof of \eqref{xlambda-x0} in Theorem \ref{thm-lambdagamma}}}]
Firstly, we can deduce from \eqref{4.9.1}, \eqref{qgg} and \eqref{qggh} that
\begin{equation}\label{4-10-1}
\begin{split}
\mbox{LHS of \eqref{quu}} &= C^2_\lambda Q\big(G(x,x_\lambda),G(x,x_\lambda)\big) + 2\sum_{h=1}^2 \frac{2\pi c_h} {\gamma_\lambda^3} C_\lambda \theta_\lambda Q\big(G(x,x_\lambda),\partial_h G(x,x_\lambda)\big) +o\Big(\frac{\theta_\lambda}{\gamma^4_\lambda}\Big) \\
&= -C^2_\lambda \frac{\partial \mathcal{R}(x_{\lambda})}{\partial{x_i}} - \sum_{h=1}^2 \frac{2\pi c_h}{\gamma_\lambda^3} C_\lambda \theta_\lambda \frac{\partial^2 \mathcal{R}(x_\lambda)}{\partial x_i \partial x_h}
+o\Big(\frac{\theta_\lambda}{\gamma^4_\lambda}\Big).
\end{split}
\end{equation}
On the other hand,
\begin{equation}\label{4-10-2}
\begin{split}
\mbox{RHS of \eqref{quu}} &= O\Big(\lambda \theta_\lambda \int_{\partial B_{\frac{d}{\theta_\lambda}}(0)} e^{\gamma^2_\lambda+2w_\lambda(z)+\frac{w^2_\lambda(z)}{\gamma^2_\lambda}} \,d\sigma\Big) = O\Big(\frac{\theta_\lambda^{2-\delta}}{\gamma^2_\lambda}\Big)=o\Big(\frac{\theta_\lambda}{\gamma^4_\lambda}\Big).
\end{split}
\end{equation}
Note that \eqref{def_C1} implies
\[
\frac{1}{C_\lambda} = \frac{\gamma_\lambda}{4\pi} \left( 1-\frac{1}{\gamma^2_\lambda} + o\Big(\frac{1}{\gamma^2_\lambda}\Big) \right),
\]
and then it follows from Lemma \ref{lem-h1}, \eqref{4-10-1} and \eqref{4-10-2} that
\begin{equation*}
\begin{split}
\frac{\partial \mathcal{R}(x_{\lambda})}{\partial{x_i}} &= -\sum_{h=1}^2 \frac{2\pi c_h}{\gamma_\lambda^3} \frac{1}{C_\lambda} \theta_\lambda \frac{\partial^2 \mathcal{R}(x_\lambda)}{\partial x_i \partial x_h}
+o\Big(\frac{\theta_\lambda}{\gamma^2_\lambda}\Big) \\
&= -\sum_{h=1}^2 \frac{2\pi c_h}{\gamma_\lambda^3} \frac{1}{C_\lambda} \theta_\lambda \frac{\partial^2 \mathcal{R}(x_0)}
{\partial x_i \partial x_h} + o\Big(\frac{\theta_\lambda}{\gamma^2_\lambda}\Big) = -\sum_{h=1}^2 \frac{c_h}{2 \gamma_\lambda^2} \theta_\lambda \frac{\partial^2 \mathcal{R}(x_0)}{\partial x_i \partial x_h} + o\Big(\frac{\theta_\lambda}{\gamma^2_\lambda}\Big).
\end{split}
\end{equation*}
Furthermore, since $x_0$ is a non-degenerate critical point of $\mathcal{R}(x)$, we derive
\begin{equation}\label{4.10-3}
-\sum_{h=1}^2 \frac{c_h}{2 \gamma_\lambda^2} \theta_\lambda \frac{\partial^2 \mathcal{R}(x_0)}{\partial x_i \partial x_h} + o\Big(\frac{\theta_\lambda}{\gamma^2_\lambda}\Big) = \Big\langle \nabla \frac{\partial \mathcal{R}(x_0)}{\partial x_i}, x_\lambda-x_0 \Big\rangle + O\big(|x_\lambda-x_0|^2\big),
\end{equation}
which implies \eqref{xlambda-x0}.
\end{proof}

Using the above results, we can obtain a more accurate estimate than \eqref{DT-lam_gam} for the single-bubble case.
\begin{proof}[\underline{\textbf{Proof of \eqref{lambda_gamma} in Theorem \ref{thm-lambdagamma}}}]
By the Green's representation theorem, we have
\begin{align}\label{4.11-1}
\gamma_\lambda &= u_\lambda(x_{\lambda})= \lambda \int_{\Omega} G(y,x_{\lambda}) u_{\lambda}(y)e^{u^2_{\lambda}(y)} \,dy \notag\\
&= \lambda \int_{B_d(x_{\lambda})} G(y,x_{\lambda}) u_{\lambda}(y)e^{u^2_{\lambda}(y)}\,dy
+\lambda \int_{\Omega\backslash   B_{d}(x_{\lambda})} G(y,x_{\lambda}) u_{\lambda}(y)e^{u^2_{\lambda}(y)}\,dy.
\end{align}
Scaling and using the properties of Green's function, we know
\begin{align*}
 &\lambda  \int_{B_d(x_{\lambda})}  G(y,x_{\lambda}) u_{\lambda}(y)e^{u^2_{\lambda}(y)}\,dy \\
=&~ \lambda \theta_\lambda^2 \int_{B_{\frac{d}{\theta_\lambda}}(0)} G\big(x_{\lambda}+\theta_{\lambda}z,x_{\lambda}\big)
u_\lambda\big(x_{\lambda}+\theta_{\lambda}z\big) e^{u^2_\lambda(x_{\lambda}+\theta_{\lambda}z)} \, dz \\
=&~ -\lambda \theta_\lambda^2 \int_{B_{\frac{d}{\theta_\lambda}}(0)} H\big(x_{\lambda}+\theta_{\lambda}z,x_{\lambda}\big)
\Big(\gamma_{\lambda}+\frac{w_\lambda(z)}{\gamma_\lambda}\Big) e^{\big(\gamma_{\lambda}+\frac{w_\lambda(z)}
{\gamma_\lambda}\big)^2} \, dz \\
&~ -\frac{\lambda \theta_\lambda^2}{2\pi} \int_{B_{\frac{d}{\theta_\lambda}}(0)} \log \big|\theta_\lambda z\big|
\Big(\gamma_{\lambda}+\frac{w_\lambda(z)}{\gamma_\lambda}\Big) e^{\big(\gamma_{\lambda}+\frac{w_\lambda(z)}
{\gamma_\lambda}\big)^2} \, dz \\
=&~ -\frac{1}{\gamma_\lambda}\int_{B_{\frac{d}{\theta_\lambda}}(0)} H \big(x_{\lambda}+\theta_{\lambda}z,x_{\lambda}\big)
\Big(1+\frac{w_\lambda(z)}{\gamma^2_\lambda}\Big) e^{2w_\lambda(z)+\frac{w^2_\lambda(z)}{\gamma^2_\lambda}} \, dz \\
&~ -\frac{1}{2\pi \gamma_\lambda}\int_{B_{\frac{d}{\theta_\lambda}}(0)} \log \big|\theta_\lambda z
\big|\Big(1+\frac{w_\lambda(z)}{\gamma^2_\lambda}\Big) e^{2w_\lambda(z)+\frac{w^2_\lambda(z)}{\gamma^2_\lambda}} \,dz.
\end{align*}
Next by \eqref{xlambda-x0}, \eqref{2.4.1} and the dominated convergence theorem, we have
\begin{equation*}
\begin{split}
&\int_{B_{\frac{d}{\theta_\lambda}}(0)}
         H\big(x_{\lambda}+\theta_{\lambda}z,x_{\lambda}\big)
         \Big(1+\frac{w_\lambda(z)}{\gamma^2_\lambda}\Big) e^{2w_\lambda(z)+\frac{w^2_\lambda(z)}{\gamma^2_\lambda}} \, dz
\\=& \mathcal{R}(x_0)
\int_{B_{\frac{d}{\theta_\lambda}}(0)}
         \Big(1+\frac{w_\lambda(z)}{\gamma^2_\lambda}\Big) e^{2w_\lambda(z)+\frac{w^2_\lambda(z)}{\gamma^2_\lambda}} \, dz+O\Big(\frac{\theta_\lambda}{\gamma_\lambda^2}\Big).
\end{split}
\end{equation*}
Also we recall from Proposition \ref{prop_wlambda} that
\begin{equation}\label{4.11-2}
w_\lambda=U+\frac{v_\lambda}{\gamma^2_\lambda}=U+\frac{v_0}
{\gamma^2_\lambda}+\frac{k_{\lambda}}{\gamma^4_{\lambda}} = U+\frac{v_0}
{\gamma^2_\lambda}+\frac{k_0}{\gamma^4_{\lambda}} +\frac{s_{\lambda}}{\gamma^6_{\lambda}}.
\end{equation}
By \eqref{4.11-2}, we derive
\begin{align}\label{4.11-3}
& \int_{B_{\frac{d}{\theta_\lambda}}(0)}\Big(1+\frac{w_\lambda(z)}{\gamma^2_\lambda}\Big) e^{2w_\lambda(z)+\frac{w^2_\lambda(z)}{\gamma^2_\lambda}} \, dz \notag\\
=& \int_{B_{\frac{d}{\theta_\lambda}}(0)} \Big(1+\frac{U(z)}{\gamma^2_\lambda}+\frac{v_0(z)}{\gamma^4_\lambda}
+\frac{k_0(z)}{\gamma^6_\lambda}+\frac{s_\lambda(z)}{\gamma^8_\lambda}\Big) e^{2\big(U+\frac{v_0}{\gamma^2_\lambda}+\frac{k_0}{\gamma^4_{\lambda}}+\frac{s_\lambda}{\gamma^6_\lambda}\big)
+\frac{1}{\gamma^2_\lambda}\big(U+\frac{v_0}{\gamma^2_\lambda}+\frac{k_0}{\gamma^4_{\lambda}}
+\frac{s_\lambda}{\gamma^6_\lambda}\big)^2} \, dz \notag\\
=& \int_{B_{\frac{d}{\theta_\lambda}}(0)} e^{2U(z)} \bigg(1+\frac{U(z)+U^2(z)+2v_0(z)}{\gamma^2_\lambda}+
\frac{2k_0(z) + v_0(z)+4U(z)v_0(z)+U^3(z)}{\gamma^4_\lambda} \notag\\
& +\frac{4v_0^2(z)+4v_0(z)U^2(z)+ U^4(z)}{2 \gamma^4_\lambda} + \frac{2s_\lambda(z)+3v_0^2(z)+4U(z)k_0(z)+4v_0(z)k_0(z)+6v_0^2(z)U(z)}{\gamma^6_\lambda} \notag\\
& +\frac{2k_0(z)U^2(z)+4U^3(z)v_0(z)+2v_0^2(z)U^2(z)+v_0(z)U^4(z)+3v_0(z)U^2(z)+k_0(z)}{\gamma^6_\lambda} \notag\\
& +\frac{8v_0^3(z)+U^6(z)+3U^5(z)}{6 \gamma^6_\lambda} \bigg) \,dz + o\Big(\frac{1}{\gamma^6_\lambda}\Big) \notag\\
=& 4\pi\Big(1+\frac{A_1}{\gamma^2_\lambda} +\frac{A_2}{\gamma^4_\lambda} +\frac{A_3}{\gamma^6_\lambda} \Big)
+o\Big(\frac{1}{\gamma^6_\lambda}\Big),
\end{align}
where
\begin{equation*}
A_1:=\frac{1}{4\pi}\int_{\R^2} e^{2U(z)}\Big(U(z)+U^2(z)+2v_0(z)\Big)   \,dz,
\end{equation*}
\begin{equation*}
A_2:=\frac{1}{4\pi}\int_{\R^2} e^{2U(z)}
\Big(  2k_0(z)+  v_0(z)+4U(z)v_0(z)+U^3(z) + 2v_0^2(z)+ 2v_0(z)U^2(z)+ \frac{1}{2} U^4(z) \Big) \,dz,
\end{equation*}
and
\begin{equation*}
\begin{split}
A_3:=& \frac{1}{4\pi}\int_{\R^2} e^{2U}\Big[2s_0(z)+k_0(z)+3v_0^2(z)+4U(z)k_0(z)+4v_0(z)k_0(z)+6v_0^2(z)U(z)
+2k_0(z)U^2(z) \\
& +4U^3(z)v_0(z)+2v_0^2(z)U^2(z)+v_0(z)U^4(z)+3v_0(z)U^2(z)+\frac{4}{3}v_0^3(z)+\frac{1}{6}U^6(z)+\frac{1}{2}U^5(z) \Big] \,dz.
\end{split}
\end{equation*}
Hence
\begin{equation*}
\begin{split}
\log \theta_\lambda \displaystyle\int_{B_{\frac{d}{\theta_\lambda}}(0)} & \Big(1+\frac{w_\lambda(z)}
{\gamma^2_\lambda}\Big) e^{2w_\lambda(z)+\frac{w^2_\lambda(z)}{\gamma^2_\lambda}} \, dz \\
& \overset{\eqref{4.11-3}} = -2\pi \Big(\log (\lambda\gamma_\lambda^2)+\gamma^2_\lambda\Big)
\left(1+\frac{A_1}{\gamma^2_\lambda}+\frac{A_2}{\gamma^4_\lambda} + \frac{A_3}{\gamma^6_\lambda}
+o\Big(\frac{1}{\gamma^6_\lambda}\Big)\right).
\end{split}
\end{equation*}
Similarly, it holds
\begin{equation*}
\begin{split}
& \int_{B_{\frac{d}{\theta_\lambda}}(0)} \log |z|\Big(1+\frac{w_\lambda(z)}{\gamma^2_\lambda}\Big) e^{2w_\lambda(z)+\frac{w^2_\lambda(z)}{\gamma^2_\lambda}} \,dz \\
=& \int_{B_{\frac{d}{\theta_\lambda}}(0)} \log |z| \Big(1+\frac{U(z)}{\gamma^2_\lambda}+\frac{v_0(z)}{\gamma^4_\lambda}
+\frac{k_\lambda(z)}{\gamma^6_\lambda}\Big) e^{2\big(U+\frac{v_0}{\gamma^2_\lambda}
+\frac{k_{\lambda}}{\gamma^4_{\lambda}}\big) + \frac{1}{\gamma^2_\lambda}\big(U+\frac{v_0}{\gamma^2_\lambda}
+\frac{k_{\lambda}}{\gamma^4_{\lambda}}\big)^2} \,dz \\
=& \int_{B_{\frac{d}{\theta_\lambda}}(0)} e^{2U(z)}\log|z| \bigg(1+\frac{U(z)+U^2(z)+2v_0(z)}{\gamma^2_\lambda}  +\frac{2k_\lambda(z) + v_0(z)+4U(z)v_0(z)+U^3(z)}{\gamma^4_\lambda} \\
& +\frac{4v_0^2(z)+4v_0(z)U^2(z)+ U^4(z)}{2 \gamma^4_\lambda} \bigg) \,dz + O\Big(\frac{1}{\gamma^6_\lambda}\Big) \\
=& 2\pi \Big(A_4+\frac{A_5}{\gamma^2_\lambda}+\frac{A_6}{\gamma^4_\lambda}\Big)+O\Big(\frac{1}{\gamma^6_\lambda}\Big),
\end{split}
\end{equation*}
where
\begin{eqnarray*}
A_4&:=&\frac{1}{2\pi}\int_{\R^2} \log |z|e^{2U(z)} \,dz,\\
A_5&:=&\frac{1}{2\pi}\int_{\R^2} \log |z|e^{2U(z)}\big(U(z)+U^2(z)+2v_0(z)\big) \,dz,
\end{eqnarray*}
and
\begin{equation*}
A_6:=\frac{1}{2\pi}\int_{\R^2} \log |z| e^{2U} \Big(  2k_0(z)+  v_0(z)+4U(z)v_0(z)+U^3(z) + 2v_0^2(z)+ 2v_0(z)U^2(z)+ \frac{1}{2} U^4(z) \Big) \,dz.
\end{equation*}
Hence  we have
\begin{equation}\label{4.11-4}
\begin{split}
& \lambda  \int_{B_d(x_{\lambda})} G(y,x_{\lambda}) u_{\lambda}(y)e^{u^2_{\lambda}(y)} \,dy \\
=& \gamma_\lambda +\frac{ \log (\lambda \gamma_\lambda^2)}{\gamma_\lambda} \left(1 +\frac{A_1}{\gamma^2_\lambda} +\frac{A_2}{\gamma^4_\lambda} + O\Big(\frac{1}{\gamma^6_\lambda}\Big)\right) +\frac{1}{\gamma_\lambda}\Big( A_1 -4\pi \mathcal{R}(x_0)- A_4 \Big) \\
& +\frac{1}{\gamma^3_\lambda}\Big( A_2 - 4\pi A_1 \mathcal{R}(x_0)- A_5 \Big) + \frac{1}{\gamma^5_\lambda}\Big( A_3 - 4\pi A_2 \mathcal{R}(x_0)- A_6 \Big) + o\Big(\frac{1}{\gamma^5_\lambda}\Big).
\end{split}
\end{equation}
Moreover, it holds
\begin{equation}\label{add_uxlambda}
\begin{split}
&\lambda \int_{\Omega\backslash B_{d}(x_{\lambda})}  G(y,x_{\lambda}) u_{\lambda}(y) e^{u^2_{\lambda}(y)}\,dy \\
=& O\bigg( \frac{1}{\gamma_\lambda} \int_{\Omega_\lambda \backslash B_{\frac{d}{\theta_\lambda}}(0)} \Big(1+\frac{w_\lambda(z)}{\gamma_\lambda^2}\Big) e^{2w_\lambda(z)+\frac{w^2_\lambda(z)}{\gamma^2_\lambda}} \,dz \bigg)
= O\Big(\frac{\theta_\lambda^{2-\delta}}{\gamma_\lambda}\Big).
\end{split}
\end{equation}
Substituting \eqref{4.11-4} and \eqref{add_uxlambda} into \eqref{4.11-1}, we find
\begin{equation*}
\begin{split}
& \log (\lambda \gamma_\lambda^2) \left(1 +\frac{A_1}{\gamma^2_\lambda} +\frac{A_2}{\gamma^4_\lambda} + O\Big(\frac{1}{\gamma^6_\lambda}\Big)\right) \\
=& -\big( A_1 -4\pi \mathcal{R}(x_0)- A_4 \big) -\frac{1}{\gamma^2_\lambda}\big( A_2 -4\pi A_1 \mathcal{R}(x_0)- A_5 \big) - \frac{1}{\gamma^4_\lambda}\big( A_3 - 4\pi A_2 \mathcal{R}(x_0)- A_6 \big) + o\Big(\frac{1}{\gamma^4_\lambda}\Big).
\end{split}
\end{equation*}
And we know that
\[
\left(1 +\frac{A_1}{\gamma^2_\lambda} + \frac{A_2}{\gamma^4_\lambda} + O\Big(\frac{1}{\gamma^6_\lambda}\Big) \right)^{-1}
= 1-\frac{A_1}{\gamma^2_\lambda}+\frac{A_1^2-A_2}{\gamma^4_\lambda}+ O\Big(\frac{1}{\gamma^6_\lambda}\Big),
\]
then it holds
\begin{equation*}
\begin{split}
\log (\lambda \gamma_\lambda^2)
=& -\big( A_1 -4\pi \mathcal{R}(x_0)- A_4 \big)-\frac{1}{\gamma^2_\lambda}\big( A_2 -4\pi A_1 \mathcal{R}(x_0)- A_5 \big)
-\frac{1}{\gamma^4_\lambda}\big( A_3 -4\pi A_2 \mathcal{R}(x_0)- A_6 \big) \\
& +\frac{A_1}{\gamma_\lambda^2}\big( A_1 -4\pi \mathcal{R}(x_0)- A_4 \big) +\frac{A_1}{\gamma^4_\lambda}\big(A_2 -4\pi A_1 \mathcal{R}(x_0)- A_5 \big) \\
& + \frac{A_2-A_1^2}{\gamma_\lambda^4}\big( A_1 -4\pi \mathcal{R}(x_0)-A_4 \big) + o\Big(\frac{1}{\gamma^4_\lambda}\Big) \\
=& -\big( A_1 -4\pi \mathcal{R}(x_0)- A_4 \big) + \frac{1}{\gamma^2_\lambda} \Big[ A_1 \big( A_1 -4\pi \mathcal{R}(x_0)- A_4 \big)- \big( A_2 -4\pi A_1 \mathcal{R}(x_0)- A_5 \big) \Big] \\
& +\frac{1}{\gamma^4_\lambda} \Big[ A_1 \big(A_2 -4\pi A_1 \mathcal{R}(x_0)-A_5\big)+\big(A_2-A_1^2\big)\big( A_1 -4\pi \mathcal{R}(x_0)-A_4\big) \\
& -\big( A_3 -4\pi A_2 \mathcal{R}(x_0)-A_6\big) \Big] + o\Big(\frac{1}{\gamma^4_\lambda}\Big).
\end{split}
\end{equation*}
Take
\begin{equation*}
\begin{split}
B_1:= -\big( A_1 -4\pi \mathcal{R}(x_0)- A_4\big),~~~~~
B_2:= A_1 \big( A_1 -4\pi \mathcal{R}(x_0)- A_4 \big)- \big( A_2 -4\pi A_1 \mathcal{R}(x_0)- A_5 \big),
\end{split}
\end{equation*}
and
\begin{equation*}
B_3:=  A_1 \big(A_2 -4\pi A_1 \mathcal{R}(x_0)-A_5\big) +\big(A_2-A_1^2\big)\big(A_1 -4\pi \mathcal{R}(x_0)-A_4\big) -\big( A_3 -4\pi A_2 \mathcal{R}(x_0)-A_6\big).
\end{equation*}
We have
\begin{equation*}
\begin{split}
\log (\sqrt{\lambda }\gamma_\lambda)
=& \frac{B_1}{2} +\frac{B_2}{2\gamma^2_\lambda} +\frac{B_3}{2\gamma^4_\lambda} +o\Big(\frac{1}{\gamma^4_\lambda}\Big),
\end{split}
\end{equation*}
which implies
\begin{align*}
\sqrt{\lambda }\gamma_\lambda
=&~ e^{\frac{B_1}{2}} \left( 1+\frac{B_2}{2\gamma^2_\lambda} +\frac{B_3}{2\gamma^4_\lambda} +\frac{B_2^2}{8\gamma^4_\lambda} +o\Big(\frac{1}{\gamma^4_\lambda}\Big) \right) \\
=&~ e^{\frac{B_1}{2}} + \frac{B_2 e^{\frac{B_1}{2}}}{2} \lambda \cdot e^{-B_1} \bigg( 1-\frac{B_2}{\gamma^2_\lambda} -\frac{B_3}{\gamma^4_\lambda} -\frac{B_2^2}{4\gamma^4_\lambda} +\frac{3B_2^2}{4\gamma^4_\lambda} +o\Big(\frac{1}{\gamma^4_\lambda}\Big) \bigg) \\
&~ +\frac{B_3 e^{\frac{B_1}{2}}}{2} \lambda^2 \cdot e^{-2B_1} \bigg( 1-\frac{2 B_2}{\gamma^2_\lambda} -\frac{2 B_3}{\gamma^4_\lambda} -\frac{B_2^2}{2\gamma^4_\lambda} +\frac{5 B_2^2}{2\gamma^4_\lambda}
+o\Big(\frac{1}{\gamma^4_\lambda}\Big) \bigg) \\
&~ +\frac{B_2^2 e^{\frac{B_1}{2}}}{8} \lambda^2 \cdot e^{-2B_1} \bigg( 1-\frac{2 B_2}{\gamma^2_\lambda} -\frac{2 B_3}{\gamma^4_\lambda} -\frac{B_2^2}{2\gamma^4_\lambda} +\frac{5 B_2^2}{2\gamma^4_\lambda} +o\Big(\frac{1}{\gamma^4_\lambda}\Big) \bigg) +o\big(\lambda^2\big) \\
=&~ e^{\frac{B_1}{2}} + \frac{B_2 e^{-\frac{B_1}{2}}}{2} \lambda - \frac{B_2^2 e^{-\frac{3 B_1}{2}}}{2} \lambda^2
+\frac{B_3 e^{-\frac{3 B_1}{2}}}{2} \lambda^2 +\frac{B_2^2 e^{-\frac{3B_1}{2}}}{8} \lambda^2 +o\big(\lambda^2\big) \\
=&~ e^{\frac{B_1}{2}} + \frac{B_2 e^{-\frac{B_1}{2}}}{2} \lambda +\frac{4 B_3-3 B_2^2}{8} e^{-\frac{3 B_1}{2}} \lambda^2 +o\big(\lambda^2\big),
\end{align*}
here we use the facts that
\[
e^t =1+t+\frac{t^2}{2}+O\big(|t|^3\big)~\,\mbox{and}~\,
\frac{1}{(1+t)^p} = 1-pt+\frac{p(p+1)}{2}t^2+O\big(|t|^3\big),~\,\mbox{for}\,~t\,~\mbox{small}.
\]
\end{proof}

\section{Uniqueness of the positive solutions}\label{s6}

\subsection{Regularization and some basic estimates}~

\vskip 0.1cm

Let $u_\lambda^{(1)}$ and $u_\lambda^{(2)}$ be two positive solutions to \eqref{1.1} with
\eqref{energy1}.
From Theorem A in the introduction, we know that the concentrated point is a critical point of the Robin function. Now let us assume that they concentrate at the same point $x_{0}$ and that $x_{0}$ is a non-degenerate critical point of Robin function $\mathcal{R}(x)$. We denote  $x^{(l)}_{\lambda}$ for $l=1,2$, the points such that
 \begin{equation*}
\gamma_\lambda^{(l)}:=u^{(l)}_{\lambda}\big(x^{(l)}_{\lambda}\big)=\max_{\overline{B_{2r}(x_{0})}}u^{(l)}_{\lambda}(x)\,\,~\mbox{for some small fixed}~r>0,
\end{equation*}
and
$$\theta^{(l)}_{\lambda}:=\Big(\lambda \big(\gamma_\lambda^{(l)}\big)^2e^{(\gamma_\lambda^{(l)})^2}\Big)^{-1/2}, ~~\mbox{for} ~~l=1,2.$$
Firstly, we have some basic results on $\frac{\theta^{(1)}_{\lambda}}{\theta^{(2)}_{\lambda}}$ and $\big| x^{(1)}_{\lambda}-x^{(2)}_{\lambda} \big|$.
\begin{Lem}\label{theta12}
It holds
\begin{equation*}\label{llsb}
\frac{\theta^{(1)}_{\lambda}}{\theta^{(2)}_{\lambda}}=1+o\big(\lambda\big).
\end{equation*}
\end{Lem}
\begin{proof}
Firstly, we have
\begin{equation}\label{5.1-1}
\frac{\theta^{(1)}_{\lambda}}{\theta^{(2)}_{\lambda}}=
\frac{\gamma^{(2)}_\lambda}{\gamma^{(1)}_\lambda} e^{-\frac{1}{2}\big((\gamma^{(1)}_\lambda)^2-(\gamma^{(2)}_\lambda)^2\big)}.
\end{equation}
From \eqref{lambda_gamma}, we find
\begin{equation}\label{5.1-2}
\frac{\gamma^{(2)}_\lambda}{\gamma^{(1)}_\lambda} =1+o\big(\lambda^2\big),
\end{equation}
and
\begin{equation}\label{5.1-3}
\begin{split}
\big(\gamma^{(1)}_\lambda\big)^2-\big(\gamma^{(2)}_\lambda\big)^2
&= \big(\gamma^{(1)}_\lambda+\gamma^{(2)}_\lambda\big)\cdot \big(\gamma^{(1)}_\lambda-\gamma^{(2)}_\lambda\big) \\
&= \frac{2}{\sqrt{\lambda}} \Big(e^{\frac{B_1}{2}} + \frac{B_2 e^{-\frac{B_1}{2}}}{2} \lambda + \frac{4B_3-3B_2^2}{8} e^{-\frac{3 B_1}{2}} \lambda^2 +o\big(\lambda^2\big)\Big)  o\big(\lambda^\frac{3}{2}\big)=o\big(\lambda \big).
\end{split}
\end{equation}
Hence we can deduce from \eqref{5.1-1}, \eqref{5.1-2} and \eqref{5.1-3} that
\begin{equation*}
\frac{\theta^{(1)}_{\lambda}}{\theta^{(2)}_{\lambda}}= \Big(1+o\big(\lambda^2\big)\Big)  \big(1+o(\lambda)\big)=1+o\big(\lambda\big).
\end{equation*}
\end{proof}

\begin{Lem}\label{xlambda12}
If $x_{0}$ is a nondegenerate critical point of $\mathcal{R}(x)$, then it holds
\begin{equation*}
\Big|\frac{ x^{(1)}_{\lambda}-x^{(2)}_{\lambda}}{ \theta^{(1)}_{\lambda} }\Big|=o\big(\lambda\big).
\end{equation*}
\end{Lem}

\begin{proof}
Recall that
\[
v_0^{(l)}(x) = w^{(l)}(x) + c_0^{(l)} \frac{4-|x|^2}{4+|x|^2} + \sum^2_{h=1} c_h^{(l)} \frac{x_h}{4+|x|^2},~\mbox{for}~l=1,2,
\]
where $w^{(1)}(x)$ and $w^{(2)}(x)$ are both radial solutions of equation
\[
-\Delta v-2e^{2U}v = e^{2U}\big(U^2+U\big).
\]
By \eqref{4.10-3}, we know that for $l=1,2$,
\begin{equation*}
\begin{split}
\Big\langle \nabla \frac{\partial \mathcal{R}(x_0)}{\partial x_i}, x^{(l)}_\lambda-x_0 \Big\rangle
= -\sum_{h=1}^2 \frac{c^{(l)}_h}{2 \big(\gamma_\lambda^{(l)}\big)^2} \theta_\lambda^{(l)} \frac{\partial^2 \mathcal{R}(x_0)}{\partial x_i \partial x_h} + o\bigg(\frac{\theta_\lambda^{(l)}}{\big(\gamma_\lambda^{(l)}\big)^2}\bigg).
\end{split}
\end{equation*}
Then we can deduce from \eqref{lambda_gamma} and Lemma \ref{theta12} that
\begin{equation}\label{5.2-1}
\begin{split}
\Big\langle \nabla \frac{\partial \mathcal{R}(x_0)}{\partial x_i}, x^{(1)}_\lambda-x^{(2)}_\lambda \Big\rangle
&= -\sum_{h=1}^2 \frac{c^{(1)}_h}{2 \big(\gamma_\lambda^{(1)}\big)^2} \theta_\lambda^{(1)} \frac{\partial^2 \mathcal{R}(x_0)}{\partial x_i \partial x_h} + \sum_{h=1}^2 \frac{c^{(2)}_h}{2 \big(\gamma_\lambda^{(2)}\big)^2} \theta_\lambda^{(2)} \frac{\partial^2 \mathcal{R}(x_0)}{\partial x_i \partial x_h} + o\Big( \theta_\lambda^{(1)} \lambda \Big) \\
&= \sum_{h=1}^2 \frac{c^{(2)}_h-c^{(1)}_h}{2 \big(\gamma_\lambda^{(1)}\big)^2} \theta_\lambda^{(1)} \frac{\partial^2 \mathcal{R}(x_0)}{\partial x_i \partial x_h} + o\Big( \theta_\lambda^{(1)} \lambda \Big).
\end{split}
\end{equation}
Furthermore, by the definition of $v_\lambda^{(l)}$, we find that
\begin{equation*}
\begin{split}
& \nabla v_\lambda^{(1)}(0) - \nabla v_\lambda^{(2)}(0) \\ =& \big(\gamma_\lambda^{(1)}\big)^2 \nabla w_\lambda^{(1)}(x) \Big|_{x=0} - \big(\gamma_\lambda^{(2)}\big)^2 \nabla w_\lambda^{(2)}(x)  \Big|_{x=0} \\
=& \big(\gamma_\lambda^{(1)}\big)^3 \theta_\lambda^{(1)} \nabla u_\lambda^{(1)}\big(\theta_\lambda^{(1)} x + x_\lambda^{(1)} \big) \Big|_{x=0} - \big(\gamma_\lambda^{(2)}\big)^3 \theta_\lambda^{(2)} \nabla u_\lambda^{(2)}
\big(\theta_\lambda^{(2)} x + x_\lambda^{(2)} \big) \Big|_{x=0} = 0,
\end{split}
\end{equation*}
since $u_\lambda^{(l)}(x)$ achieves its maximum at $x_\lambda^{(l)}$ for $l=1,2$.
Then \eqref{lim-v} gives
\begin{equation}\label{grav1-grav1}
\nabla v_0^{(1)}(0) - \nabla v_0^{(2)}(0) = 0.
\end{equation}
On the other hand, since $\big(w^{(1)}-w^{(2)}\big)(x) $ is a radial function satisfying $-\Delta v-2e^{2U}v =0$ and $v_0^{(l)}$ satisfies \eqref{v0-bound}, we have
\begin{equation*}
\big(w^{(1)}-w^{(2)}\big)(x) = c'\frac{4-|x|^2}{4+|x|^2},
\end{equation*}
which implies
\begin{equation}\label{v01-v02}
v_0^{(1)}(x)-v_0^{(2)}(x) = \big(c' + c_0^{(1)} - c_0^{(2)}\big) \frac{4-|x|^2}{4+|x|^2} +
\sum^2_{h=1} \big(c_h^{(1)}-c_h^{(2)}\big) \frac{x_h}{4+|x|^2}.
\end{equation}
Hence, it follows from \eqref{grav1-grav1} and \eqref{v01-v02} that
$c^{(2)}_h-c^{(1)}_h = 0$ for $h=1,2$. Substituting it into \eqref{5.2-1}, since $x_{0}$ is a nondegenerate critical point of $\mathcal{R}(x)$, we obtain
\[
\big|x_\lambda^{(1)} - x_\lambda^{(2)}\big| = o\Big( \theta_\lambda^{(1)} \lambda \Big).
\]
\end{proof}

In the following, we will consider the same quadratic forms already introduced in \eqref{P} and \eqref{Q}.
\begin{equation}\label{p1uv}
\begin{split}
P^{(1)}(u,v):=&- 2d\int_{\partial B_d(x^{(1)}_{\lambda})} \big\langle \nabla u ,\nu\big\rangle \big\langle \nabla v,\nu\big\rangle \,d\sigma + d  \int_{\partial B_d(x^{(1)}_{\lambda})} \big\langle \nabla u , \nabla v \big\rangle \,d\sigma,
\end{split}
\end{equation}
and
\begin{equation}\label{q1uv}
Q^{(1)}(u,v):= - \int_{\partial B_d(x^{(1)}_{\lambda})}\frac{\partial v}{\partial \nu}\frac{\partial u}{\partial x_i} \,d\sigma -\int_{\partial B_d(x^{(1)}_{\lambda})}\frac{\partial u}{\partial \nu}\frac{\partial v}{\partial x_i} \,d\sigma + \int_{\partial B_d(x^{(1)}_{\lambda})}\big\langle \nabla u,\nabla v \big\rangle \nu_i \,d\sigma.
\end{equation}
Note that if $u$ and $v$ are harmonic in $ B_d(x^{(1)}_{\lambda})\backslash \{x^{(1)}_{\lambda}\}$, then similar to Lemma \ref{indep_d}, we deduce that $P^{(1)}(u,v)$ and  $Q^{(1)}(u,v)$ are independent of $\theta\in (0,d]$.
Moreover, replacing $x_{\lambda}$ by $x^{(1)}_{\lambda}$ in Proposition \ref{prop2-1}, we have the following computations.
\begin{Prop}\label{prop_p1q1}
It holds
\begin{equation}\label{p1_g}
P^{(1)}\Big(G(x^{(1)}_{\lambda},x), G(x^{(1)}_{\lambda},x)\Big)=
-\frac{1}{2\pi},
\end{equation}
 \begin{equation}\label{p1_partg}
P^{(1)}\Big(G(x^{(1)}_{\lambda},x),\partial_hG(x^{(1)}_{\lambda},x)\Big)=
-\frac{1}{2}\frac{\partial \mathcal{R}(x^{(1)}_{\lambda})}{\partial {x_h}},
\end{equation}

\begin{equation}\label{q1_g}
Q^{(1)} \Big(G(x^{(1)}_{\lambda},x),G(x^{(1)}_{\lambda},x)\Big)=
-\frac{\partial \mathcal{R}(x^{(1)}_{\lambda})}{\partial{x_i}},
\end{equation}
and
\begin{equation}\label{q1_partg}
Q^{(1)}\Big(G(x^{(1)}_{\lambda},x),\partial_h G(x^{(1)}_{\lambda},x)\Big)=
- \frac{1}{2} \frac{\partial^2 \mathcal{R}(x^{(1)}_{\lambda})}{\partial{x_i}\partial{x_h}},
\end{equation}
where $G(x,y)$ and $\mathcal{R}(x)$ are the Green and Robin functions respectively (see \eqref{greensyst}, \eqref{GreenS-H}, \eqref{Robinf}), $\partial_h G(y,x):=\frac{\partial G(y,x)}{\partial y_h}$.
\end{Prop}
Furthermore, the following results can be derived from the above proposition.
\begin{Prop}
For some small fixed $\delta>0$, it holds
\begin{equation}\label{p1_g2}
P^{(1)}\Big(G(x^{(2)}_{\lambda},x), G(x^{(1)}_{\lambda},x)\Big)=
-\frac{1}{2\pi}+ o\big( \lambda \theta^{(1)}_{\lambda}  \big),
\end{equation}
 \begin{equation}\label{p1_partg2}
P^{(1)}\Big(G(x^{(2)}_{\lambda},x),\partial_hG(x^{(1)}_{\lambda},x)\Big)=
 O\big( \lambda \theta^{(1)}_{\lambda}  \big),
\end{equation}

\begin{equation}\label{q1_g2}
Q^{(1)} \Big(G(x^{(2)}_{\lambda},x),G(x^{(1)}_{\lambda},x)\Big)=
  O\big( \lambda \theta^{(1)}_{\lambda}  \big),
\end{equation}
and
\begin{equation}\label{q1_partg2}
Q^{(1)} \Big(G(x^{(2)}_{\lambda},x),\partial_h G(x^{(1)}_{\lambda},x)\Big) =
-\frac{1}{2} \frac{\partial^2 \mathcal{R}(x_{0})}{\partial{x_i \partial x_h}}+ O\big( \lambda \theta^{(1)}_{\lambda} \big).
\end{equation}
\end{Prop}
\begin{proof}
It follows from Lemma \ref{xlambda12} that
\begin{equation*}
G\big(x^{(2)}_{\lambda},x\big) = G\big(x^{(1)}_{\lambda},x\big)+O\big(|x^{(2)}_{\lambda}-x^{(1)}_{\lambda}|\big) = G\big(x^{(1)}_{\lambda},x\big)+o\big(\lambda \theta^{(1)}_\lambda\big),
\end{equation*}
which gives
\begin{equation*}
P^{(1)}\Big(G(x^{(2)}_{\lambda},x), G(x^{(1)}_{\lambda},x)\Big) = P^{(1)}\Big(G(x^{(1)}_{\lambda},x), G(x^{(1)}_{\lambda},x)\Big) + o\big(\lambda \theta^{(1)}_\lambda\big)
\overset{\eqref{p1_g}}= -\frac{1}{2\pi}+ o\big(\lambda\theta^{(1)}_{\lambda}\big).
\end{equation*}
Similarly, \eqref{p1_partg2}, \eqref{q1_g2} and \eqref{q1_partg2} can be deduced from \eqref{xlambda-x0}, Lemma \ref{xlambda12} and Proposition \ref{prop_p1q1}.
\end{proof}

\vskip 0.4cm

Now if $u_{\lambda}^{(1)}\not \equiv u_{\lambda}^{(2)}$ in $\Omega$, we set
\begin{equation}\label{eta-def}
\eta_{\lambda}:=\frac{u_{\lambda}^{(1)}-u_{\lambda}^{(2)}}
{\|u_{\lambda}^{(1)}-u_{\lambda}^{(2)}\|_{L^{\infty}(\Omega)}},
\end{equation}
then $\eta_{\lambda}$ satisfies $\|\eta_{\lambda}\|_{L^{\infty}(\Omega)}=1$ and
\begin{equation}\label{eta-equa}
- \Delta \eta_{\lambda}=-\frac{\Delta u_{\lambda}^{(1)}-\Delta u_{\lambda}^{(2)}}
{\|u_{\lambda}^{(1)}-u_{\lambda}^{(2)}\|_{L^{\infty}(\Omega)}}= \frac{\lambda \Big(u_{\lambda}^{(1)}e^{(u_{\lambda}^{(1)})^2} -  u_{\lambda}^{(2)}e^{(u_{\lambda}^{(2)})^2}\Big)}
{\|u_{\lambda}^{(1)}-u_{\lambda}^{(2)}\|_{L^{\infty}(\Omega)}}=E_{\lambda}\eta_{\lambda},
\end{equation}
where
\begin{equation}\label{Dlambda-def}
E_{\lambda}(x):=\lambda e^{(u_{\lambda}^{(1)})^2}+2 \lambda  u_{\lambda}^{(2)}  \displaystyle\int_{0}^1
F_t(x) e^{ (F_t(x))^2} \,dt,
\end{equation}
with $F_t(x):=tu_{\lambda}^{(1)}(x)+(1-t)u_{\lambda}^{(2)}(x)$.

\begin{Lem}\label{theta-D}
It holds
\begin{equation}\label{Dlambda-equa}
\big(\theta^{(1)}_{\lambda}\big)^{2}E_{\lambda}\big(\theta^{(1)}_{\lambda}x+x^{(1)}_{\lambda}\big)  =
2e^{2U(x)} \big(1+ O(\lambda) \big),
\end{equation}
uniformly on compact sets, in particular
\begin{equation}\label{Dlambda-lim}
\big(\theta^{(1)}_{\lambda}\big)^{2}E_{\lambda}\big(\theta^{(1)}_{\lambda}x+x^{(1)}_{\lambda}\big)  \rightarrow
2e^{2U(x)} ~\, ~\mbox{in}~C_{loc}\big(\R^2\big).
\end{equation}
Also for some small fixed $\delta,d>0$, it follows
\begin{equation}\label{Dlambda-appro}
\big(\theta^{(1)}_{\lambda}\big)^{2}E_{\lambda}\big(\theta^{(1)}_{\lambda}x+x^{(1)}_{\lambda}\big) =O\Big(\frac{1}{1+|x|^{4-\delta}}\Big)~\,\,~\mbox{for}~~ x\in \Omega_\lambda.
\end{equation}
\end{Lem}
\begin{proof}
Firstly, we have
\begin{equation}\label{5.5-1}
\begin{split}
\big(\theta^{(1)}_{\lambda}\big)^{2}E_{\lambda}\big(\theta^{(1)}_{\lambda}x+x^{(1)}_{\lambda}\big)  =&
2 \lambda \big(\theta_\lambda^{(1)}\big)^2 u_{\lambda}^{(2)}\big(\theta^{(1)}_{\lambda}x+x^{(1)}_{\lambda}\big ) \displaystyle\int_{0}^1
F_t\big(\theta^{(1)}_{\lambda}x+x^{(1)}_{\lambda}\big )e^{ \big(F_t(\theta^{(1)}_{\lambda}x+x^{(1)}_{\lambda} )\big)^2}dt\\&
+\lambda  \big(\theta^{(1)}_{\lambda}\big)^{2} e^{\big(u_{\lambda}^{(1)}(\theta^{(1)}_{\lambda}x+x^{(1)}_{\lambda} )\big)^2},
\end{split}
\end{equation}
and
\begin{equation}\label{5.5-2}
\begin{split}
F_t\big(\theta^{(1)}_{\lambda}x+x^{(1)}_{\lambda}\big ) &= tu_\lambda^{(1)}
\big(\theta^{(1)}_{\lambda}x+x^{(1)}_{\lambda}\big )+(1-t) u_\lambda^{(2)}\big(\theta^{(1)}_{\lambda}x+x^{(1)}_{\lambda}\big )\\
&= t\Big(\gamma^{(1)}_\lambda+\frac{w^{(1)}_\lambda(x)}{\gamma^{(1)}_\lambda}\Big)
+\big(1-t\big) \bigg(\gamma^{(2)}_\lambda+\frac{w^{(2)}_\lambda\Big(\frac{ \theta^{(1)}_\lambda x}
{\theta^{(2)}_\lambda}+\frac{x^{(1)}_\lambda-x^{(2)}_\lambda}{\theta^{(2)}_\lambda}\Big)}
{\gamma^{(2)}_\lambda}\bigg) \\
&= \Big(\gamma^{(1)}_\lambda+\frac{w^{(1)}_\lambda(x)}{\gamma^{(1)}_\lambda}\Big)+\big(1-t\big) f_\lambda(x),
\end{split}
\end{equation}
where $$f_\lambda(x):=\gamma^{(2)}_\lambda-\gamma^{(1)}_\lambda+\frac{w^{(2)}_\lambda\Big(\frac{ \theta^{(1)}_\lambda x}
{\theta^{(2)}_\lambda}+\frac{x^{(1)}_\lambda-x^{(2)}_\lambda}{\theta^{(2)}_\lambda}\Big)}{\gamma^{(2)}_\lambda}
-\frac{w^{(1)}_\lambda\big(x\big)} {\gamma^{(1)}_\lambda}.$$
Moreover, direct computation gives
\begin{equation*}
\begin{split}
f_\lambda(x)
&= o\big(\lambda^{\frac{3}{2}}\big)+ \frac{w_\lambda^{(2)}(x)+o(\lambda)+O\Big(\big|\frac{x_\lambda^{(1)}-x_\lambda^{(2)}}{\theta_\lambda^{(2)}}\big|\Big)}
{\gamma_\lambda^{(2)}} -\frac{w^{(1)}_\lambda\big(x\big)} {\gamma^{(1)}_\lambda}  = o\big(\lambda^{\frac{3}{2}}\big)+O\left(\frac{\big|w^{(2)}_\lambda(x)-w^{(1)}_\lambda(x)\big|}
{\gamma^{(1)}_\lambda}\right).
\end{split}
\end{equation*}
Since $v_0^{(1)}$ and $v_0^{(2)}$ are the solutions of
\[
-\Delta u - 2e^{2U}u = e^{2U}\big(U^2+U\big),
\]
we have
\[
v_0^{(1)}(x)-v_0^{(2)}(x)=\sum_{i=1}^2 e_i \frac{x_i}{4+|x|^2} + e_0 \frac{4-|x|^2}{4+|x|^2}.
\]
Note that
\begin{equation*}
\begin{split}
v_\lambda^{(1)}(0)-v_\lambda^{(2)}(0) =& \big(\gamma_\lambda^{(1)}\big)^2 \big(w_\lambda^{(1)}(x)-U(x)\big)\Big|_{x=0}
-\big(\gamma_\lambda^{(2)}\big)^2 \big(w_\lambda^{(2)}(x)-U(x)\big)\Big|_{x=0} \\
=& \big(\gamma_\lambda^{(1)}\big)^3 \Big(u_\lambda^{(1)}\big(\theta_\lambda^{(1)}x+x_\lambda^{(1)}\big)
-\gamma_\lambda^{(1)}\Big)\bigg|_{x=0} - \big(\gamma_\lambda^{(1)}\big)^2 U(0) \\
& -\big(\gamma_\lambda^{(2)}\big)^3 \Big(u_\lambda^{(2)}\big(\theta_\lambda^{(2)}x+x_\lambda^{(2)}\big)
-\gamma_\lambda^{(2)}\Big)\bigg|_{x=0} +\big(\gamma_\lambda^{(2)}\big)^2 U(0) = 0,
\end{split}
\end{equation*}
then \eqref{lim-v} gives
$v_0^{(1)}(0)-v_0^{(2)}(0)=0$,
which implies that $e_0=0$. Also \eqref{grav1-grav1} implies $e_i=0$ for $i=1,2$. Therefore, we derive $v_0^{(1)}(x)-v_0^{(2)}(x)=0$. Thus, \eqref{v-def} and \eqref{lim-v} give that
\begin{equation}\label{5.5-3}
\begin{split}
f_\lambda(x) &= o\big(\lambda^{\frac{3}{2}}\big)+O\left(\frac{\big|w^{(2)}_\lambda(x)-w^{(1)}_\lambda(x)\big|}
{\gamma^{(1)}_\lambda}\right)\\
&= o\big(\lambda^{\frac{3}{2}}\big)+O\left(\frac{\big|v^{(2)}_0(x)-v^{(1)}_0(x)\big|}
{\big(\gamma^{(1)}_\lambda\big)^3}\right) = o\big(\lambda^{\frac{3}{2}}\big).
\end{split}
\end{equation}
Hence we can deduce from \eqref{5.5-2} and \eqref{5.5-3} that
\begin{equation*}
\begin{split}
e^{\big( F_t (\theta^{(1)}_{\lambda}x+x^{(1)}_{\lambda}) \big)^2}
&= e^{\big(\gamma^{(1)}_\lambda+\frac{w^{(1)}_\lambda(x)}{\gamma^{(1)}_\lambda}\big)^2
+(1-t)^2 f^2_\lambda(x)+2\big(\gamma^{(1)}_\lambda+\frac{w^{(1)}_\lambda(x)}{\gamma^{(1)}_\lambda}\big) (1-t)f_\lambda(x)} \\
&= e^{\big(\gamma^{(1)}_\lambda+\frac{w^{(1)}_\lambda(x)}{\gamma^{(1)}_\lambda}\big)^2} \big(1+o(\lambda)\big),
\end{split}
\end{equation*}
which implies
\begin{equation}\label{5.5-4}
\begin{split}
& \big(\theta_\lambda^{(1)}\big)^2 \int^1_0 F_t\big(\theta^{(1)}_{\lambda}x+x^{(1)}_{\lambda}\big ) e^{\big( F_t(\theta^{(1)}_{\lambda}x+x^{(1)}_{\lambda}) \big)^2} \,dt \\
=& \big(\theta_\lambda^{(1)}\big)^2 e^{\big(\gamma^{(1)}_\lambda+\frac{w^{(1)}_\lambda(x)}{\gamma^{(1)}_\lambda}\big)^2} \big(1+o(\lambda)\big)  \int^1_0 \Big(\gamma^{(1)}_\lambda+\frac{w^{(1)}_\lambda(x)}{\gamma^{(1)}_\lambda}\Big)
+\big(1-t\big) f_\lambda(x) \,dt  \\
=& \frac{1}{\lambda \gamma_\lambda^{(1)}} e^{2w^{(1)}_\lambda(x)+\frac{(w^{(1)}_\lambda(x))^2}{( \gamma^{(1)}_\lambda)^2} }\big(1+o(\lambda)\big)\Big(1+\frac{ w^{(1)}_\lambda(x) }{\big(\gamma^{(1)}_\lambda\big)^2}
+o\big(\lambda^2\big)\Big).
\end{split}
\end{equation}
Letting  $t=0$ in \eqref{5.5-2}, we have
\begin{equation}\label{5.5-5}
\begin{split}
u_{\lambda}^{(2)}\big(\theta^{(1)}_{\lambda}x+x^{(1)}_{\lambda}\big )
=\gamma^{(1)}_\lambda+\frac{w^{(1)}_\lambda\big(x\big)}{\gamma^{(1)}_\lambda}+f_\lambda(x).
\end{split}
\end{equation}
Hence, by \eqref{v-def}, \eqref{v-bound}, \eqref{5.5-4} and \eqref{5.5-5}, we get
\begin{align}\label{5.5-6}
&2 \lambda \big(\theta_\lambda^{(1)}\big)^2 u_{\lambda}^{(2)}\big(\theta^{(1)}_{\lambda}x+x^{(1)}_{\lambda}\big )  \displaystyle\int_{0}^1 F_t\big(\theta^{(1)}_{\lambda}x+x^{(1)}_{\lambda}\big )
e^{ \big(F_t(\theta^{(1)}_{\lambda}x+x^{(1)}_{\lambda})\big)^2}dt \notag\\
=&~ \frac{2}{\gamma_\lambda^{(1)}} e^{2w^{(1)}_\lambda(x)+\frac{(w^{(1)}_\lambda(x))^2}{( \gamma^{(1)}_\lambda)^2} }
\big(1+o(\lambda)\big)\Big(1+\frac{ w^{(1)}_\lambda(x) }{\big( \gamma^{(1)}_\lambda\big)^2}+o\big(\lambda^2
\big)\Big)  \Big(\gamma^{(1)}_\lambda+\frac{w^{(1)}_\lambda\big(x\big)}{\gamma^{(1)}_\lambda}+f_\lambda(x)\Big) \notag\\
=&~ 2e^{2w^{(1)}_\lambda(x)+\frac{(w^{(1)}_\lambda(x))^2}{( \gamma^{(1)}_\lambda)^2} }
\big(1+o(\lambda)\big)  \Big(1+\frac{ w^{(1)}_\lambda(x) }{\big( \gamma^{(1)}_\lambda\big)^2}+o\big(\lambda^2
\big)\Big)  \Big(1+\frac{ w^{(1)}_\lambda(x) }{\big( \gamma^{(1)}_\lambda\big)^2}+o\big(\lambda^2
\big)\Big) \notag\\
=&~ 2e^{2U(x)} \big(1+ O(\lambda)\big)  \big(1+o(\lambda)\big)  \Big(1+\frac{ w^{(1)}_\lambda(x) }{\big( \gamma^{(1)}_\lambda\big)^2}+o\big(\lambda^2\big)\Big)^2 \notag\\
=&~ 2e^{2U(x)} \big(1+ O(\lambda ) \big).
\end{align}
Next we can find that
\begin{equation}\label{5.5-7}
\begin{split}
&\lambda \big(\theta^{(1)}_{\lambda}\big)^{2} e^{\big(u_{\lambda}^{(1)}(\theta^{(1)}_{\lambda}x+x^{(1)}_{\lambda})
\big)^2} \\
=& \frac{1}{\big(\gamma_\lambda^{(1)}\big)^2} e^{ 2w_\lambda^{(1)}(x)+\frac{(w_\lambda^{(1)}(x))^2}
{(\gamma_\lambda^{(1)})^2} }
= O\bigg(\frac{1}{\big(\gamma_\lambda^{(1)}\big)^2} e^{2U(x)} \bigg)=O\big(\lambda e^{2U(x) } \big).
\end{split}
\end{equation}
Then, substitute \eqref{5.5-6} and \eqref{5.5-7} into \eqref{5.5-1}, we derive \eqref{Dlambda-equa}. Also \eqref{Dlambda-lim} and \eqref{Dlambda-appro} can be obtained from the above computations and Lemma \ref{lem2.4}.
\end{proof}

Now using blow up analysis, we establish an estimate on $\eta_\lambda$.
\begin{Prop}\label{prop_tildeeta}
Let $\widetilde{\eta}_{\lambda}(x):=\eta_{\lambda}\big(\theta^{(1)}_{\lambda}x+x^{(1)}_{\lambda}\big)$, where $\eta_{\lambda}$ is defined in \eqref{eta-def}. Then by taking
a subsequence if necessary, we have
\begin{equation}\label{tildeeta-lim}
\widetilde{\eta}_{\lambda}(x) \to \widetilde{a} \frac{4-|x|^2}{4+|x|^2}+\sum_{i=1}^2\frac{\widetilde{b}_{i}x_i}{4+|x|^2}
~~~\mbox{in}~~~C^1_{loc}\big(\R^2\big),~\mbox{as}~\lambda \to 0,
\end{equation}
where $\widetilde{a}$ and $\widetilde{b}_{i}$ with $i=1,2$ are some constants.
\end{Prop}
\begin{proof}
Since $|\widetilde{\eta}_{\lambda}|\leq 1$,
by Lemma \ref{theta-D} and the standard elliptic regularity theory, we find that
\[
\widetilde{\eta}_{\lambda}(x) \in C^{1,\alpha}\big(B_R(0)\big)~\mbox{and}~
\|\widetilde{\eta}_{\lambda}\|_{C^{1,\alpha}(B_R(0))} \leq C,
\]
for any fixed large $R$ and some $\alpha \in (0,1)$, where $C$ is independent of $\lambda$. Then there exists a subsequence (still denoted by $\widetilde{\eta}_{\lambda}$) such that
$$\widetilde{\eta}_{\lambda}(x)\to \eta_0(x)~~\mbox{in}~C^1\big(B_R(0)\big).$$
By \eqref{eta-equa}, it is easy to check that $\widetilde{\eta}_{\lambda}$ satisfies
\begin{equation*}
-\Delta \widetilde{\eta}_{\lambda} =\big(\theta^{(1)}_{\lambda}\big)^{2}E_{\lambda}\big(\theta^{(1)}_{\lambda}x+x^{(1)}_{\lambda}\big) \widetilde{\eta}_{\lambda} \,\,\,~\mbox{in}~\Omega_\lambda.
\end{equation*}
Then by Lemma \ref{theta-D},
letting $\lambda \to 0$, we find that $\eta_0$ satisfies
\begin{equation*}
-\Delta \eta_0 = 2e^{2U} \eta_0~~~\mbox{in}~~~\R^2.
\end{equation*}
Hence, by Lemma \ref{lem3.1}, we have
\[
\eta_0(x)= \widetilde{a} \frac{4-|x|^2}{4+|x|^2}+\sum_{i=1}^2\frac{\widetilde{b}_{i}x_i}{4+|x|^2},
\]
where $\widetilde{a}$ and $\widetilde{b}_i$ with $i=1,2$ are some constants.

\end{proof}

Similar to the estimate of $u_\lambda$, we have a pointwise estimate away from the maximum point $x^{(1)}_\lambda$.
\begin{Prop}\label{prop_eta_lambda}
It holds
\begin{equation}\label{etalambda-equa}
\begin{split}
\eta_{\lambda}(x)=& \widetilde{A}_{\lambda}G\big(x^{(1)}_{\lambda},x\big)+
\sum^2_{i=1}\widetilde{B}_{\lambda,i}\theta^{(1)}_{\lambda}  \partial_i G\big(x^{(1)}_{\lambda},x\big)
 +o\big(\theta^{(1)}_{\lambda}\big) \,~\mbox{in}~ C^1
\Big(\Omega\backslash  B_{2d}\big(x^{(1)}_{\lambda}\big)\Big),
\end{split}
\end{equation}
where $d>0$ is a small fixed constant,
$\partial_iG(y,x)=\frac{\partial G(y,x)}{\partial y_i}$,
\begin{equation}\label{tildeA-equa}
\widetilde{A}_{\lambda}:=
\int_{B_d(x^{(1)}_{\lambda})}  E_{\lambda}(x)\eta_{\lambda}(x) \,dx,
\end{equation}
and
\begin{equation}\label{tildeB-equa}
\widetilde{B}_{\lambda,i}:= \frac{1}{\theta^{(1)}_{\lambda}}\int_{B_d(x^{(1)}_{\lambda})} \big(x_i-(x^{(1)}_{\lambda})_{i}\big)  E_{\lambda}(x)\eta_{\lambda}(x) \,dx.
\end{equation}
\end{Prop}
\begin{proof}
By \eqref{eta-equa} and the Green's representation theorem, we have
\begin{align}\label{5.7-1}
\eta_\lambda(x) = & \int_{\Omega} G(y,x) E_\lambda(y) \eta_\lambda(y) \,dy \notag\\
=& \int_{B_{d}(x_\lambda^{(1)})} G(y,x) E_\lambda(y) \eta_\lambda(y) \,dy +\int_{\Omega\backslash B_{d}(x_\lambda^{(1)})}
G(y,x) E_\lambda(y) \eta_\lambda(y) \,dy \notag\\
=& G(x_\lambda^{(1)},x) \int_{B_{d}(x_\lambda^{(1)})} E_\lambda(y) \eta_\lambda(y) \,dy + \int_{B_{d}(x_\lambda^{(1)})} \big( G(y,x)-G(x_\lambda^{(1)},x) \big) E_\lambda(y) \eta_\lambda(y) \,dy \notag\\
&+ \int_{\Omega\backslash B_{d}(x_\lambda^{(1)})} G(y,x) E_\lambda(y) \eta_\lambda(y) \,dy.
\end{align}
Similar to \eqref{2.5-3}, by Lemma \ref{theta-D}, for any fixed $\alpha>1$, 
we calculate that
\begin{equation}\label{5.7-2}
\begin{split}
&\int_{\Omega\backslash B_{d}(x_\lambda^{(1)})}  G(y,x) E_\lambda(y) \eta_\lambda(y) \,dy \\
 =& O \left( \bigg(\int_{\Omega\backslash B_{d}(x_\lambda^{(1)})} \big|E_\lambda(y) \eta_\lambda(y)\big|^{\alpha} \,dy\bigg)^{\frac{1}{\alpha}} \cdot \bigg(\int_{\Omega\backslash B_{d}(x_\lambda^{(1)})} \big|G(y,x)\big|^{\frac{\alpha}{\alpha-1}} \,dy\bigg)^{1-\frac{1}{\alpha}} \right) \\
=& O\Bigg( \big(\theta_\lambda^{(1)}\big)^{\frac{2}{\alpha}-2} \bigg(\int_{\Omega_\lambda\backslash B_{\frac{d}{\theta_\lambda^{(1)}}}(0)} \frac{1}{\big(1+|z|^{4-\delta}\big)^{\alpha}} \,dz\bigg)^{\frac{1}{\alpha}}  \Bigg) = O\Big( \big(\theta_\lambda^{(1)}\big)^{2-\delta} \Big).
\end{split}
\end{equation}
On the other hand, we find
\begin{align}\label{5.7-3}
&\int_{B_{d}(x_\lambda^{(1)})} \big( G(y,x)-G(x_\lambda^{(1)},x) \big) E_\lambda(y) \eta_\lambda(y) \,dy \notag\\
=& \int_{B_{d}(x_\lambda^{(1)})} \big\langle \nabla G(x_\lambda^{(1)},x),y-x_\lambda^{(1)} \big\rangle E_\lambda(y) \eta_\lambda(y) \,dy + O\bigg( \int_{B_{d}(x_\lambda^{(1)})} \big|y-x_\lambda^{(1)}\big|^2 |E_\lambda(y)|\cdot |\eta_\lambda(y)| \,dy  \bigg) \notag\\
=&  \theta_\lambda^{(1)} \bigg( \frac{1}{\theta_\lambda^{(1)} } \sum\limits_{i=1}^2 \partial_i G(x_\lambda^{(1)},x) \int_{B_d(x^{(1)}_{\lambda})} \Big(y_i-\big(x^{(1)}_\lambda\big)_{i}\Big)  E_{\lambda}(y)\eta_{\lambda}(y) \,dy \bigg) \notag\\
& +O\bigg( \int_{B_{\frac{d}{\theta_\lambda^{(1)}}}(0)} \big(\theta_\lambda^{(1)}\big)^2 |z|^2 \big| E_\lambda\big(\theta_\lambda^{(1)}z+x_\lambda^{(1)}\big) \big| \big(\theta_\lambda^{(1)}\big)^2 \,dz  \bigg) \notag\\
=& \theta_\lambda^{(1)} \sum\limits_{i=1}^2 \widetilde{B}_{\lambda,i} \partial_i G(x_\lambda^{(1)},x) +O\Big( \big(\theta_\lambda^{(1)}\big)^{2-\delta} \Big).
\end{align}
Then \eqref{5.7-1}, \eqref{5.7-2} and \eqref{5.7-3} imply that
\begin{equation}\label{5.7-4}
\begin{split}
\eta_{\lambda}(x)=& \widetilde{A}_{\lambda}G(x^{(1)}_{\lambda},x)+
\sum^2_{i=1}\widetilde{B}_{\lambda,i}\theta^{(1)}_{\lambda}  \partial_iG(x^{(1)}_{\lambda},x)
 +o\big(\theta^{(1)}_{\lambda}\big).
\end{split}
\end{equation}

Furthermore, we know that
\begin{equation*}
D_{x_i}\eta_\lambda(x) = \int_{\Omega} D_{x_i} G(y,x) E_\lambda(y) \eta_\lambda(y) \,dy,
\end{equation*}
then similar to \eqref{5.7-4}, we can deduce that
\begin{equation*}
\begin{split}
\eta_{\lambda}(x)=& \widetilde{A}_{\lambda}G(x^{(1)}_{\lambda},x)+
\sum^2_{i=1}\widetilde{B}_{\lambda,i}\theta^{(1)}_{\lambda}  \partial_iG(x^{(1)}_{\lambda},x)
+o\big(\theta^{(1)}_{\lambda}\big) \,~\mbox{in}~ C^1 \Big(\Omega\backslash B_{2d} \big(x^{(1)}_{\lambda}\big)\Big).
\end{split}
\end{equation*}

\end{proof}

\subsection{Proofs of Theorem \ref{th_uniq} and Theorem \ref{th_uniq2}}
\begin{Prop}
For any small fixed constant $d>0$, we have the following local Pohozaev identities:
\begin{equation}\label{q1_ueta}
\begin{split}
Q^{(1)}\big(u_\lambda^{(1)}+u_\lambda^{(2)},\eta_{\lambda}\big)=
2 \int_{\partial  B_{d}(x^{(1)}_{\lambda})}\widetilde{E}_{\lambda} \eta_{\lambda} \nu_i\,d\sigma,
\end{split}
\end{equation}
and
\begin{equation}\label{p1_ueta}
\begin{split}
P^{(1)}\big(u_\lambda^{(1)}+u_\lambda^{(2)},\eta_{\lambda}\big)
= 2d \int_{\partial  B_{d}(x^{(1)}_{\lambda})} \widetilde{E}_{\lambda}\eta_{\lambda}\,d\sigma
-4\int_{  B_{d}(x^{(1)}_{\lambda})} \widetilde{E}_{\lambda}\eta_{\lambda}\,dx,
\end{split}
\end{equation}
where $P^{(1)}$ and $Q^{(1)}$ are the quadratic forms in \eqref{p1uv} and \eqref{q1uv},   $\nu=\big(\nu_{1},\nu_2\big)$ is the unit  outward normal of $\partial  B_{d}(x^{(1)}_{\lambda})$ and
\begin{equation}\label{def_tildeE}
\widetilde{E}_{\lambda}(x):= \lambda  \displaystyle\int_{0}^1
\Big(tu_{\lambda}^{(1)}(x)+(1-t)u_{\lambda}^{(2)}(x)\Big) e^{\big(tu_{\lambda}^{(1)}(x)+(1-t)u_{\lambda}^{(2)}(x)\big)^2}dt.
\end{equation}
\end{Prop}
\begin{proof}
We can calculate that
\begin{equation}\label{5.8-1}
\begin{split}
Q^{(1)}\big(u_\lambda^{(1)}+u_\lambda^{(2)},\eta_{\lambda}\big)
&= \frac{1}{\|u_\lambda^{(1)}-u_\lambda^{(2)}\|_{L^{\infty}(\Omega)}} Q^{(1)}\big(u_\lambda^{(1)}+u_\lambda^{(2)},u_\lambda^{(1)}-u_\lambda^{(2)}\big)\\
&= \frac{1}{\|u_\lambda^{(1)}-u_\lambda^{(2)}\|_{L^{\infty}(\Omega)}} \left(Q^{(1)}\big(u_\lambda^{(1)},u_\lambda^{(1)}\big) - Q^{(1)}\big(u_\lambda^{(2)},u_\lambda^{(2)}\big)\right).
\end{split}
\end{equation}
Next, by \eqref{quu}, we know that
\begin{equation*}
\begin{split}
& Q^{(1)}\big(u_\lambda^{(1)},u_\lambda^{(1)}\big) - Q^{(1)}\big(u_\lambda^{(2)},u_\lambda^{(2)}\big) \\
=&~ \lambda \int_{\partial B_{d}(x_{\lambda}^{(1)})} e^{(u^{(1)}_\lambda(x))^2} \nu_i \,d\sigma -\lambda \int_{\partial B_{d}(x_{\lambda}^{(1)})} e^{(u^{(2)}_\lambda(x))^2} \nu_i \,d\sigma \\
=&~ \lambda \int_{\partial B_{d}(x_{\lambda}^{(1)})} \left( \int_0^1 \frac{d}{dt} e^{(tu^{(1)}_\lambda(x)+(1-t)u^{(2)}_\lambda(x))^2} \,dt \right) \nu_i \,d\sigma \\
=&~ \lambda \int_{\partial B_{d}(x_{\lambda}^{(1)})} \left(\int_0^1 2 e^{(tu^{(1)}_\lambda(x)+(1-t)u^{(2)}_\lambda(x))^2} \big(tu^{(1)}_\lambda(x)+(1-t)u^{(2)}_\lambda(x)\big) \big(u^{(1)}_\lambda(x)-u^{(2)}_\lambda(x)\big) \,dt \right) \nu_i \,d\sigma \\
=&~ 2\int_{\partial B_{d}(x_{\lambda}^{(1)})} \widetilde{E}_\lambda(x) \big(u^{(1)}_\lambda(x)-u^{(2)}_\lambda(x)\big) \nu_i \,d\sigma,
\end{split}
\end{equation*}
which together with \eqref{5.8-1} implies \eqref{q1_ueta}.

\vskip 0.2cm
Similarly, by \eqref{puu}, we have
\begin{equation*}
\begin{split}
& P^{(1)}\big(u_\lambda^{(1)}+u_\lambda^{(2)},\eta_{\lambda}\big) \\
=&~ \frac{1}{\|u_\lambda^{(1)}-u_\lambda^{(2)}\|_{L^{\infty}(\Omega)}} \left(P^{(1)}\big(u_\lambda^{(1)},u_\lambda^{(1)}\big) - P^{(1)}\big(u_\lambda^{(2)},u_\lambda^{(2)}\big)\right) \\
=&~ \frac{1}{\|u_\lambda^{(1)}-u_\lambda^{(2)}\|_{L^{\infty}(\Omega)}} \left( d\lambda  \int_{\partial B_{d}(x^{(1)}_{\lambda})} \Big( e^{(u^{(1)}_\lambda)^2} - e^{(u^{(2)}_\lambda)^2} \Big) \,d\sigma - 2\lambda \int_{B_{d}(x^{(1)}_{\lambda})}  \Big( e^{(u^{(1)}_\lambda)^2}-e^{(u^{(2)}_\lambda)^2} \Big) \,dx \right) \\
=&~ 2d \int_{\partial B_{d}(x^{(1)}_{\lambda})} \widetilde{E}_{\lambda}\eta_{\lambda} \,d\sigma
-4\int_{B_{d}(x^{(1)}_{\lambda})} \widetilde{E}_{\lambda}\eta_{\lambda} \,dx.
\end{split}
\end{equation*}

\end{proof}

\begin{Prop}\label{prop-AB}
Let $\widetilde{A}_\lambda$ and  $\widetilde{B}_{\lambda,i}$ be as in \eqref{tildeA-equa} and \eqref{tildeB-equa}, $\widetilde{b}_{i}$ as in Proposition \ref{prop_tildeeta}, then
\begin{equation}\label{A_B}
\widetilde{A}_{\lambda}=
o\big(1\big),~~~~\widetilde{B}_{\lambda,i}=o(1)~~~\mbox{and}~~~\widetilde{b}_{i}=0~\mbox{for}~i=1,2.
\end{equation}
\end{Prop}
\begin{proof}
Firstly, using \eqref{Dlambda-lim} and \eqref{tildeeta-lim}, we know
\begin{equation}\label{5.9-1}
\begin{split}
\widetilde{A}_{\lambda} &=\int_{B_d(x^{(1)}_{\lambda})}  E_{\lambda}(x)\eta_{\lambda}(x) \,dx =\int_{B_{\frac{d}{\theta_\lambda^{(1)}}}(0)} \big(\theta_\lambda^{(1)}\big)^2 E_{\lambda}\big(\theta_\lambda^{(1)}z+x_\lambda^{(1)}\big) \widetilde{\eta}_{\lambda}(z) \,dz \\
&= 2\widetilde{a} \int_{\R^2} e^{2U(z)}\frac{4-|z|^2}{4+|z|^2} \,dz +o\big(1\big)
\overset{\eqref{equa2}}=o\big(1\big).
\end{split}
\end{equation}
Similar to the proof of \eqref{5.5-4}, we have
\begin{equation}\label{tilde_D-equa}
\begin{split}
\big(\theta^{(1)}_\lambda\big)^2 \widetilde{E}_\lambda\big(\theta^{(1)}_\lambda x + x^{(1)}_\lambda\big)
&= \lambda \big(\theta^{(1)}_\lambda\big)^2 \int_{0}^{1} F_t\big(\theta^{(1)}_{\lambda}x+x^{(1)}_{\lambda}\big)
e^{ \big(F_t(\theta^{(1)}_{\lambda}x+x^{(1)}_{\lambda})\big)^2 } \,dt \\
&= \frac{1}{\gamma_\lambda^{(1)}} e^{2w^{(1)}_\lambda(x)+\frac{(w^{(1)}_\lambda(x))^2}{( \gamma^{(1)}_\lambda)^2}}
\big(1+o(\lambda)\big)\Big(1+\frac{ w^{(1)}_\lambda(x) }{\big( \gamma^{(1)}_\lambda\big)^2}
+o\big(\lambda^2\big)\Big) \\
&= \frac{1}{\gamma_\lambda^{(1)}} e^{2U(x)} \big( 1+O(\lambda) \big)~\,\,\mbox{uniformly~on~compact~sets}.
\end{split}
\end{equation}
Hence it holds
\begin{equation}\label{5.9-2}
\begin{split}
\text{RHS of (\ref{q1_ueta})} &= 2\int_{\partial B_d(x^{(1)}_{\lambda})} \widetilde{E}_{\lambda} \eta_{\lambda} \nu_i \,d\sigma \\
 &= O\bigg( \int_{\partial B_{\frac{d}{\theta^{(1)}_\lambda}}(0)} \theta^{(1)}_\lambda \widetilde{E}_{\lambda}\big(\theta^{(1)}_\lambda x + x^{(1)}_\lambda\big) \,d\sigma \bigg)
=o\Big( \sqrt{\lambda}\theta^{(1)}_\lambda\Big).
\end{split}
\end{equation}
Moreover using \eqref{4.9.1} and \eqref{etalambda-equa}, we have
\begin{equation}\label{5.9-3}
\begin{split}
 \text{LHS of (\ref{q1_ueta})}=& \widetilde{A}_{\lambda} \sum^2_{l=1}C^{(l)}_{\lambda}
 Q^{(1)}
 \Big(G\big(x^{(l)}_{\lambda},x\big),G\big(x^{(1)}_{\lambda},x\big)\Big)
\\&+  \theta^{(1)}_{\lambda} \sum^2_{l=1} \sum^2_{h=1}\big(\widetilde{B}_{\lambda,h}C^{(l)}_{\lambda}\big)  Q^{(1)}
\Big(G\big(x^{(l)}_{\lambda},x\big),\partial_h G\big(x^{(1)}_{\lambda},x\big)\Big)
+o\big( \sqrt{\lambda}\theta^{(1)}_\lambda\big).
\end{split}
\end{equation}
From \eqref{def_C1}, \eqref{q1_g} and \eqref{q1_g2}, we know
\begin{equation}\label{5.9-4}
\begin{split}
 \sum^2_{l=1} C^{(l)}_{\lambda}
 Q^{(1)} \Big(G\big(x^{(l)}_{\lambda},x\big),G\big(x^{(1)}_{\lambda},x\big)\Big) =
o\big( \sqrt{\lambda}\theta^{(1)}_\lambda\big).
\end{split}
\end{equation}
Then it follows from \eqref{5.9-1}, \eqref{5.9-2}, \eqref{5.9-3} and \eqref{5.9-4} that
\begin{equation} \label{5.9-5}
\begin{split}
 \theta^{(1)}_{\lambda} \sum^2_{l=1} \sum^2_{h=1}\big(\widetilde{B}_{\lambda,h}C^{(l)}_{\lambda}\big)  Q^{(1)}
\Big(G\big(x^{(l)}_{\lambda},x\big),\partial_h G\big(x^{(1)}_{\lambda},x\big)\Big)
=o\big( \sqrt{\lambda}\theta^{(1)}_\lambda\big).
\end{split}
\end{equation}
 Putting \eqref{q1_partg} and \eqref{q1_partg2} into \eqref{5.9-5},  we then have
\begin{equation}\label{5.9-6}
\begin{split}
\sum^2_{h=1} &\widetilde{B}_{\lambda,h} \Big(\frac{\partial^2 \mathcal{R}(x_{0})}{\partial x_i \partial x_h}+o\big(1\big)\Big)
=o(1).
\end{split}
\end{equation}
Since $x_0$ is a nondegenerate critical point of $\mathcal{R}(x)$, we can deduce from \eqref{5.9-6} that
\begin{equation}\label{5.9-7}
\widetilde{B}_{\lambda,i}=o(1),~\mbox{for}~i=1,2.
\end{equation}
On the other hand, it holds
\begin{equation}\label{5.9-8}
\begin{split}
\widetilde{B}_{\lambda,i} &= \frac{1}{\theta^{(1)}_{\lambda}}\int_{B_d(x^{(1)}_{\lambda})} \big(x_i-(x^{(1)}_{\lambda})_{i}\big)  E_{\lambda}(x)\eta_{\lambda}(x) \,dx
\\& =\big(\theta^{(1)}_{\lambda}\big)^2 \int_{B_{\frac{d}{\theta^{(1)}_\lambda}}(0)} z_i \cdot  E_{\lambda}\big(\theta^{(1)}_\lambda z+x^{(1)}_\lambda \big) \widetilde{\eta}_{\lambda}(z) \,dz \\
&= \widetilde{b}_{i} \Big(\int_{\R^2} 2e^{2U(z)}\frac{z_i^2 }{4+|z|^2} \,dz
+o(1)\Big) = 2\pi \widetilde{b}_{i}+o\big(1\big),\,\,~\mbox{for}~i=1,2.
\end{split}
\end{equation}
Then  $\widetilde{b}_{1}=\widetilde{b}_{2}=0$  by \eqref{5.9-7} and \eqref{5.9-8}.
\end{proof}

\begin{Prop}\label{prop-A}
Let $\widetilde{A}_\lambda$ be as in \eqref{tildeA-equa} and $\widetilde{a}$ be the constant in \eqref{tildeeta-lim}, then
\begin{equation*}
\widetilde{A}_{\lambda}=
o\big(\lambda \big)~\,\mbox{and}~\,\widetilde{a}=0.
\end{equation*}
\end{Prop}
\begin{proof}
Now using \eqref{4.9.1}, \eqref{p1_g}, \eqref{p1_g2} and \eqref{etalambda-equa}, we can deduce that
\begin{equation}\label{5.10-1}
\begin{split}
 \text{LHS of (\ref{p1_ueta})}=& \widetilde{A}_{\lambda}
\sum^2_{h=1}C^{(h)}_{\lambda}
P^{(1)}\Big(G\big(x^{(h)}_{\lambda},x\big),G\big(x^{(1)}_{\lambda},x\big)\Big) \\&
+  \theta^{(1)}_{\lambda} \sum^2_{h=1} \sum^2_{i=1}\big(\widetilde{B}_{\lambda,i}C^{(h)}_{\lambda}\big)
P^{(1)}\Big(G\big(x^{(h)}_{\lambda},x\big),\partial_i G\big(x^{(1)}_{\lambda},x\big)\Big) +
o\big( \sqrt{\lambda}\theta^{(1)}_\lambda\big) \\=&
-\widetilde{A}_{\lambda} \sum^2_{h=1} \frac{C^{(h)}_{\lambda}}{2\pi}+\theta^{(1)}_{\lambda}
 \sum^2_{h=1}\sum^2_{i=1} \big(\widetilde{B}_{\lambda,i}C^{(h)}_{\lambda}\big)
P^{(1)}\Big(G\big(x^{(h)}_{\lambda},x\big),\partial_i G\big(x^{(1)}_{\lambda},x\big)\Big) \\
& +o\big( \sqrt{\lambda}\theta^{(1)}_\lambda\big).
\end{split}
\end{equation}
Using \eqref{xlambda-x0}, \eqref{p1_partg}, \eqref{5.9-7} and Lemma \ref{xlambda12}, we find
\begin{equation}\label{5.10-2}
\begin{split}
& \sum^2_{h=1}\sum^2_{i=1} \big(\widetilde{B}_{\lambda,i}C^{(h)}_{\lambda}\big)
P^{(1)}\Big(G\big(x^{(h)}_{\lambda},x\big),\partial_i G\big(x^{(1)}_{\lambda},x\big)\Big)\\=&
  \sum^2_{h=1}\sum^2_{i=1} \big(\widetilde{B}_{\lambda,i}C^{(h)}_{\lambda}\big)
P^{(1)}\Big(G\big(x^{(1)}_{\lambda},x\big),\partial_i G\big(x^{(1)}_{\lambda},x\big)\Big)+o\big( \lambda^{\frac{3}{2}}\theta^{(1)}_\lambda\big)=o\big( \lambda^{\frac{3}{2}}\theta^{(1)}_\lambda\big).
\end{split}
\end{equation}
Then from  \eqref{def_C1}, \eqref{5.10-1} and \eqref{5.10-2}, we have
\begin{equation}\label{lhs-5.41}
\begin{split}
\text{LHS of (\ref{p1_ueta})}  =& -\widetilde{A}_{\lambda} \sum^2_{h=1} \frac{C^{(h)}_{\lambda}}{2\pi} +
o\Big( \sqrt{\lambda}\theta^{(1)}_\lambda\Big) \\
=& -4\widetilde{A}_{\lambda} \bigg(\frac{1}{\gamma_\lambda^{(1)}} +\frac{1}{\big(\gamma_\lambda^{(1)}\big)^3} +o\big(\lambda^{\frac{3}{2}}\big)\bigg) +o\Big( \sqrt{\lambda}\theta^{(1)}_\lambda\Big).
\end{split}
\end{equation}
Using \eqref{tilde_D-equa}, we know that
\begin{equation*}
\begin{split}
\int_{\partial  B_{d}(x^{(1)}_{\lambda})} \widetilde{E}_{\lambda}\eta_{\lambda}\,d\sigma =
O\bigg(\theta^{(1)}_\lambda\int_{\partial B_{\frac{d}{\theta^{(1)}_\lambda}}(0)} \widetilde{E}_\lambda
\big(\theta^{(1)}_\lambda x+x^{(1)}_\lambda\big) \,d\sigma \bigg)
= O\Big(\sqrt{\lambda}\big(\theta^{(1)}_\lambda\big)^{2-\delta}\Big).
\end{split}
\end{equation*}
Hence by the definition of $E_\lambda(x)$ and $\widetilde{E}_\lambda(x)$ \big(see \eqref{Dlambda-def} and \eqref{def_tildeE}\big), we get
\begin{align}\label{5.10-4}
\text{RHS of (\ref{p1_ueta})}
=& -4\int_{B_{d}(x^{(1)}_{\lambda})} \widetilde{E}_{\lambda}\eta_{\lambda}\,dx +O\Big(\sqrt{\lambda}\big(\theta^{(1)}_\lambda\big)^{2-\delta}\Big) \nonumber\\
=& -\frac{2}{\gamma_\lambda^{(1)}} \int_{B_{d}(x^{(1)}_{\lambda})} \Big( E_\lambda-\lambda e^{\big(u_\lambda^{(1)}\big)^2} \Big)\eta_\lambda \,dx - \frac{4}{\gamma_\lambda^{(1)}} \int_{B_{d}(x^{(1)}_{\lambda})} \Big(u_\lambda^{(1)}-u_\lambda^{(2)}\Big)\widetilde{E}_\lambda\eta_\lambda \,dx \nonumber\\
& -\frac{4}{\gamma_\lambda^{(1)}} \int_{B_{d}(x^{(1)}_{\lambda})} \Big(\gamma_\lambda^{(1)}-u_\lambda^{(1)}(x)\Big)\widetilde{E}_\lambda(x)\eta_\lambda(x) \,dx +O\Big(\sqrt{\lambda}\big(\theta^{(1)}_\lambda\big)^{2-\delta}\Big) \nonumber\\
=& -\frac{2\widetilde{A}_\lambda}{\gamma_\lambda^{(1)}}+\frac{2}{\gamma_\lambda^{(1)}}
\int_{B_d(x^{(1)}_{\lambda})} \lambda e^{\big(u_\lambda^{(1)}(x)\big)^2}\eta_\lambda(x) \,dx \nonumber\\
& +\frac{4}{\gamma_\lambda^{(1)}} \underbrace{ \int_{B_d(x^{(1)}_{\lambda})} \Big(u_\lambda^{(1)}(x)-\gamma^{(1)}_\lambda\Big) \widetilde{E}_\lambda(x)\eta_\lambda(x) \,dx }_{:=\tau_\lambda} +o\bigg(\frac{1}{\big(\gamma_\lambda^{(1)}\big)^5}\bigg).
\end{align}
Now, we estimate the second term.
\begin{align}\label{5.10-5}
& \frac{2}{\gamma_\lambda^{(1)}} \int_{B_d(x^{(1)}_{\lambda})} \lambda e^{\big(u_\lambda^{(1)}(x)\big)^2}\eta_\lambda(x) \,dx \nonumber\\
=& \frac{2\lambda}{\big(\gamma_\lambda^{(1)}\big)^2} \int_{B_d(x^{(1)}_{\lambda})} e^{\big(u_\lambda^{(1)}(x)\big)^2} \eta_\lambda(x) u_\lambda^{(1)}(x) \,dx -\frac{2\lambda}{\big(\gamma_\lambda^{(1)}\big)^2} \int_{B_d(x^{(1)}_{\lambda})} e^{\big(u_\lambda^{(1)}(x)\big)^2} \eta_\lambda(x) \big(u_\lambda^{(1)}(x)- \gamma_\lambda^{(1)}\big) \,dx \nonumber\\
=& \frac{2\lambda}{\big(\gamma_\lambda^{(1)}\big)^2} \int_{B_d(x^{(1)}_{\lambda})} e^{\big(u_\lambda^{(1)}(x)\big)^2} \eta_\lambda(x) u_\lambda^{(1)}(x) \frac{u_\lambda^{(1)}(x)-\big(u_\lambda^{(1)}(x)-\gamma_\lambda^{(1)}
\big)}{\gamma_\lambda^{(1)}} \,dx \nonumber\\
&-\frac{2\widetilde{a}}{\big(\gamma_\lambda^{(1)}\big)^5} \bigg(\int_{\R^2} e^{2U(z)} \frac{4-|z|^2}{4+|z|^2} U(z) \,dz+o(1)\bigg) \nonumber\\
=& -\frac{4\pi \widetilde{a}}{\big(\gamma_\lambda^{(1)}\big)^5} +o\bigg(\frac{1}{\big(\gamma_\lambda^{(1)}\big)^5}\bigg) +\frac{2\lambda}{\big(\gamma_\lambda^{(1)}\big)^3} \int_{B_d(x^{(1)}_{\lambda})} e^{\big(u_\lambda^{(1)}(x)\big)^2} \eta_\lambda(x) \big(u_\lambda^{(1)}(x)\big)^2 \,dx  \nonumber\\
& -\frac{2\lambda}{\big(\gamma_\lambda^{(1)}\big)^3} \int_{B_d(x^{(1)}_{\lambda})} e^{\big(u_\lambda^{(1)}(x)\big)^2} \eta_\lambda(x) u_\lambda^{(1)}(x) \big(u_\lambda^{(1)}(x)-\gamma_\lambda^{(1)}\big)\,dx \nonumber\\
=& \frac{2\lambda}{\big(\gamma_\lambda^{(1)}\big)^3} \int_{B_d(x^{(1)}_{\lambda})} e^{\big(u_\lambda^{(1)}(x)\big)^2} \eta_\lambda(x) \big(u_\lambda^{(1)}(x)\big)^2 \,dx -\frac{8\pi \widetilde{a}}{\big(\gamma_\lambda^{(1)}\big)^5}
+o\bigg(\frac{1}{\big(\gamma_\lambda^{(1)}\big)^5}\bigg) \nonumber\\
=& \frac{2}{\big(\gamma_\lambda^{(1)}\big)^3} \int_{B_d(x^{(1)}_{\lambda})} \Big(\lambda e^{\big(u_\lambda^{(1)}(x)\big)^2} \big(u_\lambda^{(1)}(x)\big)^2 -\frac{1}{2} E_\lambda(x)\Big) \eta_\lambda(x)  \,dx \nonumber\\
&+\frac{\widetilde{A}_\lambda}{\big(\gamma_\lambda^{(1)}\big)^3}
-\frac{8\pi \widetilde{a}}{\big(\gamma_\lambda^{(1)}\big)^5}+o\bigg(\frac{1}{\big(\gamma_\lambda^{(1)}\big)^5}\bigg).
\end{align}
Furthermore, using \eqref{Dlambda-def} and \eqref{5.5-6}, we calculate that
\begin{align*}
& \frac{2}{\big(\gamma_\lambda^{(1)}\big)^3} \int_{B_d(x^{(1)}_{\lambda})} \Big(\lambda e^{\big(u_\lambda^{(1)}(x)\big)^2} \big(u_\lambda^{(1)}(x)\big)^2 -\frac{1}{2} E_\lambda(x)\Big)\eta_\lambda(x) \,dx \\
=& \frac{2}{\big(\gamma_\lambda^{(1)}\big)^3} \int_{B_d(x^{(1)}_{\lambda})} \left(\lambda e^{\big(u_\lambda^{(1)}(x)\big)^2} \big(u_\lambda^{(1)}(x)\big)^2 -\frac{1}{2} \lambda
e^{(u_{\lambda}^{(1)}(x))^2} - \lambda u_{\lambda}^{(2)} \displaystyle\int_{0}^1 F_t(x) e^{(F_t(x))^2} \,dt \right)
\eta_\lambda(x) \,dx \\
=& \frac{2}{\big(\gamma_\lambda^{(1)}\big)^3} \int_{B_{\frac{d}{\theta^{(1)}_\lambda}}(0)} e^{2w^{(1)}_\lambda(x)+\frac{(w^{(1)}_\lambda(x))^2}{(\gamma^{(1)}_\lambda)^2}} \bigg(\frac{w^{(1)}_\lambda(x)}{\big(\gamma^{(1)}_\lambda\big)^2}+1\bigg)^2 \widetilde{\eta}_\lambda(x) \,dx \\
& -\frac{1}{\big(\gamma_\lambda^{(1)}\big)^5} \int_{B_{\frac{d}{\theta^{(1)}_\lambda}}(0)} e^{2w^{(1)}_\lambda(x)+\frac{(w^{(1)}_\lambda(x))^2}{(\gamma^{(1)}_\lambda)^2}}\widetilde{\eta}_\lambda(x) \,dx
-\frac{2}{\big(\gamma_\lambda^{(1)}\big)^3} \int_{B_{\frac{d}{\theta^{(1)}_\lambda}}(0)}
\lambda \big(\theta_\lambda^{(1)}\big)^2 u_{\lambda}^{(2)}\big(\theta_\lambda^{(1)}x+x_\lambda^{(1)}\big) \\
& \times \displaystyle\int_{0}^1 F_t\big(\theta_\lambda^{(1)}x+x_\lambda^{(1)}\big) e^{(F_t(\theta_\lambda^{(1)}x+x_\lambda^{(1)}))^2} \,dt \cdot \widetilde{\eta}_\lambda(x) \,dx \\
=& \frac{2}{\big(\gamma_\lambda^{(1)}\big)^3} \int_{B_{\frac{d}{\theta^{(1)}_\lambda}}(0)} e^{2w^{(1)}_\lambda(x)+\frac{(w^{(1)}_\lambda(x))^2}{(\gamma^{(1)}_\lambda)^2}} \bigg(\frac{w^{(1)}_\lambda(x)}{\big(\gamma^{(1)}_\lambda\big)^2}+1\bigg)^2 \widetilde{\eta}_\lambda(x) \,dx \\
& -\frac{\widetilde{a}}{\big(\gamma_\lambda^{(1)}\big)^5} \bigg( \int_{\R^2} e^{2U(x)} \frac{4-|x|^2}{4+|x|^2} \,dx +o(1)\bigg) \\
& -\frac{2}{\big(\gamma_\lambda^{(1)}\big)^3} \int_{B_{\frac{d}{\theta^{(1)}_\lambda}}(0)}
e^{2w^{(1)}_\lambda(x)+\frac{(w^{(1)}_\lambda(x))^2}{( \gamma^{(1)}_\lambda)^2}}
\big(1+o(\lambda)\big)  \bigg(1+\frac{ w^{(1)}_\lambda(x) }{\big( \gamma^{(1)}_\lambda\big)^2}+o\big(\lambda^2
\big)\bigg)^2 \widetilde{\eta}_\lambda(x) \,dx \\
=& \frac{2}{\big(\gamma_\lambda^{(1)}\big)^3} \int_{B_{\frac{d}{\theta^{(1)}_\lambda}}(0)} e^{2w^{(1)}_\lambda(x)+\frac{(w^{(1)}_\lambda(x))^2}{(\gamma^{(1)}_\lambda)^2}} \bigg(\frac{w^{(1)}_\lambda(x)}{\big(\gamma^{(1)}_\lambda\big)^2}+1\bigg)^2 \widetilde{\eta}_\lambda(x) \,dx
+o\bigg(\frac{1}{\big(\gamma_\lambda^{(1)}\big)^5}\bigg) \\
& -\frac{2}{\big(\gamma_\lambda^{(1)}\big)^3} \int_{B_{\frac{d}{\theta^{(1)}_\lambda}}(0)}
e^{2w^{(1)}_\lambda(x)+\frac{(w^{(1)}_\lambda(x))^2}{( \gamma^{(1)}_\lambda)^2}}
\left(\bigg(\frac{w^{(1)}_\lambda(x)}{\big(\gamma^{(1)}_\lambda\big)^2}+1\bigg)^2 +o(\lambda)\right) \widetilde{\eta}_\lambda(x) \,dx \\
=& o\bigg(\frac{1}{\big(\gamma_\lambda^{(1)}\big)^5}\bigg),
\end{align*}
which together with \eqref{5.10-5} implies
\begin{equation}\label{5.10-6}
\frac{2}{\gamma_\lambda^{(1)}} \int_{B_d(x^{(1)}_{\lambda})} \lambda e^{\big(u_\lambda^{(1)}(x)\big)^2}\eta_\lambda(x) \,dx = \frac{\widetilde{A}_\lambda}{\big(\gamma_\lambda^{(1)}\big)^3}
-\frac{8\pi \widetilde{a}}{\big(\gamma_\lambda^{(1)}\big)^5}+o\bigg(\frac{1}{\big(\gamma_\lambda^{(1)}\big)^5}\bigg).
\end{equation}
Substituting \eqref{5.10-6} into \eqref{5.10-4}, we derive
\begin{equation}\label{rhs-5.41}
\begin{split}
\text{RHS of (\ref{p1_ueta})} &= -\frac{2\widetilde{A}_\lambda}{\gamma_\lambda^{(1)}} +\frac{4}{\gamma_\lambda^{(1)}} \tau_\lambda +\frac{2}{\gamma_\lambda^{(1)}} \int_{B_d(x^{(1)}_{\lambda})} \lambda e^{\big(u_\lambda^{(1)}(x)\big)^2} \eta_\lambda(x) \,dx +o\bigg(\frac{1}{\big(\gamma_\lambda^{(1)}\big)^5}\bigg) \\
&=  -\frac{2\widetilde{A}_\lambda}{\gamma_\lambda^{(1)}} +\frac{\widetilde{A}_\lambda}{\big(\gamma_\lambda^{(1)}\big)^3}
+\frac{4}{\gamma_\lambda^{(1)}} \tau_\lambda
-\frac{8\pi \widetilde{a}}{\big(\gamma_\lambda^{(1)}\big)^5}+o\bigg(\frac{1}{\big(\gamma_\lambda^{(1)}\big)^5}\bigg).
\end{split}
\end{equation}
Combined with \eqref{lhs-5.41} and \eqref{rhs-5.41}, we find
\begin{equation*}
-4\widetilde{A}_{\lambda} \bigg(\frac{1}{\gamma_\lambda^{(1)}} +\frac{1}{\big(\gamma_\lambda^{(1)}\big)^3} +o\big(\lambda^{\frac{3}{2}}\big)\bigg) =
\widetilde{A}_\lambda \bigg( \frac{1}{\big(\gamma_\lambda^{(1)}\big)^3}-\frac{2}{\gamma_\lambda^{(1)}} \bigg) +\frac{4}{\gamma_\lambda^{(1)}} \tau_\lambda
-\frac{8\pi \widetilde{a}}{\big(\gamma_\lambda^{(1)}\big)^5}+o\bigg(\frac{1}{\big(\gamma_\lambda^{(1)}\big)^5}\bigg),
\end{equation*}
which indicates that
\begin{equation}\label{A_tilde1}
-\widetilde{A}_{\lambda} \bigg(\frac{2}{\gamma_\lambda^{(1)}} +\frac{5}{\big(\gamma_\lambda^{(1)}\big)^3} +o\big(\lambda^{\frac{3}{2}}\big)\bigg) = \frac{4}{\gamma_\lambda^{(1)}} \tau_\lambda
-\frac{8\pi \widetilde{a}}{\big(\gamma_\lambda^{(1)}\big)^5}+o\bigg(\frac{1}{\big(\gamma_\lambda^{(1)}\big)^5}\bigg).
\end{equation}

Next, in order to deal with $\tau_\lambda$, we need some additional estimates.
Recall that
\[
-\Delta u_\lambda^{(l)} = \lambda u_\lambda^{(l)} e^{(u_\lambda^{(l)})^2},~\mbox{for}~l=1,2,~~~~~\mbox{and}~~~~~
-\Delta \eta_\lambda = E_\lambda \eta_\lambda.
\]
Then it follows from \eqref{4.9.1}, \eqref{def_C1}, \eqref{etalambda-equa} and \eqref{A_B} that
\begin{align}\label{lhs-add}
&\int_{B_d(x^{(1)}_{\lambda})} \Big(E_\lambda(x) \eta_\lambda(x) u_\lambda^{(1)}(x) - \lambda u_\lambda^{(1)}(x) e^{(u_\lambda^{(1)}(x))^2} \eta_\lambda(x)\Big) \,dx \notag\\
=& \int_{B_d(x^{(1)}_{\lambda})} \Big( -\Delta \eta_\lambda(x)u_\lambda^{(1)}(x)+\Delta u_\lambda^{(1)}(x) \eta_\lambda(x)\Big) \,dx \notag\\
=& \int_{\partial B_d(x^{(1)}_{\lambda})} \Big( -\frac{\partial \eta_\lambda}{\partial \nu}(x) u_\lambda^{(1)}(x) +\frac{\partial u_\lambda^{(1)}}{\partial \nu}(x) \eta_\lambda(x)\Big) \,d\sigma \notag\\
=& \int_{\partial B_d(x^{(1)}_{\lambda})} \Big( \widetilde{A}_\lambda G(x,x_\lambda^{(1)})
+o\big(\theta_\lambda^{(1)}\big) \Big) \left( C_\lambda^{(1)} \frac{\partial G(x,x_\lambda^{(1)})}{\partial \nu}+O\bigg(\frac{\theta_\lambda^{(1)}}{\big(\gamma_\lambda^{(1)}\big)^3}\bigg) \right) \notag\\
& -\bigg( \widetilde{A}_\lambda \frac{\partial G(x,x_\lambda^{(1)})}{\partial \nu}
+o\big(\theta_\lambda^{(1)}\big) \bigg) \left( C_\lambda^{(1)} G(x,x_\lambda^{(1)}) +O\bigg(\frac{\theta_\lambda^{(1)}}{\big(\gamma_\lambda^{(1)}\big)^3}\bigg) \right) \,d\sigma \notag\\
=& o\Big(\theta_\lambda^{(1)} C_\lambda^{(1)}\Big) = o\Big(\theta_\lambda^{(1)}\sqrt{\lambda}\Big).
\end{align}
On the other hand, we can also calculate that
\begin{align}\label{rhs-add}
&\int_{B_d(x^{(1)}_{\lambda})} \Big(E_\lambda(x) -\lambda e^{(u_\lambda^{(1)}(x))^2}\Big) u_\lambda^{(1)}(x) \eta_\lambda(x) \,dx \notag\\
=& \int_{B_d(x^{(1)}_{\lambda})} 2 u_\lambda^{(2)}(x) \widetilde{E}_\lambda(x) u_\lambda^{(1)}(x) \eta_\lambda(x) \,dx \notag\\
=& \int_{B_{\frac{d}{\theta^{(1)}_\lambda}(0)}} 2 \big(\theta_\lambda^{(1)}\big)^2 u_\lambda^{(1)}\big(\theta_\lambda^{(1)}x+x_\lambda^{(1)}\big) \widetilde{\eta}_\lambda(x) \widetilde{E}_\lambda\big(\theta_\lambda^{(1)}x+x_\lambda^{(1)}\big) \Big(u_\lambda^{(1)}\big(\theta_\lambda^{(1)}x+x_\lambda^{(1)}\big)+f_\lambda(x)\Big) \,dx \notag\\
=& \int_{B_{\frac{d}{\theta^{(1)}_\lambda}(0)}} 2\big(\theta_\lambda^{(1)}\big)^2 \Big(u_\lambda^{(1)}\big(\theta_\lambda^{(1)}x+x_\lambda^{(1)}\big)\Big)^2 \widetilde{\eta}_\lambda(x) \widetilde{E}_\lambda\big(\theta_\lambda^{(1)}x+x_\lambda^{(1)}\big) \,dx \notag\\
& +o\big(\lambda^{\frac{3}{2}}\big) \int_{B_{\frac{d}{\theta^{(1)}_\lambda}(0)}} 2 \big(\theta_\lambda^{(1)}\big)^2 \widetilde{E}_\lambda\big(\theta_\lambda^{(1)}x+x_\lambda^{(1)}\big) u_\lambda^{(1)}\big(\theta_\lambda^{(1)}x+x_\lambda^{(1)}\big) \widetilde{\eta}_\lambda(x) \,dx \notag\\
=& 2\big(\theta_\lambda^{(1)}\big)^2 \int_{B_{\frac{d}{\theta^{(1)}_\lambda}(0)}} \widetilde{\eta}_\lambda(x) \widetilde{E}_\lambda\big(\theta_\lambda^{(1)}x+x_\lambda^{(1)}\big) \frac{\big(w_\lambda^{(1)}(x)\big)^2}{\big(\gamma_\lambda^{(1)}\big)^2} \,dx + 2\big(\gamma_\lambda^{(1)}\big)^2 \int_{B_d(x^{(1)}_{\lambda})} \widetilde{E}_\lambda(x) \eta_\lambda(x) \,dx \notag\\
& +4\gamma_\lambda^{(1)}\int_{B_d(x^{(1)}_{\lambda})} \widetilde{E}_\lambda(x) \eta_\lambda(x) \big(u_\lambda^{(1)}(x)-\gamma_\lambda^{(1)}\big) \,dx +o\big(\lambda^{\frac{3}{2}}\big) \notag\\
=& \frac{2\widetilde{a}}{\big(\gamma_\lambda^{(1)}\big)^3} \int_{\R^2} U^2(x) e^{2U(x)} \frac{4-|x|^2}{4+|x|^2} \,dx +4\gamma_\lambda^{(1)} \tau_\lambda +2\big(\gamma_\lambda^{(1)}\big)^2 \int_{B_d(x^{(1)}_{\lambda})} \widetilde{E}_\lambda(x) \eta_\lambda(x) \,dx +o\big(\lambda^{\frac{3}{2}}\big) \notag\\
=& -\frac{12\pi \widetilde{a}}{\big(\gamma_\lambda^{(1)}\big)^3} +4\gamma_\lambda^{(1)} \tau_\lambda +2\big(\gamma_\lambda^{(1)}\big)^2 \int_{B_d(x^{(1)}_{\lambda})} \widetilde{E}_\lambda(x) \eta_\lambda(x) \,dx +o\bigg(\frac{1}{\big(\gamma_\lambda^{(1)}\big)^3}\bigg).
\end{align}
Thus, \eqref{lhs-add} and \eqref{rhs-add} give that
\begin{equation}\label{5.10-7}
\int_{B_d(x^{(1)}_{\lambda})} \widetilde{E}_\lambda(x) \eta_\lambda(x) \,dx
= \frac{6\pi \widetilde{a}}{\big(\gamma_\lambda^{(1)}\big)^5} -\frac{2}{\gamma_\lambda^{(1)}} \tau_\lambda +o\bigg(\frac{1}{\big(\gamma_\lambda^{(1)}\big)^5}\bigg).
\end{equation}
Using \eqref{5.10-4} and \eqref{rhs-5.41}, we have
\begin{equation*}
\int_{B_d(x^{(1)}_{\lambda})} \widetilde{E}_\lambda(x) \eta_\lambda(x) \,dx =
\frac{\widetilde{A}_\lambda}{2\gamma_\lambda^{(1)}} -\frac{\widetilde{A}_\lambda}{4\big(\gamma_\lambda^{(1)}\big)^3}
-\frac{1}{\gamma_\lambda^{(1)}} \tau_\lambda
+\frac{2\pi \widetilde{a}}{\big(\gamma_\lambda^{(1)}\big)^5}+o\bigg(\frac{1}{\big(\gamma_\lambda^{(1)}\big)^5}\bigg),
\end{equation*}
which together with \eqref{5.10-7} implies that
\begin{equation}\label{A_tilde2}
\widetilde{A}_\lambda \bigg(\frac{1}{2\gamma_\lambda^{(1)}} -\frac{1}{4\big(\gamma_\lambda^{(1)}\big)^3}\bigg) =
\frac{4\pi \widetilde{a}}{\big(\gamma_\lambda^{(1)}\big)^5}-\frac{1}{\gamma_\lambda^{(1)}} \tau_\lambda +o\bigg(\frac{1}{\big(\gamma_\lambda^{(1)}\big)^5}\bigg).
\end{equation}
Combined with \eqref{A_tilde1} and \eqref{A_tilde2}, we find
\begin{equation}\label{A_tilde3}
-6 \widetilde{A}_\lambda \left(\frac{1}{\big(\gamma_\lambda^{(1)}\big)^3}+o\bigg(\frac{1}{\big(\gamma_\lambda^{(1)}\big)^3}\bigg)\right)
= \frac{8\pi \widetilde{a}}{\big(\gamma_\lambda^{(1)}\big)^5}+o\bigg(\frac{1}{\big(\gamma_\lambda^{(1)}\big)^5}\bigg).
\end{equation}

Moreover, we have
\begin{equation}\label{5.10-16}
\begin{split}
 &\int_{B_d(x^{(1)}_{\lambda})} u^{(1)}_\lambda(x) E_\lambda(x) \eta_{\lambda}(x) \,dx
\\=&
\big(\theta^{(1)}_\lambda\big)^2\int_{B_{\frac{d}{\theta^{(1)}_\lambda}}(0)}
\Big(\gamma^{(1)}_\lambda+ \frac{w_\lambda^{(1)}(x)}{\gamma^{(1)}_\lambda}\Big)
E_\lambda\big(\theta^{(1)}_\lambda x+x_\lambda^{(1)}\big) \widetilde{\eta}_\lambda(x) \,dx\\=&
\gamma^{(1)}_\lambda \widetilde{A}_\lambda
+\frac{\big(\theta^{(1)}_\lambda\big)^2}{\gamma^{(1)}_\lambda}
\int_{B_{\frac{d}{\theta^{(1)}_\lambda}(0)}} w_\lambda^{(1)}(x) E_\lambda
 \big(\theta_\lambda^{(1)}x+x_\lambda^{(1)}\big) \widetilde{\eta}_\lambda(x) \,dx \\=&
 \gamma^{(1)}_\lambda \widetilde{A}_\lambda
+\frac{2 \widetilde{a} }{ \gamma^{(1)}_\lambda }
 \int_{\R^2} \frac{4-|x|^2}{4+|x|^2} e^{2U(x)}U(x) \,dx+ o\Big(\frac{1}{\gamma^{(1)}_\lambda}\Big) =
  \gamma^{(1)}_\lambda \widetilde{A}_\lambda +\frac{4\pi\widetilde{a} }{ \gamma^{(1)}_\lambda } + o\Big(\frac{1}{\gamma^{(1)}_\lambda}\Big).
\end{split}
\end{equation}
On the other hand, using \eqref{lhs-add}, we can compute that
\begin{align*}
&\int_{B_d (x^{(1)}_{\lambda})}  u^{(1)}_\lambda(x) E_\lambda(x) \eta_{\lambda}(x) \,dx \\
=& \int_{B_d (x^{(1)}_{\lambda})}  \lambda u^{(1)}_\lambda(x) e^{(u_\lambda^{(1)}(x))^2} \eta_{\lambda}(x) \,dx +o\Big(\theta_\lambda^{(1)} \sqrt{\lambda}\Big)\\
=& \frac{1}{\gamma_\lambda^{(1)}} \int_{B_{\frac{d}{\theta_\lambda^{(1)}}}(0)} e^{2w^{(1)}_\lambda(x)+\frac{(w^{(1)}_\lambda(x))^2}{( \gamma^{(1)}_\lambda)^2}} \bigg(\frac{w_\lambda^{(1)}(x)}{\big(\gamma_\lambda^{(1)}\big)^2}+1\bigg) \widetilde{\eta}_\lambda(x) \,dx +o\Big(\theta_\lambda^{(1)} \sqrt{\lambda}\Big) \\
=& \frac{\widetilde{a}}{\gamma_\lambda^{(1)}} \int_{\R^2} e^{2U(x)} \frac{4-|x|^2}{4+|x|^2} \,dx +o\Big(\frac{1}{\gamma^{(1)}_\lambda}\Big) = o\Big(\frac{1}{\gamma^{(1)}_\lambda}\Big),
\end{align*}
which together with \eqref{5.10-16} gives
\begin{equation}\label{A_tilde4}
\begin{split}
\widetilde{A}_\lambda = -\frac{4\pi\widetilde{a} }{ \big(\gamma^{(1)}_\lambda\big)^2 } + o\bigg(\frac{1}{\big(\gamma^{(1)}_\lambda\big)^2}\bigg).
\end{split}
\end{equation}
Substituting \eqref{A_tilde4} into \eqref{A_tilde3}, we derive
\begin{equation*}
\frac{16\pi\widetilde{a} }{ \big(\gamma^{(1)}_\lambda\big)^5} = o\bigg(\frac{1}{\big(\gamma_\lambda^{(1)}\big)^5}\bigg).
\end{equation*}
Therefore we prove that $\widetilde{a}=0$ and \eqref{A_tilde4} implies that $\widetilde{A}_\lambda=o(\lambda)$.
\end{proof}

 \vskip 0.1cm

\begin{proof}[\underline{\textbf{Proof of Theorem \ref{th_uniq}}}]
Suppose $u^{(1)}_\lambda\not\equiv u^{(2)}_\lambda$, and let $\eta_{\lambda}:=\frac{u^{(1)}_\lambda- u^{(2)}_\lambda}{
\|u^{(1)}_\lambda-u^{(2)}_\lambda\|_{L^\infty{(\Omega)}}}$. We have $\|\eta_{\lambda}\|_{L^\infty{(\Omega)}}=1$.
Taking $\widetilde{\eta}_{\lambda}(x):=
\eta_{\lambda}\big(\theta^{(1)}_{\lambda}x+x^{(1)}_{\lambda}\big)$, by Propositions \ref{prop_tildeeta}, \ref{prop-AB} and \ref{prop-A}, we have
\begin{equation}\label{loc_o}
\|\widetilde{\eta}_{\lambda}\|_{L^{\infty}(B_R(0))}=o\big(1\big)~\mbox{for~any}~R>0.
\end{equation}
Let $y_\lambda$ be a maximum point of $\eta_{\lambda}$ in
$\Omega$. We can assume that $\eta_{\lambda}(y_\lambda)=1$, then $y_\lambda\in \Omega\backslash  B_{R\theta^{(1)}_{\lambda}}(x^{(1)}_{\lambda})$.
Now we write
$$ \Omega\backslash  B_{R\theta^{(1)}_{\lambda}}\big(x^{(1)}_{\lambda}\big)=
\Big(\Omega\backslash B_{d}\big(x^{(1)}_{\lambda}\big)\Big) \bigcup
\Big(B_{d}\big(x^{(1)}_{\lambda}\big) \backslash B_{\frac{2 \theta^{(1)}_\lambda}{\lambda}}
\big(x^{(1)}_{\lambda}\big)\Big) \bigcup
\Big(B_{\frac{2 \theta^{(1)}_\lambda}{\lambda}}\big(x^{(1)}_{\lambda}\big) \backslash  B_{R\theta^{(1)}_{\lambda}}\big(x^{(1)}_{\lambda}\big)\Big),$$
and divide the proof into the following three steps.

\vskip 0.2cm

\noindent \textbf{Step 1. }We show that $y_\lambda\not\in \Omega\backslash B_{d}\big(x^{(1)}_{\lambda}\big)$.

It is enough to prove
\begin{equation}\label{t1.3-1}
 \eta_{\lambda}=o(1)~~\mbox{uniformly in}~~\Omega\backslash B_{d}\big(x^{(1)}_{\lambda}\big).
\end{equation}
By Green's representation theorem, \eqref{etalambda-equa} and \eqref{A_B}, we can deduce that
\begin{equation}\label{t1.3-1-1}
\begin{split}
{\eta}_\lambda(x)=& \widetilde{A}_{\lambda} G\big(x^{(1)}_{\lambda},x\big)  +o\big(\theta^{(1)}_{\lambda}\big),~\,\mbox{uniformly in}~\Omega\backslash  B_{d}\big(x^{(1)}_{\lambda}\big).
\end{split}
\end{equation}
Since $\widetilde{A}_{\lambda}=o\big(\lambda\big)$ by Proposition \ref{prop-A} and
we can calculate that
\begin{equation*}
\sup_{\Omega\backslash  B_{d}(x^{(1)}_{\lambda})}G\big(x^{(1)}_{\lambda},x\big) =
\sup_{\Omega\backslash  B_{d}(x^{(1)}_{\lambda})}\Big(-\frac{1}{2\pi}\log \big|x-x^{(1)}_{\lambda}\big|
-H\big(x,x^{(1)}_{\lambda}\big)\Big) =O\big(1\big),
\end{equation*}
hence \eqref{t1.3-1} is derived from \eqref{t1.3-1-1}.

\vskip 0.2cm

\noindent \textbf{Step  2. } We show that $y_\lambda \notin  \displaystyle B_{d}\big(x_\lambda^{(1)}\big) \backslash B_{ \frac{2 \theta^{(1)}_\lambda}{\lambda} }\big(x_\lambda^{(1)}\big)$.

We claim that
\begin{equation}\label{t1.3-2}
\eta_\lambda(x)=o\big(1\big)~~\mbox{for}~x\in  \displaystyle B_{d}\big(x_\lambda^{(1)}\big)\backslash B_{ \frac{2 \theta^{(1)}_\lambda}{\lambda} }\big(x_\lambda^{(1)}\big).
\end{equation}
By the Green's representation theorem and \eqref{5.7-2}, we get
\begin{equation}\label{t1.3-2-1}
\begin{split}
\eta_\lambda(x){=}& \int_{B_{d}(x_\lambda^{(1)})} G(y,x)
 E_\lambda(y) \eta_\lambda(y) \,dy+ \int_{\Omega\backslash  B_{d}(x_\lambda^{(1)})}G(y,x)
 E_\lambda(y)\eta_\lambda(y) \,dy\\=&    \int_{B_{ \frac{  \theta^{(1)}_\lambda}{\lambda}}(x_\lambda^{(1)})} G(y,x) E_\lambda(y) \eta_\lambda(y) \,dy\\&
 +  \int_{B_{d}(x_\lambda^{(1)})\backslash B_{ \frac{  \theta^{(1)}_\lambda}{\lambda}}(x_\lambda^{(1)})}G(y,x)
 E_\lambda(y)  \eta_\lambda(y) \,dy+ o\big(\theta^{(1)}_\lambda\big).
\end{split}
\end{equation}
By Taylor's expansion, we find
\begin{equation}\label{t1.3-2-2}
\begin{split}
& \int_{B_{ \frac{\theta^{(1)}_\lambda}{\lambda}}(x_\lambda^{(1)})}  G(y,x) E_\lambda(y) \eta_\lambda(y) \,dy \\
=& \int_{B_{ \frac{\theta^{(1)}_\lambda}{\lambda}}(x_\lambda^{(1)})} G\big(x_\lambda^{(1)},x\big) E_\lambda(y) \eta_\lambda(y) \,dy \\
& + O\bigg(\int_{B_{ \frac{\theta^{(1)}_\lambda}{\lambda}}(x_\lambda^{(1)})}
\big|y-x_\lambda^{(1)}\big| \Big|\nabla G\big((1-\xi)x_\lambda^{(1)}+\xi y,x\big)\Big| E_\lambda(y) \eta_\lambda(y) \,dy  \bigg) \\
=& \int_{ B_{\frac{\theta^{(1)}_\lambda}{\lambda}}(x_\lambda^{(1)})}  G\big(x_\lambda^{(1)},x\big) E_\lambda(y) \eta_\lambda(y) \,dy +O\bigg(\big(\theta_\lambda^{(1)}\big)^2 \int_{ B_{\frac{1}{\lambda}}(0)} \theta_\lambda^{(1)} |y| E_\lambda\big(\theta^{(1)}_\lambda y+x^{(1)}_\lambda) \,dy  \bigg) \\
=& G\big(x_\lambda^{(1)},x\big) \int_{ B_{ \frac{\theta^{(1)}_\lambda}{\lambda} }(x_\lambda^{(1)})} E_\lambda(y)\eta_\lambda(y) \,dy +O\big(\theta_\lambda^{(1)}\big), ~\,\,\mbox{where}~\xi \in (0,1).
\end{split}
\end{equation}
Furthermore, recalling the definitions of  $\widetilde{A}_{\lambda}$ and $w_{\lambda}$,
for $x\in  B_{d}\big(x_\lambda^{(1)}\big)\backslash  B_{ \frac{2\theta^{(1)}_\lambda}{\lambda}}\big(x_\lambda^{(1)}\big)$, we deduce from \eqref{Dlambda-appro} that
\begin{align}\label{t1.3-2-3}
& G\big(x_\lambda^{(1)},x\big) \int_{ B_{ \frac{\theta^{(1)}_\lambda}{\lambda} }(x_\lambda^{(1)})}
E_\lambda(y)\eta_\lambda(y) \,dy  \notag\\
=& G\big(x_\lambda^{(1)},x\big) \widetilde{A}_{\lambda}-G\big(x_\lambda^{(1)},x\big) \int_{ B_d(x_\lambda^{(1)})
\backslash B_{\frac{\theta^{(1)}_\lambda}{\lambda}}(x_\lambda^{(1)})} E_\lambda(y) \eta_\lambda(y) \,dy \notag\\
=& G\big(x_\lambda^{(1)},x\big) \widetilde{A}_{\lambda} +O\bigg(\Big|\log \frac{\theta^{(1)}_\lambda}{\lambda} \Big|  \int^{ \frac{d}{\theta^{(1)}_\lambda}}_{\frac{1}{\lambda} }\frac{1}{r^{3-\delta}} \,dr \bigg) \notag\\
=& G\big(x_\lambda^{(1)},x\big) \widetilde{A}_{\lambda}+O\bigg(\lambda^{2-\delta}
\Big|\log \frac{\theta^{(1)}_\lambda}{\lambda} \Big| \bigg)
=G\big(x_\lambda^{(1)},x\big) \widetilde{A}_{\lambda}+O\Big(\lambda^{2-\delta}
\big|\log \theta^{(1)}_\lambda \big| \Big).
\end{align}
Hence \eqref{t1.3-2-2} and \eqref{t1.3-2-3} imply that for $x\in   B_{d}\big(x_\lambda^{(1)}\big)\backslash  B_{ \frac{2\theta^{(1)}_\lambda}{\lambda}}\big(x_\lambda^{(1)}\big)$, it holds
\begin{equation}\label{t1.3-2-4}
\begin{split}
 \int_{B_{ \frac{\theta^{(1)}_\lambda}{\lambda}}(x_\lambda^{(1)})} G(y,x)
E_\lambda(y) \eta_\lambda(y) \,dy=G\big(x_\lambda^{(1)},x\big) \widetilde{A}_{\lambda}+O\Big(\lambda^{2-\delta}
\big|\log \theta^{(1)}_\lambda \big| \Big).
\end{split}
\end{equation}

Next we estimate the second term of \eqref{t1.3-2-1}.
\begin{align}\label{t1.3-2-5}
&\int_{B_{d}(x_\lambda^{(1)})\backslash  B_{ \frac{\theta^{(1)}_\lambda}{\lambda}}(x_\lambda^{(1)})}  G(y,x)
E_\lambda(y)  \eta_\lambda(y) \,dy \notag\\=&~
O\Bigg(\int_{B_{\frac{d}{\theta^{(1)}_\lambda}}(0)\backslash  B_{\frac{1}{\lambda}}(0)}\frac{ G\big(
\theta^{(1)}_\lambda y+x_\lambda^{(1)},x\big) }{1+|y|^{4-\delta}}  \,dy \Bigg) \notag\\=&~
O\Bigg( \int_{B_{\frac{d}{\theta^{(1)}_\lambda}}(0)\backslash  B_{\frac{1}{\lambda}}(0)} \frac{\big|\log \theta^{(1)}_\lambda\big|+ \Big|\log \big|y+\frac{x_\lambda^{(1)}-x}{\theta^{(1)}_\lambda}\big|\Big|  }
{1+|y|^{4-\delta}} \,dy\Bigg) \notag\\=&~
O\Bigg( \int_{B_{\frac{d}{\theta^{(1)}_\lambda}}(0)\backslash  B_{\frac{1}{\lambda}}(0)} \frac{\Big| \log \big |y+\frac{x_\lambda^{(1)}-x}{\theta^{(1)}_\lambda}\big|\Big|  }{1+|y|^{4-\delta}} \,dy\Bigg)+O\Big(
 \lambda^{2-\delta} \big|\log \theta^{(1)}_\lambda\big| \Big).
\end{align}
Furthermore, using H\"older's inequality, we have
\begin{equation*}
\begin{split}
& \int_{B_{\frac{d}{\theta^{(1)}_\lambda}}(0)\backslash  B_{\frac{1}{\lambda}}(0)} \frac{ \Big|\log \big |y+\frac{x_\lambda^{(1)}-x}{\theta^{(1)}_\lambda}\big|\Big| }{1+|y|^{4-\delta}} \,dy \\
=& O\Bigg( \bigg(\int_{B_{\frac{d}{\theta^{(1)}_\lambda}}(0)\backslash  B_{\frac{1}{\lambda}}(0)}  \Big|\log \big |y+\frac{x_\lambda^{(1)}-x}{\theta^{(1)}_\lambda} \big|\Big|^{q}  \,dy \bigg)^{\frac{1}{q}}  \bigg(
\int_{B_{\frac{d}{\theta^{(1)}_\lambda}}(0)\backslash  B_{\frac{1}{\lambda}}(0)} \frac{1}
{\big(1+|y|\big)^{\frac{(4-\delta)q}{q-1}}} \,dy \bigg)^{\frac{q-1}{q}}  \Bigg) \\
=& O\left(\lambda^{\big((4-\delta)\frac{q}{q-1}-2\big)\frac{q-1}{q}}  \big(\theta^{(1)}_\lambda\big)^{-\frac{2}{q}} \big|\log \theta^{(1)}_\lambda \big|   \right),~\mbox{for~any~fixed~large}~q.
\end{split}
\end{equation*}
Let $q=\frac{1}{\lambda}$, one can derive from \eqref{lambda_gamma} that
\begin{equation}\label{t1.3-2-6}
\int_{B_{\frac{d}{\theta^{(1)}_\lambda}}(0)\backslash  B_{\frac{1}{\lambda}}(0)} \frac{\Big|\log \big |y+\frac{x_\lambda^{(1)}-x}{\theta^{(1)}_\lambda} \big|\Big| }{1+|y|^{4-\delta}} \,dy
= O\Big( \lambda^{2-\delta}  \big|\log \theta_\lambda^{(1)}\big| \Big).
\end{equation}
By definition, it is easy to see that $\big|\log \theta^{(1)}_\lambda\big|=O\big(\frac{1}{\lambda}\big)$, then it follows from \eqref{t1.3-2-1}, \eqref{t1.3-2-4}, \eqref{t1.3-2-5} and \eqref{t1.3-2-6} that
\begin{equation}\label{t1.3-2-7}
\begin{split}
\eta_\lambda(x)=& G\big(x_\lambda^{(1)},x\big) \widetilde{A}_{\lambda}+O\big(\lambda^{1-\delta} \big),~\mbox{uniformly in}~ B_{d}\big(x_\lambda^{(1)}\big)\backslash  B_{ \frac{2\theta^{(1)}_\lambda}{\lambda}}\big(x_\lambda^{(1)}\big).
\end{split}
\end{equation}
Finally, since
$\widetilde{A}_{\lambda}=o(\lambda)$ and
\begin{equation*}
\begin{split}
\sup_{B_{d}(x_\lambda^{(1)}) \backslash  B_{\frac{2\theta^{(1)}_\lambda}{\lambda}}(x_\lambda^{(1)})} G\big(x^{(1)}_{\lambda},x\big)
=& \sup_{B_{d}(x_\lambda^{(1)})\backslash  B_{\frac{2\theta^{(1)}_\lambda}{\lambda}}(x_\lambda^{(1)})}
\Big(-\frac{1}{2\pi}\log \big|x-x^{(1)}_{\lambda}\big|-H\big(x,x^{(1)}_{\lambda}\big)\Big) \\
=& O\Big(\big|\log \theta^{(1)}_\lambda\big|\Big)=O\Big(\frac{1}{\lambda}\Big),
\end{split}
\end{equation*}
we find \eqref{t1.3-2} can be deduced from \eqref{t1.3-2-7}.

\vskip 0.2cm

\noindent \textbf{Step  3.} By the above analysis, we have
$$y_\lambda\in B_{\frac{2\theta^{(1)}_\lambda}{\lambda}}\big(x_\lambda^{(1)}\big) \backslash  B_{R\theta^{(1)}_\lambda}\big(x_\lambda^{(1)}\big)~~\mbox{and}~~r_\lambda:=|y_\lambda|.$$
By translation, we assume that  $x_\lambda^{(1)}=0$, then $\frac{r_\lambda}{\theta^{(1)}_\lambda}\geq R\gg 1$. Set
\begin{equation}\label{t1.3-3-1}
\widehat{\eta}_\lambda(y):=\eta_\lambda(r_\lambda y),~y \in \Omega_{r_\lambda}:=\big\{x:\,r_\lambda x\in\Omega \big\},
\end{equation}
then $\|\widehat{\eta}_\lambda\|_{L^\infty(\Omega_{r_\lambda})}=1$, and for any $y\in \Omega_{r_\lambda}\backslash \{0\} $,
\begin{equation*}
\begin{split}
-\Delta \widehat{\eta}_\lambda (y)=&r^2_\lambda E_\lambda (r_\lambda y) \widehat{\eta}_\lambda(y)
\le C \Big(\frac{r_\lambda }{\theta^{(1)}_\lambda}\Big)^2 \frac{1}{\Big(1+\big|\frac{r_\lambda y} {\theta^{(1)}_\lambda}\big|\Big)^{4-\delta}}\rightarrow 0.
\end{split}\end{equation*}
Hence, by standard elliptic regularity, there exists a bounded function $\eta$ such that
\[
\widehat{\eta}_\lambda \rightarrow \eta~\mbox{in}~C\big(B_R(0)\backslash B_\tau(0)\big),~~\mbox{for~any}~R,\tau>0,
\]
and
\[
\Delta \eta=0~~\mbox{in}~\R^2 \backslash \{0\}.
\]
Since $\eta$ is a bounded function, direct computation gives $\eta=\widetilde{C}$, and we have $\widetilde{C}=1$ since $\widehat{\eta}_\lambda(\frac{y_\lambda}{r_\lambda})=\eta_\lambda(y_\lambda)=1$. Hence, we derive that
\begin{equation}\label{t1.3-3-2}
\widehat{\eta}_\lambda\rightarrow 1~\mbox{in}~C\big(B_R(0)\backslash B_\tau(0)\big).
\end{equation}

Now we know that $\widetilde{\eta}_{\lambda}(y):=\eta_\lambda\big(\theta^{(1)}_\lambda y\big)$ solves
\begin{equation*}
-\Delta \widetilde{\eta}_{\lambda}(y)-2e^{2U(y)}\widetilde{\eta}_{\lambda}(y) =
\widetilde{f}_\lambda(y):=\Big(\big(\theta_\lambda^{(1)}\big)^2 E_\lambda\big(\theta^{(1)}_\lambda y\big)
-2e^{2U(y)}\Big)\widetilde{\eta}_{\lambda}(y).
\end{equation*}
Furthermore, according to the proof of Lemma \ref{theta-D} and the expansion of $w_\lambda$, using Lemma \ref{lem2.4} and Proposition \ref{prop4.2}, we can verify that
\begin{equation*}
|\widetilde{f}_\lambda(y)|\leq C \lambda \frac{1}{\big(1+|y|\big)^{4-\delta-\tau_1/2}},~\mbox{for~some~small~fixed}~\tau_1 \in (0,1).
\end{equation*}
The average of $\widetilde{\eta}_\lambda$, denoted by
$\eta_\lambda^*(r):=\displaystyle\int^{2\pi}_0 \widetilde{\eta}_\lambda(r,\theta) \,d\theta$, solves the ODE
\begin{equation*}
-\left(\eta_\lambda^*\right)''-\frac{1}{r}\left(\eta_\lambda^*\right)'-2e^{2U}\eta_\lambda^*=
\int^{2\pi}_0 \widetilde{f}_\lambda(r,\theta)d\theta:=f^*_\lambda(r).
\end{equation*}
Note that $u_0(r):=\frac{4-r^2}{4+r^2}$ is a bounded solution of
\begin{equation}\label{t1.3-3-3}
-v''(r)-\frac{1}{r}v'(r)-2e^{2U(r)}v(r)=0.
\end{equation}
Then by ODE's theory, the second solution of \eqref{t1.3-3-3} is given by
\begin{equation*}
u_1(r):=\frac{(4-r^2)\log r+8}{r^2+4},
\end{equation*}
and  the function
\begin{equation*}
V_\lambda(r):=u_0(r)\int^r_0su_1(s)f_\lambda^*(s)ds-u_1(r)\int^r_0su_0(s)f_\lambda^*(s)ds
\end{equation*}
is a solution of $$-v''(r)-\frac{1}{r}v'(r)-2e^{2U(r)}v(r)=f_\lambda^*(r).$$
For detailed calculations, readers can refer to the appendix in \cite{CPY2021}.
Therefore, there exist two constants $C_{0,\lambda}$, $C_{1,\lambda}\in\mathbb R$ such that
\begin{equation*}
\eta^*_\lambda(r)=C_{0,\lambda}u_0(r)+C_{1,\lambda}u_1(r)+V_\lambda(r).
\end{equation*}
Since $\eta^*_\lambda$ must be bounded at $r=0$ and $V_\lambda(0)=0$, then we have that $C_{1,\lambda}=0$. Moreover, since $u_0(0)=1$, $V_\lambda(0)=0$ and
\begin{equation*}
\eta^*_\lambda(0)=\displaystyle\int^{2\pi}_0 \widetilde{\eta}_\lambda(0,\theta) \,d\theta \overset{\eqref{loc_o}}=o(1).
\end{equation*}
Hence we get $C_{0,\lambda}=o(1)$ and then
 we find
\begin{equation}\label{t1.3-3-4}
\begin{split}
|\eta^*_\lambda(r)| =&~ |V_\lambda(r)|+o(1) \\
\leq &~ {C}\lambda\int^r_0 \frac{s\log s}{(1+s)^{4-\delta-\frac{\tau_1}{2}}} \,ds +C\lambda \int^r_0  \frac{s}{(1+s)^{4-\delta-\frac{\tau_1}{2}}} \,ds \\
&~ + C\lambda \log r \int^r_0  \frac{s}{(1+s)^{4-\delta-\frac{\tau_1}{2}}} \,ds +o(1) \\
\leq &~ {C\lambda\log r} +o(1).
\end{split}
\end{equation}
Furthermore, since $y_\lambda\in B_{\frac{2\theta^{(1)}_\lambda}{\lambda}}(0) \backslash  B_{R\theta^{(1)}_\lambda}(0)$, we have $\frac{r_\lambda}{\theta^{(1)}_\lambda} \leq \frac{2}{\lambda}$ and it follows from \eqref{t1.3-3-4} that
\begin{equation}\label{t1.3-3-5}
\eta^*_\lambda\Big(\frac{r_\lambda}{\theta^{(1)}_\lambda}\Big)\leq  C\lambda \log \frac{r_\lambda}{\theta^{(1)}_\lambda} +o(1) \leq C\lambda \log \frac{2}{\lambda} +o(1) \to 0.
\end{equation}

On the other hand, by \eqref{t1.3-3-2}, we know
\begin{equation}\label{t1.3-3-6}
\widehat{\eta}_\lambda(x)\to 1,~~\mbox{for any}~~|x|=1.
\end{equation}
Then by \eqref{t1.3-3-1} and \eqref{t1.3-3-6}, we find
\begin{equation*}
\begin{split}
\eta^*_\lambda\Big(\frac{r_\lambda}{\theta^{(1)}_\lambda}\Big) =&
\int^{2\pi}_0 \widetilde{\eta}_\lambda \Big(\frac{r_\lambda}{\theta^{(1)}_\lambda},\theta\Big) \,d\theta =
\int^{2\pi}_0 \eta_\lambda(r_\lambda,\theta) \,d\theta
\\=&\int^{2\pi}_0 \widehat{\eta}_\lambda(1,\theta) \,d\theta  \geq c_0>0,~\mbox{for some constant}~c_0,
\end{split}\end{equation*}
which is a contradiction to \eqref{t1.3-3-5}. Hence we complete the proof of Theorem \ref{th_uniq}.
\end{proof}

\begin{proof}[\underline{\textbf{Proof of Theorem \ref{th_uniq2}}}]
Suppose that $u_\lambda$ is a solution to problem \eqref{1.1} with
\begin{equation}\label{energy-fty}
\limsup_{\lambda \to 0} J_\lambda(u_\lambda) < \infty,
\end{equation}
then it follows from \cite{DT2020} that there exists $k \geq 1$ such that
\begin{equation*}
J_\lambda(u_\lambda) \to 2k\pi, ~ \mbox{as}~ \lambda \to 0,
\end{equation*}
and $u_\lambda$ is a $k$-peak solution concentrated at $k$ different points $\{x_1,\cdots,x_k\} \subset \Omega$. In addition, $x:=(x_1,\cdots,x_k)\in \Omega^k$ is a critical point of the Kirchoff-Routh function. On the other hand, it is shown in \cite{GT2010} that the Kirchoff-Routh function has no critical point for $k\geq 2$ if $\Omega \subset \R^2$ is a convex domain. Then we can deduce that any solution satisfying \eqref{energy-fty} must be a single-peak solution concentrated at the critical point of the Robin function. Besides, since $\Omega$ is a bounded convex domain in $\R^2$,
we obtain from \cite{CF1985} that the Robin function $\mathcal{R}(x)$ is strictly convex and it has a unique critical point which is non-degenerate. Therefore, by Theorem \ref{th_uniq}, we prove that there exists $\lambda_0> 0$ such that \eqref{1.1} admits a unique solution for any $\lambda \in (0,\lambda_0)$.

\vskip 0.1cm

Furthermore, if $\Omega$ is also symmetric with respect to $x_1,x_2$, it is easy to verify that $u_\lambda(-x_1,x_2)$ and $u_\lambda(x_1,-x_2)$ are also the solutions to problem \eqref{1.1}. Using the above proved uniqueness result, we know that $u_\lambda(x) =u_\lambda(-x_1,x_2)= u_\lambda(x_1,-x_2)$ for $\lambda$ small. Then, the fact that $u_\lambda$ is even in $x_1,x_2$ implies that $$u_\lambda(0)=\max\limits_{x\in \Omega} u_\lambda(x).$$

\end{proof}

\appendix

\section{Proof of Proposition \ref{prop_klambda}}\label{s-k}

\renewcommand{\theequation}{A.\arabic{equation}}
\setcounter{equation}{0}

\vskip 0.2cm

 Firstly, let $M_\lambda:=\displaystyle\max_{|x|\leq \frac{d_0}{\theta_{\lambda}}}\frac{|k_{\lambda}(x)|}{(1+|x|)^{\tau_1}}$,
then $$\eqref{k-inequa} \Leftrightarrow M_\lambda \leq C.$$
We will prove that $M_\lambda\leq C$ by contradiction.  Set
\begin{equation*}
M_\lambda^*:=\max_{|x|\leq \frac{d_0}{\theta_{\lambda}}}\max_{|x'|=|x|}\frac{|k_{\lambda}(x)-k_{\lambda}(x')|}{(1+|x|)^{{\tau_1}}}.
\end{equation*}

\begin{Lem}\label{lemB.1}
If $M_\lambda\to +\infty$, then it holds
\begin{equation}\label{B.1.1}
M_\lambda^*=o(1)M_\lambda.
\end{equation}
\end{Lem}

\begin{proof}
Suppose this is not true, then there exists $c_0>0$ such that $M^*_\lambda\geq c_0M_\lambda$. Let $y'_\lambda$ and
$y''_\lambda$ satisfy $|y'_\lambda|=|y''_\lambda| {\leq\frac{d_{0}}{\theta_{\lambda}}}$ and
\begin{equation*}
M_\lambda^*=  \frac{|k_{\lambda}(y'_\lambda)-k_{\lambda}(y''_\lambda)|}{(1+|y'_\lambda|)^{{\tau_1}}}.
\end{equation*}
Without loss of generality, we may assume that $y'_\lambda$ and
$y''_\lambda$ are symmetric with respect to the $x_1$ axis. Set
\begin{equation*}
l^*_\lambda(x):=k_{\lambda}(x)-k_{\lambda}(x^-),~\,~x^-:=(x_1,-x_2),\  \mbox{ for }x=(x_{1},x_{2}),  ~\,x_2>0.
\end{equation*}
 Hence $y''_{\lambda}={y'_{\lambda}}^{-}$ and
$M_\lambda^*=\frac{ |l^*_\lambda(y'_{\lambda})| }{ (1+|y'_{\lambda}| )^{\tau_1} }$.
Let us define
\begin{equation}\label{llambda-def}
l_\lambda(x):=\frac{l^*_\lambda(x)}{(1+x_2)^{\tau_1}}.
\end{equation}
Then direct calculations show that $l_\lambda$ satisfies
\begin{equation}\label{llambda-equa}
-\Delta l_\lambda-\frac{2{\tau_1}}{1+x_2}\frac{\partial l_\lambda}{\partial x_2}+
\frac{{\tau_1}(1-{\tau_1})}{(1+x_2)^2}l_\lambda= 2 l_\lambda e^{2U(x)} +\frac{h_\lambda(x)}{(1+x_2)^{\tau_1}} + \frac{\widetilde{h}_\lambda(x)}{(1+x_2)^{\tau_1}},
\end{equation}
where $h_\lambda(x):=h^*_\lambda(x)-h^*_\lambda(x^-)$ and $\widetilde{h}_\lambda(x):=h^{**}_\lambda(x)-h^{**}_\lambda(x^-)$. $h^*_\lambda$ and $h^{**}_\lambda$ are defined in \eqref{h-ast} and \eqref{h-ast2} respectively.
Also let $y^{**}_\lambda$ satisfies
\begin{equation}\label{M_ast2-def}
\mbox{$|y^{**}_\lambda|\leq\frac{d_{0}}{\theta_{\lambda}}$, $y^{**}_{\lambda,2}\geq 0$ and}~~~ \big|l_\lambda(y^{**}_\lambda)\big|= M^{**}_\lambda:=\max_{|x|\leq \frac{d_0}{\theta_{\lambda}},~x_2\geq 0} |l_\lambda(x)|.
\end{equation}
Then it follows
\begin{equation}\label{M_ast2-2}
M^{**}_\lambda \geq \frac{\big|l^*_\lambda(y'_\lambda)\big|}{\big(1+y'_{\lambda,2}\big)^{\tau_1}}
\geq \frac{\big|k_{\lambda}(y'_\lambda)-k_{\lambda}(y''_\lambda)\big|}{\big(1+|y'_\lambda|\big)^{\tau_1}}=M^*_\lambda\to +\infty.
\end{equation}
We may assume that $l_\lambda(y^{**}_\lambda)>0$. We claim that
\begin{equation}\label{y_ast2-bound}
|y^{**}_\lambda|\leq C.
\end{equation}
The proof of  \eqref{y_ast2-bound} is divided into   two steps.

\vskip 0.2cm

\noindent
\emph{Step 1. We prove that $|y^{**}_\lambda|\leq \frac{d_0}{2\theta_\lambda}$.}

\vskip 0.2cm
Suppose this is not true. By the definition of $k_{\lambda}$ and \eqref{4.2-part1-5}, we have
\begin{equation*}
\begin{split}
k_{\lambda}(y^{**}_\lambda)-k_{\lambda}(y^{**-}_\lambda)
&= \gamma^2_\lambda\Big(\big(v_{\lambda}(y^{**}_\lambda)-v_{\lambda}(y^{**-}_\lambda)\big)
-\big(v_0(y^{**}_\lambda)-v_0(y^{**-}_\lambda)\big)\Big) \\
&= O\big( \gamma^4_\lambda \theta_\lambda y^{**}_{\lambda,2} \big)+o\big( \gamma^4_\lambda \theta_{\lambda}\big)-\gamma^2_\lambda \big(v_0(y^{**}_\lambda)-v_0(y^{**-}_\lambda)\big).
\end{split}
\end{equation*}
Also we recall that
\begin{equation*}
\begin{split}
v_0(x)=w(x)+ c_0\frac{4-|x|^2}{4+|x|^2}+\sum^2_{i=1}{c_i}\frac{x_i}{4+|x|^2},
\end{split}\end{equation*}
where $w(x)$ is the radial solution of $-\Delta u-2e^{2U}u=e^{2U}\big(U^2+U\big)$. Then
\begin{equation*}
\begin{split}
v_0(y^{**}_\lambda)-v_0(y^{**-}_\lambda) =& \sum^2_{i=1}{c_i}\frac{y^{**}_{\lambda,i}}{4+|y^{**}_\lambda|^2}-
\sum^2_{i=1}{c_i}\frac{y^{**-}_{\lambda,i}}{4+|y^{**-}_\lambda|^2}
=O\Big(\big|y^{**}_\lambda \big|^{-1}\Big).
\end{split}\end{equation*}
As a result,
\begin{equation}\label{klambda-}
l^*_\lambda(y_\lambda^{**})=k_{\lambda}(y^{**}_\lambda)-k_{\lambda}(y^{**-}_\lambda)
=O\big( \gamma^4_\lambda \theta_\lambda y^{**}_{\lambda,2} \big)+o\big( \gamma^4_\lambda \theta_{\lambda}\big)+O\Big(\gamma^2_\lambda \big|y^{**}_\lambda \big|^{-1}\Big),
\end{equation}
which implies that
\begin{equation*}
\begin{split}
M^{**}_\lambda &= \frac{l^*_\lambda(y_\lambda^{**})}{\big(1+y^{**}_{\lambda,2}\big)^{\tau_1}}
= O\Big( \gamma^4_\lambda  \theta_\lambda \big(y^{**}_{\lambda,2}\big)^{1-\tau_1} \Big) +o\big( \gamma^4_\lambda  \theta_\lambda \big) + O\Big(\gamma^2_\lambda \big|y^{**}_{\lambda}\big|^{-1}\Big) \\
&=  O\big( \gamma^4_\lambda  \theta_\lambda^{\tau_1} \big) + o\big( \gamma^4_\lambda  \theta_\lambda \big)
+O\big(\gamma^2_\lambda \theta_\lambda \big)=o(1),
\end{split}
\end{equation*}
since $\frac{d_0}{2\theta_\lambda}<|y_\lambda^{**}| \leq \frac{d_0}{\theta_\lambda}$.
This is a contradiction with \eqref{M_ast2-2}.

\vskip 0.2cm

\noindent
\emph{Step 2. We prove that $|y^{**}_\lambda|\leq C$.}

\vskip 0.2cm
 Suppose this is not true.
Now by Step 1, we have $y^{**}_\lambda\in B_{\frac{d_0}{2\theta_\lambda}}(0)$. Thus
\begin{equation}\label{y-sdsd}
|\nabla l_\lambda(y^{**}_\lambda)|=0~\mbox{and}~\Delta l_\lambda(y^{**}_\lambda)\leq 0.
\end{equation}
So from \eqref{llambda-equa}  and \eqref{y-sdsd}, we find
\begin{equation}\label{y-sdsd1}
\begin{split}
0&\leq -\Delta l_\lambda(y^{**}_\lambda) = -\frac{\tau_1(1-\tau_1)}{(1+y^{**}_{\lambda,2})^2}l_\lambda(y^{**}_\lambda)
+\frac{h_\lambda(y^{**}_\lambda)}{(1+y^{**}_{\lambda,2})^{\tau_1}} +\frac{\widetilde{h}_\lambda(y^{**}_\lambda)}{(1+y^{**}_{\lambda,2})^{\tau_1}}
+2 l_\lambda(y^{**}_\lambda)e^{2U(y^{**}_\lambda)} \\
&\leq -\frac{\tau_1(1-\tau_1)}{(1+y^{**}_{\lambda,2})^2}l_\lambda(y^{**}_\lambda)
+\frac{h_\lambda(y^{**}_\lambda)}{(1+y^{**}_{\lambda,2})^{\tau_1}} +\frac{\widetilde{h}_\lambda(y^{**}_\lambda)}{(1+y^{**}_{\lambda,2})^{\tau_1}}
+\frac{\tau_1(1-\tau_1)}{2(1+y^{**}_{\lambda,2})^2} l_\lambda(y^{**}_\lambda) \\
&= -\frac{\tau_1(1-\tau_1)}{2(1+y^{**}_{\lambda,2})^2}l_\lambda(y^{**}_\lambda)
+\frac{h_\lambda(y^{**}_\lambda)}{(1+y^{**}_{\lambda,2})^{\tau_1}} +\frac{\widetilde{h}_\lambda(y^{**}_\lambda)}{(1+y^{**}_{\lambda,2})^{\tau_1}},
\end{split}
\end{equation}
here the inequality holds since $|y^{**}_\lambda|\to \infty$.
Hence using \eqref{y-sdsd1}, we deduce that
\begin{equation*}
\frac{l_\lambda(y^{**}_\lambda)}{(1+y^{**}_{\lambda,2})^2} \leq
C \bigg( \frac{h_\lambda(y^{**}_\lambda)}{(1+y^{**}_{\lambda,2})^{\tau_1}} +
\frac{\widetilde{h}_\lambda(y^{**}_\lambda)}{(1+y^{**}_{\lambda,2})^{\tau_1}} \bigg),
\end{equation*}
which implies
\begin{equation}\label{y-sdsd2}
l_\lambda(y^{**}_\lambda) \leq C \Big( h_\lambda(y^{**}_\lambda)\big(1+y^{**}_{\lambda,2}\big)^{2-\tau_1} + \widetilde{h}_\lambda(y^{**}_\lambda)\big(1+y^{**}_{\lambda,2}\big)^{2-\tau_1} \Big).
\end{equation}
Combined with \eqref{h-ast}, \eqref{h-ast2} and \eqref{y-sdsd2}, the fact that $|y^{**}_\lambda|\to \infty$ indicates
\begin{equation*}
M^{**}_\lambda = l_\lambda(y^{**}_\lambda) \leq
C \bigg( \frac{1}{(1+|y_\lambda^{**}|)^{2+\tau_1-\delta-2\tau}} +
\frac{1}{\gamma_\lambda^2 (1+|y_\lambda^{**}|)^{2+\tau_1-\delta-3\tau}} \bigg)
\to 0~~\mbox{as}~~\lambda\to 0,
\end{equation*}
which is a contradiction with \eqref{M_ast2-2}. Hence \eqref{y_ast2-bound} follows.\\

Now let $l_\lambda^{**}(x)=\frac{l_\lambda(x)}{M_\lambda^{**}}$,
where $l_\lambda$ is defined in \eqref{llambda-def} and $M_\lambda^{**}$ is defined in \eqref{M_ast2-def}.
From \eqref{llambda-equa}, it follows that $l_\lambda^{**}(x)$ solves
\begin{equation}\label{l_ast2-equa}
-\Delta l^{**}_\lambda-\frac{2\tau_1}{1+x_2}\frac{\partial l^{**}_\lambda}{\partial x_2}+
\frac{\tau_1(1-\tau_1)}{(1+x_2)^2}l^{**}_\lambda = 2 l^{**}_\lambda e^{2U(x)}
+\frac{h_\lambda(x)}{M_\lambda^{**}(1+x_2)^{\tau_1}}+ \frac{\widetilde{h}_\lambda(x)}{M_\lambda^{**}(1+x_2)^{\tau_1}}.
\end{equation}
Moreover $l^{**}_{\lambda}(y^{**}_\lambda)=1$ and
\[
|l^{**}_{\lambda}(x)|\leq 1~ \mbox{in}~ \Big\{x:|x|\leq\frac{d_0}{\theta_\lambda},~x_2\geq 0\Big\},
\]
hence $l^{**}_\lambda(x)\to \gamma(x)$ uniformly in any compact subset of $\R^2$. And $\gamma(x)\not\equiv0$ because $l_\lambda^{**}(y_{\lambda}^{**})=1$ and  $|y_{\lambda}^{**}|\le C$. Observe that
$$\frac{h_\lambda(x)}{M_\lambda^{**}(1+x_2)^{\tau_1}}\to 0~\mbox{uniformly in any compact set of}~\Big\{x: |x|\leq \frac{d_0}{\theta_\lambda},~x_2 \geq 0\Big\}~\mbox{as}~\lambda\to 0,$$
and
$$\frac{\widetilde{h}_\lambda(x)}{M_\lambda^{**}(1+x_2)^{\tau_1}}\to 0~\mbox{uniformly in any compact set of}~\Big\{x: |x|\leq \frac{d_0}{\theta_\lambda},~x_2 \geq 0\Big\}~\mbox{as}~\lambda\to 0.$$
Then passing to the limit $\lambda\to 0$ into \eqref{l_ast2-equa}, we can deduce that   $\gamma$ solves
\begin{equation*}
\begin{cases}
\displaystyle -\Delta \gamma -\frac{2\tau_1}{1+x_2}\frac{\partial \gamma}{\partial x_2}
+\Big(\frac{\tau_1(1-\tau_1)}{(1+x_2)^2}-2e^{2U(x)} \Big)\gamma =0,~x_2>0,\\[3mm]
\displaystyle \gamma(x_1,0)=0.
\end{cases}
\end{equation*}
Hence $\bar \gamma=(1+x_2)^{\tau_1}\gamma$ satisfies
\begin{equation*}
\begin{cases}
-\Delta \bar \gamma-2e^{2U} \bar \gamma =0,~x_2>0,\\[2mm]
\bar\gamma(x_1,0)=0.
\end{cases}
\end{equation*}
Now using Lemma \ref{lem3.1}, we have
\begin{equation}\label{gamma_bar-equa}
\bar \gamma(x)=e_0\frac{4-|x|^2}{4+|x|^2}+\sum^2_{i=1}{e_i}\frac{x_i}{4+|x|^2},~~
\mbox{for some constants $e_0$, $e_1$ and $e_2$}.
\end{equation}
On the other hand, by the definition of $k_{\lambda}(x)$ and $l^*_\lambda(x)$, it follows from \eqref{omega_0} that
\begin{equation}\label{add_k0}
\begin{split}
\nabla l_\lambda^*(0)=&\nabla \big( k_{\lambda}(x)-k_{\lambda}(x^-)\big)\big|_{x=0}\\=&
\gamma^2_\lambda \big(\nabla v_{\lambda}(x)-\nabla v_0(x)\big)\big|_{x=0}-
\gamma^2_\lambda \big(\nabla v_{\lambda}(x^-)-\nabla v_0(x^-)\big)\big|_{x=0} \\=& -\gamma^2_\lambda \nabla  \big( v_0(x)- v_0(x^-)\big)\big|_{x=0}.
\end{split}
\end{equation}
Also we recall that
\[
\lim_{\lambda \to 0} v_\lambda =v_0~~\mbox{in}~~C_{loc}^2(\R^2),
\]
hence by \eqref{omega_0}, we can verify that
\begin{equation}\label{add_k01}
 \nabla \big(v_0(x)-v_0(x^-)\big)\big|_{x=0}=0.
\end{equation}
Therefore,  it holds
\begin{equation*}
\nabla \big((1+x_2)^{\tau_1}l_\lambda^{**}(x)\big)\big|_{x=0}=
\frac{\nabla l_\lambda^*(0)}{M_\lambda^{**}}= 0,
\end{equation*}
which means $\nabla \bar\gamma(0)=0$. Then from \eqref{gamma_bar-equa}, we find $e_1=e_2=0$.
Moreover from $\bar\gamma(x_1,0)=0$, we also have $e_0=0$. So $\bar\gamma=0$, which is a contradiction.
Hence \eqref{B.1.1} follows.
\end{proof}

 \vskip 0.1cm

\begin{proof}[\underline{\textbf{Proof of Proposition \ref{prop_klambda}}}] First we recall that
$$\eqref{k-inequa} \Leftrightarrow M_\lambda \leq C.$$
Suppose by contradiction that $M_\lambda\to +\infty$, as $\lambda\to 0$ and set
\begin{equation*}
 \varphi_{\lambda}(r)=\frac{1}{2\pi}\int^{2\pi}_0k_{\lambda}(r,\theta) \,d\theta,~\,\,r=|x|.
\end{equation*}
Then by Lemma \ref{lemB.1}, similar to the proof of \eqref{psi-lamda}, it follows
\begin{equation*}
\max_{r\leq \frac{d_0}{\theta_{\lambda}}}\frac{|\varphi_{\lambda}(r)|}{(1+r)^{\tau_1}}=M_\lambda\big(1+o(1)\big).
\end{equation*}
Assume that $\frac{|\varphi_{\lambda}(r)|}{(1+r)^{\tau_1}}$ attains its maximum at $\widehat{r}_\lambda$.
Then we claim
\begin{equation}\label{hat-r}
\widehat{r}_\lambda\leq C.
\end{equation}
In fact, let $\phi(x)=\frac{4-|x|^2}{4+|x|^2}$ and  we recall that
$-\Delta\phi(x)=2e^{2U(x)}\phi(x)$.
Now multiplying \eqref{k-equa} by $\phi(x)$ and using integration by parts, it follows from \eqref{h-ast} and \eqref{h-ast2} that
\begin{equation*}
\begin{split}
\int_{|x|=r} \Big(\frac{\partial k_{\lambda}(x)}{\partial \nu}\phi(x)-  \frac{\partial \phi(x)}{\partial \nu}k_{\lambda}(x)\Big) \,d\sigma = -\int_{|x|\leq r} \big( h_\lambda^*(x) +h_\lambda^{**}(x) \big) \phi(x) \,dx
=O\big(1\big).
\end{split}
\end{equation*}
Hence similar to \eqref{v-phi-2}--\eqref{r-bound-2}, we have
$(1+\widehat{r}_\lambda)^{\tau_1}M_\lambda= O\big(M_\lambda\big)$,
 which implies \eqref{hat-r}.
\vskip 0.2cm

Now integrating \eqref{k-equa}, we get
\begin{equation}\label{iden-philamda}
-\Delta \varphi_{\lambda}=2\varphi_{\lambda}e^{2U(x)}+\frac{1}{2\pi}\displaystyle\int^{2\pi}_0 h^*_\lambda(r,\theta) \,d\theta +\frac{1}{2\pi}\displaystyle\int^{2\pi}_0 h^{**}_\lambda(r,\theta) \,d\theta,~\,~\mbox{for}~|x|\leq \frac{d_0}{\theta_{\lambda}}.
\end{equation}
Define $\varphi^*_{\lambda}(x)=\frac{\varphi_{\lambda}(|x|)}{\varphi_{\lambda}(\widehat{r}_\lambda)}$, we have
\begin{equation*}
|\varphi^*_{\lambda}(x)|\leq \frac{C\big|\varphi_{\lambda}(|x|)\big|}{M_\lambda}\leq C\big(1+|x|\big)^{\tau_1}.
\end{equation*}
Also we know
\begin{equation*}
\frac{1}{\varphi_{\lambda}(\widehat{r}_\lambda)} \displaystyle\int^{2\pi}_0h^*_\lambda(r,\theta) \,d\theta \leq  \frac{C}{M_\lambda}\to 0 ~~\,\,\mbox{and}~~\,\,
\frac{1}{\varphi_{\lambda}(\widehat{r}_\lambda)} \displaystyle\int^{2\pi}_0 h^{**}_\lambda(r,\theta) \,d\theta \leq  \frac{C}{M_\lambda \gamma_\lambda^2}\to 0.
\end{equation*}
Hence from the above computations and dominated convergence theorem, passing to the limit in the equation \eqref{iden-philamda} divided by $\varphi_{\lambda}(\widehat{r}_\lambda)$, we can deduce that $$\varphi^*_{\lambda}\to \varphi(|x|)~~\mbox{in}~~ C^2_{loc}(\R^2),$$ and $\varphi$ satisfies
\begin{equation*}
-\varphi''-\frac{1}{r}\varphi'=2e^{2U}\varphi.
\end{equation*}
Therefore, it follows
\begin{equation}\label{phi}
\varphi=c_0\frac{4-|x|^2}{4+|x|^2},~\mbox{with some constant}~c_0.
\end{equation}
Since $\varphi^*_{\lambda}{ (\widehat{r}_\lambda)}=1$ and $\widehat{r}_\lambda \leq C$, we find $\varphi(|x|)\not\equiv 0$.

\vskip 0.1cm

On the other hand, we know that
\begin{equation}\label{phi0}
\varphi_{\lambda}(0) =k_{\lambda}(0)=\gamma^2_\lambda \big(v_{\lambda}(0)-v_0(0)\big)= -\gamma^2_\lambda v_0(0) = -\gamma^2_\lambda \lim_{\lambda\to 0} v_{\lambda}(0)=0.
\end{equation}
Hence $\varphi_{\lambda}(0)=0$ and then we have $\varphi(0)=0$, which together with \eqref{phi} implies $\varphi\equiv 0$. This is a contradiction and we complete the proof of \eqref{k-inequa}.
\end{proof}
\vskip 0.2cm

\section{Proof of  Proposition \ref{prop_slambda}}\label{s-s}

\renewcommand{\theequation}{B.\arabic{equation}}
\setcounter{equation}{0}
\vskip 0.2cm

Let $L_\lambda:=\displaystyle\max_{|x|\leq \frac{d_0}{\theta_{\lambda}}}\frac{|s_{\lambda}|}{(1+|x|)^{\tau_2}}$,
then $\eqref{s-inequa} \Leftrightarrow L_\lambda \leq C$.
We will prove that $L_\lambda\leq C$ by contradiction.  Set
\begin{equation*}
L_\lambda^*:=\max_{|x|\leq \frac{d_0}{\theta_{\lambda}}}\max_{|x'|=|x|}\frac{|s_{\lambda}(x)-s_{\lambda}(x')|}{(1+|x|)^{{\tau_2}}}.
\end{equation*}

\begin{Lem}\label{lemB.2}
If $L_\lambda\to +\infty$, then it holds
\begin{equation}\label{B.2.1}
L_\lambda^*=o(1)L_\lambda.
\end{equation}
\end{Lem}

\begin{proof}
Suppose this is not true. Then there exists $c_0>0$ such that $L^*_\lambda\geq c_0L_\lambda$. Let $z'_\lambda$ and
$z''_\lambda$ satisfy $|z'_\lambda|=|z''_\lambda| {\leq\frac{d_{0}}{\theta_{\lambda}}}$ and
\begin{equation*}
L_\lambda^*=  \frac{|s_{\lambda}(z'_\lambda)-s_{\lambda}(z''_\lambda)|}{(1+|z'_\lambda|)^{{\tau_2}}}.
\end{equation*}
Without loss of generality, we may assume that $z'_\lambda$ and
$z''_\lambda$ are symmetric with respect to the $x_1$ axis. Set
\begin{equation*}
p^*_\lambda(x):=s_{\lambda}(x)-s_{\lambda}(x^-),~\,~x^-:=(x_1,-x_2),\  \mbox{ for }x=(x_{1},x_{2}),  ~\,x_2>0.
\end{equation*}
 Hence $z''_{\lambda}={z'_{\lambda}}^{-}$ and
$L_\lambda^*=\frac{|p^*_\lambda(z'_{\lambda})| }{ (1+|z'_{\lambda}| )^{\tau_2} }$.
Let us define
\begin{equation}\label{plambda-def}
p_\lambda(x):=\frac{p^*_\lambda(x)}{(1+x_2)^{\tau_2}}.
\end{equation}
Then direct calculations show that $p_\lambda$ satisfies
\begin{equation}\label{plambda-equa}
-\Delta p_\lambda-\frac{2{\tau_2}}{1+x_2}\frac{\partial p_\lambda}{\partial x_2}+
\frac{{\tau_2}(1-{\tau_2})}{(1+x_2)^2}p_\lambda=\frac{q_\lambda(x)}{(1+x_2)^{\tau_2}}+2e^{2U(x)}p_\lambda,
\end{equation}
where $q_\lambda(x):=q^*_\lambda(x)-q^*_\lambda(x^-)$ and $\displaystyle q^*_\lambda(x) =O\Big(
\frac{1}{1+|x|^{4-\delta-3\tau}}\Big)$.
Also let $z^{**}_\lambda$ satisfy
\begin{equation}\label{L_ast2-def}
\mbox{$|z^{**}_\lambda|\leq\frac{d_{0}}{\theta_{\lambda}}$, $z^{**}_{\lambda,2}\geq 0$ and}~~~ \big|p_\lambda(z^{**}_\lambda)\big|= L^{**}_\lambda:=\max_{|x|\leq \frac{d_0}{\theta_{\lambda}},~x_2\geq 0} |p_\lambda(x)|.
\end{equation}
Then it follows
\begin{equation}\label{L_ast2-2}
L^{**}_\lambda \geq \frac{\big|p^*_\lambda(z'_\lambda)\big|}{\big(1+z'_{\lambda,2}\big)^{\tau_2}}
\geq \frac{\big|s_{\lambda}(z'_\lambda)-s_{\lambda}(z''_\lambda)\big|}{\big(1+|z'_\lambda|\big)^{\tau_2}}=L^*_\lambda\to +\infty.
\end{equation}
We may assume that $p_\lambda(z^{**}_\lambda)>0$. We claim that
\begin{equation}\label{z_ast2-bound}
|z^{**}_\lambda|\leq C.
\end{equation}
The proof of  \eqref{z_ast2-bound} is divided into   two steps.

\vskip 0.2cm

\noindent
\emph{Step 1. We prove that $|z^{**}_\lambda|\leq \frac{d_0}{2\theta_\lambda}$.}

\vskip 0.2cm
Suppose this is not true. By the definition of  $s_{\lambda}$ and \eqref{klambda-}, we have
\begin{equation}\label{plambda_ast}
\begin{split}
p_{\lambda}^{*}(z^{**}_\lambda) &=s_{\lambda}(z^{**}_\lambda)-s_{\lambda}(z^{**-}_\lambda)
= \gamma^2_\lambda\Big(\big(k_{\lambda}(z^{**}_\lambda)-k_{\lambda}(z^{**-}_\lambda)\big)
-\big(k_0(z^{**}_\lambda)-k_0(z^{**-}_\lambda)\big)\Big)\\ &=
 O\Big( \gamma^6_\lambda{\theta_\lambda z^{**}_{\lambda,2}}\Big)+o\Big( \gamma^6_\lambda \theta_{\lambda}\Big)+O\Big(\gamma^4_\lambda \big|z_\lambda^{**}\big|^{-1} \Big)
 -\gamma^2_\lambda\Big(k_0(z^{**}_\lambda)-k_0(z^{**-}_\lambda)\Big).
\end{split}
\end{equation}
Also we recall that
$k_0$ solves the non-homogeneous linear equation
\begin{equation*}
-\Delta v-2e^{2U}v=  e^{2U } \Big(\frac{1}{2}U^4+2v_0^2+ U^3+  v_0\big(2U^2+4U+1\big) \Big)~~\mbox{in}~~\R^2,
\end{equation*}
and $v_0(x)$ can be written as
\begin{equation*}
\begin{split}
v_0(x)=\overline{v}(x)+ \sum^2_{i=1}{c_i}\frac{\partial U}{\partial x_i},
\end{split}
\end{equation*}
where $\overline{v}(x)$ is the radial solution of $-\Delta u-2e^{2U}u=e^{2U}(U^2+U)$.
Now we write
$k_0=k_1+k_2$ with
\begin{equation*}
\begin{cases}
\displaystyle -\Delta k_1-2e^{2U}k_1=  e^{2U } \Big(\frac{1}{2}U^4+2 \overline{v}^2+ U^3+ \overline v \big(2U^2+4U+1\big)
\Big),&\mbox{in}~~\R^2,\\[4mm]
\displaystyle -\Delta k_2-2e^{2U}k_2=  e^{2U} \left( 2 \Big(\displaystyle\sum^2_{i=1}{c_i}\frac{\partial U}{\partial x_i}\Big)^2 +4\overline{v} \displaystyle\sum^2_{i=1}{c_i}\frac{\partial U}{\partial x_i} +\big(2U^2+4U+1\big)\displaystyle\sum^2_{i=1}{c_i}\frac{\partial U}{\partial x_i} \right),&\mbox{in}~~\R^2.
\end{cases}
\end{equation*}
Then we have
\begin{equation}\label{k1}
 k_1(x) = w(x) + e_{1,0}\frac{4-|x|^2}{4+|x|^2} + \sum^2_{i=1}{e_{1,i}}\frac{x_i}{4+|x|^2},
\end{equation}
where $w(x)$ is the radial solution of
$$-\Delta u - 2e^{2U}u = e^{2U} \Big(\frac{1}{2}U^4+2 \overline{v}^2+ U^3+ \overline v \big(2U^2+4U+1\big)\Big).$$
On the other hand, we find that
\begin{equation*}
-\Delta(c_i \overline{v})-2e^{2U}(c_i \overline{v}) = c_i e^{2U}\big(U^2+U\big),\quad i=1,2,
\end{equation*}
which implies that
\begin{equation*}
-\Delta\Big(c_i \frac{\partial\overline{v}}{\partial x_i}\Big) -2e^{2U}\Big(c_i \frac{\partial\overline{v}}{\partial x_i}\Big) = c_i e^{2U}\big(2U^2+4U+1\big)\frac{\partial U}{\partial x_i} + 4e^{2U}(c_i \overline{v}) \frac{\partial U}{\partial x_i},\quad i=1,2.
\end{equation*}
Then, we derive that
\begin{equation}\label{k2-1}
-\Delta\Big(\sum_{i=1}^2 c_i \frac{\partial\overline{v}}{\partial x_i}\Big) -2e^{2U}\Big(\sum_{i=1}^2 c_i \frac{\partial\overline{v}}{\partial x_i}\Big) = e^{2U}\big(2U^2+4U+1\big) \sum_{i=1}^2 c_i \frac{\partial U}{\partial x_i} + 4e^{2U} \overline{v} \sum_{i=1}^2 c_i \frac{\partial U}{\partial x_i}.
\end{equation}
Moreover, we know that
\[
-\Delta \frac{\partial U}{\partial x_1} -2e^{2U}\frac{\partial U}{\partial x_1}=0,
\]
which yields that
\[
-\Delta \frac{\partial^2 U}{\partial x_1 \partial x_2} -2e^{2U} \frac{\partial^2 U}{\partial x_1 \partial x_2} = 4e^{2U}\Big(\frac{\partial U}{\partial x_1} \frac{\partial U}{\partial x_2}\Big).
\]
It means that
\begin{equation}\label{k2-2}
-\Delta \Big(c_1 c_2 \frac{\partial^2 U}{\partial x_1 \partial x_2}\Big) - 2e^{2U} \Big(c_1 c_2 \frac{\partial^2 U}
{\partial x_1 \partial x_2}\Big) = 4e^{2U}\Big(c_1 c_2 \frac{\partial U}{\partial x_1} \frac{\partial U}{\partial x_2}\Big).
\end{equation}
Similarly, for $i=1,2$, we have
\[
-\Delta \Big(\frac{c_i^2}{2} \frac{\partial^2 U}{\partial x_i^2}\Big) - 2e^{2U} \Big(\frac{c_i^2}{2} \frac{\partial^2 U}
{\partial x_i^2}\Big) = 2e^{2U} \Big(c_i \frac{\partial U}{\partial x_i}\Big)^2,
\]
that is
\begin{equation}\label{k2-3}
-\Delta \Big(\sum_{i=1}^2 \frac{c_i^2}{2} \frac{\partial^2 U}{\partial x_i^2}\Big) - 2e^{2U} \Big(\sum_{i=1}^2 \frac{c_i^2}{2} \frac{\partial^2 U}{\partial x_i^2}\Big) = 2e^{2U} \sum_{i=1}^2 c_i^2 \Big(\frac{\partial U}{\partial x_i}\Big)^2.
\end{equation}

Therefore, it follows from \eqref{k2-1}, \eqref{k2-2} and \eqref{k2-3} that
\[
k_2 = \sum_{i=1}^2 c_i \frac{\partial\overline{v}}{\partial x_i} + \frac{1}{2} \sum_{i=1}^2 c_i^2 \frac{\partial^2 U}{\partial x_i^2} + c_1 c_2 \frac{\partial^2 U}{\partial x_1 \partial x_2} + e_{2,0}\frac{4-|x|^2}{4+|x|^2} + \sum_{i=1}^2 e_{2,i}\frac{x_i}{4+|x|^2},
\]
which together with \eqref{k1} implies
\begin{equation*}
k_0(x) = w(x) + \sum^2_{i=1} e_i \frac{x_i}{4+|x|^2} + e_0 \frac{4-|x|^2}{4+|x|^2} + \sum_{i=1}^2 c_i \frac{\partial\overline{v}}{\partial x_i} + \frac{1}{2} \sum_{i=1}^2 c_i^2 \frac{\partial^2 U}{\partial x_i^2} + c_1 c_2 \frac{\partial^2 U}{\partial x_1 \partial x_2}.
\end{equation*}
Then
\begin{equation*}
\begin{split}
k_0(z^{**}_\lambda)-k_0(z^{**-}_\lambda) &= e_2 \frac{2 z^{**}_{\lambda,2}}{4+|z^{**}_\lambda|^2} + \sum_{i=1}^2 c_i \Big( \frac{\partial\overline{v}}{\partial x_i}(z^{**}_\lambda) - \frac{\partial\overline{v}}{\partial x_i}(z^{**-}_\lambda) \Big) + c_1 c_2 \frac{8 z^{**}_{\lambda,1} z^{**}_{\lambda,2}}{(4+|z^{**}_\lambda|^2)^2} \\
&= 2 e_2 \frac{z^{**}_{\lambda,2}}{4+|z^{**}_\lambda|^2} + 8 c_1 c_2 \frac{z^{**}_{\lambda,1} z^{**}_{\lambda,2}}{(4+|z^{**}_\lambda|^2)^2} + 2 c_2 \partial_r \overline{v}\big(|z^{**}_\lambda|\big) \frac{z^{**}_{\lambda,2}}{|z^{**}_\lambda|}.
\end{split}
\end{equation*}
Since we suppose that $|z^{**}_\lambda| > \frac{d_0}{2\theta_\lambda}$ and $\overline{v}$ is a radial function solves
\[
-\Delta u-2e^{2U}u=e^{2U}\big(U^2+U\big),
\]
similar to the proof of \eqref{parw-4}, we know that
$\partial_r \overline{v} \big(|z^{**}_\lambda|\big) = O\big( |z^{**}_\lambda|^{-1} \big)$.
Hence
\[
k_0(z^{**}_\lambda)-k_0(z^{**-}_\lambda) = O\Big( \big|z^{**}_\lambda \big|^{-1} \Big),
\]
and inserting it into \eqref{plambda_ast}, we get
\[
p_{\lambda}^{*}(z^{**}_\lambda) = O\Big( \gamma^6_\lambda{\theta_\lambda z^{**}_{\lambda,2}}\Big)+o\Big( \gamma^6_\lambda \theta_{\lambda}\Big)+O\Big(\gamma^4_\lambda \big|z_\lambda^{**}\big|^{-1} \Big).
\]
As a result, since $\frac{d_0}{2\theta_\lambda}<|z_\lambda^{**}| \leq \frac{d_0}{\theta_\lambda}$, we derive
\begin{equation*}
\begin{split}
L^{**}_\lambda &= \frac{p^*_\lambda(z_\lambda^{**})}{\big(1+z^{**}_{\lambda,2}\big)^{\tau_2}}
= O\Big( \gamma^6_\lambda \theta_\lambda \big(z^{**}_{\lambda,2}\big)^{1-\tau_2} \Big) + o\big(\gamma^6_\lambda \theta_\lambda\big) + O\Big(\gamma^4_\lambda \big|z^{**}_\lambda \big|^{-1}\Big) \\
&=  O\big( \gamma^6_\lambda \theta_\lambda^{\tau_2}\big) + o\big(\gamma^6_\lambda \theta_\lambda\big) + O\big(\gamma^4_\lambda \theta_\lambda \big) = o\big(1\big).
\end{split}
\end{equation*}
This is a contradiction with \eqref{L_ast2-2}.

\vskip 0.2cm

\noindent
\emph{Step 2. We prove that $|z^{**}_\lambda|\leq C$.}

\vskip 0.2cm
 Suppose this is not true.
Now by Step 1, we have $z^{**}_\lambda\in B_{\frac{d_0}{2\theta_\lambda}}(0)$. Thus
\begin{equation}\label{4.7-step2-1}
|\nabla p_\lambda(z^{**}_\lambda)|=0~\mbox{and}~\Delta p_\lambda(z^{**}_\lambda)\leq 0.
\end{equation}
So similar to \eqref{y-sdsd1}, from \eqref{plambda-equa}, \eqref{4.7-step2-1} and the fact that $|z^{**}_\lambda|\to \infty$, we find
\begin{equation}\label{4.7-step2-2}
\begin{split}
0\leq -\Delta p_\lambda(z^{**}_\lambda) &= -\frac{\tau_2(1-\tau_2)}{(1+z^{**}_{\lambda,2})^2}p_\lambda(z^{**}_\lambda)
+\frac{q_\lambda(z^{**}_\lambda)}{(1+z^{**}_{\lambda,2})^{\tau_2}}+2e^{2U(z_\lambda^{**})}p_\lambda(z^{**}_\lambda) \\
&\leq -\frac{\tau_2(1-\tau_2)}{2(1+z^{**}_{\lambda,2})^2}p_\lambda(z^{**}_\lambda)
+\frac{q_\lambda(z^{**}_\lambda)}{(1+z^{**}_{\lambda,2})^{\tau_2}}.
\end{split}
\end{equation}
Hence using \eqref{4.7-step2-2}, we deduce
\begin{equation*}
\frac{p_\lambda(z^{**}_\lambda)}{(1+z^{**}_{\lambda,2})^2}\leq \frac{Cq_\lambda(z^{**}_\lambda)}{(1+z^{**}_{\lambda,2})^{\tau_2}},
\end{equation*}
that is
\begin{equation}\label{4.7-step2-3}
p_\lambda(z^{**}_\lambda) \leq C q_\lambda(z^{**}_\lambda)(1+z^{**}_{\lambda,2})^{2-\tau_2}.
\end{equation}
Since $q_\lambda(x):=q^*_\lambda(x)-q^*_\lambda(x^-)$ with $\displaystyle q^*_\lambda(x) =O\Big(\frac{1}
{1+|x|^{4-\delta-3\tau}}\Big)$, it follows from \eqref{4.7-step2-3} that
\begin{equation*}
L^{**}_\lambda = p_\lambda(z^{**}_\lambda) \leq
\frac{C}{(1+|z_\lambda^{**}|)^{2+\tau_2-\delta-3\tau}}\to 0~~\mbox{as}~~\lambda\to 0,
\end{equation*}
which is a contradiction with \eqref{L_ast2-2}. Then \eqref{z_ast2-bound} holds.\\

Now let $p_\lambda^{**}(x)=\frac{p_\lambda(x)}{L_\lambda^{**}}$,
where $p_\lambda$ is defined in \eqref{plambda-def} and $L_\lambda^{**}$ is defined in \eqref{L_ast2-def}.
From \eqref{plambda-equa}, it follows that $p_\lambda^{**}(x)$ solves
\begin{equation}\label{p_ast2-equa}
-\Delta p^{**}_\lambda-\frac{2\tau_2}{1+x_2}\frac{\partial p^{**}_\lambda}{\partial x_2}+ \frac{\tau_2(1-\tau_2)}
{(1+x_2)^2}p^{**}_\lambda=\frac{q_\lambda(x)}{L_\lambda^{**}(1+x_2)^{\tau_2}}+2e^{2U(x)}p^{**}_\lambda.
\end{equation}
Moreover $p^{**}_{\lambda}(z^{**}_\lambda)=1$ and
\[
|p^{**}_{\lambda}(x)|\leq 1 ~\mbox{in}~ \Big\{x:|x|\leq\frac{d_0}{\theta_\lambda},~x_2\geq 0\Big\},
\]
hence $p^{**}_\lambda(x)\to \gamma(x)$ uniformly in any compact subset of $\R^2$. And $\gamma(x)\not\equiv0$ because $p_\lambda^{**}(z_{\lambda}^{**})=1$ and  $|z_{\lambda}^{**}|\le C$. Observe that
$$\frac{q_\lambda(x)}{L_\lambda^{**}(1+x_2)^{\tau_2}}\to 0~\mbox{uniformly in any compact set of}~\Big\{x:|x|\leq\frac{d_0}{\theta_\lambda},~x_2\geq 0\Big\}~\mbox{as}~\lambda\to 0.$$
Hence passing to the limit $\lambda\to 0$ into \eqref{p_ast2-equa},  we can deduce that   $\gamma$ solves
\begin{equation*}
\begin{cases}
\displaystyle -\Delta \gamma -\frac{2\tau_2}{1+x_2}\frac{\partial \gamma}{\partial x_2}
+\left(\frac{\tau_2(1-\tau_2)}{(1+x_2)^2}-2e^{2U(x)} \right)\gamma =0,~x_2>0,\\[4mm]
\displaystyle \gamma(x_1,0)=0.
\end{cases}
\end{equation*}
Hence $\bar \gamma=(1+x_2)^{\tau_2}\gamma$ satisfies
\begin{equation*}
\begin{cases}
-\Delta \bar \gamma  -2e^{2U(x)} \bar \gamma =0,~x_2>0,\\[2mm]
\bar\gamma(x_1,0)=0,
\end{cases}
\end{equation*}
and we have
\begin{equation}\label{gamma1_bar-equa}
\bar \gamma(x)=c_0\frac{4-|x|^2}{4+|x|^2}+\sum^2_{i=1}{c_i}\frac{x_i}{4+|x|^2},~~
\mbox{for some constants $c_0$, $c_1$ and $c_2$}.
\end{equation}
On the other hand, \eqref{add_k0} and \eqref{add_k01} imply that
\[
\nabla \big(k_\lambda(x)- k_\lambda(x^-)\big)\big|_{x=0}=0.
\]
Also by \eqref{k-lim}, we can verify that
\begin{equation*}
\begin{split}
\nabla \big(k_0(x)-k_0(x^-)\big)\big|_{x=0}=0.
\end{split}
\end{equation*}
Then by the definition of $s_{\lambda}(x)$ and $p^*_\lambda(x)$, we know
\begin{equation*}
\begin{split}
\nabla p_\lambda^*(0) &= \nabla \big(s_{\lambda}(x)-s_{\lambda}(x^-)\big)\big|_{x=0} \\
&= \gamma^2_\lambda \big(\nabla k_{\lambda}(x)-\nabla k_0(x)\big)\big|_{x=0}-
\gamma^2_\lambda \big(\nabla k_{\lambda}(x^-)-\nabla k_0(x^-)\big)\big|_{x=0} =0.
\end{split}
\end{equation*}
Hence it holds
\begin{equation*}
\nabla \big((1+x_2)^{\tau_2}p_\lambda^{**}(x)\big)\big|_{x=0}=
\frac{\nabla p_\lambda^*(0)}{L_\lambda^{**}}= 0,
\end{equation*}
which means $\nabla \bar\gamma(0)=0$. Then from \eqref{gamma1_bar-equa}, we find $c_1=c_2=0$.
Moreover from $\bar\gamma(x_1,0)=0$, we also have $c_0=0$. So $\bar\gamma=0$, which is a contradiction.
Hence \eqref{B.2.1} follows.
\end{proof}

 \vskip 0.1cm

\begin{proof}[\underline{\textbf{Proof of Proposition \ref{prop_slambda}}}] First we recall that
$$\eqref{s-inequa} \Leftrightarrow L_\lambda \leq C.$$
Suppose by contradiction that $L_\lambda\to +\infty$, as $\lambda\to 0$ and set
\begin{equation*}
 \overline{s}_{\lambda}(r)=\frac{1}{2\pi}\int^{2\pi}_0s_{\lambda}(r,\theta) \,d\theta,~\,\,r=|x|.
\end{equation*}
Then by Lemma \ref{lemB.2}, similar to the proof of \eqref{psi-lamda}, it follows
\begin{equation*}
\max_{r\leq \frac{d_0}{\theta_{\lambda}}}\frac{|\overline{s}_{\lambda}(r)|}{(1+r)^{\tau_2}}=L_\lambda\big(1+o(1)\big).
\end{equation*}
Assume that $\frac{|\overline{s}_{\lambda}(r)|}{(1+r)^{\tau_2}}$ attains its maximum at $t_\lambda$.
Then we claim
\begin{equation}\label{t-bound}
t_\lambda\leq C.
\end{equation}
In fact, let $\phi(x)=\frac{4-|x|^2}{4+|x|^2}$ and  we recall that
$-\Delta\phi(x)=2e^{2U(x)}\phi(x)$.
Now multiplying \eqref{s-equa} by $\phi(x)$ and using integration by parts, we have
\begin{equation*}
\begin{split}
\int_{|x|=r}&\Big(\frac{\partial s_{\lambda}(x)}{\partial \nu}\phi(x)-  \frac{\partial \phi(x)}{\partial \nu}s_{\lambda}(x)\Big) \,d\sigma=O\big(1\big).
\end{split}
\end{equation*}
Hence similar to \eqref{v-phi-2}--\eqref{r-bound-2}, we have
$(1+t_\lambda)^{\tau_2}L_\lambda= O\big(L_\lambda\big)$,
 which implies \eqref{t-bound}.
\vskip 0.2cm

Now integrating \eqref{s-equa}, we get
\begin{equation}\label{equazionePsid}
-\Delta \overline{s}_{\lambda} =2\overline{s}_{\lambda}e^{2U(x)}+\frac{1}{2\pi}\displaystyle\int^{2\pi}_0q^*_\lambda(r,\theta) \,d\theta, ~\,~\mbox{for}~|x|\leq \frac{d_0}{\theta_{\lambda}}.
\end{equation}
Next we define $s^*_{\lambda}(x)=\frac{\overline{s}_{\lambda}(|x|)}{\overline{s}_{\lambda}(t_\lambda)}$ and   pass to the limit in the equation \eqref{equazionePsid} divided by $\overline{s}_{\lambda}(t_\lambda)$. One has
\begin{equation*}
|s^*_{\lambda}(x)|\leq \frac{C\big|\overline{s}_{\lambda}(|x|)\big|}{L_\lambda}\leq C\big(1+|x|\big)^{\tau_2}.
\end{equation*}
Also we know
\begin{equation*}
\frac{1}{\overline{s}_{\lambda}(t_\lambda)} \displaystyle\int^{2\pi}_0q^*_\lambda(r,\theta) \,d\theta \leq  \frac{C}{L_\lambda}\to 0.
\end{equation*}
Hence from the above computations and dominated convergence theorem, we can deduce that $$s^*_{\lambda}\to s(x)~~\mbox{in}~~ C^2_{loc}(\R^2),$$ with $s(x)=s(|x|)$ satisfying
$-s''-\frac{1}{r}s'=2e^{2U}s$.
Therefore, it follows
\begin{equation}\label{s-res}
s(x)=c_0\frac{4-|x|^2}{4+|x|^2},~\mbox{with some constant}~c_0.
\end{equation}
Since $s^*_{\lambda}{ (t_{\lambda})}=1$ and $t_\lambda\leq C$, we find $s(|x|)\not\equiv 0$.

\vskip 0.1cm

On the other hand, we know that
\[
\overline{s}_{\lambda}(0) =s_{\lambda}(0) =\gamma^2_\lambda \big(k_{\lambda}(0)-k_0(0)\big) \overset{\eqref{phi0}}= -\gamma^2_\lambda k_0(0) = -\gamma^2_\lambda \lim_{\lambda\to 0} k_{\lambda}(0) =0.
\]
Hence $s^*_{\lambda}(0)=0$ and then we have $s(0)=0$, which together with \eqref{s-res}, implies $s(x)\equiv 0$. This is a contradiction. So we complete the proof of \eqref{s-inequa}.
\end{proof}

\noindent\textbf{Acknowledgments}
Peng Luo  and  Shuangjie Peng were supported by National Natural Science Foundation of China (No.11831009).
Peng Luo  was also   supported  by National Natural Science Foundation of China (No.12171183).

\renewcommand\refname{References}
\renewenvironment{thebibliography}[1]{%
\section*{\refname}
\list{{\arabic{enumi}}}{\def\makelabel##1{\hss{##1}}\topsep=0mm
\parsep=0mm
\partopsep=0mm\itemsep=0mm
\labelsep=1ex\itemindent=0mm
\settowidth\labelwidth{\small[#1]}%
\leftmargin\labelwidth \advance\leftmargin\labelsep
\advance\leftmargin -\itemindent
\usecounter{enumi}}\small
\def\newblock{\ }
\sloppy\clubpenalty4000\widowpenalty4000
\sfcode`\.=1000\relax}{\endlist}
\bibliographystyle{model1b-num-names}

\end{document}